\documentclass[11pt,reqno]{amsart}
\usepackage[colorlinks=true,allcolors=blue,backref=page]{hyperref}
\usepackage{color}
\usepackage{booktabs}
\usepackage{longtable}
\usepackage{array}
\usepackage{amsmath, amssymb, amsthm}

\usepackage{mathrsfs}
\usepackage{mathtools}
\usepackage[noabbrev,capitalize,nameinlink]{cleveref}
\crefname{equation}{}{}
\usepackage{fullpage}
\usepackage[noadjust]{cite}
\usepackage{graphics}
\usepackage{pifont}
\usepackage{tikz}
\usepackage{tikz-cd}
\usepackage{bbm}
\usepackage[T1]{fontenc}

\DeclareMathOperator*{\argmin}{arg\,min}
\usetikzlibrary{arrows.meta}

\usepackage{environ}
\usepackage{framed}
\usepackage{url}
\usepackage[linesnumbered,ruled,vlined]{algorithm2e}
\usepackage[noend]{algpseudocode}
\usepackage[labelfont=bf]{caption}
\usepackage{cite}
\usepackage{framed}
\usepackage[framemethod=tikz]{mdframed}
\usepackage{appendix}
\usepackage{graphicx}
\usepackage[textsize=tiny]{todonotes}
\usepackage{tcolorbox}
\usepackage{enumerate}
\usepackage[shortlabels]{enumitem}
\allowdisplaybreaks[1]

\apptocmd{\sloppy}{\hbadness 10000\relax}{}{} 

\crefname{algocf}{Algorithm}{Algorithms}

\crefname{equation}{}{} 
\crefname{conjecture}{Conjecture}{Conjectures} 
\AtBeginEnvironment{appendices}{\crefalias{section}{appendix}} 

\crefformat{enumi}{#2#1#3}
\crefrangeformat{enumi}{#3#1#4 to~#5#2#6}
\crefmultiformat{enumi}{#2#1#3}%
{ and~#2#1#3}{, #2#1#3}{ and~#2#1#3}

\usepackage[color,final]{showkeys} 

\colorlet{refkey}{orange!20}
\colorlet{labelkey}{blue!30}

\crefname{algocf}{Algorithm}{Algorithms}

\numberwithin{equation}{section}
\newtheorem{theorem}{Theorem}[section]
\newtheorem{proposition}[theorem]{Proposition}
\newtheorem{lemma}[theorem]{Lemma}

\crefname{claim}{Claim}{Claims}

\newtheorem{corollary}[theorem]{Corollary}

\newtheorem*{question*}{Question}

\theoremstyle{definition}
\newtheorem{definition}[theorem]{Definition}

\newtheorem*{definition*}{Definition}

\theoremstyle{remark}
\newtheorem*{remark}{Remark}
\newtheorem*{remarks}{Remarks}

\newcommand\Av{\operatorname{Av}}

\newcommand{\mb}{\mathbb}
\newcommand{\mbf}{\mathbf}

\newcommand{\mf}{\mathfrak}

\newcommand{\on}{\operatorname}
\newcommand{\wh}{\widehat}

\newcommand\edge{\operatorname{edge}}
\newcommand\nonedge{\operatorname{nonedge}}
\newif\ifshowedits
\showeditstrue   

\newcommand{\eps}{\varepsilon}

\newcommand\Supp{\operatorname{Supp}}

\newcommand\Pois{\operatorname{Pois}}
\newcommand\rp{\operatorname{rp}}

\newcommand{\dist}[2]{{\operatorname{d}_{\operatorname{BL}}[#1;#2]}}
\def\BL{\operatorname{BL}}
\def\boldeps{\boldsymbol\eps}
\def\boldbeta{\boldsymbol\beta}
\def\boldgamma{\boldsymbol\gamma}
\def\boldxi{\boldsymbol\xi}

\renewcommand{\le}{\leqslant}
\renewcommand{\ge}{\geqslant}
\renewcommand{\Re}{\on{Re}}

\newcommand\Z{\mathbb{Z}}
\newcommand\Q{\mathbb{Q}}
\newcommand\C{\mathbb{C}}
\newcommand\R{\mathbb{R}}
\newcommand\N{\mathbb{N}}

\newcommand\jump{\operatorname{jump}}

\def\sml{\operatorname{small}}
\def\lrg{\operatorname{large}}
\def\med{\operatorname{med}}

\usepackage{bm}

\allowdisplaybreaks

\newcommand\samedist{\overset{d}{=}}

\newcommand\Lip{\operatorname{Lip}}

\title{The proportion of permutations fixing a $k$-set}
\author[A1]{Ben Green}
\address{Mathematical Institute, Andrew Wiles Building, Radcliffe Observatory Quarter, Woodstock Rd, Oxford OX2 6QW, UK}
\email{ben.green@maths.ox.ac.uk}

\author[A2]{Mehtaab Sawhney}
\address{Department of Mathematics, Columbia University and OpenAI}
\email{m.sawhney@columbia.edu}
\begin{document}

\begin{abstract} 
Denote by $p(k)$ the limit, as $n \rightarrow \infty$, of the probability that a random permutation on a set of size $n$ has an invariant set of size $k$. We give an asymptotic formula for $p(k)$, showing that it is asymptotically $f(\{\log_2 k\}) k^{-\delta} (\log k)^{-3/2}$ where $\delta = 1 - \frac{1 + \log \log 2}{\log 2} \approx 0.086$ and $f$ is a smooth, positive, function on $\R/\Z$, which we will describe explicitly. The function $f$ satisfies $\frac{\max f}{\min f} < 1 + 2 \times 10^{-7}$ and we conjecture that it is not constant.

Estimating $p(k)$ is a model for the more well-known question which asks for an estimation of $M(n)$, the number of distinct elements in the $n$-by-$n$ multiplication table. By elaborating on the techniques in this paper, we will give an asymptotic for $M(n)$ in forthcoming work.
\end{abstract}

\maketitle
\setcounter{tocdepth}{1}

\tableofcontents

\part{Introduction and setup}

\section{Introduction}

\subsection{Main result and history}
The well-known multiplication table problem of Erd\H{o}s \cite{erdos-mult-1,erdos-mult-2} asks for an asymptotic for $M(n)$, the number of elements in the $n$-by-$n$ multiplication table. In forthcoming work \cite{green-sawhney-forthcoming} we offer a solution to this problem.

It is well-known that (at least in certain regimes) both the logs of the prime factors of a random integer and the cycle lengths of a random permutation have Poisson behaviour. A consequence of this is an analogy between certain problems about the divisors of typical integers (which the estimation of $M(n)$ can easily be reduced to) and problems about invariant sets of random permutations.  In our forthcoming paper we will go into much greater detail on this connection; the introduction to \cite{EFG16} may also be consulted. The objective of the present work is to fully explore the (limiting version of) the permutation analogue of the multiplication table problem, which we now describe.

Denote by $\mf{S}_n$ the symmetric group of permutations on $n$ letters. Let $i(n,k)$ denote the probability that $\pi\in \mf{S}_n$ has an invariant set of size $k$. Let $p(k) = \lim_{n\to \infty} i(n,k)$; this limit exists, as we will recall below, and we have (for example) $p(1) = 1 - \frac{1}{e}$ and $p(2) = 1 - 2 e^{-3/2}$. 

Our main interest in this paper will be the estimation of $p(k)$. We believe this is of interest in its own right, and indeed several previous works such as \cite{dfg,EFG16,luzcak-pyber} are devoted to the problem. However, our main rationale is that the methods we develop are relatively clean when it comes to estimating $p(k)$, but will adapt to solve the more well-known (but more technical) question of estimating $M(n)$. Large parts of the present paper can be reused almost verbatim for the latter question.

The following is our main theorem. Let 
\begin{equation}\label{eft-def} \delta = 1 - \frac{1+\log\log 2}{\log 2}\end{equation} be the ``Erd\H{o}s--Ford--Tenenbaum constant''. In this theorem, and throughout the paper, we write $\log_2 k$ for log to base 2 (\emph{not} $\log \log k$ as one occasionally sees in analytic number theory). We also write $\{x\}$ for the fractional part of a real number $x$, that is to say $0 \le \{x \} < 1$ and $x - \{x \} \in \Z$; we have $\{x\} = x - \lfloor x \rfloor$ where the $\lfloor x \rfloor$ is the greatest integer less than or equal to $x$.
\begin{theorem}\label{thm:main}
There exists $f \in C^{\infty}(\R/\Z)$, nowhere zero, such that 
\[ p(k) =  (1 + o(1)) f(\{ \log_2 k \}) k^{-\delta} (\log k)^{-3/2}.\]

The function $f$ can be expressed as $c_0 g \ast \mu \ast \mu'$, where $g \in C^{\infty}(\R/\Z)$ is the function defined by 
\begin{equation}\label{g-funct-def} g(x) := \sum_{D \in \Z} (\log 2)^{D - x} (1 - e^{-2^{D - x}}),\end{equation} $c_0 = (\frac{\pi}{2})^{1/2} (\log 2)^{3/2} e^{\gamma(\frac{1}{\log 2} - 1)}$, and $\mu, \mu'$ are two nonzero, positive Borel measures\footnote{\emph{Positive} means that $\mu(E) \ge 0$ for any measurable set $E$.} on $\R/\Z$ which we will define in \cref{prop16.2,sec17-main} below. We have
\[ \frac{\max f}{\min f} \le \frac{\max g}{\min g} < 1 + 2\cdot 10^{-7}.\]
\end{theorem}
This may be compared with the main result of \cite{EFG16} where, using the methods of Ford \cite{For08}, it was shown that $k^{-\delta} (\log k)^{-3/2} \ll p(k) \ll k^{-\delta} (\log k)^{-3/2}$. It is a slightly odd historical quirk of the subject that this was only worked out \emph{after} Ford proved the strictly more difficult corresponding result for $M(n)$. (That this should be possible was anticipated by Diaconis and Soundararajan, as noted in \cite{sound-pk-remark}.)

We make some remarks on the theorem.

\begin{enumerate} \item
 We presume that $f$ is not the constant function, but we have not been able to prove this. We will show that $\wh{g}(m) \ne 0$ for all $m \in \Z \setminus \{0\}$, which means that $f$ is constant if and only if $\mu \ast \mu'$ is the uniform measure. Proving this, or obtaining any numerical simulations of $\mu, \mu'$, remains an interesting open problem.

\item  The most interesting feature of \cref{thm:main} is arguably the appearance of a periodic function of $\{\log_2 k\}$. This sort of behaviour has been seen in related problems before, notably in the work of Balazard, Nicolas, Pomerance and Tenenbaum \cite{bnpt} where they examine the number of $n \le X$ with $\tau(n) \ge \log X$, where $\tau$ is the divisor function. Citing that paper, Brent, Pomerance, Purdum and Webster \cite{bppw} speculate that the asymptotic for $M(n)$ may have an oscillatory term.
For squarefree $n$, the condition $\tau(n) \ge \log X$ is the same as $\omega(n) \ge \log \log X/\log 2$, where $\omega$ denotes the number of prime factors. The asymptotic turns out to be dominated by the contribution of $n$ for which $\omega(n) = \log \log X/\log 2 + O(1)$. Since $\omega(n)$ is constrained to be an integer, the fractional part of $\log \log X/\log 2$ starts to have a significant effect on the distribution. The same phenomenon is at play in the present paper, with $\omega(n)$ replaced by the number of cycles of the permutation $\pi$ of length $\le k$, and the important contribution being from $\pi$ with $\log_2 k + O(1)$ such cycles.

\item The almost constant nature of $g$ seemed extremely surprising and mysterious to us at first sight. One explanation for this, as we shall see below, is that the nonzero Fourier coefficients $\wh{g}(m)$ are given by values of $\Gamma$ with imaginary part $2 \pi m /\log 2$, and so are very small when $m \ne 0$. Moreover, it transpires that similar phenomena have been observed on a number of occasions before, with the earliest source we are aware of being formula (2.11.2) in Hardy's explanation \cite{hardy-ramanujan} of how Ramanujan failed to prove the prime number theorem. Several further related references are given on \cite[p 311]{flajolet-sedgewick}. As pointed out to us by Soundararajan, the exact same function $g$ appears as $\psi_1^*$ in the work of Granville, Sedunova and Sabuncu \cite{granville-sedunova-sabuncu}; although the contexts are somewhat related, this appears to be at least partially coincidental. Finally, we remark that Sean Eberhard pointed out to us that similar functions can arise in rather basic problems, such as the following question on ``binomial thinning''. Suppose we start on day 0 with a population of size $k$. On day $i$, each remaining member of the population is eliminated with probability $1/2$, these events being independent. What is the probability that there is some day on which precisely one person remains? It is not hard to see that the answer tends to $g_0(\{ \log_2 k\})$ where $g_0(t) := \sum_{D \in \Z} 2^{D + t} e^{-2^{D + t}}$ as $k$ grows. 
\end{enumerate}

An immediate corollary of \cref{thm:main} is the following.

\begin{corollary}\label{powers-two-cor}
   There is a positive constant $C$ such that we have the asymptotic $p(k) = (C + o(1)) k^{-\delta} (\log k)^{-3/2}$ as $k \rightarrow \infty$ along powers of two.
\end{corollary}
Although \cref{powers-two-cor} is an immediate consequence of \cref{thm:main}, it is in fact a slightly easier result (we shall see the proof at the end of \cref{decouple-sec}).

\subsection{The $\varrho$ statistics and the measure $\mu$}\label{rho-stat-sec}
We turn now to a description of the measure $\mu$ in the statement of our main result.

Let $u = (u_i)_{i = 1}^{\infty}$ be a sequence of real numbers. In our applications, this will either be the arrival times of a rate 1 Poisson process on $[0,\infty)$ or a minor variant of this. Then for each positive integer $\ell$ we define
\begin{equation}\label{x-stat-def} \varrho_{u}(\ell) :=  2^{u_{\ell} - \ell}\sum_{\eps \in \{0,1\}^{\ell}} 1 \big(\sum_{i = 1}^{\ell} \eps_i 2^{-u_i} \in [1 - 2^{-u_{\ell}}, 1]\big),\end{equation} where $1(E)$ denotes the indicator of the event $E$.
Since we will typically have $u_i \approx i$, roughly speaking one expects $\varrho_{u}(\ell)$ to have size $\asymp 1$. One should think of the $\varrho_{u}(\ell)$ as describing some sort of ``density'' of the set of subset sums $\sum_{i \in I} 2^{-u_i}$ near $1$.

From now on we fix the \emph{upper process} $\mbf{u} = (\mbf{u}_i)_{i = 1}^{\infty}$ to be precisely the sequence of arrival times of a rate 1 Poisson process on $[0,\infty)$. The following is our main result concerning the measure $\mu$. Here, and throughout the paper, $x^+ := \max(x, 0)$.

\begin{proposition}\label{prop16.2}
Let $c := -\frac{\log \log 2}{\log 2}$. For every Lipschitz function $\psi \in \Lip(\R/\Z)$ the measure $\mu$ satisfies 
\begin{equation}\label{mu-formula} \int_{\R/\Z} \psi \, \mathrm{d}\mu := (2/\pi)^{1/2}\lim_{\ell \rightarrow \infty}\mb{E}(\ell - \mbf{u}_{\ell})^+ \varrho_{\mbf{u}}(\ell)^c \psi (\log_2 \varrho_{\mbf{u}}(\ell)) ,\end{equation} where the expectation $\mb{E}$ is over the random process $\mbf{u}$, and $x^+ := \max(x, 0)$. We adopt the convention that the integrand is zero when $\varrho_{\mbf{u}}(\ell) = 0$.

In particular, the Fourier coefficients of $\mu$ are given by
\[ \wh{\mu}(r) =  (2/\pi)^{1/2}\lim_{\ell \rightarrow \infty}\mb{E}(\ell - \mbf{u}_{\ell})^+\varrho_{\mbf{u}}(\ell)^{c - \frac{2\pi i r}{\log 2}}\] for $r \in \Z$.
The existence of these limits is part of the proposition. Moreover, the bounded Lipschitz norm $\Vert \mu \Vert_{\BL}$ is bounded, and $\mu$ is not the zero measure.
\end{proposition}
\begin{remarks} For the definition of the bounded Lipschitz norm $\Vert \cdot \Vert_{\BL}$ of a Borel measure on $\R/\Z$, see \cref{bl-app}.

At first sight, it might appear that the limit as $\ell \rightarrow \infty$ in \cref{mu-formula} diverges. However, this is not the case, essentially since if $\mbf{u}_{\ell}$ is significantly smaller than $\ell$ then $\varrho_{\mbf{u}}(\ell)$ is only nonzero with rather small probability; for rigorous instances of this assertion, see \cref{X-upper-gen}.

A remark of a different nature is that, despite the existence of the above formul{\ae} for $\mu$, obtaining any interesting numerical or computational data about $\mu$ seems a very difficult task. This is because for any sequence $u$ the determination of $\varrho_{u}(\ell)$ is closely related to the subset sum problem and all known algorithms for this have complexity exponential in $\ell$, whilst we expect the rate of convergence in \cref{mu-formula} to be only polynomial.
\end{remarks}

\subsection{The $\tau$ statistics and the measure $\mu'$}\label{tau-stat-intro}

Now let $x = (x_i)_{i = 1}^{\infty}$ be a sequence of positive integers. Then for each positive integer $\ell$ we define
\begin{equation}\label{tau-stat-def} \tau_{x}(\ell) := 2^{-\ell} \# \{ \sum_{i = 1}^{\ell} \eps_i x_i : \eps \in \{0,1\}^{\ell}\}.\end{equation} That is, the number of \emph{distinct} subset sums $\sum_{i = 1}^{\ell} \eps_i x_i$, normalised by the number of such sums. Note in particular that $0 \le \tau_{x}(\ell) \le 1$ for all $\ell$. It is also not hard to see that $\tau_{x}(\ell)$ is a non-increasing function of $\ell$. 

Typically we will have $x_i \asymp 2^i$ and so roughly speaking one expects $\tau_{x}(\ell)$ to have size $\asymp 1$ (that is, the subset sums $\sum_{i = 1}^{\ell} \eps_i x_i$ have a good chance of being mostly distinct; for comparison, when $x_i = 2^i$ exactly, these sums are genuinely distinct).

Define the \emph{lower process} $\mbf{x} = (\mbf{x}_i)_{i = 1}^{\infty}$ by ordering the elements of a random (multi-)set of positive integers, where the multiplicity of $i$ is a $\Pois(1/(i\log 2))$ random variable, with these variables being independent. Here is our main result concerning the measure $\mu'$.

\begin{proposition}\label{sec17-main}
Let $c = -\frac{\log \log 2}{\log 2}$. For every Lipschitz function $\psi \in \Lip(\R/\Z)$, the measure $\mu'$ satisfies
\[ \int_{\R/\Z} \psi \, \mathrm{d}\mu' := (2/\pi)^{1/2}\lim_{\ell \rightarrow \infty}\mb{E} (\log_2 \mbf{x}_{\ell} - \ell)^+ \tau_{\mbf{x}} (\ell)^c \psi(\log_2 \tau_{\mbf{x}}(\ell)).\] In particular the Fourier coefficients of $\mu'$ are given by
\[  \wh{\mu'}(r) =  (2/\pi)^{1/2}\lim_{\ell \rightarrow \infty}\mb{E} (\log_2 \mbf{x}_{\ell} - \ell)^+ \tau_{\mbf{x}}(\ell)^{c - \frac{2\pi i r}{\log 2}}.\] The existence of these limits is part of the proposition. Moreover, the bounded Lipschitz norm $\Vert \mu' \Vert_{\BL}$ is bounded, and $\mu'$ is not the zero measure.
\end{proposition}

\subsection{Acknowledgments}
MS thanks Ivan Corwin, Milind Hegde, Yang Liu, Ashwin Sah, Dominik Schmid and Mark Sellke for helpful and motivating conversations. BG thanks Sean Eberhard, Kevin Ford and Dimitris Koukoulopoulos for discussions on related matters over a number of years prior to this work.
Both authors thank Kannan Soundararajan for his immediate and somewhat remarkable observation that the function $g$ appears in \cite{granville-sedunova-sabuncu}. Finally, we are both very grateful to Kevin Ford for numerous comments on a draft of this paper.

BG is supported by Simons Investigator Award 376201. For the purpose of Open Access, the first-named author has applied a CC BY public copyright licence to any Author Accepted Manuscript (AAM) version arising from this submission. This research was conducted during the period MS served as a Clay Research Fellow. 

\subsection{Notation} 

The following pieces of notation will be in force throughout the paper. The letter $k$ will always be as in \cref{thm:main} (that is, we want to compute $p(k)$). We always write
\begin{equation}\label{k-integer-decomp} k = 2^{n + \xi}
\end{equation} where $n \in \Z$ and $\xi \in [0,1)$. That is, $\xi = \{ \log_2 k\}$. \vspace*{8pt}

\emph{Bold font.} There are many random variables in the paper. We have tried to consistently use bold font to denote these, which hopefully helps avoid some confusion, particularly when various conditionings are in place. To avoid conflict we have used blackboard bold for $\N, \Z, \R$ (the naturals, integers and reals) and for $\mb{P}, \mb{E}$ (probability and expectation).\vspace*{8pt}

\emph{Asymptotic notation.} We will use the asymptotic notations $\gg, \ll$ and $O(\cdot)$ in the standard manner for analytic number theorists. That is, $X = O(Y)$ and $X \ll Y$ both mean that $|X| \le C Y$ for some absolute constant $C$ (which may change from line to line). If $X, Y$ are positive quantities then we write $X \asymp Y$ to mean $Y \ll X \ll Y$. We will sometimes use $\Omega(X)$ to mean a quantity which is bounded below by $cX$ for some absolute constant $c > 0$ (which may change from line to line). We will use the $o(\cdot)$ notation only sparingly; $o(1)$ means a quantity which tends to $0$ as $k \rightarrow \infty$. The notations $\ggg$ and $\lll$ are informal and mean `much greater than' and `much less than' respectively.\vspace*{8pt}

\emph{Further notation.} We will use the following conventions and further pieces of notation.
\begin{itemize}
\item Write $H_X = \sum_{j \le X} 1/j$; here $X$ need not necessarily be an integer.
\item $1_S(x)$ and $1_{x \in S}$ both mean $1$ if $x \in S$ and $0$ otherwise.
\item If $N$ is a positive integer, $[N] = \{1,\dots, N\}$.
\item If $f : [N] \rightarrow \R$ is a function (for some $N$) then $\argmin(f)$ denotes the first $q$ for which $f(q) = \min_{1 \le i \le N} f(i)$. 
\item If $E_1, E_2$ are two events, $\mb{P}(E_1, E_2)$ denotes the probability that both $E_1$ and $E_2$ occur. Occasionally in more complicated expressions we will write $\mb{P}(E_1 \cap E_2)$ for the same thing. We sometimes write $\neg E$ for the complement of the event $E$.
\item $|S|$ denotes the cardinality of the set $S$. This notation will also be used for multisets, in which case elements are counted with multiplicity. The notation $\# \{x : \text{conditions on $x$}\}$ denotes the cardinality of the set of $x$ satisfying the specified conditions. 
\item $\Supp(f)$ denotes the support of the function $f$, that is to say $\{ x : f(x) \ne 0\}$.
\item $x^+$ denotes $\max(x, 0)$.
\item $\mbf{X} \samedist \Pois(\lambda)$ means that $\mbf{X}$ is a random variable with the $\Pois(\lambda)$ distribution, that is to say $\mb{P}(X = i) = e^{-\lambda}\lambda^i/i!$ for nonnegative integers $i$. Occasionally we will abuse notation by writing $\Pois(\lambda)$ to mean a random variable with the $\Pois(\lambda)$ distribution. Most particularly, we will talk about random walks with $\Pois(1)-1$ steps, which means the increments $\boldxi$ are distributed as $X - 1$ with $X \samedist \Pois(1)$. Similarly a walk with $1 - \Pois(1)$ steps has increments $\boldxi'$ distributed as $1 - X$ with $X \samedist \Pois(1)$.
\end{itemize}

\subsection*{Index of Key Notation}
For convenience we collect an index of the most frequently used notation in the paper.

\renewcommand{\arraystretch}{1.2}

\begin{longtable}{>{$}c<{$} >{\raggedright\arraybackslash}p{0.45\textwidth} >{\raggedright\arraybackslash}p{0.2\textwidth}}
\toprule
\textbf{Notation} & \textbf{Brief Description} & \textbf{Defined in} \\
\midrule
\endfirsthead

\toprule
\textbf{Notation} & \textbf{Brief Description} & \textbf{Defined in} \\
\midrule
\endhead

\bottomrule
\endfoot

k,n,\xi & global parameters $k = 2^{n + \xi}$ &  \cref{k-integer-decomp}\\
\kappa & small parameter ($= \frac{1}{100}$ say) in definition of bounded random walk & \cref{bounded-walk-def}\\
\eta & small parameter ($= \frac{1}{100}$ say) in definition of positive random walk & \cref{good-walk-def} \\
\eps_0,\eps_1 & small positive constants, power savings in random walk estimates  &  \cref{lem13.5}, \cref{decoupling} \\
W 
& $W(x) = x e^{-x^2/2} 1_{x \ge 0}$ (density of the Rayleigh distribution)
& \cref{raleigh-intro} \\

\delta
& Erd\H{o}s--Ford--Tenenbaum constant 
& \cref{eft-def} \\
\gamma & Euler's constant & \\

g
& (except \cref{sec3,sec4}) specific function in $C^{\infty}(\R/\Z)$ & \cref{g-funct-def}\\
g & (\cref{sec3,sec4}) Gaussian density function & \cref{gaussian-def}\\
\phi & specific map from $[0,\infty) \rightarrow \N$ with $\phi(x) \asymp 2^x$ & \cref{phis-def} \\
\tilde\phi & technical variant of $\phi$ & \cref{phis-def} \\
\Pois(\lambda) & Poisson random variable parameter $\lambda$ & \cref{tau-stat-intro} \\

\boldbeta
& random walk with $\Pois(1) - 1$ steps $\boldxi_i$
&  \\
\boldbeta'
& random walk with $1 - \Pois(1)$ steps $\boldxi'_i$
&  \\

\mbf{A}
& random multiset with dyadic rate & \cref{a-law} \\

\mbf{u} & upper process (rate 1 Poisson process) & \cref{upper-lower-process-def} \\
\mbf{x} & lower process (random lacunary sequence of integers) & \cref{upper-lower-process-def} \\
\mbf{s} & auxiliary rate 1 Poisson process & \cref{section5} \\
\tau_{\mbf{x} \mid \beta'},\tau^*_{\mbf{x} \mid \beta'} & variants of $\tau$ functional & \cref{tau-sec} \\
\varrho_{\mbf{u} \mid \beta}, \varrho^*_{\mbf{u} \mid \beta} & variants of $\varrho$ functional &\cref{rho-sec} \\
h(m), h'(m) & almost positivity probabilities for random walks & \cref{hm-def,hm-def-2} \\ \Sigma(S) & set of subset sums of $S$ & \cref{sig-tau-def}
\end{longtable}

\section{Heuristic explanation} \label{heuristic-sec}

The form of our main result is complicated. We offer here a heuristic explanation for the result, which also serves as a rather rough overview of the actual argument. This heuristic makes certain simplifications and should not be taken too literally; a more technical overview may be found in \cref{outline-sec4} below.

Let us first recall from \cite{EFG16} that we can immediately pass from the language of permutations to a formulation in terms of a so-called Poisson random multiset. A multiset $S \subset \N$ has the usual definition; $r_S(i)$ denotes the multiplicity of $i$ in $S$.

\begin{definition}\label{sig-tau-def} If $S$ is a multiset (finite or infinite), let $s_1, s_2,\dots$ be some listing of its elements (with multiplicity). We define the set $\Sigma(S)$ of subset sums to consist of all sums $\sum_{i} \eps_i s_i$ with $\eps_i \in \{0,1\}$ and all but finitely many $\eps_i$ equal to $0$.
\end{definition}
\begin{remark} Note carefully that we always regard $\Sigma(S)$ as a \emph{set} and not a multiset. \end{remark}
Now let $\mbf{A}_0 \subset \N$ denote the random multiset with $r_{\mbf{A}_0}(i) \samedist \Pois(1/i)$ for all $i \in \N$, with these variables being independent. We then have
\begin{equation}\label{pk-sumset}  p(k) = \mb{P} (k \in \Sigma(\mbf{A}_0)).\end{equation}
This follows from the fact that a permutation $\pi$ has an invariant set of size $k$ if and only if it has disjoint cycles of lengths summing to $k$, together with the theorem \cite{AT} that the distribution of the number of cycles of length $i$, $i = 1,\dots, k$, of a random $\pi \in \mf{S}_N$ tends to that of independent $\Pois(1/i)$ random variables in the large $n$ limit. We will not mention permutations again, and from now on our focus will be on \cref{pk-sumset}. 

An observation which in some form goes back to Erd\H{o}s \cite{erdos-mult-1} is that the most important contribution to $p(k)$ comes from instances of $\mbf{A}_0$ with $|\mbf{A}_0 \cap [k]| \approx \log_2 k$, that is to say from sets which are bigger by a factor $\frac{1}{\log 2}$ than the expected size. We will explain why this is so in \cref{new-sec-3} below.

It is tempting to model the relevant sets $\mbf{A}_0$ by `changing measure' and using $\Pois(1/(i \log 2))$ variables to define a random multiset $\mbf{A}_1$. However, to better capture the large-scale behaviour we use the following variant of this. First we recall the definitions of the upper and lower processes $\mbf{u},\mbf{x}$ from the introduction.

\begin{definition}\label{upper-lower-process-def}
Define the \emph{upper process} $\mbf{u} = (\mbf{u}_i)_{i = 1}^{\infty}$ to be simply the sequence of arrival times of a rate 1 Poisson process on $[0,\infty)$. Define the \emph{lower process} $\mbf{x} = (\mbf{x}_i)_{i = 1}^{\infty}$ by ordering the elements of a random (multi-)set of positive integers, where the multiplicity of $i$ is a $\Pois(1/(i\log 2))$ random variable, with these variables being independent. 
\end{definition}

Now we define our random multiset $\mbf{A}_1$. Pick some $k_0$ with $1 \lll k_0 \lll k$. For definiteness one could take $k_0 := \sqrt{k}$, although the precise choice is unimportant ($k_0$ does not need to be an integer). The `small' elements of $\mbf{A}_1$ are simply the elements $\mbf{x}_i$ of the lower process which are less than or equal to $k_0$, which we write $\mbf{x}_1 \le \mbf{x}_2 \le \cdots \le \mbf{x}_{L'}$ in order (where $L'$ is the number of such elements, which is itself a random variable). The `large' elements of $\mbf{A}_1$ are $\lfloor k 2^{-\mbf{u}_i}\rfloor$, where $\mbf{u}_1 \le \mbf{u}_2 \le \dots \le \mbf{u}_L$ are the elements of the upper process which are less than $\log_2(k/k_0)$. Again, $L$ is a random variable. We remark that the probability that some $i \ge k_0$ lies in this upper portion of $\mbf{A}_1$ equals the probability that some element of $\mbf{u}$ lies in the interval $(\log 2)^{-1}\big( \log (k/i) - \log(1 + \frac{1}{i}) , \log (k/i)\big]$, which is $\frac{1}{i \log 2} + O(i^{-2})$.

One may compute that this change of measure affects probabilities as follows. If $|A| = n + D$ then we have
\begin{equation}\label{change-measure-result} \mb{P}(\mbf{A}_0 \cap [k] = A) \approx e^{\gamma(\frac{1}{\log 2} - 1)} k^{-\delta} (\log 2)^{D - \xi} \mb{P}(\mbf{A}_1  = A),\end{equation} where $\gamma$ is Euler's constant. (An essentially identical computation to this one is done carefully at the end of \cref{new-sec-3}.) Thus we see that the desired quantity $p(k)$ can be understood with a sufficiently accurate understanding of the probabilities $\mb{P}(k \in \Sigma(\mbf{A}_1), \; |\mbf{A}_1| = n + D)$.

The next observation is (in slightly different language) the key advance of Kevin Ford \cite{For08}, which is to establish that the main contribution is from sets $\mbf{A}_1$ for which roughly speaking we have
\begin{equation}\label{xi-rough-pos} \log_2 \mbf{x}_i \ge i - O(1)\end{equation}
and
\begin{equation}\label{yi-rough-pos} \mbf{u}_i \le i + O(1)\end{equation} for all $i$. Let us briefly explain in rather heuristic terms where this comes from. In what follows we condition to $|\mbf{A}_1| = n + D$ with $D = O(1)$; thus \[ L + L' = n + D.\] For the discussion, we will otherwise ignore the effect of this conditioning on the distribution of the $\mbf{x}_i$ and $\mbf{u}_i$ variables.

Suppose first that \cref{xi-rough-pos} is violated, say $\log_2 \mbf{x}_q \le q - m$ for some $m$ (which we think of as ``not $O(1)$'') and some $q \le L'$. Then, since the $\mbf{x}_i$ grow roughly dyadically, one expects $\Sigma(\mbf{x}_1,\dots, \mbf{x}_q)$ to be contained in $[0, O(2^{q - m})]$, and hence to have size $\lessapprox 2^{q - m}$. It then follows that 
\[ |\Sigma(\mbf{A}_1)| \le |\Sigma(\mbf{x}_1,\dots, \mbf{x}_q)| \cdot |\Sigma(\mbf{x}_{q+1},\dots, \mbf{x}_{L'})| \cdot |\Sigma(\mbf{u}_1,\dots, \mbf{u}_L)|\lessapprox  2^{q - m} \cdot 2^{n + D - q} \sim k 2^{ - m}.\] This leads us to expect that \begin{equation}\label{x-2m-heur} \mb{P}(k \in \Sigma(\mbf{A}_1)) \lessapprox 2^{-m},\end{equation} or in other words the contribution from those $\mbf{x}$ for which \cref{xi-rough-pos} is violated is quite small.

Now suppose that \cref{yi-rough-pos} is violated in the sense that  $\mbf{u}_q \ge q + m$ for some $q \le L$. First observe that $k \in \Sigma(\mbf{A}_1)$ requires that \begin{equation}\label{member-a1} \big| k - \sum_{i = 1}^L \eps_i \lfloor k 2^{-\mbf{u}_i}\rfloor\big| =  \sum_{i = 1}^{L'} \eps'_i \mbf{x}_i \lessapprox k_0\end{equation} for some $\eps_i, \eps'_i \in \{0,1\}$, where the second inequality is what we typically expect since the $\mbf{x}_i$ grow dyadically and the largest is $\mbf{x}_{L'} \le k_0$.
Now one expects $\sum_{i \ge q} \lfloor k2^{-\mbf{u}_i}\rfloor \lessapprox k2^{-q - m}$, and so \cref{member-a1} requires \begin{equation}\label{conc-heur} \big|1 - \sum_{i = 1}^{q - 1} \eps_i 2^{-\mbf{u}_i}\big| \lessapprox 2^{-q - m}\end{equation} for some $\eps_1,\dots, \eps_{q-1} \in \{0,1\}$. There are $2^{q - 1}$ choices for $\eps = (\eps_i)_{i = 1}^{q - 1}$, and for each the inequality \cref{conc-heur} can typically only happen if $\eps_i \ne 0$ for some $i = O(1)$, since $\mbf{u}_i \approx i$. For any such $\eps$, by revealing all variables except $\mbf{u}_i$ one sees that the probability of \cref{conc-heur} is $\lessapprox 2^{-q - m}$. Summing over $\eps$ again yields a bound like \cref{x-2m-heur}. 

This concludes our discussion of why the important contributions come from \cref{xi-rough-pos,yi-rough-pos}; as previously remarked, a rigorous version of this analysis and its implications was one of the key advances of Ford \cite{For08}.

Now we come to the new part of the analysis. The point of departure is that random $\mbf{A}_1$ satisfying the properties \cref{xi-rough-pos,yi-rough-pos} will in fact (with high probability) robustly satisfy them in the sense that
\begin{equation}\label{xi-rough-pos-robust} \log_2 \mbf{x}_i \ge i + i^{1/2 - \eta}\end{equation}
and
\begin{equation}\label{yi-rough-pos-robust} \mbf{u}_i \le i - i^{1/2 - \eta}\end{equation} for large $i$. (Here $\eta$ is a small constant, say $\eta = \frac{1}{100}$.)
These statements are closely analogous to the fact that (under mild hypotheses on the increments) a mean zero random walk $\boldbeta(1),\boldbeta(2),\dots, \boldbeta(N)$ of length $N$, conditioned to stay positive, will almost surely have $\boldbeta(i) \ge i^{1/2 - \eta}$ for all large $i$. (In fact we will make extensive use of the detailed theory of random walks conditioned to stay almost positive in our arguments.)

The conditions \cref{xi-rough-pos-robust,yi-rough-pos-robust} (or closely related ones) have two important ramifications for us. First, they imply a stability property for the $\varrho$ and $\tau$ statistics defined in \cref{x-stat-def,tau-stat-def}; roughly speaking we have
\begin{equation}\label{tau-stability} \tau_{\mbf{x}}(\ell) \approx \tau_{\mbf{x}}(\ell')\end{equation} for $\ell, \ell' \ggg 1$ and similarly
\begin{equation}\label{rho-stability} \varrho_{\mbf{u}}(\ell) \approx \varrho_{\mbf{u}}(\ell').\end{equation}
Heuristically, the explanation for \cref{tau-stability} is that the slightly faster-than-$2^i$ growth of the $\mbf{x}_i$ means that new coincidences $\sum_{i = 1}^{\ell} \eps_i \mbf{x}_i = \sum_{i = 1}^{\ell}\eps'_i \mbf{x}_i$ become unlikely as $\ell$ increases. A similar heuristic can be provided for \cref{rho-stability}.
We caution that \cref{tau-stability,rho-stability} should be considered only as heuristics; the actual statements, established in \cref{tau-sec,rho-sec}, are more technical.

The second important ramification of \cref{xi-rough-pos-robust,yi-rough-pos-robust}, and in some sense the heart of the paper, is that with conditions such as these in place we are able to make the following heuristic rigorous.

Consider the portion of $\Sigma(\mbf{A}_1)$ coming from the $\mbf{x}_i$, that is to say 
\[ \Sigma_{\sml}(\mbf{A}_1) := \{ \sum_{i = 1}^{L'} \eps'_i \mbf{x}_i : \eps'_i \in \{0,1\}\}.\]
Observe that by definition \cref{tau-stat-def} we have
\begin{equation}\label{sig-s-small} |\Sigma_{\sml}(\mbf{A}_1)| = 2^{L'} \tau_{\mbf{x}}(L').\end{equation}
The event $k \in \Sigma(\mbf{A}_1)$ is then precisely the event that there is some $\eps = (\eps_i)_{i = 1}^L$ such that 
\begin{equation}\label{unrescale}  k - \sum_{i = 1}^L \eps_i \lfloor k 2^{-\mbf{u}_i}\rfloor \in \Sigma_{\sml}(\mbf{A}_1).\end{equation} 

Let $\psi(k)$ be some very slowly-growing function. Then almost surely we have $\max (\Sigma_{\sml}(\mbf{A}_1)) \le \psi(k) k_0 \le \psi(k) k 2^{-\mbf{u}_L}$, since $\mbf{x}_{L'} \le k_0 \le k 2^{-\mbf{u}_L}$ and the $\mbf{x}_i$ grow roughly as powers of two. That is, the condition \cref{unrescale} is contained in the condition
 \begin{equation}\label{rescale-2} k - \sum_{i = 1}^L \eps_i \lfloor k 2^{-\mbf{u}_i}\rfloor  \in [0, \psi(k) k 2^{-\mbf{u}_L}].\end{equation} 
By the definition \cref{x-stat-def} of $\varrho$, we expect the number of $\eps$ satisfying this weaker condition to be roughly $r:=  \psi(k) 2^{L - \mbf{u}_L} \varrho_{\mbf{u}}(L)$. (Here we are assuming, for the purpose of the heuristic, that \cref{x-stat-def} does not change much if we replace the interval $[1 - 2^{-\mbf{u}_{\ell}}, 1]$ with $[1 - (j+1) 2^{-\mbf{u}_{\ell}}, 1 - j 2^{-\mbf{u}_{\ell}}]$ for $j = 1,\dots, \psi(k) - 1$.)

We now adopt the most na\"{\i}ve heuristic, namely that the probability that a given $\eps$ satisfying \cref{rescale-2} also satisfies \cref{unrescale} should be
\[ \approx \lambda := \frac{ |\Sigma_{\sml}(\mbf{A}_1)|}{\psi(k) k 2^{-\mbf{u}_L}} .\]
Assuming suitable independence, the natural guess for the probability that \cref{unrescale} holds for at least one of the $r$ choices of $\eps$ for which \cref{rescale-2} holds is $1 - (1 - \lambda)^r \approx 1 - e^{-\lambda r}$, that is to say
\begin{equation}\label{prob-k-s} \mb{P}(k \in \Sigma(\mbf{A}_1)) \approx 1 - e^{-\lambda r}.\end{equation}
Note that using \cref{sig-s-small} we have
\[ \lambda r = \frac{ |\Sigma_{\sml}(\mbf{A}_1)|}{\psi(k) k 2^{-\mbf{u}_L}}   \psi(k) 2^{L - \mbf{u}_L} \varrho_{\mbf{u}}(L) = \frac{1}{k}2^{L + L'} \tau_{\mbf{x}}(L') \varrho_{\mbf{u}}(L). \]
Since $L + L' = n + D$ and $k = 2^{n + \xi}$, it follows that
\[ \lambda r = 2^{D - \xi}\tau_{\mbf{x}}(L') \varrho_{\mbf{u}}(L). \] Finally, using the invariance properties \cref{tau-stability,rho-stability}, it follows from the above that (heuristically!)
\begin{equation}\label{s-cond} \mb{P}\big(k \in \Sigma(\mbf{A}_1) \mid |\mbf{A}_1| = n + D \big) \approx 1 - \exp\big({-}2^{D - \xi}\tau_{\mbf{x}}(\ell)\varrho_{\mbf{u}}(\ell)\big), \end{equation} where here $1 \lll \ell \lll L, L'$ (one should think of $\ell$ as being `fixed but large' while if $k_0 = \sqrt{k}$ then we expect $L, L' \sim \frac{1}{2}\log_2 k$ to grow with $k$). 

Now that we have written the key relation \cref{s-cond}, for the rest of this heuristic discussion we are no longer conditioning on $|\mbf{A}_1| = n + D$.

Now condition to $\mbf{x}_1 = x_1, \dots, \mbf{x}_{\ell} = x_{\ell}$ and $\mbf{u}_1 = u_1,\dots, \mbf{u}_{\ell} = u_{\ell}$ for fixed choices of $x_1 \le \dots \le x_{\ell}$ and $u_1 \le\dots \le u_{\ell}$. We wish to calculate the probability that $|\mbf{A}_1| = n + D$, or in other words that $L + L' = n + D$. In order to do this we will use the following two facts. Here, $W(t) := t e^{-t^2/2} 1_{t \ge 0}$ denotes the density function of the Rayleigh distribution.\vspace*{5pt}

\emph{Fact 1.} Conditioned on $\mbf{u}_1 = u_1,\dots, \mbf{u}_{\ell} = u_{\ell}$ with $0 \le u_1 \le \cdots \le u_{\ell} < \ell$, the probability that $\mbf{u}_{\ell'} \le \ell'$ for all $\ell'$ with $\ell \le \ell' \le L$ and a given value of $L$ is $\approx (2/\pi)^{1/2} (\ell - u_\ell) N^{-1} W(\frac{L - N}{\sqrt{N}})$, where $N = \lfloor \log_2 (k/k_0)\rfloor$.\vspace*{5pt}

\emph{Fact 2.} Conditioned on $\mbf{x}_1 = x_1, \dots, \mbf{x}_{\ell} = x_{\ell}$ with $0 \le x_1 \le \cdots \le x_{\ell}$ and $\log_2 x_{\ell} > \ell$, the probability that $\log_2 \mbf{x}_{\ell'} \ge \ell'$ for all $\ell'$ with $\ell \le \ell' \le L'$ and a given $L'$ is $\approx (2/\pi)^{1/2} (\log_2 x_{\ell} - \ell) (N')^{-1} W(\frac{N' - L'}{\sqrt{N'}})$ where $N' = \lfloor \log_2 k_0\rfloor$.\vspace*{5pt}

Both of these facts come from the same source, which is the behaviour of random walks conditioned to stay (almost) positive. For Fact 1, there is a natural random walk whose $i$th step is the number of elements of $\mbf{u}$ in the interval $[u_{\ell} + i - 1, u_{\ell} + i]$, minus 1; the increments of this walk are i.i.d. $\Pois(1) - 1$ variables, and its length $T$ is $\approx \log_2(k/k_0) - u_{\ell}$. The conditions $u_{\ell'} \le \ell'$ are then essentially that this random walk stays above $-(\ell - u_\ell)$, and the event that we are interested in is essentially that the walk ends at $L - \ell - T$. We use the approximations $T \sim N$ and $L - \ell - T \sim L - N$. Fact 1 is then a consequence of known results about the distribution of the endpoint of random walk of length $N$, conditioned to stay above $-m$. (For precise statements of the relevant statements, which are consequences of work of Denisov, Tarasov and Wachtel \cite{DTW24b}, see \cref{random-walk-sec} and in particular \cref{first-rand-walk-pos,lem13.5}). The explanation of Fact 2 is analogous and we omit a detailed discussion.

We wish to sum the probabilities in Facts 1 and 2 over $L, L'$ such that $L + L' = n + D$ (where here $D = O(1)$). Since the $W( \cdot)$ terms vary on a scale of $\sqrt{N}$ or $\sqrt{N'}$, which we assume $\ggg 1$, this sum does not essentially depend on $D$; ignoring minor issues of integer parts, we may thus assume $L = N + t$ and $L' = N' - t$ (so $L + L' = N + N' = \log_2 k = n + O(1)$). The sum is then 
\[ \frac{2}{\pi} (\ell - u_\ell) (\log_2 x_{\ell} - \ell) \sum_t N^{-1} (N')^{-1} W\big(\frac{t}{\sqrt{N}}\big) W\big(\frac{t}{\sqrt{N'}}\big).\]
Writing $M := \big(\frac{N N'}{N + N'}\big)^{1/2}$ and approximating the sum by an integral, the sum over $t$ is
\begin{align}\nonumber (NN')^{-3/2} \sum_{t \in \Z_{\ge 0}} t^2 e^{-\frac{1}{2} t^2 (\frac{1}{N} + \frac{1}{N'})} & = \frac{M^2}{(NN')^{3/2}} \sum_{t \in \Z_{\ge 0}} \big(\frac{t}{M}\big)^2 e^{-\frac{1}{2}(t/M)^2} \approx  \big(\frac{M^2}{NN'}\big)^{3/2} \int^{\infty}_0 x^2 e^{-x^2/2} dx \\ & =  (\pi/2)^{1/2} (N + N')^{-3/2} \approx (\pi/2)^{1/2} (\log_2 k)^{-3/2}. \label{coupling-heuristic}\end{align}
In summary, conditioned on $\mbf{x}_1 = x_1,\dots, \mbf{x}_{\ell} = x_{\ell}$ and $\mbf{u}_1 = u_1,\dots, \mbf{u}_{\ell} = u_{\ell}$ we have
\[ \mb{P}\big(|\mbf{A}_1| = n + D\big) \approx (2/\pi)^{1/2}(\log_2 k)^{-3/2} (\ell - u_\ell) (\log_2 x_{\ell} - \ell) .\]
We now combine this with \cref{s-cond} and undo the conditioning to $\mbf{x}_1 = x_1,\dots, \mbf{x}_{\ell} = x_{\ell}$ and $\mbf{u}_1 = u_1,\dots, \mbf{u}_{\ell} = u_{\ell}$. This gives that $\mb{P}\big(k \in \Sigma(\mbf{A}_1) , \,|\mbf{A}_1| = n + D\big)$ is approximately
\[  \approx\big( \frac{2}{\pi}\big)^{1/2}(\log_2 k)^{-3/2} \mb{E}_{\mbf{x}, \mbf{u}} (\log_2 \mbf{x}_{\ell} - \ell)^+ (\ell - \mbf{u}_{\ell})^+ \big(1 - \exp\big({-}2^{D - \xi} \tau_{\mbf{x}}(\ell) \varrho_{\mbf{u}}(\ell)\big)\big) .\]
Finally, summing over $D$ using \cref{change-measure-result} gives
\[ p(k) = \mb{P}(k \in \Sigma(\mbf{A}_0)) \approx  f(\xi) k^{-\delta} (\log k)^{-3/2},\]
where \[ f(\xi) = (\frac{\pi}{2})^{1/2} (\log 2)^{3/2} e^{\gamma(\frac{1}{\log 2} - 1)}\mb{E}_{\mbf{x},\mbf{u}} (\log_2 \mbf{x}_{\ell} - \ell)^+ (\ell - \mbf{u}_{\ell})^+  g_{\tau_{\mbf{x}}(\ell)\varrho_{\mbf{u}}(\ell)}(\xi), \]
with
\[ g_{\lambda}(\xi) := \sum_{D \in \Z} (\log 2)^{D - \xi} (1 -  e^{-\lambda 2^{D - \xi}}).\]
Note that
\begin{align}
g_{\lambda}(\xi)  & = \sum_{D \in \Z} (\log 2)^{D - \xi} (1 -  e^{-2^{D - \xi + \log_2 \lambda}}) \\ & = (\log 2)^{-\log_2\lambda} \sum_{D \in \Z} (\log 2)^{D - \xi + \log_2\lambda} (1 -  e^{-2^{D - \xi + \log_2 \lambda}}) = \lambda^{-\frac{\log \log 2}{\log 2}} g(\xi  -  \log_2\lambda)\label{g-lambda-trans}
\end{align}
with $g$ the function in \cref{thm:main}
and so we indeed see that $f = c_0 g \ast \mu \ast \mu'$ with $\mu, \mu'$ defined as in \cref{prop16.2,sec17-main} and $c_0 = (\frac{\pi}{2})^{1/2} (\log 2)^{3/2} e^{\gamma(\frac{1}{\log 2} - 1)}$.

\begin{remark}
Properties of uniform order statistics analogous to \cref{xi-rough-pos-robust,yi-rough-pos-robust} (so-called ``strong barrier conditions'') have been exploited in a related context by Schlitt in recent work \cite{schlitt}; see in particular \cite[page 5]{schlitt} for some further references.
\end{remark}

\section{Changing measure}\label{new-sec-3}
We will give a detailed outline of the rest of the paper in \cref{outline-sec4} below. However, it is expedient to first give a proper treatment of the `change of measure' step (that is, the move from $\mbf{A}_0$ to $\mbf{A}_1$) described at the start of \cref{heuristic-sec}, since this sits slightly to one side of the rest of the paper. 

Rather than working with exactly the random mulitset $\mbf{A}_1$ as described in \cref{heuristic-sec}, we will instead consider the random multiset $\mbf{A}$ in which  \begin{equation}\label{a-law} r_{\mbf{A}}(i) \samedist \Pois\big(\frac{n}{iH_k}\big)\end{equation} for $i \in [k]$, where $H_k = \sum_{i \le k} 1/i$ is the harmonic sum. Note here that $n/H_k \approx 1/\log 2$, so $\mbf{A}$ is very closely related to $\mbf{A}_1$. Observe that 
\[ \big|\mbf{A} \big| \samedist \Pois \big( \frac{n}{H_k} \sum_{i \le k} \frac{1}{i}\big) = \Pois(n).\]
The fact that $n$ is an integer is the reason for this particular rescaling; this will enable us to make effective use of embedded random walks, where we can bring into play the existing literature on the distribution of almost positive walks.

In this short section, we accomplish two things. First, we show that only the contribution to \cref{pk-sumset} from sets with $|\mbf{A}_0 \cap [k]| \approx \log_2 k$ is important. Second, for sets of roughly this critical size, we analyse the change in passing from the random multiset $\mbf{A}_0$ to the reweighted random multiset $\mbf{A}$. The following is the key result of the section, and the only one we will need in the rest of the paper. Recall that $\gamma$ is Euler's constant and $\delta$ the Erd\H{o}s-Ford-Tenenbaum constant \cref{eft-def}.
\begin{proposition} 
\label{key-measure-change} Write $k = 2^{n + \xi}$. We have \begin{align*} p(k) = \Big(1 + O\big(\frac{\log \log k}{\log k} \big)\Big)e^{\gamma(\frac{1}{\log 2} - 1)} k^{-\delta} \sum_{|D| \le 20 \log n}   (\log 2)^{D - \xi} & \mb{P}\big(k \in \Sigma(\mbf{A}), \,|\mbf{A} | = n + D\big) \\ & + O\big(k^{-\delta} (\log k)^{-2}\big).\end{align*} \end{proposition}
\begin{proof}
The starting point is of course \cref{pk-sumset}. We first eliminate the contribution from sets with size not close to $\log_2 k$. The key claim is that 
\begin{equation}\label{lemma6.1} \mb{P}\big(k \in \Sigma(\mbf{A}_0) \; \mbox{and} \; \big| |\mbf{A}_0 \cap [k]| - n\big| \ge 20 \log n \big) \ll k^{-\delta} (\log k)^{-2} .\end{equation}
A short computation with Stirling's formula using the fact that $|\mbf{A}_0 \cap [k]| \samedist \Pois(H_k)$ shows that 
\[ \mb{P}\big( |\mbf{A}_0 \cap [k]| - n  \ge 20 \log n \big) \ll k^{-\delta} (\log k)^{-2} ,\] so it is enough to show 
\begin{equation}\label{enough-10.1} \mb{P}\big(k \in \Sigma(\mbf{A}_0) \; \mbox{and} \; |\mbf{A}_0 \cap [k]| \le  m \big) \ll k^{-\delta} (\log k)^{-2},\end{equation} where, for brevity, we have written $m := n - 20 \log n$. Define $\mbf{A}_0^{\sml} := \mbf{A}_0 \cap [1,\frac{k}{2\log k})$ and $\mbf{A}_0^{\lrg} := \mbf{A}_0 \cap [\frac{k}{2\log k},k]$. Then the LHS of \cref{enough-10.1} is bounded above by
\begin{equation}\label{int-10.2} \sum_{\substack{m' \le m \\ 1 \le r \le m'}} \mb{P}\big(k \in \Sigma(\mbf{A}_0) \mid |\mbf{A}_0^{\lrg}| = r,\, |\mbf{A}_0^{\sml}| = m' - r \big) \cdot \mb{P} \big(|\mbf{A}_0^{\lrg}| = r,\, |\mbf{A}_0^{\sml}| = m' - r \big); \end{equation} the key point here is that we do not need to include $r = 0$ since, if $\mbf{A}_0^{\lrg}$ is empty and $|\mbf{A}_0 \cap [k]| \le  m$ then $\max \big(\Sigma(\mbf{A}_0) \cap [k]\big) \le \frac{mk}{2\log k} < 3k/4$, so in particular $k \notin \Sigma(\mbf{A}_0)$.

Conditioned on $|\mbf{A}_0^{\lrg}| = r$, we can write $\mbf{A}_0^{\lrg} = \{ \mbf{a}_1,\dots, \mbf{a}_r\}$, where the $\mbf{a}_j$ are independent samples from $\Z \cap [\frac{k}{2\log k}, k]$ with $\mb{P}(\mbf{a}_j = x)$ proportional to $1/x$. (This is the standard relation between the conditioned Poisson and multinomial distributions.) Further conditioning on a fixed choice $A_0^{\sml}$ of $\mbf{A}_0^{\sml}$ of size $m' -r$, we have 
\begin{equation}\label{union-10.1} \{ k \in \Sigma(\mbf{A}_0) \} = \bigcup_{\eps \in \{0,1\}^r \setminus \{0^r\}} \bigcup_{s \in \Sigma(A_0^{\sml})} \{ k = \sum_{j = 1}^r \eps_j \mbf{a}_j + s\}.\end{equation}
For each $\eps$, select the minimum $i$ with $\eps_i \ne 0$, and reveal all $\mbf{a}_j$ except for $j = i$; since the probability mass function of $\mbf{a}_i$ is pointwise bounded by $O(\frac{\log k}{k})$, we see that $\mb{P} \big(k = \sum_{j = 1}^r \eps_j \mbf{a}_j + s\big) \ll \frac{\log k}{k}$. Therefore, from \cref{union-10.1} we have
\[ \mb{P}\big(k \in \Sigma(\mbf{A}_0) \mid |\mbf{A}_0^{\lrg}| = r, \, \mbf{A}_0^{\sml} = A^{\sml}_0 \big) \ll 2^{m'} \frac{\log k}{k}\] and so, summing over all $A^{\sml}_0$ weighted by $\mb{P}(\mbf{A}_0^{\sml} = A^{\sml}_0)$,
\[ \mb{P}\big(k \in \Sigma(\mbf{A}_0) \mid |\mbf{A}_0^{\lrg}| = r, \,|\mbf{A}_0^{\sml}| = m' - r\big) \ll 2^{m'} \frac{\log k}{k}.\] Substituting into \cref{int-10.2} gives
\begin{align*}
\mb{P}\big(k \in \Sigma(\mbf{A}_0) \; \mbox{and} \; |\mbf{A}_0 \cap [k]| & \le  m \big)  \ll \frac{\log k}{k}\sum_{\substack{m' \le m \\ 1 \le r \le m'}} 2^{m'} \mb{P} \big(|\mbf{A}_0^{\lrg}| = r, \, |\mbf{A}_0^{\sml}| = m' - r \big)  \\ & \le \frac{\log k}{k} \sum_{m' \le m} 2^{m'} \mb{P} \big(|\mbf{A}_0 \cap [k] | = m'\big) = \frac{\log k}{k} \sum_{m' \le m}  \frac{e^{-H_k}(2H_k)^{m'} }{(m')!}.
\end{align*}
Using the bounds $e^{-H_k} \ll 1/k$ and $(m')! \ge (m'/e)^{m'}$, one may see that this is bounded above by $\ll \frac{\log k}{k^2} \sum_{m' \le m} (2e H_k/m')^{m'}$. Since $(C/x)^x$ is increasing on $[0, C/e]$, and since $m =  (1 + o(1)) \frac{\log k}{\log 2}$, this is $\ll \frac{(\log k)^2}{k^2} (2e H_k/m)^m$. Noting that $(2e \log k/m_0)^{m_0} = k^{2 - \delta}$, where $m_0 := \log k/\log 2$, we can rewrite this bound as 
\[ (\log k)^2 k^{-\delta} \cdot \big( \frac{m_0}{m}\big)^m e^{m - m_0} \cdot \big( \frac{H_k}{\log k}\big)^m \cdot (2 \log 2)^{m - m_0}.\] Using $(m_0/m)^m e^{m - m_0} \le 1$ and $(H_k/\log k)^m = (1 + O(1/\log k))^m \ll 1$, we see that this is $\ll (\log k)^2k^{-\delta} (2 \log 2)^{-20 \log \log k} < k^{-\delta}( \log k)^{-2}$. This completes the proof of the claim \cref{lemma6.1}. 

From this and \cref{pk-sumset} we evidently have
\[
p(k) = \mb{P}\big(k \in \Sigma(\mbf{A}_0) \; \mbox{and} \; \big| |\mbf{A}_0 \cap [k]| - n\big| \le 20 \log n \big) + O\big( k^{-\delta} (\log k)^{-2}\big).
\]
Now we change measure and express the probabilities using $\mbf{A}$ as given by the law \cref{a-law}.

If $A$ is a fixed multiset of elements from $[k]$ (with the multiplicity of $i$ in $A$ being $r_A(i)$) then 
\[ \mb{P}(\mbf{A}_0 \cap [k] = A) = \prod_{i \in [k]} \frac{e^{-1/i} (1/i)^{r_A(i)}}{r_A(i)!}, \qquad \mb{P}(\mbf{A} = A) =  \prod_{i \in [k]} \frac{e^{-n/H_k i} (n/(H_k i))^{r_A(i)}}{r_A(i)!} ,\] and so if $|A| = n + D$ then 
\[ \mb{P}(\mbf{A}_0 \cap [k]= A) =  e^{n - H_k} \big(\frac{H_k}{n}\big)^{n + D}\mb{P}(\mbf{A} = A).\]
Using that 
\[ n = \frac{\log k}{\log 2} - \xi, \quad \frac{H_k}{n} =  \Big( 1 + \frac{1}{\log k} (\gamma + \xi \log 2) + O\big(\frac{1}{\log^2 k}\big) \Big) \log 2,\]
it follows that we have the key relation \[ \mb{P}(\mbf{A}_0 \cap [k] = A) = \Big(1 + O\big(\frac{1 + |D|}{\log k}\big) \Big) e^{\gamma(\frac{1}{\log 2} - 1)} (\log 2)^{D - \xi} k^{-\delta} \mb{P}(\mbf{A} = A).\] 
Summing over all sets $A$ with cardinality in the interval $[n - 20\log n, n + 20\log n]$ (and thus $|D| \ll \log \log k$) and with $k \in \Sigma(A)$, and applying \cref{lemma6.1}, the proof of \cref{key-measure-change} is complete.
\end{proof}

\section{Outline of the rest of the paper}\label{outline-sec4}

We now give an outline of the rest of the paper. \cref{key-measure-change} has reduced our main task of proving \cref{thm:main} to that of estimating $\mb{P}(k \in \Sigma(\mbf{A}))$ (and that $|\mbf{A}|$ has a certain fixed size close to $\frac{\log k}{\log 2}$), where the random (multi-)set $\mbf{A}$ is defined using \cref{a-law}. An important tool for understanding $\mbf{A}$ -- and indeed the reason for the particular definition \cref{a-law} -- will be the use of random walks.

Roughly, the idea is that an equivalent way to sample $\mbf{A}$ is via the use of two random walks $\boldbeta, \boldbeta'$. The random walk $\boldbeta = (\boldbeta(i))_{i \in \N}$ is a random walk with $\Pois(1) - 1$ steps $\boldxi_i = \boldbeta(i) - \boldbeta(i-1)$, where we adopt the convention that $\boldbeta(0) = 0$ deterministically. Similarly $\boldbeta'$ has $1 - \Pois(1)$ steps $\boldxi'_i$. 

Only the $\boldxi_i$ with $i \le \lceil n/2\rceil$ and the $\boldxi'_i$ with $i \le \lfloor n/2\rfloor$ will be relevant. The choice of $n/2$ here is somewhat arbitrary; any pair of cutoffs summing to $n$ and both tending to infinity would do, and our particular choice corresponds to taking $k_0 \sim \sqrt{k}$ in the heuristic of \cref{heuristic-sec}.

The precise manner in which this sampling of $\mbf{A}$ is done is a little technical and is described carefully in \cref{section5} below. Roughly, one should think of $1 + \boldxi_i$ being the number of elements of $\mbf{A}$ of size $\sim k 2^{-i}$, and $1 - \boldxi'_i$ the number of elements of $\mbf{A}$ of size $\sim 2^i$. The precise meaning of $\sim$ will be clarified in \cref{section5} but the definitions will be such that we have
\begin{equation} |\mbf{A} |  = \sum_{i \le \lceil n/2\rceil}(1 + \boldxi_i) + \sum_{i \le \lfloor n/2\rfloor} (1 - \boldxi'_i)  = n + \boldbeta(\lceil n/2\rceil) - \boldbeta'(\lfloor n/2\rfloor).\label{a-discrep}\end{equation}

Note that $1 + \boldxi_i, 1 - \boldxi'_i \samedist \Pois(1)$, which makes this modelling of $\mbf{A}$ natural. The reason for using two `opposite' random walks (rather than simply $\Pois(1)$ random variables) is completely analogous to the modelling of $\mbf{A}_1$ using the upper and lower processes $\mbf{u},\mbf{x}$ as described in \cref{heuristic-sec}, and has to do with the fact that we will exploit positivity properties of these walks.

Throughout many of our arguments we will condition to a fixed choice $(\beta, \beta')$ of (the initial segments of) $(\boldbeta, \boldbeta')$. We use the term \emph{upper walk} to describe any possible instance $\beta$ of $\boldbeta$; this is simply a sequence $(\beta(i))_{i \in \N}$ whose increments $\xi_i := \beta(i) - \beta(i-1)$ are bounded below by $-1$. Similarly a lower walk is any sequence $(\beta'(i))_{i \in \N}$ whose increments $\xi'_i = \beta'(i) - \beta'(i-1)$ are bounded above by $1$.

Write $(\mbf{A} \mid \beta,\beta')$ for the corresponding (random) choice of $\mbf{A}$. The definition of this object is described carefully in \cref{subsec-53} below, but this outline can be followed without detailed knowledge of that section. By \cref{a-discrep} we have
\begin{equation*}  \# (\mbf{A} \mid \beta, \beta')  = n + \beta(\lceil n/2\rceil) - \beta'(\lfloor n/2\rfloor).  \end{equation*}
We therefore have from \cref{key-measure-change} that 
\begin{align}\nonumber  p(k) = \Big(& 1 + O\big(\frac{\log \log k}{\log k}\big) \Big)e^{\gamma(\frac{1}{\log 2} - 1)} k^{-\delta} \sum_{|D| \le 20 \log n}  (\log 2)^{D - \xi} \times \\ & \times \!\!\!\!\!\!\!\!\!\!\!\!\!\sum_{\substack{\beta,\beta' \\ \beta(\lceil n/2\rceil) - \beta'(\lfloor n/2\rfloor) = D}}\!\!\!\!\!\!\!\!\!\!\!\mb{P}\big((\boldbeta,\boldbeta') = (\beta, \beta')\big)  \mb{P}\big(k \in \Sigma(\mbf{A} \mid \beta,\beta')\big) + O(k^{-\delta} (\log k)^{-2}). \label{key-measure-change-cond}\end{align}
This formula will be the basis for much of our analysis, and in particular the bulk of our paper will be devoted to trying to understand the quantity 
\begin{equation}\label{key-a-beta} \mb{P}(k \in \Sigma(\mbf{A} \mid \beta,\beta')).\end{equation} Broadly, our strategy for doing this will follow the initial steps of the heuristic laid out in \cref{heuristic-sec}. However, by conditioning to fixed $\beta, \beta'$ the actual details look somewhat different.

The analogue of \cref{xi-rough-pos,yi-rough-pos} is that the main contribution will come from $\beta, \beta'$ which are almost positive in the sense that \begin{equation}\label{min-cond-overview} \min \big( \min_{i \le \lceil n/2\rceil} \beta(i), \min_{i \le \lfloor n/2\rfloor} \beta'(i)\big) \ge - O(1). \end{equation}

The main point of departure of the paper will then be the fact that most such walks satisfy $\beta(i),\beta'(i) \ge i^{1/2 - \frac{1}{100}}$ for large $i$, which is the appropriate analogue of \cref{xi-rough-pos-robust,yi-rough-pos-robust}, as well as some weak boundedness properties on their increments. We give the appropriate technical definitions now, since they will feature throughout the paper. In the following definitions $\kappa$ and $\eta$ are fixed small positive constants, global to the paper. We can take $\kappa = \eta = \frac{1}{100}$, but it makes the exposition clearer to keep these constants unspecified (and separate).

\begin{definition}\label{bounded-walk-def}
Let $R \ge 1$ be a parameter and let $L \in \N$. Let $\beta$ be a walk (upper or lower) with increments $\xi_i = \beta(i) - \beta(i - 1)$. Then we say that $\beta$ is $R$-bounded to length $L$ if $|\xi_i| \le Ri^{\kappa}$ for $i = 1,\dots, L$. 
\end{definition}

\begin{definition}\label{good-walk-def}
Let $T \ge 0$ be a parameter and let $L \in \N$. Let $\beta$ be a walk (upper or lower) with increments $\xi_i = \beta(i) - \beta(i - 1)$. Then we say that $\beta$ is $T$-\emph{positive} to length $L$ if $-T + \ell^{1/2 - \eta} \le \beta(\ell) \le T + \ell^{1/2 + \eta}$
for $\ell = 1,\dots, L$.
\end{definition}

Continuing to follow the heuristic in \cref{heuristic-sec}, the next task is to define appropriate analogues of the $\tau$ and $\varrho$ statistics. These will be called $\tau^*_{\mbf{x} \mid \beta'}(\ell)$ and $\varrho^*_{\mbf{u} \mid \beta}(\ell)$ respectively, where $\ell$ is a positive integer parameter. They are defined, and their key properties explored, in \cref{tau-sec} and \cref{rho-sec} respectively. In particular, we will show that if $\beta,\beta'$ are appropriately bounded and positive then these statistics have good convergence properties as $\ell \rightarrow \infty$, statements which correspond to the heuristic assertions \cref{tau-stability,rho-stability}. The precise statements here are \cref{lem23} and \cref{lemma7point3}.\vspace*{10pt}

\cref{part2} of the paper is an interlude consisting of three sections, whose purpose is to develop various background material. These three sections may be read independently of each other and the rest of the paper.

\cref{random-walk-sec} contains material on almost positive random walks, in particular concerning boundedness and positivity properties of such walks. Rigorous statements to the effect that most almost positive walks are appropriately bounded (resp. positive) may be found in \cref{walks-bd-min-est} and \cref{pos-paths-cor} respectively. \cref{random-walk-sec} also introduces the technical concept of a \emph{jump step}. This is a single time $t$ at which the walk $\beta$ has a large increment, followed by suitable boundedness and positivity properties; the definition is \cref{jump-step-7}. Almost positive walks generically have a jump step in appropriate ranges (see \cref{lemma7.19}).

\cref{01-matrix-subsection} assembles various bounds for $\{0,1\}$-matrices satisfying linear-algebraic conditions, and finally \cref{g-function-sec} gives the basic properties of the function $g$ featuring in our main result (see \cref{g-funct-def}).\vspace*{10pt}

\cref{part3} of the paper, consisting of two sections \cref{sec3,sec4}, is the beating heart of the paper. This is the portion of the paper where the analysis corresponding to \cref{sig-s-small} to \cref{s-cond} in our heuristic overview is carried out rigorously. In particular, the phrase `assuming suitable independence' leading up to \cref{prob-k-s} is justified in \cref{sec3} by a (lengthy) local limit theorem analysis; the main result is \cref{lem:Poisson}. Here, the existence of a jump step $t$ is used crucially in order to characterise the relevant distributions as Gaussians. 

The analysis is continued in \cref{sec4}, where the heuristic \cref{prob-k-s} (suitably rephrased in the random walk formalism) is justified via an inclusion--exclusion argument. The apparent dependence on the location of the jump step $t$ is shown to be illusory, and the main result of the section is \cref{sec6-main}, which is a rigorous version of \cref{s-cond}. Roughly speaking it provides the desired expression for \cref{key-a-beta}, namely that under suitable conditions we have 
\begin{equation}\label{key-a-beta-form} \mb{P}\big(k \in \Sigma(\mbf{A} \mid \beta,\beta')\big) \approx 1 - \mb{E} \exp\big( -2^{D - \xi}  \varrho^*_{\mbf{u} \mid \beta}(L)\tau^*_{\mbf{x} \mid \beta'}(L)\big),\end{equation} where here $\xi = \{ \log_2 k\}$,  $D = \beta(\lceil n/2\rceil) - \beta'(\lfloor n/2\rfloor)$ and $L$ is an appropriate large parameter. This should be compared with \cref{s-cond}.\vspace*{10pt}

\cref{part4} of the paper contains the main synthesis, substituting \cref{key-a-beta-form} (or rather the more technical \cref{sec6-main}) into \cref{key-measure-change} in order to prove the weak form of \cref{thm:main} in which the measures $\mu, \mu'$ are shown to exist, but are not yet described explicitly. The first task here, accomplished in \cref{sec:reduc-to-crit}, is to reduce to the case where (essentially) $D = O(1)$ and \cref{min-cond-overview} holds. Here, with an eye to applying the main results from \cref{part3}, one must additionally restrict to walks which are suitably bounded, positive, and have a jump step.

\cref{decouple-sec} then contains the main analysis, which is parallel to the leveraging of Facts 1 and 2 in our heuristic discussion. The point here is to evaluate the probability that two walks $\boldbeta, \boldbeta'$ of lengths $\lceil n/2\rceil,\lfloor n/2\rfloor$ are almost positive and satisfy $\boldbeta(\lceil n/2\rceil) - \boldbeta'(\lfloor n/2\rfloor) = D$. This is done by using a suitable local limit theorem, which states that the endpoints of the two almost positive walks, appropriately scaled, each have a Rayleigh distribution. The result we apply follows from work of Denisov, Tarasov and Wachtel \cite{DTW24b}, though we will give our own self-contained treatment. The endpoints of the two walks may then be coupled via a fairly straightforward analysis very similar to \cref{coupling-heuristic}.

Throughout the above argument, the walks $\boldbeta, \boldbeta'$ under consideration are bounded below by $-M$, for some parameter $M$. By the end of \cref{decouple-sec} we have established an approximate form of \cref{thm:main} in which ``$f$'' is of the form $g \ast \mu_M \ast \mu'_{M} + O(e^{-cM})$ for certain Borel measures $\mu_M, \mu'_M$ on $\R/\Z$.  To conclude the proof of \cref{thm:main} (and, in particular, to show that $f$ actually exists) we must show appropriate convergence of $\mu_M, \mu'_M$ to nonzero limit measures $\mu, \mu'$ as $M \rightarrow \infty$. This task, which is rather delicate, is accomplished in \cref{sec15-mainthm-proof}. It makes considerable further use of the results from \cref{tau-sec,rho-sec,random-walk-sec}.\vspace*{10pt}

The final part of the main paper, \cref{part5}, gives the explicit characterisations of the measures $\mu, \mu'$ as detailed in \cref{prop16.2} and \cref{sec17-main} respectively. The proofs occupy \cref{sec16,sec17} respectively. These involve a translation from the formalism of random walks $\boldbeta, \boldbeta'$ to the language of sequences $\mbf{u}, \mbf{x}$ as in the statements of \cref{prop16.2,sec17-main} (and the heuristic sketch). A surprising amount of care is required to carry this out correctly. A key technical issue is that when defining the measures $\mu$ and $\mu'$ we have not specified that the corresponding walks are ``almost positive'' and one needs to prove that the contribution from remaining paths is essentially negligible. \vspace*{10pt}

There are six appendices. \cref{poisson-app} collects various facts about random walks and a basic estimate from renewal theory. \cref{dtw-appendix} then gives a self-contained account of a version of the local limit theorem of Denisov, Tarasov and Wachtel which is sufficient for our purposes. \cref{appB} pulls together some Fourier estimates of fairly standard type. \cref{appC} gives assorted real-variable inequalities whose proofs would have diverted attention to too great an extent if given in the main text. \cref{bl-app} gives the basic properties of the so-called bounded Lipschitz norm on Borel measures, which we use in \cref{sec15-mainthm-proof} to study the convergence of the measures $\mu_M,\mu'_M$, and again in \cref{sec16,sec17} when characterising $\mu, \mu'$. Finally, \cref{appF} collects some standard large deviation estimates.

\section{Sampling using walks }\label{section5}

In this section we give the details of how the random (multi-)set $\mbf{A}$ defined by $r_{\mbf{A}}(i) \samedist \Pois\big(\frac{n}{iH_k}\big)$ (that is, as in \cref{a-law}) may be sampled via a pair of random walks $\boldbeta, \boldbeta'$. 

As in the previous section, and for the rest of the paper, $\boldbeta$ denotes a random walk with $\Pois(1) - 1$ steps $\boldxi_i = \boldbeta(i) - \boldbeta(i-1)$, and $\boldbeta'$ has $1 - \Pois(1)$ steps $\boldxi'_i = \boldbeta'(i) - \boldbeta'(i-1)$. By convention, $\boldbeta(0) = \boldbeta'(0) = 0$. The walks $\boldbeta,\boldbeta'$ are independent.

At the same time we will show how to sample the upper and lower processes (see \cref{upper-lower-process-def}) $\mbf{u},\mbf{x}$ in a similar manner. More precisely we will explain how to couple them to $\boldbeta, \boldbeta'$ respectively.

\subsection{The upper and lower processes}
To couple the upper process $\mbf{u}$ to $\boldbeta$, we first sample $\boldbeta$ and then, for each $i$, sample $1 + \boldxi_i$ uniform random elements from $(i-1, i]$ and order them to create $\mbf{u} \cap (i-1,i]$. It is a standard fact that $\mbf{u}$ is indeed a rate 1 Poisson process. Observe that $\boldbeta(i) = \# (\mbf{u} \cap [0,i]) - i$.

Handling the lower process is a little more involved. We begin by associating to $\boldbeta'$ an auxiliary rate 1 Poisson process $\mbf{s}$ by first sampling $\boldbeta'$, then for each $i$ sampling $1 - \boldxi'_i$ uniform random elements from $(i-1,i]$ and ordering them to create $\mbf{s} \cap (i-1,i]$. Again, $\mbf{s}$ is a rate 1 Poisson process, and now $\boldbeta'(i) = i - \# (\mbf{s} \cap [0,i])$.

We then introduce the following pair of closely related maps. Here, $H_N$ denotes the harmonic sum and by convention $H_0 = 0$.

\begin{definition}\label{phis-def}
Define maps $\phi, \tilde\phi : [0,\infty) \rightarrow \N$ as follows. $\phi(0) = \tilde\phi(0) = 1$, and if $x > 0$ then $\tilde\phi$ is the unique map so that $x \in (\frac{n}{H_k} H_{\tilde\phi(x)-1}, \frac{n}{H_k}H_{\tilde\phi(x)}]$. $\phi : [0,\infty) \rightarrow \N$ is the unique map defined by $x \in (\frac{1}{\log 2} H_{\phi(x)-1}, \frac{1}{\log 2}H_{\phi(x)}]$.
\end{definition}
We will use $\phi$ now, and $\tilde\phi$ will be used later when discussing how to sample $\mbf{A}$.
For orientation, we note that $\phi(x), \tilde\phi(x) \asymp 2^x$, and also that $\tilde\phi$ depends on $k$, but $\phi$ is `universal'. For a detailed proof, see \cref{lemma53} below.

We now couple the lower process $\mbf{x}$ (\cref{upper-lower-process-def}) to $\boldbeta'$ as follows. First sample $\boldbeta'$, then associate the auxiliary process $\mbf{s}$. Define $\mbf{x} = \phi(\mbf{s})$. We observe that, unravelling the definitions, 
\[ r_{\mbf{x}}(i) = \big| \mbf{s} \cap \frac{1}{\log 2} (H_{i-1},  H_i] \big|\samedist \Pois(1/(i \log 2)). \] Thus $\mbf{x}$ indeed has the required distribution of the lower process.

\subsection{Sampling $\mbf{A}$}

We now turn to the discussion of sampling $\mbf{A}$ (the random multiset defined by \cref{a-law}) via the random walks $\boldbeta,\boldbeta'$.  Let $\mbf{u}$ be the upper process coupled to $\boldbeta$ as above, and let $\mbf{s}$ be the auxiliary rate 1 Poisson process coupled to $\boldbeta'$, again as described above. Now set 
\begin{equation*} \mbf{A} = \tilde\phi\big(\mbf{s} \cap[0, \lfloor n/2\rfloor]\big) \cup \tilde\phi \big(n - (\mbf{u} \cap [0,\lceil n/2\rceil])\big).\end{equation*}
That is, we use $(\boldbeta(i))_{i \le \lceil n/2\rceil}$ to describe the large elements of $\mbf{A}$, and $(\boldbeta'(i))_{i \le \lfloor n/2\rfloor}$ to describe the small elements.
Observe that $[0, \lfloor n/2\rfloor] \cup (n - [0,\lceil n/2\rceil]) = [0,n]$ and so the distribution of $\mbf{A}$ is exactly that of $\tilde\phi(\mbf{p} \cap [0,n])$, where $\mbf{p}$ is a Poisson process of rate $1$. Therefore we have
\[ r_{\mbf{A}}(i) = r_{\tilde\phi(\mbf{p} \cap [0,n])}(i) =  \big| \mbf{p} \cap \frac{n}{H_k}(H_{i-1}, H_i] \big| \samedist \Pois\big(\frac{n}{H_k i}\big)\] for $i = 1,\dots, k$, and so this $\mbf{A}$ indeed coincides with the random multiset defined by \cref{a-law}.

\subsection{Conditioning to fixed walks.} \label{subsec-53} We recall in a formal definition the notions of upper and lower walk.

\begin{definition}
An \emph{upper walk} of length $N$ is simply a sequence $(\beta(i))_{i = 1}^{N}$ where the increments $\xi_i := \beta(i) - \beta(i-1)$ are integers satisfying $\xi_i \ge -1$ for all $i$. A \emph{lower walk} of length $N'$ is a sequence $(\beta'(i))_{i = 1}^{N'}$ where the increments $\xi'_i := \beta'(i) - \beta'(i-1)$ are integers satisfing $\xi'_i \le 1$ for all $i$. By convention we set $\beta(0) = \beta'(0) = 0$.    
\end{definition}

For much of the analysis we will condition to the event that $(\boldbeta(i))_{i \le \lceil n/2\rceil} = (\beta(i))_{i \le \lceil n/2\rceil}$ and that $(\boldbeta'(i))_{i \le \lfloor n/2\rfloor} = (\beta'(i))_{i \le \lfloor n/2\rfloor}$ for some pair of upper- and lower walks $\beta, \beta'$, which we will always take to have lengths $\lceil n/2\rceil$, $\lfloor n/2\rfloor$ respectively. It is convenient to write $(\boldbeta,\boldbeta') = (\beta,\beta')$ for this event. 

We now consider the upper and lower processes $\mbf{u}, \mbf{x}$ and random multiset $\mbf{A}$ conditioned to fixed walks. We take the opportunity to spell out these definitions and at the same time specify notation for listing their elements.\vspace*{8pt}

\emph{Conditioned upper process.} The conditioned upper process $(\mbf{u} \mid \boldbeta = \beta)$, which we write $(\mbf{u} \mid \beta)$ for brevity, is obtained by sampling, for each $i \le \lceil n/2\rceil$, $1 + \xi_i$ uniform random elements from $(i-1,i]$, and ordering them. (We remark that we will never need to discuss any values of $\mbf{u}$ outside the interval $[0, \lceil n/2\rceil]$, so we happily ignore the slight inconsistency of notation here.) We write the elements of $(\mbf{u} \mid \beta)$ as $(\mbf{u}_{i,j})_{i \le \lceil n/2\rceil; j \in [1 + \xi_i]}$, but leave these randomly ordered on each $(i-1, i]$. When this notation is used, the underlying $\beta$ will be clear from context and so we do not indicate an explicit dependence. Thus $\mbf{u}_{i,j}$ is uniform on $(i-1,i]$, and these variables are independent for $j = 1,2,\dots, 1 + \xi_i$. \vspace*{8pt}

\emph{Conditioned lower process.} To describe $(\mbf{x} \mid  \beta') = (\mbf{x} \mid \boldbeta' = \beta')$ we first consider the conditioned auxiliary Poisson process $(\mbf{s} \mid \beta')$. This is obtained by sampling $1 - \xi'_i$ uniform random elements from $(i-1,i]$, $i \le \lfloor n/2\rfloor$. We write its elements as $(\mbf{s}_{i,j})_{i \le \lfloor n/2\rfloor, j \in [1 - \xi'_i]}$, again randomly ordered on $(i-1, i]$. Thus $\mbf{s}_{i,j}$ is uniform on $(i-1,i]$, and these variables are independent for $j = 1,2,\dots, 1 - \xi'_i$. We have, by definition, $(\mbf{x} \mid \beta') = \phi(\mbf{s} \mid \beta')$. We write $\mbf{x}_{i,j} = \phi(\mbf{s}_{i,j})$, $i \le \lfloor n/2\rfloor$ and $j \in [1 - \xi'_i]$ for the elements of $(\mbf{x} \mid \beta')$. \vspace*{8pt}

\emph{Conditioned random multiset.} Now we define $(\mbf{A} \mid \beta,\beta')$ for $\mbf{A}$ conditioned to $(\boldbeta , \boldbeta') = (\beta,\beta')$. First of all we introduce notation which unifies the $\lceil n/2\rceil$ steps of $\beta$ and the $\lfloor n/2\rfloor$ steps of $\beta'$ which are relevant to the construction. Thus we define non-negative integers $b_1,\dots, b_n$ by  \begin{equation}\label{bi-defs} b_{i} := 1 - \xi'_i \; \;  (i = 1,\dots, \lfloor n/2\rfloor) \;\;\; \mbox{and}\;\;\;  b_{n + 1 - i} := 1 + \xi_i \; \; (i = 1,\dots, \lceil n/2\rceil).\end{equation}
$b_i$ should be thought of as roughly the number of elements of $(\mbf{A} \mid \beta,\beta')$ which are of size $\asymp 2^{i}$, as we will see in \cref{aij-dist} below.

The elements of $(\mbf{A} \mid \beta,\beta')$  may be listed as $(\mbf{a}_{i,j})_{i \in [n], j \in [b_i]}$ with
\begin{equation}\label{a-list} \mbf{a}_{i,j} = \left\{\begin{array}{ll}  \tilde\phi(\mbf{s}_{i,j}) &   1 \le i \le \lfloor n/2\rfloor; \\ \tilde\phi(n - \mbf{u}_{n + 1 - i,j}) & \lfloor n/2\rfloor < i \le n. \end{array}\right.   \end{equation}
Note in particular that we have 
\begin{equation}\label{a-size-2} \# (\mbf{A} \mid \beta,\beta') = \sum_{i = 1}^n b_i.\end{equation}

\subsection{Size and anticoncentration estimates}

We will make extensive use of \cref{aij-dist} which gives the approximate distribution of the $\mbf{a}_{i,j}$. In the proof of this we will require some simple properties of the function $\tilde\phi$. We give these, together with essentially identical properties of $\phi$, in the following lemma.

\begin{lemma}\label{lemma53}
Uniformly for $x \in [0,n]$ we have
\begin{equation}\label{phi-tilde-x-size-log} \log_2 \phi(x),\log_2 \tilde\phi(x) = x + O(1).\end{equation}    Thus $\phi(x), \tilde\phi(x) \asymp 2^x$.
\end{lemma}
\begin{proof}
We give the proofs for $\tilde\phi$. Recall throughout that $n = \lfloor \log_2 k \rfloor$. We have $\frac{n}{H_k} = \frac{1}{\log 2} + O(\frac{1}{n})$ and $H_{\tilde{\phi}(x)} = \log \tilde\phi(x) + O(1)$. Thus from $x \le \frac{n}{H_k} H_{\tilde\phi(x)}$ we obtain $x \le (\frac{1}{\log 2} + O(\frac{1}{n}))(\log \tilde\phi(x) + O(1)) = \log_2 \tilde\phi(x) + O(1)$, or in other words the RHS of \cref{phi-tilde-x-size-log} is at most the LHS. Here, we used that $\tilde\phi(x) \le k$ for $x \le n$ in order to bound the cross term $O(\frac{1}{n}) \log \tilde\phi(x)$ by an absolute constant.

In a very similar manner one may obtain an inequality in the opposite direction as well, and putting these together gives \cref{phi-tilde-x-size-log}.

The proof for $\phi$ is extremely similar (and slightly simpler). This estimate in fact holds for all $x$.
\end{proof}

Before turning to \cref{aij-dist} itself, we first give distributional estimates for the components $\mbf{x}_{i,j}$ of the conditioned lower process $(\mbf{x} \mid \boldbeta' = \beta')$.

\begin{lemma} We have
\begin{equation} \label{aij-prob-bd} \mbf{x}_{i,j} \asymp 2^i \end{equation} deterministically and
\begin{equation}\label{non-conc-n} \sup_{t\in \Z} \mb{P}(\mbf{x}_{i,j} = t) \ll 2^{-i}.\end{equation}

\end{lemma}
\begin{proof}
For both parts we use that $\mbf{x}_{i,j} = \phi(\mbf{s}_{i,j})$. Item \cref{aij-prob-bd} is immediate from \cref{lemma53} and the fact that $\mbf{s}_{i,j} \in (i-1,i]$. For \cref{non-conc-n}, observe that $\mb{P}(\phi(\mbf{s}_{i,j}) = t)$ is the probability that a uniform random variable on $(i-1, i]$ lies in $(\frac{1}{\log 2} H_{t-1}, \frac{1}{\log 2} H_t]$, which is an interval of length $1/t\log 2$. For this interval to intersect $(i-1, i]$ we must (since $\phi$ is nondecreasing) have $t \ge \phi(i-1)$, and so the desired probability is at most $1/\phi(i-1) \log 2$. The stated bound \cref{non-conc-n} then follows using \cref{lemma53}.
\end{proof}

Now we give the key estimate on the $\mbf{a}_{i,j}$, the elements of the conditioned set $(\mbf{A} \mid \beta,\beta')$.

\begin{lemma}\label{aij-dist}
Deterministically we have
\begin{equation}\label{dyadic-aij}  \mbf{a}_{i,j} \asymp 2^{i}.\end{equation}
We also have the anticoncentration estimate
\begin{equation}\label{anti-concentration-aij} \sup_{x \in \Z}\mb{P}(\mbf{a}_{i,j} = x) \ll 2^{-i}.\end{equation}  
Here, the implied constants are absolute \textup{(}and, in particular, do not depend on $k$\textup{)}.
\end{lemma}
\begin{proof} 
Item \cref{dyadic-aij} is immediate from \cref{lemma53,a-list} and the fact that $\mbf{s}_{i,j}, \mbf{u}_{i,j} = i + O(1)$.

For \cref{anti-concentration-aij}, suppose first that $\lfloor n/2\rfloor < i \le n$. Observe that 
\[ \mb{P}(\mbf{a}_{i,j} = x) = \mb{P}(\tilde\phi(n - \mbf{u}_{n + 1 - i,j}) = x) = \mb{P} ( n - \mbf{u}_{n + 1 - i,j} \in I )\] where $I := (\frac{n}{H_k} H_{x - 1}, \frac{n}{H_k} H_x ]$. Now $\mbf{u}_{n + 1 - i,j}$ is uniformly distributed on $(n -i , n + 1 - i]$, so it follows that 
\[ \mb{P}(\mbf{a}_{i,j} = x) \le |I| = \frac{n}{H_k x} \ll \frac{1}{x}.\]
Since $\mbf{a}_{i,j}$ is deterministically $\asymp 2^{i}$, the result follows.
To proof in the case $i \le \lfloor n/2\rfloor$ is very similar and left to the reader.
\end{proof}

\subsection{Coupling top and bottom to universals} 

The material in the following subsection will not be required until the end of \cref{sec4}. The idea is to couple $\mbf{A}$, whose definition is quite heavily dependent on $k$, to the `universal' upper and lower processes $\mbf{u}, \mbf{x}$. Essentially this consists of relating $\mbf{A}$ to the set $\mbf{A}_1$ which featured in our heuristic discussion. 

We first handle the large elements of $\mbf{A}$.

\begin{lemma}\label{u-couple}
Let $\beta, \beta'$ be upper and lower walks, and consider the conditioned upper process $(\mbf{u} \mid \beta)$ and the conditioned random multiset $(\mbf{A} \mid \beta,\beta')$. Let $i$ satisfy $\lfloor n/2\rfloor < i \le n$. Then $|\frac{1}{k} \mbf{a}_{i,j} - 2^{-\mbf{u}_{n + 1-i,j}} | \ll 2^{i - n} (n + 1 - i)/n$ deterministically.
\end{lemma}
\begin{proof}  By \cref{a-list} we have $\mbf{a}_{i,j}  = \tilde\phi( n - \mbf{u}_{n + 1 -i,j})$. From \cref{phis-def}, this implies that $n - \mbf{u}_{n + 1 - i,j} \in (\frac{n}{H_k} H_{\mbf{a}_{i,j} - 1}, \frac{n}{H_k}H_{\mbf{a}_{i,j}}]$.
Since $H_t, H_{t-1} = \log t + \gamma + O(\frac{1}{t})$ uniformly for $t \ge 2$, it follows that
\[ n - \mbf{u}_{n + 1 -i,j} = \frac{n}{H_k}\big(\log \mbf{a}_{i,j} + \gamma + O(2^{-i})\big).\] Multiplying by $\frac{H_k}{n}$ and using $H_k = \log k + \gamma + O(\frac{1}{k})$, $\frac{H_k}{n} = \log 2 + O(\frac{1}{n})$ and $\mbf{u}_{n + 1 - i,j} \le n + 1 -i$, this implies that
\[ \log k - \mbf{u}_{n + 1 -i,j} \log 2 = \log \mbf{a}_{i,j} + O\big(\frac{n + 1 -i}{n}\big).\] That is, writing $x := \frac{1}{k} \mbf{a}_{i,j}$ and $y := 2^{-\mbf{u}_{n + i - 1,j}}$, we have $|\log x - \log y| \ll \frac{n + 1 - i}{n}$. The desired result then follows using the inequality $|x - y| \le \max(x,y) |\log x - \log y|$ (valid for all real $x,y \in (0,1]$) and the fact that $x,y \asymp 2^{i - n}$.
\end{proof}

Now we turn to the small elements of $\mbf{A}$.

\begin{lemma}\label{s-couple}
Let $\beta, \beta'$ be upper and lower walks, and consider the conditioned lower process $(\mbf{x} \mid \beta)$, with elements $\mbf{x}_{i,j}$, and the conditioned random multiset $(\mbf{A} \mid \beta,\beta')$, with elements $\mbf{a}_{i,j}$. Then, uniformly for $i \le \lfloor n/2\rfloor$, we have $\mb{P}(\mbf{x}_{i,j} \neq \mbf{a}_{i,j})  \ll i/n$.\end{lemma}
\begin{proof}
Let $(\mbf{s} \mid \beta') = (\mbf{s}_{i,j})_{i \le \lfloor n/2\rfloor, j \in [1 - \xi'_i]}$ be the auxiliary process introduced above. Thus $\mbf{x}_{i,j} = \phi(\mbf{s}_{i,j})$ and $\mbf{a}_{i, j} = \tilde\phi(\mbf{s}_{i,j})$. Thus we must bound $\mb{P}(\phi(\mbf{s}_{i,j}) \ne \tilde\phi(\mbf{s}_{i,j}))$; that is, it suffices to prove that \[ \sum_x \big|\mb{P}(\tilde\phi(\mbf{s}_{i,j}) = x) - \mb{P}(\phi(\mbf{s}_{i,j}) = x)\big|  \ll i/n.\] From the definitions (\cref{phis-def}) of $\phi, \tilde\phi$, the task is to show
\[ \sum_x \Big| \mb{P} \big( \mbf{s}_{i,j} \in (\frac{1}{\log 2} H_{x-1}, \frac{1}{\log 2}H_x] \big) - \mb{P} \big( \mbf{s}_{i,j} \in (\frac{n}{H_k} H_{x-1}, \frac{n}{H_k}H_x] \big)\Big| \ll i/n. \]
Recall that here $\mbf{s}_{i,j}$ is uniform on $(i-1, i]$. In particular, both probabilities vanish unless $x \asymp 2^i$. For each $x$, the difference in probabilities is the length of the symmetric difference between the two intervals, and therefore the above expression is bounded by
\[ \sum_{x \asymp 2^i} \Big|\frac{n}{H_k} - \frac{1}{\log 2}\Big| \big( H_{x - 1} + H_x\big) \ll \frac{1}{n} \sum_{x \asymp 2^i} \big(H_{x-1} + H_x\big) \ll \frac{i}{n}, \] as required.
\end{proof}

\section{The random variables \texorpdfstring{$\tau$}{} and \texorpdfstring{$\tau^*$}{}}
\label{tau-sec}
In this section we develop the basic properties of the $\tau$-statistics as defined in \cref{tau-stat-def}, which feature prominently in the measure $\mu'$ (see \cref{sec17-main}). 
 
Let $\beta'$ be a lower walk with increments $\xi'_i = \beta'(i) - \beta'(i-1)$. In this section we suppose that $\beta'$ is infinite (rather than of length $\lfloor n/2\rfloor$) although we will only be dealing with finite segments of $\beta'$.

Recall the definition of the conditioned lower process $(\mbf{x} \mid \beta')$ and its elements $\mbf{x}_{i,j}$, $j \in [1 - \xi'_i]$, as described in \cref{subsec-53}. We define two statistics related to $\tau$, as defined in \cref{tau-stat-def}.

\begin{definition}\label{y-beta-def} Let $\ell \in \N$. We define
\begin{equation}\label{y-beta-form} \tau^*_{\mbf{x} \mid \beta'}(\ell) := 2^{\beta'(\ell) - \ell} \# \big\{ \sum_{i \le \ell; j \in [1 - \xi'_i]} \eps_{i,j} \mbf{x}_{i,j} : \eps_{i,j} \in \{0,1\}\big\}.\end{equation}   By convention we set $\tau_{\mbf{x} \mid \beta'}(0) = \tau^*_{\mbf{x} \mid \beta'}(0) := 1$. Suppose that 
\begin{equation}\label{lim-beta-j-assump} \lim_{j \rightarrow \infty} (j - \beta'(j)) = \infty.\end{equation}
Then we also define
\begin{equation}\label{tau-unstar-def} \tau_{\mbf{x} \mid \beta'}(\ell) := 2^{-\ell} \# \big\{ \sum_{i,j} \eps_{i,j} \mbf{x}_{i,j} : \eps_{i,j} \in \{0,1\}\big\},\end{equation} where the $(i,j)$-sum is taken over the $\ell$ smallest $\mbf{x}_{i,j}$s. 
\end{definition}
\begin{remarks}
The assumption \cref{lim-beta-j-assump} guarantees that $\tau_{\mbf{x} \mid \beta'}(\ell)$ is defined for all $\ell$. We note that the definition of $\tau_{\mbf{x} \mid \beta'}$ is compatible with the definition \cref{tau-stat-def}, being the $\tau$ statistic defined there applied to the conditioned sequence $(\mbf{x} \mid \beta')$.
We think of $\tau^*_{\mbf{x} \mid \beta'}(\ell)$ as rather analogous to $\tau_{\mbf{x} \mid \beta'}(\ell)$, but `adapted to the random walk'; instead of working with the first $\ell$ elements of $(\mbf{x} \mid \beta')$, we work with the elements up to $\ell$.
\end{remarks}

Observe that the $\ell - \beta'(\ell)$ elements $(\mbf{x}_{i,j})_{i \le \ell; j \in [1 - \xi'_i]}$ are precisely the initial $\ell - \beta'(\ell)$ elements of the conditioned lower process $(\mbf{x} \mid \beta')$. Therefore we have
\begin{equation}\label{tau-y-relation}
\tau^*_{\mbf{x} \mid \beta'}(\ell) = \tau_{\mbf{x} \mid \beta'}(\ell - \beta'(\ell)).
\end{equation}

It is straightforward to check the monotonicity properties
\begin{equation}\label{y-unstar-monotonicity} 1 = \tau_{\mbf{x} \mid \beta'}(0) \ge \dots \ge \tau_{\mbf{x} \mid \beta'}(\ell) \ge \tau_{\mbf{x} \mid \beta'}(\ell+1) \ge \dots \ge  0\end{equation} and
\begin{equation}\label{y-monotonicity} 1 = \tau^*_{\mbf{x} \mid \beta'}(0) \ge \dots \ge \tau^*_{\mbf{x} \mid \beta'}(\ell) \ge \tau^*_{\mbf{x} \mid \beta'}(\ell+1) \ge \dots \ge  0.\end{equation}

The main aim in this section is to develop some basic properties of the variables $\tau_{\mbf{x} \mid \beta'}(\ell)$ and $\tau^*_{\mbf{x} \mid \beta'}(\ell)$, particularly the latter. The first statement says that, under reasonable conditions, $\tau^*_{\mbf{x} \mid \beta'}(\ell) \lessapprox 2^{-m}$ if $\beta'$ attains the value $-m$. This should be considered a rigorous version of the heuristics just before \cref{x-2m-heur}.

Recall from \cref{bounded-walk-def} the notion of a walk being $R$-bounded to some length $L$.

\begin{lemma}\label{Y-upper}
Let $m \ge 0$ be an integer. Suppose that $\beta'$ is a lower walk and that $\beta'(q) = -m$ for some $q$ with $0 \le q \le \ell$. Suppose that $\beta'$ is $R$-bounded to length $\ell$. Then 
\begin{equation}\label{Y-minm-bd}  \tau^*_{\mbf{x} \mid \beta'}(\ell) \ll R (1+q)^{\kappa} 2^{-m}.\end{equation} 
\end{lemma}
\begin{proof} If $q = 0$ then $m = 0$ and the result is trivial from \cref{y-monotonicity}. Suppose, then, that $q \ge 1$. By the monotonicity property \cref{y-monotonicity} it suffices to handle the case $\ell = q$. By \cref{aij-prob-bd} we have
 \begin{equation*}\big| \sum_{i \le \ell;j \in [1 - \xi'_i]} \eps_{i,j} \mbf{x}_{i,j}\big| \ll \sum_{i \le \ell} 2^{i} |1 - \xi'_i| \ll R 2^{\ell}\ell^{\kappa},\end{equation*} where in the last step we used the $R$-boundedness of $\beta'$. In particular the number of distinct elements  $\sum_{i \le \ell;j \in [1 - \xi'_i]} \eps_{i,j} \mbf{x}_{i,j}$ is $\ll R 2^{\ell} \ell^{\kappa}$.  From the definition \cref{y-beta-form}, it follows that $\tau^*_{\mbf{x} \mid \beta'}(\ell) \le 2^{\beta'(\ell)} R \ell^{\kappa} = 2^{-m} R \ell^{\kappa}$, which implies \cref{Y-minm-bd}.
\end{proof}

We will also require the following variant for the unstarred $\tau_{\mbf{x} \mid \beta'}$ quantities. 

\begin{corollary}\label{Y-upper-nonstar}
Let $m \ge 0$ and $\ell \ge 1$ be integers. Suppose that \cref{lim-beta-j-assump} holds, and denote by $\ell'$ the smallest integer such that $\ell' - \beta'(\ell') \ge \ell$. Suppose that $\beta'(q) = -m$ for some $q$ with $0 \le q \le \ell'-1$, and that $\beta'$ is $R$-bounded to length $\ell'-1$. Then 
\begin{equation*} \tau_{\mbf{x} \mid \beta'}(\ell) \ll R (1 + q)^{\kappa} 2^{-m}.\end{equation*}    
\end{corollary}
\begin{proof}
By the monotonicity of the $\tau_{\mbf{x} \mid \beta'}(\ell)$ in $\ell$, the definition of $\ell'$, and \cref{tau-y-relation}, we have 
\[ \tau_{\mbf{x} \mid \beta'}(\ell) \le \tau_{\mbf{x} \mid \beta'}(\ell' - 1 - \beta'(\ell' - 1)) = \tau^*_{\mbf{x} \mid \beta'}(\ell'-1).\] The result now follows immediately from \cref{Y-upper}, replacing $\ell$ by $\ell' - 1$ in that lemma.
\end{proof}

The next fact is more difficult. It allows us to establish rapid convergence of $\tau^*_{\mbf{x} \mid \beta'}(\ell)$ to the limit as $\ell \rightarrow \infty$ under a suitable condition on the lower walk $\beta'$ (that $\beta'$ is suitably positive in the sense of \cref{good-walk-def}).

\begin{lemma}\label{lem23}
Let $\beta'$ be a lower walk. Then uniformly for non-negative integers $\ell$ we have
\begin{equation} \label{bd-on-y-beta} \mb{E}\big( \tau^*_{\mbf{x} \mid \beta'}(\ell) -  \tau^*_{\mbf{x} \mid \beta'}(\ell+1) \big) \ll 2^{-\beta'(\ell+1)}. \end{equation}
Here, the expectation is over random choices of $(\mbf{x} \mid \beta') = (\mbf{x}_{i,j})_{i;j \in [1 - \xi'_i]}$.
\end{lemma}
\begin{proof} For an integer $\ell \ge 1$, denote $\eps_{\le \ell} = (\eps_{i,j})_{i \le \ell, j \in [1 - \xi'_i]}$ and for $t \in \Z$ write
\[ r_{\mbf{x},\ell}(t) := \sum_{\eps_{\le \ell}} 1 \big( \sum_{i \le \ell, j \in [1 - \xi'_i]} \eps_{i,j} \mbf{x}_{i,j} = t \big).\] (Thus $r_{\mbf{x},\ell}(t)$ is a random variable.) Define also $r_{\mbf{x},0}(t) = 1_{t = 0}$. Thus by \cref{y-beta-form} we have, for all $\ell \ge 0$, that
\begin{equation}\label{tau-def} \tau^*_{\mbf{x} \mid \beta'}(\ell) = 2^{\beta'(\ell) - \ell} \big| \Supp(r_{\mbf{x},\ell})\big|,\end{equation} where $\Supp$ denotes support.
Define 
\[ f_{\mbf{x},\ell}(t) := \sum_{\eps_{\ell+1}} 1_{\Supp(r_{\mbf{x},\ell})} \big(t - \sum_{j = 1}^{1 - \xi'_{\ell+1}} \eps_{\ell+1, j} \mbf{x}_{\ell+1, j}\big),\] where $\eps_{\ell+1}$ denotes $(\eps_{\ell+1, j})_{j \in [1 - \xi'_{\ell+1}]}$. Since
\[ r_{\mbf{x},\ell+1}(t) = \sum_{\eps_{\ell+1}} r_{\mbf{x},\ell} \big(t - \sum_{j = 1}^{1 - \xi'_{\ell+1}} \eps_{\ell+1, j} \mbf{x}_{\ell+1, j}\big),\] we see that 
\begin{equation}\label{sup-rel-1} \Supp (f_{\mbf{x},\ell}) = \Supp (r_{\mbf{x},\ell+1}) \quad \mbox{and} \quad \sum_t f_{\mbf{x},\ell}(t) = 2^{1 - \xi'_{\ell+1}} \big|\Supp(r_{\mbf{x},\ell})\big|.\end{equation}
By expansion and a substitution, we have
\[ \sum_t f_{\mbf{x},\ell}(t)^2 = \sum_{\eps_{\ell+1}, \eps'_{\ell+1}} \sum_t 1_{\Supp(r_{\mbf{x},\ell})}(t) 1_{\Supp(r_{\mbf{x},\ell})}\Big( t - \sum_{j = 1}^{1 - \xi'_{\ell+1}} (\eps_{\ell+1, j} - \eps'_{\ell+1, j}) \mbf{x}_{\ell+1, j}\Big),\] whence
\begin{align*} \sum_t & f_{\mbf{x}, \ell}(t)^2 - 2^{1 - \xi'_{\ell+1}}\big|\Supp(r_{\mbf{x}, \ell})\big| =  \\ & = \sum_{\eps_{\ell+1} \ne \eps'_{\ell+1}} \sum_t 1_{\Supp(r_{\mbf{x}, \ell})}(t) 1_{\Supp(r_{\mbf{x}, \ell})}\Big( t - \sum_{j = 1}^{1 - \xi'_{\ell+1}} (\eps_{\ell+1, j} - \eps'_{\ell+1, j}) \mbf{x}_{\ell+1, j}\Big).\end{align*} Denote the two equal expressions here by $E(\mbf{x})$.
We may write
\begin{equation}\label{e-def-6} E(\mbf{x})= \sum_{\eps_{\ell+1} \ne \eps'_{\ell+1}} \sum_{t,t'} 1_{\Supp(r_{\mbf{x}, \ell})}(t) 1_{\Supp(r_{\mbf{x},\ell})}(t') 1\Big( t - t' = \sum_{j = 1}^{1 - \xi'_{\ell+1}} (\eps_{\ell+1, j} - \eps'_{\ell+1, j}) \mbf{x}_{\ell+1, j}\Big).\end{equation}
By Cauchy-Schwarz we have $\big|\Supp (f_{\mbf{x}, \ell})\big| \ge \big( \sum_t f_{\mbf{x}, \ell}(t) \big)^2/\sum_t f_{\mbf{x}, \ell}(t)^2$, and using this, \cref{sup-rel-1} and the definition of $E(\mbf{x})$, it follows that 
\[ \big|\Supp(r_{\mbf{x}, \ell+1})\big| \ge \frac{2^{2(1 - \xi'_{\ell+1})} |\Supp (r_{\mbf{x}, \ell})|^2}{2^{1 - \xi'_{\ell+1}} \big|\Supp (r_{\mbf{x}, \ell})\big| +  E(\mbf{x})},\] and hence from \cref{tau-def} (and using $\ell+1 - \beta'(\ell+1) = \ell - \beta'(\ell) + (1 - \xi'_{\ell+1})$) that 
\begin{equation*}  \tau^*_{\mbf{x} \mid \beta'}(\ell)  \le \tau^*_{\mbf{x} \mid \beta'}(\ell+1)\Big(1 + \frac{E(\mbf{x})}{2^{1 - \xi'_{\ell+1}} |\Supp(r_{\mbf{x}, \ell})|}\Big). \end{equation*} Taking expectations over the random choice of $(\mbf{x} \mid \beta')$ and using the trivial bound $\tau^*_{\mbf{x} \mid \beta'}(\ell+1) \le 1$ on the second term gives
\begin{equation}\label{diff-y-bet}
\mb{E} \tau^*_{\mbf{x} \mid \beta'}(\ell)  \le \mb{E}\tau^*_{\mbf{x} \mid \beta'}(\ell+1) + \mb{E}\Big( \frac{E(\mbf{x})}{2^{1 - \xi'_{\ell+1}} |\Supp(r_{\mbf{x}, \ell})|}\Big).
\end{equation}
The remaining task is to bound the final average here. Consider again the definition \cref{e-def-6} of $E(\mbf{x})$. Condition to  $\mbf{x}_{i,j} = x_{i,j}$ for $i \le \ell$ and $j \in [1 - \xi'_i]$. From \cref{e-def-6} and the anticoncentration estimate \cref{non-conc-n} it follows that 
\[ \mb{E} \big( E(\mbf{x}) \mid \mbf{x}_{i,j} = x_{i,j} \; \mbox{for $i \le \ell$ and $j \in [1 - \xi'_i]$}\big)  \ll 2^{2(1 - \xi'_{\ell+1}) - \ell} \big| \Supp(r_{x,\ell})\big|^2;\] note here that we have written $r_{x, \ell}$ rather than $r_{\mbf{x}, \ell}$ since this quantity depends only on the fixed variables $x_{i,j}$ with $i \le \ell$ and $j \in [1 - \xi'_i]$. Consequently

\[ \mb{E} \bigg( \frac{E(\mbf{x})}{2^{1 - \xi'_{\ell+1}} |\Supp(r_{x,\ell})|}  \mid \mbf{x}_{i,j} = x_{i,j} \; \mbox{for $i \le \ell$, $j \in [1 - \xi'_i]$}\bigg)  \ll 2^{1 - \xi'_{\ell+1} - \ell} \big|\Supp(r_{x,\ell})\big| \ll 2^{-\beta'(\ell+1)}, \] where here we used the trivial bound $\big|\Supp(r_{x,\ell})\big| \le 2^{\ell - \beta'(\ell)}$.
Undoing the conditioning and recalling \cref{diff-y-bet}, the estimate \cref{bd-on-y-beta} follows.

\end{proof}

\section{The random variables \texorpdfstring{$\varrho$}{} and \texorpdfstring{$\varrho^*$}{}}\label{rho-sec}

In this section we develop the basic properties of the $\varrho$-statistics as defined in \cref{rho-stat-sec}, which feature prominently in the measure $\mu$ (see \cref{prop16.2}). 
 
Let $\beta$ be an upper walk with increments $\xi_i = \beta(i) - \beta(i-1)$. In this section we suppose that $\beta$ is infinite (rather than of length $\lceil n/2\rceil$) although we will only be dealing with finite segments of $\beta$. Recall the definition of the conditioned upper process $(\mbf{u} \mid \beta)$ and its elements $\mbf{u}_{i,j}$, $j \in [1 + \xi_i]$, as described in \cref{subsec-53}.

\begin{definition}\label{xbeta-ybeta-defs}
Let $\beta$ be an upper walk and let $\ell \ge 1$ be an integer. Then we define the random variable $\varrho^*_{\mbf{u} \mid \beta}(\ell)$ by \begin{equation} \varrho^*_{\mbf{u} \mid \beta}(\ell)  := 2^{-\beta(\ell)} \sum_{\eps} 1 \Big( \sum_{i \le \ell; j \in [1 + \xi_i]} \eps_{i,j}2^{-\mbf{u}_{i, j}} \in [1 - 2^{-\ell},1]\Big).\label{x-beta-form}\end{equation} Here, $\eps$ ranges over all $\eps = (\eps_{i,j})_{i \le \ell; j \in [1 + \xi_i]}$ with $\eps_{i,j} \in \{0,1\}$.
By convention we set $\varrho^*_{\mbf{u} \mid \beta}(0) := 1$. 
\end{definition}
A very closely-related definition is that of $\varrho_{\mbf{u} \mid \beta}(\ell)$, which is defined as in \cref{x-stat-def}, with $u$ being the sequence $(\mbf{u} \mid \beta)$. It is slightly annoying to express this using the $\mbf{u}_{i,j}$ notation for the elements of $(\mbf{u} \mid \beta)$; we do this now. Fix the upper walk $\beta$. In order for the quantities $\varrho_{\mbf{u} \mid \beta}(\ell)$ to be well-defined, we will need to assume that  
\begin{equation}\label{lim-betaupper-j-assump} \lim_{j \rightarrow \infty} (j + \beta(j)) = \infty.\end{equation}
Given $\ell \ge 1$, let $\ell'$ be minimal so that $\ell' + \beta(\ell') \ge \ell$. (Note carefully that the definition of $\ell'$ here, in the context of an \emph{upper} walk $\beta$, differs from the definition in the previous section, but there should be no danger of confusion.) Define an integer $r \ge 1$ by 
\begin{equation}\label{r-defn} \ell'-1 + \beta(\ell'-1) = \ell - r.\end{equation}
We have $r \le \max(\xi_{\ell'} + 1,\ell)$. Define $\mbf{u}^*_{\ell',1},\dots, \mbf{u}^*_{\ell', r}$ to be the $r$ smallest elements among the $\mbf{u}_{\ell', j}$, $j \in [1 + \xi_{\ell'}]$, listed in ascending order so that $\mbf{u}^*_{\ell', r}$ is the largest of these elements, and therefore the $\ell$th element of $(\mbf{u} \mid \beta)$ when these are listed in ascending order. Then from the definition \cref{x-stat-def} we have
\begin{equation}\label{rho-alt} \varrho_{\mbf{u} \mid \beta}(\ell) = 2^{\mbf{u}^*_{\ell',r} - \ell}\sum_{\eps \in \{0,1\}^{\ell}} 1 \big( \sum_{i < \ell'; j} \eps_{i,j} 2^{-\mbf{u}_{i,j}} + \sum_{j = 1}^r \eps_{\ell', j} 2^{-\mbf{u}^*_{\ell', j}} \in [1 - 2^{-\mbf{u}^*_{\ell', r}},1]\big).\end{equation}

We now provide a lemma giving upper bounds for the $\varrho$ and $\varrho^*$ statistics. The gist of this is that these variables are bounded in expectation (assuming weak boundedness properties of the walk $\beta$) and that both have only a small ($\sim 2^{-m}$) probability of being nonzero if $\beta$ assumes the value $-m$. This last point corresponds to a rigorous version of the heuristic around \cref{member-a1,conc-heur}.

The proofs for $\varrho, \varrho^*$ are slightly different in minor details but are best given together.

\begin{lemma}\label{X-upper-gen} 
Let $m \ge 0$ and $\ell \ge 1$ be integers and let $R \ge 1$ be a parameter. Suppose that $\beta$ is an upper walk which is $R$-bounded to length $\ell$.  Then 
\begin{equation}\label{X-upper-simple} \mb{E} \varrho^*_{\mbf{u} \mid \beta}(\ell) \ll R^2,\end{equation} where the expectation is over random choices of $(\mbf{u} \mid \beta)$. Suppose in addition that $\beta(q) = -m$ for some $q$ with $0 \le q \le \ell$. Then 
\begin{equation}\label{X-non-vanishing} \mb{P}\big(\varrho^*_{\mbf{u} \mid \beta}(\ell) \ne 0\big) \ll R^3 (1 + q)^{\kappa} 2^{-m}.\end{equation}
Now suppose that $\beta$ satisfies \cref{lim-beta-j-assump}, and that $\beta$ is $R$-bounded to length $\ell'$, where $\ell'$ is minimal so that $\ell' + \beta(\ell') \ge \ell$. Then 
\begin{equation}\label{rho-moment-bd} \mb{E} \varrho_{\mbf{u} \mid \beta}(\ell) \ll R^{2}.\end{equation}
Suppose in addition that $\beta(q) = -m$ for some $q$ with $0 \le q \le \ell'$. Then \begin{equation}\label{rho-nonvanishing-bd} \mb{P}\big(\varrho_{\mbf{u} \mid \beta}(\ell) \ne 0\big) \ll \max(\ell, \ell')^{3} 2^{-m}.\end{equation} 
\end{lemma}
\begin{proof} An important tool in the proof will be the anticoncentration estimate
\begin{equation}\label{main-x-anti}\mb{P}(2^{-\mbf{u}_{i,j}} \in I) \ll 2^i |I|\end{equation} uniformly for all intervals $I$ and all $i$. This follows quickly from the fact that $\mbf{u}_{i,j}$ is uniformly distributed on $(i-1,i]$. An important technical point is that when $r$ is small there may not be a similar concentration estimate for the $\mbf{u}^*_{\ell, j}$, which are likely to be concentrated at one end of the interval. Throughout the proof we may assume without loss of generality that $R$ is sufficiently large. \vspace*{8pt}

\emph{Proof of \cref{X-upper-simple}.}  We observe that using the $R$-boundedness property we have (deterministically)
\begin{equation}\label{R-bded-small-1} \sum_{i > 2\log_2 R; j \in [1 + \xi_i]} \eps_{i,j} 2^{-\mbf{u}_{i,j}} \le \sum_{i > 2 \log_2 R} ( 1+\xi_i) 2^{1-i} \le \sum_{i> 2 \log_2 R} (1 + R i^{\kappa}) 2^{1-i} < \tfrac{1}{2}.\end{equation} To see the second inequality, one can bound the sum by $\ll R^{-1/2}$, then use the assumption that $R$ is sufficiently large. Now suppose that $\eps$ counts towards \cref{x-beta-form}, that is to say $\sum_{i \le \ell;j \in [1 + \xi_i]} \eps_{i,j} 2^{-\mbf{u}_{i,j}} \in [1 - 2^{-\ell},1]$. In particular, $\sum_{i \le \ell;j \in [1 + \xi_i]} \eps_{i,j} 2^{-\mbf{u}_{i,j}} \ge \frac{1}{2}$. It follows from this and \cref{R-bded-small-1} that some $\eps_{a,b}$ is nonzero with $a \le \min(2\log_2 R,\ell)$. We then have \[ \mb{P}\big(\sum_{i \le \ell; j \in [1 + \xi_i]} \eps_{i,j} 2^{-\mbf{u}_{i,j}} \in [1 - 2^{-\ell},1]\big) \ll 2^{a- \ell} \ll R^2 2^{-\ell};\] this follows by revealing all $\mbf{u}_{i,j}$ with $(i,j) \ne (a,b)$ then using \cref{main-x-anti}. Therefore \[ \mb{P}\big(\sum_{i \le \ell; j \in [1 + \xi_i]} \eps_{i,j} 2^{-\mbf{u}_{i,j}} \in [1 - 2^{-\ell},1]\big) \ll  R^22^{-\ell},\] uniformly for all $\eps$. Summing over the $2^{\sum_{i \le \ell} (1+\xi_i)} = 2^{\ell + \beta(\ell)}$ choices of $(\eps_{i,j})_{i \le \ell, j \in [1+\xi_i]}$, and multiplying by the normalising factor $2^{-\beta(\ell)}$ in the definition \cref{x-beta-form}, we obtain \cref{X-upper-simple}.\vspace*{8pt}

\emph{Proof of \cref{rho-moment-bd}.} This is similar but slightly more technical due to the failure of \cref{main-x-anti} for the $\mbf{u}^*$ variables.  First note the trivial bound \begin{equation}\label{varrho-triv-7} \varrho_{\mbf{u} \mid \beta}(\ell) \le 2^{\mbf{u}^*_{\ell', r}} \le 2^{\ell'}.\end{equation} If $\ell' \le 3$, this immediately implies the required bound (since we are assuming $R \ge 1$), so suppose in what follows that $\ell' \ge 4$. 
Suppose that $\eps$ counts towards \cref{rho-alt}. Since $\mbf{u}^*_{\ell', r} \ge \ell' - 1 \ge 3$, we have
\begin{equation}\label{pre-two-cases}  \sum_{i < \ell'; j \in [1 + \xi_i]} \eps_{i,j} 2^{-\mbf{u}_{i,j}}  + \sum_{j = 1}^r \eps_{\ell', j} 2^{-\mbf{u}^*_{\ell', j}} \ge \tfrac{3}{4} .\end{equation}

We divide into two cases: Case 1 is that there is some nonzero $\eps_{a,b}$ with $a \le \min(2 \log_2 R, \ell' - 1)$, and Case 2 is that this fails to be the case. In Case 1, we may argue as before using \cref{main-x-anti}. If we are in Case 2 then, arguing as in the proof of \cref{X-upper-simple}, we have $\sum_{i < \ell'; j \in [1 + \xi_i]} \eps_{i,j} 2^{-\mbf{u}_{i,j}} \le \frac{1}{2}$, and hence from \cref{pre-two-cases} it follows that $\sum_{j = 1}^r \eps_{\ell', j} 2^{-\mbf{u}^*_{\ell', j}} \ge \frac{1}{4}$.  Then, since $\mbf{u}_{\ell', j}^* \ge \ell'-1$, we have $r \ge 2^{\ell' - 3}$. By $R$-boundedness we have $r \le 1+\xi_{\ell'} \le 1+R(\ell')^{\kappa}$. Since $r \ge 2^{\ell' - 3} \ge 2$, this implies that $R \gg 2^{\ell'/2}$, and then \cref{rho-moment-bd} follows using \cref{varrho-triv-7}.\vspace*{8pt}

\emph{Proof of \cref{X-non-vanishing}.} Suppose first that $q \le 2 \log_2 R$. Since all steps of $\beta$ are at least $-1$, we have $q \ge m$, and so $R^2 2^{-m} \ge 1$. The stated bound is therefore trivial in this case, so we assume from now on that $q > 2 \log_2 R$.

Suppose that $\varrho^*_{\mbf{u} \mid \beta}(\ell) \ne 0$. Then there is some $\eps$ such that $\sum_{i \le \ell} \eps_{i,j} 2^{-\mbf{u}_{i,j}} \in [1 - 2^{-\ell},1]$. As before, this implies that $\eps_{a,b} \ne 0$ for some $a \le \min(2\log_2 R, \ell)$. In particular, $a \le q$. Moreover, since $\mbf{u}_{i,j} \ge i - 1$, $\sum_{i \le q} \eps_{i,j} 2^{-\mbf{u}_{i,j}}$ lies in the interval $I := [1 - \sum_{i \ge q + 1} (1 + \xi_i) 2^{1-i} - 2^{-\ell} ,    1 ]$. By revealing all $\mbf{u}_{i,j}$ except $\mbf{u}_{a,b}$, it follows (since $q \ge a$) that $\mb{P}\big( \sum_{i \le q} \eps_{i,j} 2^{-\mbf{u}_{i,j}} \in I\big) \ll 
R^2|I|$. Summing over all possibilities for $(\eps_{i,j})_{i \le q, j \in [1+\xi_i]}$, of which there are $2^{\sum_{i =1}^q (1 + \xi_i)} = 2^{q + \beta(q)} = 2^{q - m}$, it follows that 
\begin{equation}\label{nonzero-prob-inter} \mb{P}(\varrho^*_{\mbf{u} \mid \beta}(\ell) \ne 0) \ll R^2 2^{q - m}|I| = R^2 2^{-m} \big( \sum_{i \ge q +1} (1 + \xi_i) 2^{q - i} + 2^{q - \ell}\big).\end{equation}
Note that $2^{q - \ell} \le 1$. Moreover, the $R$-boundedness assumption implies that
\begin{equation}\label{r-boundedness-q} \sum_{i \ge q+1} (1 + \xi_{i}) 2^{q- i} \le q^{\kappa}  \sum_{j \ge 1}2^{-j} \Big( \frac{1 + R(j + q)^{\kappa}}{q^{\kappa}} \Big) \ll Rq^{\kappa} \end{equation} 
The desired bound \cref{X-non-vanishing} therefore follows from \cref{nonzero-prob-inter}.\vspace*{8pt}

\emph{Proof of \cref{rho-nonvanishing-bd}.} Since the steps of $\beta$ are bounded below by $-1$, we have $\ell' \ge q \ge m$. Since \cref{rho-nonvanishing-bd} is vacuous for $m \le 3$, we may assume $\ell' \ge 3$. From the definition of $\ell'$ we have that $\beta(\ell') \ge \beta(\ell' - 1)$, so the smallest $q'$ for which $\beta(q') \le -m$ has $q' \le \ell' - 1$. Replacing $q$ by $q'$ and $m$ by $m' = -\beta(q')$, we may assume that $q \le \ell' - 1$. Write $\tilde R$ for the smallest quantity such that $\beta$ is $\tilde R$ bounded to length $\ell' - 1$. Since $\sum_{i \le \ell'-1} (\xi_i + 1) = (\ell' - 1) + \beta(\ell' - 1) < \ell$, we certainly have $\xi_i \le \ell$ for all $i \le \ell' - 1$ and so (crudely) $\tilde R \le \ell$. 

Suppose that $\varrho_{\mbf{u} \mid \beta}(\ell) \ne 0$. Then there is some $\eps$ such that 
\begin{equation}\label{state-rho} \sum_{i < \ell'; j \in [1 + \xi_i]} \eps_{i,j} 2^{-\mbf{u}_{i,j}} + \sum_{j = 1}^r \eps_{\ell', j} 2^{-\mbf{u}^*_{\ell', j}} \in [1 - 2^{1 - \ell'},1].\end{equation} Since $\ell' \ge 3$, we must again be in at least one of the two cases, Case 1 and Case 2, as described following \cref{pre-two-cases}.

We first treat Case 1. Using \cref{R-bded-small-1} (with $\tilde R$ replacing $R$) we see that in this case there is some nonzero $\eps_{a,b}$ with $a \le \min(2\log_2 \tilde R,\ell'-1)$. Suppose that $q < a$. Then $q \le 2 \log_2 \tilde R$. Since $q \ge m$, $\tilde R^2 2^{-m} \ge 1$ and hence $\ell^2 2^{-m} \ge 1$. The bound \cref{rho-nonvanishing-bd} is therefore trivial in this case, so we assume from now on that $q \ge a$. 

By \cref{state-rho}, $\sum_{i \le q} \eps_{i,j} 2^{-\mbf{u}_{i,j}}$ lies in the interval $\tilde I := [1 - \sum_{q + 1 \le i \le \ell' -1} (1 + \xi_i) 2^{1 - i} - (r+1)2^{1 - \ell'}, 1]$. Here, it is important to note that $q \le \ell'-1$. Arguing as in \cref{nonzero-prob-inter} (using here that $q \ge a$, and also that $\tilde R \le \ell$) we obtain 
\[ \mb{P}\big(\varrho_{\mbf{u} \mid \beta}(\ell) \ne 0\big) \ll  \ell^2 2^{-m}\big( \sum_{q + 1 \le i \le \ell' -1} (1 + \xi_i) 2^{q - i} + (r+1)2^{q - \ell'}  \big) . \] We have $(r + 1) 2^{q - \ell'} < r + 1 \ll \ell$, so the contribution from this term is acceptable. The contribution from the sum over $i$ may be bounded exactly as in \cref{r-boundedness-q} (with $\tilde R$ replacing $R$), and this term is also acceptable. This concludes the analysis of Case 1.

Finally, we consider Case 2. As before, $r > 2^{\ell' - 3}$. We also have $r \le \ell$. Furthermore, since $\beta$ is an upper walk we have $\min_{0 \le i \le \ell'} \beta(i) \ge {-}\ell'$, and so $m \le \ell'$. Combining these facts gives $2^m \le 8 \ell$, and so \cref{rho-nonvanishing-bd} is trivial in this case.\end{proof}

Now we come to the second main result of the section, which asserts that $\varrho^*_{\mbf{u} \mid \beta}(\ell)$ has good convergence properties as $\ell \rightarrow \infty$, under the assumption that $\beta$ is appropriately $T$-positive in the sense of \cref{good-walk-def}. We will also need a related statement involving the $\varrho_{\mbf{u} \mid \beta}(\ell)$ and, as with \cref{X-upper-gen}, it makes sense to prove this at the same time. 

\begin{lemma}\label{lemma7point3}
Let $\ell\ge 1$ be an integer and let $T \ge 1$ be a parameter. Then there is some absolute constant $c \in (0,\frac{1}{2})$ such that the following are true.

Suppose that $\beta$ is an upper walk which is $T$-positive up to $\ell$. Then we have \begin{equation}\label{lem73-first} \mb{E} \big| \varrho^*_{\mbf{u} \mid \beta}(\ell-1) - \varrho^*_{\mbf{u} \mid \beta}(\ell) \big|^2\ll 2^{3T - c\ell^{1/4}}.\end{equation}

Now suppose that \cref{lim-betaupper-j-assump} holds. Given $\ell$ and $\beta$, we again define $\ell'$ to be minimal such that $\ell' + \beta(\ell') \ge \ell$. Suppose that $\beta$ is $T$-positive up to $\ell'$. Then we have \begin{equation}\label{lem73-second} \mb{E} \big| \varrho^*_{\mbf{u} \mid \beta}(\ell'-1) - \varrho_{\mbf{u} \mid \beta}(\ell) \big|^2\ll 2^{3T - c\ell^{\prime 1/4}}.\end{equation}
\end{lemma}
\begin{proof} 
We begin by observing that in both situations, the positivity of $\beta$ implies some weak boundedness properties. Starting with \cref{lem73-first}, from the definition \cref{x-beta-form} we have the trivial bound $\varrho^*_{\mbf{u} \mid \beta}(\ell) \le 2^{\ell}$, and so \cref{lem73-first} is vacuous unless $T \le \ell$. Thus we may assume $T \le \ell$ throughout. But then, using the $T$-positive nature of $\beta$, we have $\xi_i = \beta(i) - \beta(i - 1) \le  2T + \ell^{1/2 + \eta}  \le 3\ell$ for all $i \le \ell$. This implies that $\beta$ is $R$-bounded up to $\ell$ for $R := 3\ell$.

An essentially identical remark applies regarding \cref{lem73-second}, here using the trivial bound \cref{varrho-triv-7}; in this case we may assume that $\beta$ is $R' := 3\ell'$-bounded up to $\ell'$. 

We now begin the proof proper, starting with \cref{lem73-second}, indicating as we go the (minor) modifications required to prove \cref{lem73-first}.  From this point on in the proof it is convenient to adopt the following shorthand for intervals: $a + b[\eta]$ denotes the interval $[a, a + b\eta] \subset \R$. Here, $b$ is allowed to be negative (in which case we mean the interval $[a + b\eta, a]$).

Let $\lambda \in [0, \infty)$, $x \in \R$ and $L \in \N$ (we will have $L \in \{\ell', \ell\}$ later on). Set $\eps_{< L} = (\eps_{i,j})_{i < L; j \in [1 + \xi_i]} \in \{0,1\}^{L - 1 + \beta(L-1)}$. Denote by $E(\mbf{u}, \eps_{< L},\lambda, x)$ the event (for random $(\mbf{u} \mid \beta) = (\mbf{u}_{i,j})_{i,j}$) that
\begin{equation}\label{e-def} \sum_{i < L;j \in [1 + \xi_i]} \eps_{i,j} 2^{-\mbf{u}_{i,j}} \in 1- x - 2^{-\lambda} [2^{1-L}].\end{equation} Denote $\eps_{\ell'} = (\eps_{\ell',j})_{j = 1}^r \in \{0,1\}^r$.
By mild rearrangement of \cref{rho-alt} using \cref{r-defn}, we have
\begin{equation}\label{x-ell-0-def} \varrho_{\mbf{u} \mid \beta}(\ell) = 2^{-\beta(\ell'-1)} \sum_{\eps_{< \ell'}} \Av_{\eps_{\ell'}} 2^{\lambda_0(\mbf{u}^*_{\ell'})} 1\big(E(\mbf{u}, \eps_{< \ell'}, \lambda_{0}(\mbf{u}^*_{\ell'}), x_{0}(\eps_{\ell'},\mbf{u}^*_{\ell'})\big)
\end{equation}
where here $\lambda_0(\mbf{u}^*_{\ell'}) := \mbf{u}^*_{\ell',r} - (\ell' - 1)\in [0,1]$ and $x_0(\eps_{\ell'}, \mbf{u}^*_{\ell'}) := \sum_{j = 1}^r \eps_{\ell', j} 2^{-\mbf{u}^*_{\ell',j}}$, and $\Av_{\eps_{\ell'}}$ denotes the average over all tuples in $\eps_{\ell'} \in \{0,1\}^r$.
On the other hand from the definition \cref{x-beta-form} we have
\begin{equation}\label{x-ell-1-def} \varrho^*_{\mbf{u} \mid \beta}(\ell'-1) = 2^{-\beta(\ell'-1)} \sum_{\eps_{< \ell'}} \Av_{\eps_{\ell'}} 2^{\lambda_1(\mbf{u}^*_{\ell'})} 1\big(E(\mbf{u}, \eps_{< \ell'}, \lambda_{1}(\mbf{u}^*_{\ell'}), x_{1}(\eps_{\ell'},\mbf{u}^*_{\ell'})\big)
\end{equation} with $\lambda_1(\mbf{u}^*_{\ell'}) = x_1(\eps_{\ell'}, \mbf{u}^*_{\ell'}) := 0$.
Subtracting \cref{x-ell-1-def} from \cref{x-ell-0-def} gives
\begin{align} \nonumber \varrho_{\mbf{u} \mid \beta}(\ell) & - \varrho^*_{\mbf{u} \mid \beta}(\ell'-1) \\ &  = 2^{-\beta(\ell'-1)} \sum_{\delta \in \{0,1\}}(-1)^{\delta}\sum_{\eps_{< \ell'}} \Av_{\eps_{\ell'}} 2^{\lambda_{\delta}(\mbf{u}^*_{\ell'})} 1\big(E(\mbf{u}, \eps_{< \ell'}, \lambda_{\delta}(\mbf{u}^*_{\ell'}), x_{\delta}(\eps_{\ell'},\mbf{u}^*_{\ell'})\big).\label{marker-1}\end{align}
Squaring and taking expectations over the random choice of $(\mbf{u} \mid \beta)$ gives
\begin{align*} \mb{E} \big|\varrho_{\mbf{u} \mid \beta}(\ell) -\varrho^*_{\mbf{u} \mid \beta}(\ell'-1)  \big|^2  & = 2^{-2\beta(\ell'-1)} \sum_{\delta,\delta' \in \{0,1\}}(-1)^{\delta + \delta'} \times \\ & \times \sum_{\eps_{< \ell'},\eps'_{<\ell'}} \Av_{\eps_{\ell'},\eps'_{\ell'}}  p_{\eps_{< \ell'}, \eps'_{< \ell'}}\big(x_{\delta}(\eps_{\ell'}, \mbf{u}^*_{\ell'}), x_{\delta'}(\eps'_{\ell'}, \mbf{u}^*_{\ell'}), \lambda_{\delta}(\mbf{u}^*_{\ell'}\big),\lambda_{\delta'}(\mbf{u}^*_{\ell'})\big),\end{align*} where
\begin{equation}\label{p-def} p_{\eps_{< L}, \eps'_{< L}}(x,x',\lambda, \lambda') := 2^{\lambda+ \lambda'} \mb{P} \big( E(\mbf{u},\eps_{< L},\lambda, x) \cap E(\mbf{u}, \eps'_{< L},\lambda', x')\big).\end{equation}
Since $\sum_{\delta, \delta'} (-1)^{\delta + \delta'} = 0$ we have
\begin{align}\nonumber & \mb{E} \big|\varrho_{\mbf{u} \mid \beta}(\ell) -\varrho^*_{\mbf{u} \mid \beta}(\ell'-1) \big|^2 \ll \\ &  \ll 2^{-2\beta(\ell'-1)} \sum_{\eps_{< \ell'},\eps'_{<\ell'}} \max_{\substack{0 \le x,x' \le 2R\ell' 2^{-\ell'} \\ \lambda,\lambda' \in [0,1]}} \big| p_{\eps_{< \ell'}, \eps'_{< \ell'}}(x, x', \lambda,\lambda') - p_{\eps_{< \ell'},\eps'_{< \ell'}}(0,0,0,0)\big|.\label{ell2-rho-new}\end{align} Here, the constraints on $x,x'$ come from the definition of $x_0( \; , \;)$; we have $|x_0(\eps_{\ell'}, \mbf{u}^*_{\ell'})| \le r 2^{1-\ell'} \le 2R(\ell')^{\kappa} 2^{-\ell'}$.

To complete the proof of \cref{lem73-second}, it therefore suffices to show that the RHS of \cref{ell2-rho-new} is $\ll 2^{3T - \Omega(\ell^{\prime 1/4})}$. At this point, $\ell'$ can be thought of as simply a dummy variable (rather than specifically the smallest $\ell'$ such that $\ell' + \beta(\ell') \ge \ell$). Replacing it by $L$ (in both \cref{p-def} and \cref{ell2-rho-new}), the task is therefore to show that if $\beta$ is $R$-bounded and $T$-positive up to $L$, where $R := 3L$ and $T \le L$, then
\begin{equation}\label{main-prop-task} 2^{-2\beta(L-1)} \sum_{\eps_{< L},\eps'_{< L}} \max_{\substack{0 \le x,x' \le 2RL 2^{-L} \\ \lambda,\lambda' \in [0,1]}} \big| p_{\eps_{< L}, \eps'_{< L}}(x, x', \lambda,\lambda') - p_{\eps_{< L},\eps'_{< L}}(0,0,0,0)\big| \ll 2^{3T - \Omega(L^{1/4})}.\end{equation}
Before turning to the proof of this, we claim that \cref{lem73-first} also follows from it in the case $L = \ell$. For this, first recall \cref{x-ell-1-def} (with the dashes dropped from the $\ell$s):
\[ \varrho^*_{\mbf{u} \mid \beta}(\ell-1) = 2^{-\beta(\ell-1)} \sum_{\eps_{< \ell}} \Av_{\eps_{\ell}} 2^{\lambda_1(\mbf{u}_{\ell})} 1\big( \sum_{i < \ell;j \in [1 + \xi_i]} \eps_{i,j} 2^{-\mbf{u}_{i,j}} \in 1- x_1(\eps_{\ell}, \mbf{u}_{\ell}) - 2^{-\lambda_1(\mbf{u}_{\ell})} [2^{1-\ell}]\Big),\] where $\lambda_1(\mbf{u}_{\ell}) = x_1(\eps_{\ell}, \mbf{u}_{\ell}) = 0$.
On the other hand, 
\[ \varrho^*_{\mbf{u} \mid \beta}(\ell) = 2^{-\beta(\ell-1)} \sum_{\eps_{< \ell}} \Av_{\eps_{\ell}} 2^{\lambda_2(\mbf{u}_{\ell})} 1\big( \sum_{i < \ell;j \in [1 + \xi_i]} \eps_{i,j} 2^{-\mbf{u}_{i,j}} \in 1- x_2(\eps_{\ell}, \mbf{u}_{\ell}) - 2^{-\lambda_2(\mbf{u}_{\ell})} [2^{1-\ell}]\Big)\]
where $\lambda_2(\mbf{u}_{\ell}) = 1$ and $x_2(\eps_{\ell}, \mbf{u}_{\ell}) := \sum_{j = 1}^{\xi_{\ell} + 1} \eps_{\ell, j} 2^{-\mbf{u}_{\ell,j}}$. That \cref{lem73-first} is a consequence of \cref{main-prop-task} then follows almost verbatim as in the argument from \cref{marker-1} to \cref{ell2-rho-new} above.

To complete the proof of \cref{lemma7point3}, then, the remaining task is to prove \cref{main-prop-task}. The bound is trivial for $L = O(1)$, so we may assume $L$ sufficiently large. To ease notation, we write $\eps = \eps_{< L}$ and $\eps' = \eps'_{< L}$ for the rest of the proof.

For any of the $p_{\eps,\eps'}(\cdots)$ terms in \cref{main-prop-task} to be nonzero (for a given fixed $\eps, \eps'$) we must have that some event $E(\mbf{u}, \eps_{< L}, \lambda, x)$ with $0 \le x \le 2R L 2^{-L}$ and $\lambda \in [0,1]$ is nontrivial, which implies $\sum_{i < L; j \in [1 + \xi_i]} \eps_{i,j} 2^{1-i} \ge 1 - 4RL 2^{-L} > \frac{1}{2}$, using here that $R = 3L$ and that $L$ is sufficiently large. By \cref{R-bded-small-1}, it follows that there is some $(i_0, j_0)$ with $i_0 \le 2 \log_2 R$ ($< L$) such that $\eps_{i_0, j_0} \ne 0$. Say that $\eps$ is \emph{good} if it has this property. Thus we may restrict the sum in \cref{main-prop-task} to $\eps$ good, and similarly also to $\eps'$ good.

We first consider the contribution to \cref{main-prop-task} from the diagonal terms with $\eps = \eps'$ (with both being good). The number of such pairs is $< 2^{L + \beta(L-1)}$. For fixed $\eps = \eps'$, pick some $(i_0, j_0)$ with $\eps_{i_0, j_0} \ne 0$ and $i_0 \le 2 \log_2 R$. Revealing all $\mbf{u}_{i,j}$ except $\mbf{u}_{i_0,j_0}$, it follows from \cref{main-x-anti} that
\[ p_{\eps, \eps'}(x,x',\lambda, \lambda') \ll \mb{P} \big( \sum_{i < L;j \in [1 + \xi_i]} \eps_{i,j} 2^{-\mbf{u}_{i,j}} \in 1- x - 2^{-\lambda} [2^{1-L}] \big) \ll 2^{i_0 - L} \ll R^2 2^{-L}.\] Thus the contribution of these terms to \cref{main-prop-task} is $\ll R^2 2^{-\beta(L-1)}$, which is $\ll R^2 2^{T - \Omega(L^{1/4})}$ by the fact that $\beta$ is a $T$-positive walk. Thus the contribution of the diagonal terms to \cref{main-prop-task} is acceptable.

Now consider the non-diagonal terms with $\eps \ne \eps'$ and both $\eps, \eps'$ good. Order the pairs $(i,j)$ with $i < L$ and $j \in [1 + \xi_i]$ lexicographically. Let $(a, j_a)$ be the first pair in this ordering for which at least one of $\eps_{a, j_a}, \eps'_{a,j_a}$ is not zero. Let $(b, j_b)$ be the first pair for which at least one of $\eps_{b, j_b}, \eps'_{b,j_b}$ is not zero, and for which $(\eps_{a,j_a}, \eps'_{a, j_a}) \neq (\eps_{b, j_b}, \eps'_{b,j_b})$. These pairs must exist since $\eps, \eps'$ are distinct and neither is the zero tuple (since they are both good).

For future reference we count the number of pairs $\eps, \eps'$ corresponding to a given $(a,b)$. For such pairs we have $\eps_{i,j} = \eps'_{i,j} = 0$ for $i \le a - 1$, whilst 
$(\eps_{i,j}, \eps'_{i,j}) \in \{ (0,0), (\eps_{a,j_a}, \eps'_{a,j_a})\}$ for $a \le i \le b-1$, so the number of such choices is bounded by 
\begin{equation}\label{comp-W}  2^{(b + \beta(b-1)) - (a + \beta(a - 1))} 2^{2(L + \beta(L-1)) - 2(b + \beta(b-1))}.\end{equation}

Expanding out the definitions (that is, \cref{p-def} and the definition \cref{e-def} of the event $E(\cdots)$), one sees that 
\[ p_{\eps, \eps'}(x, x', \lambda, \lambda') = 2^{\lambda + \lambda'} \mb{P}\big(  Z_{\eps,\eps'}(x,x' ,\lambda,\lambda') \big) \] where
$Z_{\eps,\eps'}(x,x',\lambda,\lambda')$ is the event that 
\[ \eps_{a,j_a}  2^{-\mbf{u}_{a,j_a}} +  \eps_{b,j_b} 2^{-\mbf{u}_{b,j_b}} + t(\mbf{u}) \in 1 - x - 2^{-\lambda} [ 2^{1 - L}]\] and \[ \eps'_{a,j_a} 2^{-\mbf{u}_{a,j_a}} + \eps'_{b,j_b} 2^{-\mbf{u}_{b,j_b}} + t'(\mbf{u}) \in 1 - x' - 2^{-\lambda'} [ 2^{1 - L}] , \] where
\begin{equation*} t(\mbf{u}) :=  \sum_{(i,j) \neq (a, j_a), (b, j_b)} \eps_{i,j} 2^{-\mbf{u}_{i,j}},  \quad t'(\mbf{u}) :=  \sum_{(i,j) \neq (a, j_a), (b, j_b)} \eps'_{i,j} 2^{-\mbf{u}_{i,j}}.\end{equation*}
Write $M = M(\eps, \eps')$ for the (nonsingular) $2 \times 2$ matrix $\left(\begin{smallmatrix} \eps_{a,j_a} & \eps_{b, j_b}  \\ \eps'_{a,j_a} & \eps'_{b, j_b}\end{smallmatrix}\right)$. Then we can recast the definition of $Z_{\eps,\eps'}(x,x',\lambda,\lambda')$ as the event that 
\begin{equation}\label{star}  \left(\begin{smallmatrix} 2^{-\mbf{u}_{a, j_a}} \\ 2^{-\mbf{u}_{b, j_b}}\end{smallmatrix}\right)  \in D_{\lambda,\lambda'} - M^{-1} \left(\begin{smallmatrix} x + t(\mbf{u}) \\  x' + t'(\mbf{u})\end{smallmatrix}\right) , \end{equation} 
where $D_{\lambda,\lambda'} := M^{-1} \big( (1 - 2^{-\lambda}[2^{1-L}]) \times (1 - 2^{-\lambda'}[2^{1-L}])\big) \subset \R^2$.
To continue the computation we condition on the values of $t = t(\mbf{u})$ and $t' = t'(\mbf{u})$, thus we write 
\begin{equation}\label{pepseps-cond}  p_{\eps,\eps'}(x,x',\lambda,\lambda') = 2^{\lambda+\lambda'} \int f(t,t')  \mb{P} \big(Z_{\eps,\eps'}(x,x',\lambda,\lambda') \mid t(\mbf{u}) = t, t'(\mbf{u}) = t' \big) \,\mathrm{d} t \mathrm{d}t' ,\end{equation} where $f$ is the probability density function of the variable $(t(\mbf{u}),t'(\mbf{u}))$.
We divide into two cases for the pair $(t, t')$ via the following definition.

\begin{definition} We say that a pair $(t,t')$ is an \emph{edge pair} if at least one of the sets $S_{\lambda,\lambda', x, x'} := D_{\lambda,\lambda'} - M^{-1}\left(\begin{smallmatrix} x + t \\ x' + t'\end{smallmatrix}\right)$, $\lambda,\lambda' \in [0,1]$, $0 \le x, x' \le 2RL2^{-L}$, has a point outside the rectangle $[2^{-a}, 2^{1 - a}] \times [2^{-b},2^{1 - b}] \subset \R^2$, and also one of the $S_{\tilde\lambda,\tilde\lambda',\tilde x,\tilde x'}$ has a point inside this rectangle.\end{definition} 
Using this definition and \cref{pepseps-cond} we may decompose $p_{\eps,\eps'} = p_{\eps,\eps'}^{\edge} + p_{\eps,\eps'}^{\nonedge}$ in the obvious way. By the triangle inequality, it is sufficient to show the bound \cref{main-prop-task} with $p_{\eps,\eps'}$ replaced by each of $p_{\eps,\eps'}^{\edge}$ and $p_{\eps,\eps'}^{\nonedge}$ separately.\vspace*{8pt}

\emph{The edge case.} We begin by handling the contribution of $p_{\eps,\eps'}^{\edge}$. One may compute that the distance between any point in $S_{\lambda,\lambda',x,x'}$ and any point in $S_{\tilde\lambda,\tilde\lambda', \tilde x, \tilde x'}$ is $\ll RL 2^{-L}$. Therefore if $(t,t')$ is an edge pair, $M^{-1}\left(\begin{smallmatrix} 1-t \\  1-t'\end{smallmatrix}\right)$ (which lies in $S_{0,0,0,0}$) lies in one of the four thin rectangles $[2^{-a} - O(RL 2^{-L}), 2^{-a}] \times [2^{-b}, 2^{1 - b}]$, $[2^{1 - a} , 2^{1-a} + O(RL 2^{-L})] \times [2^{-b}, 2^{1 - b}]$, $[2^{-a}, 2^{1 - a}] \times [2^{-b} -  O(RL 2^{-L}), 2^{-b}]$ or $[2^{-a}, 2^{1 - a}] \times [2^{1-b}, 2^{1 - b} + O(RL 2^{-L})]$. Since the possible rows of $M^{-1}$ are $\pm (0,1)$, $\pm (1,0)$ and $\pm (1, -1)$, we see that one of $t, t', t - t'$ is constrained to some interval of length $O(RL 2^{-L})$. We call these Cases 1,2 and 3 respectively.

\emph{Case 1.} Observe that $\eps_{i_1,j_1} \ne 0$ for some $(i_1, j_1)$ with $i_1 \le b + 2\log_2 R$: this is a consequence of the restriction to $\eps, \eps'$ being good. Revealing all $\mbf{u}_{i,j}$ except $\mbf{u}_{i_1, j_1}$, $\mbf{u}_{a, j_a}$, $\mbf{u}_{b,j_b}$, it follows that the probability of $M^{-1}(t(\mbf{u}), t'(\mbf{u}))$ lying in one of the four thin rectangles is $\ll R^3L 2^{b - L}$.

\emph{Case 2.} The analysis is identical to Case 1 except now we use a non-zero value of $\eps'_{i_1,j_1}$.

\emph{Case 3.} Note that $\eps_{i_1,j_1} - \eps'_{i_1,j_1} \ne 0$ for some $(i_1, j_1)$ with $i_1 \le b + 2\log_2 R$, using the definition of $b$. The rest of the analysis is as in Case 1. 

By the above discussion, 
\begin{equation}\label{edge-prob} \mb{P} \big( (t(\mbf{u}), t'(\mbf{u})) \; \mbox{is an edge pair}\big) \ll R^3 L 2^{b - L}.\end{equation}

Now note that uniformly we have
\[ \mb{P}\big(  Z_{\eps,\eps'}(x,x',\lambda,\lambda') \mid t(\mbf{u}) = t, \, t'(\mbf{u}) = t'\big) \ll 2^{a + b  - 2L}.\] This follows by observing that any translate of $D_{\lambda,\lambda'}$ is contained in some $I_1 \times I_2$, where $I_1, I_2$ are intervals of length $2^{2 - L}$, and so by \cref{main-x-anti} we have 
\begin{equation}\label{p-simple-upper}  \mb{P}\big( Z_{\eps,\eps'}(x,x',\lambda,\lambda') \mid t(\mbf{u}) = t, t'(\mbf{u}) = t'\big) \le \mb{P}(2^{-\mbf{u}_{a, j_a}} \in I_1) \mb{P}(2^{-\mbf{u}_{b, j_b}} \in I_2) \ll 2^{a - L} \cdot 2^{b - L}.\end{equation}

Putting \cref{edge-prob,p-simple-upper} together, for a given $\eps, \eps'$ corresponding to $(a,b)$, uniformly in $x, x',\lambda,\lambda'$ we have
\[ p_{\eps,\eps'}^{\edge}(x, x', \lambda,\lambda) \ll R^3L 2^{a + 2b - 3L}.\] In particular this holds with $x = x' = \lambda = \lambda'=  0$.

We now bound the total contribution to \cref{main-prop-task} from the $p_{\eps,\eps'}^{\edge}$ terms by summing the above bound over $\eps, \eps'$ corresponding to a given pair $(a,b)$, then finally summing over $(a,b)$. Using \cref{comp-W}, the contribution to \cref{main-prop-task} for a fixed $(a,b)$ is bounded by 
\[ \ll R^3L \cdot 2^{-2\beta(L-1)} \cdot  2^{(b + \beta(b-1)) - (a + \beta(a - 1))} \cdot 2^{2(L + \beta(L-1)) - 2(b + \beta(b-1))}\cdot 2^{a + 2b - 3L},\] which simplifies to $R^3 L 2^{b - \beta(a - 1) - \beta(b - 1) - L}$. 
Finally, we sum over all $(a,b)$ with $a \le b \le L$. The contribution from $b < L/2$ is negligible (using the $T$-positive nature of $\beta$ to bound $-\beta(a - 1) - \beta(b - 1) \le 2T$). For the contribution from $b \ge L/2$, we use $b - L \le 0$ and $-\beta(a - 1) - \beta(b - 1) \le 2T - (L/2 - 1)^{1/2 - \eta}$, which again follows from the $T$-positive nature of $\beta$. Using these estimates we obtain our sum over all $(a,b)$, that is to say the total contribution from the edge cases, is $\ll R^32^{2T - \Omega(L^{1/4})}$. Since $R = 3L$, the $R^3$ factor can be absorbed by reducing the $\Omega$ constant, and this is acceptable towards \cref{main-prop-task}.\vspace*{8pt}

\emph{The non-edge case.} In this case, using \cref{star} we see that 
\begin{align} \nonumber 2^{\lambda + \lambda'}\mb{P}\big( Z_{\eps,\eps'}(x,x',\lambda,\lambda') \mid t(\mbf{u}) = t,\, t(\mbf{u}') = t'\big)&  - \mb{P}\big( Z_{\eps,\eps'}(0,0,0,0) \mid t(\mbf{u}) = t,\, t(\mbf{u}') = t'\big) \\ &  = \frac{1}{(\log 2)^2} \Big( 2^{\lambda + \lambda'}\int_{\mathcal{R}'} \frac{dy dy'}{y y'} - \int_{\mathcal{R}} \frac{dy dy'}{y y'}  \Big),\label{p-minus-p}\end{align}
where $\mathcal{R} := D_{0,0} - M^{-1}\left(\begin{smallmatrix} t \\ t' \end{smallmatrix} \right)$, $\mathcal{R}' := D_{\lambda,\lambda'}- M^{-1}\left(\begin{smallmatrix} x+ t \\ x'+ t' \end{smallmatrix} \right)$, and both regions of integration are contained in $[2^{-a}, 2^{1 - a}] \times [2^{-b}, 2^{1 - b}]$. Set $\left(\begin{smallmatrix} u \\ u' \end{smallmatrix} \right) := M^{-1}\left(\begin{smallmatrix} 1 - t \\ 1 - t' \end{smallmatrix}\right) \in \mathcal{R}$; thus $\left(\begin{smallmatrix} u \\ u' \end{smallmatrix} \right) \in [2^{-a}, 2^{1 - a}]\times [2^{-b}, 2^{1 - b}]$. If $\left(\begin{smallmatrix} y \\ y' \end{smallmatrix} \right) \in \mathcal{R} \cup \mathcal{R}'$ then $|u - y|, |u' - y'| \ll R L 2^{-L}$. Now since $a \le b$, both first partial derivatives of $1/yy'$ on $[2^{-a}, 2^{1 - a}] \times [2^{-b}, 2^{1 - b}]$ are bounded by $\ll 2^{a + 2b}$, and therefore on $\mathcal{R} \cup \mathcal{R}'$ we have
\[ \frac{1}{y y'} = \frac{1}{uu'} +  O(R L 2^{-L} 2^{a + 2b}).\]
Substituting into \cref{p-minus-p}, the contributions from the two $\frac{1}{uu'}$ terms cancel since $\mu(\mathcal{R}) = 2^{\lambda + \lambda'}\mu(\mathcal{R}')$ (here $\mu$ is Lebesgue measure on $\R^2$). The contribution from the $O(R L 2^{-L} 2^{a + 2b})$ error is $\ll RL 2^{a + 2b - 3L}$, since $\mu(\mathcal{R}), \mu(\mathcal{R}') \ll 2^{-2L}$.

Combining these observations gives 
\begin{align*} \Big|2^{\lambda + \lambda'}\mb{P}\big( Z_{\eps,\eps'}(x,x',\lambda,\lambda') \mid t(\mbf{u}) = t,\,t(\mbf{u}') = t'\big)  - \mb{P}\big(Z_{\eps,\eps'}(0,0,0,0) & \mid t(\mbf{u}) = t,\,t(\mbf{u}') = t'\big)\Big| \\ & \ll RL 2^{a + 2b - 3L}.\end{align*}
Integrating over all non-edge $(t,t')$ and using $\int f(t,t') \, \mathrm{d}t \mathrm{d}t' = 1$, we have
\[ \big| p^{\nonedge}_{\eps,\eps'}(x,x',\lambda,\lambda) - p^{\nonedge}_{\eps,\eps'}(0,0,0,0)\big| \ll RL 2^{a + 2b - 3L}.\]

Exactly as before, we may sum this over all $\eps, \eps'$ corresponding to a given $(a,b)$, and then over all $(a,b)$, obtaining the same bound $2^{2T - \Omega( L^{1/4})}$ for the contribution of the $p_{\eps,\eps'}^{\nonedge}$ terms to \cref{main-prop-task}. 

This concludes the proof of \cref{main-prop-task} and hence of \cref{lemma7point3}.
\end{proof}

A corollary of \cref{lemma7point3} is a square mean bound for $\varrho^*_{\mbf{u} \mid \beta}(\ell)$, which we will need in \cref{sec15-mainthm-proof}.

\begin{corollary}
\label{rho-ell-2-bound} 
Let $\ell$ be sufficiently large and let $T \ge 1$ be a parameter. Suppose that $\beta$ is an upper walk which is $T$-positive up to $\ell$. Then we have \[ \mb{E}  \varrho^*_{\mbf{u} \mid \beta}(\ell)^2 \ll  2^{3T}.\]
\end{corollary}
\begin{proof}
We use the inequalities
\[ \mb{E} |X_N - X_0 |^2  \le\mb{E}\big(  \sum_{i = 1}^N |X_i - X_{i-1}|\big)^2  \le \big( \sum_{i =1}^N i^{-2} \big) \big(\sum_{i = 1}^N i^2 \mb{E} |X_i - X_{i - 1}|^2 \big)
\] and
$\mb{E} |X_N|^2 \le 2 \mb{E}|X_N - X_0|^2 + 2\mb{E}|X_0|^2$, valid for any sequence $(X_i)_{i = 0}^N$ of real-valued random variables. Applying these with $N = \ell$ and $X_i = \varrho^*_{\mbf{u} \mid \beta}(i)$, the result follows from \cref{lemma7point3}.
\end{proof}
\begin{remark}
A somewhat shorter, direct, proof of  \cref{rho-ell-2-bound} is possible, along the lines of \cref{lemma7point3} but without the careful analysis of `edge' cases. We leave the details to the interested reader.  
\end{remark}

\part{Background material}\label{part2}

Each section of this part of the paper may be read independently of the rest of the paper. Our aim here is to assemble various background ingredients for the remaining parts of the argument. 

\section{Properties of random walks}\label{random-walk-sec}

In this section we collect the results we need on random walks constrained to be almost positive. Throughout the section $\boldbeta$ will be a random walk with increments $\boldxi_i \samedist \Pois(1) - 1$, and $\boldbeta'$ a random walk with increments $\boldxi'_i \samedist 1 - \Pois(1)$, though much of the discussion would apply to rather general random walks. 

For integers $m$, we write
\begin{equation}\label{hm-def}  h(m) = \lim_{N \rightarrow \infty} N^{1/2} \mb{P} \big(\min_{1 \le i \le N} \boldbeta(i) \ge -m\big)\end{equation} 
and
\begin{equation}\label{hm-def-2}  h'(m) = \lim_{N \rightarrow \infty} N^{1/2} \mb{P} \big(\min_{1 \le i \le N} \boldbeta'(i) \ge -m\big).\end{equation} The following proposition assembles the basic properties of these quantities.

\begin{proposition}\label{first-rand-walk-pos}
We have the following statements.
\begin{enumerate}
\item $h(m)$, $h'(m)$ exist for all $m \ge 0$ \textup{(}that is, the limits in \cref{hm-def,hm-def-2} exist for $m \ge 0$\textup{)}.
\item For $m \ge 0$ we have $h(m) = (2/\pi)^{1/2} (m + 1)$ and $h'(m) = (2/\pi)^{1/2} (m + O(1))$.
\item Let $0\le m \le \sqrt{N}$. Then, uniformly,
\begin{equation}\label{prop81-3-i} N^{1/2} \mb{P}\big(\min_{1 \le i \le N} \boldbeta(i) \ge - m\big) \asymp m + 1,\end{equation} and similarly for $\boldbeta'$. Moreover we have the upper bound
\begin{equation}\label{prop81-3-ii} N^{1/2} \mb{P}\big(\min_{1 \le i \le N} \boldbeta(i) \ge - m\big) \ll m + 1\end{equation}
with no restriction on $m$.
\end{enumerate}
\end{proposition}
\begin{proof} These may mostly be regarded as classical results from the fluctuation theory of random walks. However, it is not easy to locate appropriate references. See \cref{poisson-app} for further details and proofs. We remark that \cref{prop81-3-ii} is a trivial consequence of \cref{prop81-3-i}.
\end{proof}

Once it is established that $h(m)$ exists for all $m \ge 0$, it follows by conditioning on the first step of the walk that $h(m)$ exists for all $m \in \Z$ and that we have the recurrence relation
\begin{equation}\label{vbar-recur} h(m) = \sum_{r \ge -m} \mb{P}(\Pois(1) - 1 = r) h(m + r)\end{equation} (and similarly for $h'$). $h(-1)$ is a convenient shorthand for $\lim_{N \rightarrow \infty} N^{1/2} \mb{P} \big( \min_{1 \le i \le N} \boldbeta(i) > 0\big)$, and similarly for $h'(-1)$. We remark that $h'(m) = 0$ for $m \le -2$, since $\boldbeta'(1) \le 1$.

\subsection{Spaces of almost positive paths}\label{sec7.1}
In this subsection we describe how, associated to random walks such as $\boldbeta,\boldbeta'$, there are certain measures on $\Z^\N$ supported on sequences which are bounded below. One should heuristically think of this as allowing us to sample `a random $\boldbeta$ which is bounded below', though some care should be taken with this interpretation since these are not probability measures. The construction is mostly standard and is an example of Doob conditioning; we give brief details for the convenience of the reader.

Let $m \ge -1$. Denote by $\Z^{\N}_{\ge - m}$ the space of sequences $\beta = (\beta(i))_{i=1}^{\infty}$ with $\beta(i) \in \Z$ and $\beta(i) \ge -m$ for all $i$.
Consider the $\sigma$-algebra on $\Z^\N_{\ge -m}$ generated by `cylinder' sets of the form $\Gamma(c_1,\dots, c_{\ell}) := \{ \beta \in \Z^\N_{\ge -m} : \beta(i) = c_i \; \mbox{for $1 \le i \le \ell$}\}$ for some $\ell \in \N$, where here we may restrict to $\min_{1 \le i \le \ell} c_i \ge -m$.

Now consider the upper random walk $\boldbeta$, that is to say the random walk with $\Pois(1) - 1$ increments. Define $\omega_m(\Z^{\N}_{\ge -m}) = h(m)$ and
\begin{equation}\label{omega-meas-def} \omega_m \big(\Gamma(c_1,\dots, c_{\ell})\big) := h(m + c_{\ell}) \mb{P}\big(\boldbeta(i) = c_i \; \mbox{for $1 \le i \le \ell$}\big).\end{equation}

\begin{lemma}\label{72lemma}
$\omega_m$ is well-defined and extends to a measure on $\Z^{\N}_{\ge -m}$ \textup{(}with the $\sigma$-algebra described above\textup{)}. $\omega_m$ is supported on sequences $\beta$ for which $\beta(i+1) \ge \beta(i) - 1$ for all $i$.
\end{lemma}
\begin{proof} By the Daniell-Kolmogorov extension theorem it is enough to check the compatibility condition 
\begin{equation*}\sum_{c_{\ell+1}} \omega_m \big( \Gamma(c_1,\dots, c_{\ell},c_{\ell+1})\big) = \omega_m \big( \Gamma(c_1,\dots, c_{\ell})\big).\end{equation*}
We have
\begin{align*}   \sum_{c_{\ell+1}} \omega_m \big( \Gamma(c_1,\dots, c_{\ell},c_{\ell+1})\big)  & = \sum_{c_{\ell+1}  }h(m + c_{\ell+1}) \mb{P}\big(\boldbeta(i) = c_i \; \mbox{for $1 \le i \le \ell+1$}\big) \\ & = \sum_{c_{\ell+1}  }\mb{P}\big(\boldxi_{\ell+1} = c_{\ell+1} - c_{\ell}\big)h(m + c_{\ell+1}) \mb{P}\big(\boldbeta(i) = c_i \; \mbox{for $1 \le i \le \ell$}\big) \\ & = h(m + c_{\ell}) \mb{P}\big(\boldbeta(i) = c_i \; \mbox{for $1 \le i \le \ell$}\big) = \omega_m \big( \Gamma(c_1,\dots, c_{\ell})\big), \end{align*}
where the $\boldxi_i$ are the increments of $\boldbeta$ and in the penultimate step we used \cref{vbar-recur} with $m$ replaced by $m + c_{\ell}$ and the dummy variable $r$ equal to $c_{\ell+1} - c_{\ell}$. The key point to note is that this is valid since $m + c_{\ell+1} \ge 0$ by assumption.
\end{proof}
In an essentially identical manner we define measures $\omega'_m$ associated to the lower walk $\boldbeta'$ (in fact a similar construction holds for any random walk $\boldgamma$ with identical increments satisfying suitable assumptions, but we only need $\omega_m$ and $\omega'_m$ in the present paper).

The following lemma provides a reasonably intuitive way to think about these measures.
\begin{lemma}\label{lemma7.2}
Let $m \ge -1$ and $L \ge 1$ be integers, and suppose that $S \subset [-m, \infty)^L \subset \Z^L$. Denote by $E$ the set of walks $\beta$ with $(\beta(1),\dots, \beta(L)) \in S$. Then 
\begin{equation}\label{limit-omega-m} \omega_m \{ \beta \in E\}  = \lim_{N \rightarrow \infty} N^{1/2}\mb{P} \big(\boldbeta \in E , \min_{1 \le i \le N} \boldbeta(i) \ge -m\big). \end{equation}
Suppose moreover that $L \le N/2$ and that $m \le \sqrt{N}$. Then
\begin{equation}\label{lem72-eq}  \mb{P} \big(\boldbeta \in E, \; \min_{1 \le i \le N} \boldbeta(i) \ge -m\big) \ll  N^{-1/2} \omega_m \{ \beta \in E \}.\end{equation}
Analogous results hold for $\boldbeta'$ and $\omega'_m \{ \beta' \in E\}$.
\end{lemma}
\begin{proof}
We first establish \cref{limit-omega-m}. First of all, fix $(c_1,\dots, c_L) \in S$, and let $N$ be much larger than $L$. Considering the translated walk $\boldgamma(i) := \boldbeta(i + L) - \boldbeta(L)$, we see that 
\begin{align*} \mb{P}\big(\boldbeta(1) = & c_1,\dots, \boldbeta(L) = c_L , \min_{1 \le i \le N} \boldbeta(i) \ge - m\big)\\ & = \mb{P}\big(\boldbeta(1) = c_1,\dots, \boldbeta(L) = c_L \big) \mb{P}\big( \min_{1 \le i \le N - L} \boldgamma(i) \ge -m - c_L\big).\end{align*} By the definitions of $h(m + c_L)$ and of $\omega_m$, this is
\begin{align*}  \mb{P}\big(\boldbeta(1) = c_1, & \dots, \boldbeta(L) = c_L \big) \cdot (1 + o_{N \rightarrow \infty}(1)) (N - L)^{-1/2} h(m + c_L) \\ & = \mb{P}\big(\boldbeta(1) = c_1,\dots, \boldbeta(L) = c_L \big) \cdot (1 + o_{N \rightarrow \infty}(1)) N^{-1/2} h(m + c_L) \\ & = (1 + o_{N \rightarrow \infty}(1)) N^{-1/2} \omega_m (\Gamma(c_1,\dots, c_L)).\end{align*}
Summing over all $(c_1,\dots, c_L) \in E$ gives \cref{limit-omega-m}. The proof of \cref{lem72-eq} is almost identical, but now using the second statement of \cref{first-rand-walk-pos} (2) and (3) in place of the definition of $h( \cdot)$, and applying the bound $N \ll N - L$, which is valid in the range $L \le N/2$.
\end{proof}

\emph{Path spaces.} We recall that, associated to $\boldbeta$, there is the standard notion of the path space $\Z^{\N}$ together with its associated measure which, with a slight abuse of notation, we denote by $\mb{P}$. Here the $\sigma$-algebra is generated by cylinder sets $\Gamma(c_1,\dots, c_{\ell})$ (but now in $\Z^{\N}$) and $\int 1_{\Gamma(c_1,\dots, c_{\ell})}(\beta) \, \mathrm{d} \mb{P}(\beta) := \mb{P}\big(\boldbeta(1) = c_1,\dots, \boldbeta(\ell) = c_{\ell})$. It follows from the definitions that if $F$ is a function supported on paths $\beta$ with $\beta(i) \ge -m$ for all $i$, but otherwise depending only on $\beta(1),\dots, \beta(\ell)$, then 
\begin{equation}\label{doob-path} \int F(\beta) \, \mathrm{d} \omega_m = \int F(\beta) h(m + \beta(\ell)) \,\mathrm{d} \mb{P}(\beta).\end{equation}
There is a similar construction with respect to $\boldbeta'$, and we will use $\mb{P}'$ to denote the corresponding path measure.

\subsection{Crude bounds for endpoints of almost positive walks} 

The following lemma may be found in various places, for instance \cite[Lemma 20]{VW2009} and \cite{addario-berry-reed-1}. A very succinct argument is presented as Theorem 1 in the unpublished note \cite{addario-berry-reed-2}; we sketch this in \cref{a3-sec}.

\begin{lemma}\label{rand-walk-aux}
Uniformly for integers $N \ge 1$ and $x \ge 0$ we have
\[ \mb{P}\big(\boldbeta(N) = x, \, \min_{1 \le i \le N} \boldbeta(i) \ge 0\big) \ll \min \big(N^{-1}, (1 + x) N^{-3/2}\big),\] and similarly for $\boldbeta'$.
\end{lemma}\begin{proof} See the references above or \cref{a3-sec}.\end{proof}

We will need the following generalisation in which the condition on $\min_{1 \le i \le N} \boldbeta(i)$ is relaxed.

\begin{lemma}
Uniformly for integers $N \ge 1$, $m \ge 0$ and $x \ge 0$ we have
\begin{equation}\label{est-72-geq} \mb{P}\big(\boldbeta(N) = x,\, \min_{1 \le i \le N} \boldbeta(i) \ge -m\big) \ll (1 + m)^2(1 + x+ m) N^{-3/2}.\end{equation}
and
\begin{equation}\label{cor745} \mb{P}\big( \boldbeta(N) = x, \, \min_{1 \le i \le N} \boldbeta(i) \ge -m\big) \ll (1 + m)^2 N^{-1}.\end{equation}
Similar estimates hold for $\boldbeta'$.
\end{lemma}
\begin{proof}
In both cases we may assume that $\min_{1 \le i \le N} \boldbeta(i) = -s$ for some $s$ with $0 \le s \le m$ since the contribution from walks for which this is not the case can be bounded using \cref{rand-walk-aux}. For $j \in [N]$ define the probability
\[ p(j,s,x) := \mb{P}\big( \boldbeta(N) = x,\, \boldbeta(j) = \min_{1 \le i \le N} \boldbeta(i)  = -s\big).\] Thus it suffices to obtain upper bounds on 
\begin{equation}\label{sum-psx} \sum_{0 \le s \le m}\sum_{1 \le j \le N} p(j,s,x).\end{equation}
To bound $p(j,s,x)$ we employ a reversal principle, running a walk with $\boldxi'$ increments backwards from $j$ and a walk with $\boldxi$ increments forwards. (Specifically, these walks are given by $\boldgamma'(i) := \boldbeta(j - i) + s$ and $\boldgamma(i) := \boldbeta(j + i) + s$ respectively.) From this we see that 
\begin{equation}\label{reflection-applied} p(j,s,x) =  \mb{P}\big(\boldgamma'(j) = s, \;\min_{1 \le i \le j} \boldgamma'(i) \ge 0\big) \mb{P}\big(\boldgamma(N - j) = x+ s, \min_{1 \le i \le N - j} \boldgamma(i) \ge 0\big),\end{equation} where the second term is absent (equal to 1) when $j = N$, and for the first term to be valid we needed $s \ge 0$.

The second probability term in \cref{reflection-applied} is $\ll (N + 1 - j)^{-1}$ by \cref{rand-walk-aux} (considering the cases $j < N$ and $j = N$ separately). The first probability term in \cref{reflection-applied} is $\ll (1 + s) j^{-3/2}$, again using \cref{rand-walk-aux}. Inputting these bounds into \cref{reflection-applied} gives the estimate
\begin{equation}\label{argim-each-x} p(j,s,x) \ll (1 + s) j^{-3/2} (N + 1 - j)^{-1}.\end{equation} 
It then follows easily that the sum in \cref{sum-psx} is $\ll(1 + m)^2 N^{-1}$, which gives \cref{cor745}.

Alternatively, one can use the second estimate in \cref{rand-walk-aux} on the second probability term in \cref{reflection-applied}, obtaining
\begin{equation*} p(j,s,x) \ll (1 + s)  (1 + x + s) j^{-3/2}(N + 1 - j)^{-3/2}.\end{equation*} 
From this it follows that the sum in \cref{sum-psx} is $\ll (1 + m)^2 (1 + x + m) N^{-3/2}$, which gives \cref{est-72-geq}.
\end{proof}
\begin{remarks}
The statements \cref{est-72-geq,cor745} are provided as natural generalisations of \cref{rand-walk-aux}. In certain applications below one could save occasional powers of $1 + m$ by slightly more efficient use of the same underlying ideas. However, these powers of $1 + m$ are not important to us in this paper.  The bound \cref{argim-each-x} will be used a few times later on in its own right. Finally, we note that a similar statement with a similar proof may be found as \cite[Theorem 2]{ford-prob}.
\end{remarks}

As an application, we have the following bounds for measures of certain sets of walks.

\begin{lemma}\label{argmin-omega-lem}
Uniformly for $m \ge 0$, $i \ge 1$ and $j \ge -m$ we have
\begin{equation}\label{om-bd-3} \omega_m \{\beta \colon \beta(i) = j\} \ll (1 + m)^2(1 + j + m)^2 i^{-3/2},\end{equation} and similar bounds hold for $\omega'_m$.
In particular for all $q \ge 0$ we have
\begin{equation}\label{qm-beta-bd}  \omega_m \{ \beta \colon\beta(q) = -m \} \ll (1 + m)^2 (1 + q)^{-3/2}.\end{equation} 
\end{lemma}
\begin{proof}
For \cref{om-bd-3}, the value is $h(m + j) \mb{P}\big( \boldbeta(i) = j,\; \min_{1 \le i' \le i} \boldbeta(i') \ge -m)$.
We have $h(m + j) \ll 1 + m + j$ by \cref{first-rand-walk-pos} (2), and the probability term may be bounded using \cref{est-72-geq}. This gives \cref{om-bd-3}. The bound \cref{qm-beta-bd} follows from \cref{om-bd-3} by taking $i = q$ and $j = -m$. 
\end{proof}

We also have the following statement about the position of the minimum.

\begin{lemma}\label{min-pos-lem}
Uniformly for integers $N \ge 1$, $0 \le q \le N$, and $m \ge 0$, we have
\[ \mb{P}\big( \boldbeta(q) = -m, \, \min_{1 \le i \le N} \boldbeta(i) \ge -m\big) \ll (1 + m)(1 +  q)^{-3/2}(N + 1 - q)^{-1/2}. \]
\end{lemma}
\begin{proof}
When $q = 0$ we must have $m = 0$, and the result is immediate from the existence of $h(0)$. Suppose that $q \ge 1$. Then by \cref{est-72-geq} we have $\mb{P}\big( \boldbeta(q) = -m, \, \min_{1 \le i \le q} \boldbeta(i) = -m\big) \ll (1 + m)^2 q^{-3/2}$. When $q = N$ this gives the result immediately. Suppose that $q < N$. The shifted walk $\boldgamma(i) := \boldbeta(q + i) - \boldbeta(q)$ of length $N - q$ must have $\min_{1 \le i \le N - q} \boldgamma(i) \ge 0$ which, by the existence of $h(0)$, occurs with probability $\ll (N - q)^{-1/2} \ll (N + 1 - q)^{-1/2}$. Since this walk is independent of $(\boldbeta(i))_{1 \le i \le q}$, the stated bound follows.
  \end{proof}

\subsection{Local limit theorems for almost positive walks and applications}
The main result of this subsection is a quantitative version of the local limit theorem for random walks constrained to be almost positive.  Here is the statement we need. In this result \begin{equation}\label{raleigh-intro} W(x) =  x e^{-x^2/2} 1_{x \ge 0}\end{equation} denotes the density of the Rayleigh distribution.
\begin{proposition}\label{lem13.5} There is an absolute constant $\eps_0 > 0$ such that the following holds.
Uniformly for integers $m, x, N$ with $m ,x \ge 0$ and $N \ge 1$ we have
\begin{equation}\label{lem13.5-expr} \mb{P} \big(\boldbeta(N) = x,\; \min_{1 \le i \le N} \boldbeta(i) \ge -m\big) = N^{-1} h(m) W\big(\frac{x}{\sqrt{N}}\big) + O\big((1 + m)^{O(1)} N^{-1 - \eps_0}\big)\end{equation} and
\begin{equation}\label{lem13.5-exprb} \mb{P} \big(\boldbeta'(N) = x,\; \min_{1 \le i \le N} \boldbeta'(i) \ge -m\big) = N^{-1} h'(m) W\big(\frac{x}{\sqrt{N}}\big) + O\big((1 + m)^{O(1)} N^{-1 - \eps_0}\big).\end{equation} 
\end{proposition}
\begin{remarks}
For the proof, see \cref{loc-lim-thm-app} and \cref{dtw-appendix}. We give some brief bibliographical remarks here. The result also holds in the case $m = -1$ (with $(1 + m)^{O(1)}$ replaced by $1$) and this is a result of Denisov, Tarasov and Wachtel \cite[Theorem 1]{DTW24b}; however, this is a very precise asymptotic with a correspondingly lengthy proof, and moreover depends on \cite{DTW24}. For these reasons, we give an essentially self-contained account in our appendices. Other relevant references are Caravenna \cite{caravenna} (who was seemingly the first to prove a local limit theorem for random walks conditioned to stay positive, but with no quantitative dependencies) and Grama and Xiao \cite[Theorem 6.1]{grama-xiao}, which (with a little work) would imply \cref{lem13.5}. Let us be clear that we claim no originality for this result.
    \end{remarks}

For the remainder of the paper $\eps_0$ will denote a fixed constant for which \cref{lem13.5} holds true.

\subsection{Coupling results}

We now derive some results about coupling $\boldbeta,\boldbeta'$ together to have nearby endpoints. These will be critical inputs in our argument.

\begin{lemma}\label{decoupling} Let $\eps_1 > 0$ be a sufficiently small constant. Let $m,m'$ be non-negative integers. Let $L, N, N'$ be positive integers with $|N' - N| \le 1$. Let $c_1,\dots, c_L$ and $c'_1,\dots, c'_L \in \Z$ with $c_i \ge -m$ and $c'_i \ge -m'$ for all $i$, $1 \le i \le L$. If $m > 0$ suppose that $\min_{1 \le i \le L} c_i = -m$, and if $m' > 0$ that $\min_{1 \le i \le L} c'_i = -m'$. Assume that \[ \max_{1 \le i \le L} c_i, \; \max_{1 \le i \le L} c'_i \le N^{\eps_1} \quad \mbox{and that} \quad L,|d|,m,m' \le N^{\eps_1}.\] Then
\begin{align*} \mb{P} \big(\min_{0 \le i \le N} \boldbeta(i) = -m,\;  \min_{0 \le i \le N'} \boldbeta'(i) & = -m',\;  \boldbeta(N) - \boldbeta'(N') = d \mid  \boldbeta(i) = c_i, \; \boldbeta'(i) = c'_i, \; i = 1\dots,L\big) \\ & =  \frac{\sqrt{\pi}}{4} N^{-3/2} h(m + c_L) h'(m' + c'_L) + O(N^{-3/2 - \eps_1}). \end{align*}   
\end{lemma}
\begin{proof} The result is vacuous for $N = O(1)$ so we assume $N$ sufficiently large throughout. 
Let $\boldgamma$ be the random walk of length $N- L$ defined by $\boldgamma(i) := \boldbeta(L+i) - c_L$, thus $\boldgamma$ has $\Pois(1) - 1$ increments. Define $\boldgamma'$ similarly. The condition that $\min_{0 \le i \le N} \boldbeta(i) = -m$ is then precisely the condition that $\min_{1 \le i \le N-L} \boldgamma(i) \ge -m-c_L$, and similarly for $\boldbeta'$.
Therefore the probability we are interested in equals
\[  \sum_{x - x'= d} \mb{P} \big( \boldgamma(N - L) = x,\; \min_{1 \le i \le N - L} \boldgamma(i) \ge  -m - c_L \big) \cdot \mb{P} \big( \boldgamma(N' - L) = x', \;\min_{1 \le i \le N' - L} \boldgamma'(i) \ge -m' - c'_L \big) . \]
We can localise the sum to $x, x' \le N^{1/2 + \eps_1}$ with simple large deviation estimates on $\mb{P}(\boldgamma(N - L) \ge N^{1/2 + \eps})$ and the corresponding estimate with $\boldgamma'$ (see \cref{lemmaF.1}).
By \cref{lem13.5}, the remaining sum is
\[ \sum_{\substack{x - x'= d \\ x, x' \le N^{1/2 + \eps_1}}} \!\!\!\!\!\!\Big( \frac{h(m + c_L)}{N - L} W\big(\frac{x}{\sqrt{N - L}}\big) + O\big(  N^{C\eps_1 - \eps_0 - 1}\big)\Big)\Big( \frac{h(m' + c'_L)}{N' - L} W\big(\frac{x'}{\sqrt{N' - L}}\big) + O\big(  N^{C\eps_1 - \eps_0 - 1}\big)\Big) \] for some absolute constant $C$. 
By the assumptions and \cref{first-rand-walk-pos} (2) we have $h(m + c_L), h'(m' + c'_L) \ll N^{\eps_1}$ and so one may see that the contribution from all terms involving one of the $O(\cdot)$ errors is acceptable, provided $\eps_1$ is sufficiently small in terms of $\eps_0$. We are then left with the sum
\begin{equation*}\frac{h(m + c_L) h'(m' + c'_L)}{(N - L)(N' - L)}  \sum_{\substack{x - x'= d \\ x, x' \le N^{1/2 + \eps_1}}} W\big(\frac{x}{\sqrt{N - L}}\big) W\big(\frac{x'}{\sqrt{N' - L}}\big) . \end{equation*}
Now one may check that $\frac{x'}{\sqrt{N' - L}}, \frac{x}{\sqrt{N - L}} = \frac{x}{\sqrt{N}} + O(N^{\eps_1 - 1/2})$ in the range of summation, and so the above is
\[ N^{-2} h(m + c_L) h'(m' + c'_L)\big(1 + O(\frac{L}{N})\big) \sum_{x, x - d \le N^{1/2 + \eps_1}} \Big( W\big(\frac{x}{\sqrt{N}}\big)^2 + O(N^{\eps_1 - 1/2})\Big). \] The contributions from both the $O(L/N)$ and the $O(N^{\eps_1 - 1/2})$ terms are easily seen to be acceptable, so we are left with
\[ N^{-2} h(m + c_L) h'(m' + c'_L) \sum_{x, x - d \le N^{1/2 + \eps_1}}  W(\frac{x}{\sqrt{N}})^2. \] Comparing the sum to an integral we see that 
\[ N^{-1/2}  \sum_{x, x - d \le N^{1/2 + \eps_1}} W\big(\frac{x}{\sqrt{N}}\big)^2 = \int^{\infty}_0 W(t)^2 \, \mathrm{d}t + O(N^{-1/2}) = \frac{\sqrt{\pi}}{4} + O(N^{-1/2}).\]
This concludes the proof.
\end{proof}

The next result allows us to dispense with `rare' events in coupled situations like the above.

\begin{lemma}\label{lem13.7a}
Let $M, N,N'$ be positive integers with $|N' - N| \le 1$. Suppose that $\mathcal{E}$ is an event on the sample space of walks $\boldbeta$ of lengths $\le N$. Then uniformly for nonnegative integers $m, m'\le M$ and $d \in \Z$ we have
\begin{align*} \mb{P}&\big(\boldbeta \in \mathcal{E}, \; \min_{0 \le i \le N}  \boldbeta(i) = -m, \; \min_{0 \le i\le N'}  \boldbeta'(i) = -m', \;\boldbeta(N) - \boldbeta'(N') = d\big) \\ & \ll M^2 N^{-1} \mb{P} \big( \boldbeta \in \mathcal{E}, \;\min_{0 \le i \le N} \boldbeta(i) = -m\big)  \ll M^{3} N^{-3/2} \mb{P} \big( \boldbeta \in \mathcal{E} \mid \min_{1 \le i \le N} \boldbeta(i) \ge -m\big)  .\end{align*}
 There is a symmetric estimate with the roles of $\boldbeta, \boldbeta'$ reversed.   
\end{lemma}
\begin{proof}
We can write the probability as
\[ \sum_x \mb{P} \big( \boldbeta \in \mathcal{E}, \;\min_{0 \le i \le N} \boldbeta(i) = -m, \;\boldbeta(N) = x\big)\mb{P} \big( \min_{0 \le i \le N'} \boldbeta'(i) = -m',\; \boldbeta'(N') = x - d\big).\]
Applying \cref{cor745} to the second factor gives the first estimate. To get the second estimate we relax the $\min_{0 \le i \le N} \boldbeta(i) = - m$ condition to $\min_{1 \le i \le N}\boldbeta(i) \ge -m$ and apply \cref{prop81-3-ii}.
\end{proof}

\subsection{Boundedness estimates}  In this section we bound the probability of an almost positive walk having an unusually large increment. We start with a definition which is closely related to \cref{bounded-walk-def}. Recall that $\kappa = \frac{1}{100}$ is a globally defined small constant.

\begin{definition}\label{r-beta-def}
Let $\beta$ be a walk. Then we define $R(\beta)$ to be the minimal positive integer $R$ for which $|\xi_i| \le R i^{\kappa}$ for all $i$. Set $R(\beta) := \infty$ if no such $R$ exists. More generally denote by $R_{\le L}(\beta)$ the minimal positive integer for which this holds for all $i \le L$, that is to say the minimal positive integer $R$ for which $\beta$ is $R$-bounded to length $L$.
\end{definition} 
\begin{remark}
The restriction of the values of $R$ to a discrete set (the integers) is a technical convenience.   
\end{remark}

\begin{lemma}\label{walks-bd-min-est}
Uniformly in integers $L, N$ with $1 \le L \le N$ and non-negative integers $m, r$ we have
\[ \mb{P} \big( R_{\le L}(\boldbeta) > r, \, \min_{1 \le i \le N} \boldbeta(i) \ge -m\big) \ll (m + r + 1)e^{-r} N^{-1/2}.\] A similar estimate holds for $\boldbeta'$.
\end{lemma}
\begin{proof}
We may assume $r \ge 1$, since for $r = 0$ the result is immediate from \cref{prop81-3-ii}. If $R_{\le L}(\boldbeta) > r$ then there is some (minimal) index $\ell$, $1 \le \ell \le L$, such that $|\xi_{\ell}| \ge r \ell^{\kappa}$. Suppose that $|\xi_{\ell}| = i$. Then, by the minimality of $\ell$, we have $\boldbeta(\ell) \le |\xi_1| + \cdots + |\xi_{\ell}| \le r \ell^2 + i$. The shifted walk $\boldgamma(j) := \boldbeta(\ell + j) - \boldbeta(\ell)$ of length $N - \ell$ must therefore stay above $-(m + r \ell^2 + i)$, and so the contribution of a given pair $(\ell, i)$ to our probability is, by \cref{prop81-3-ii},
\[ \ll \mb{P}\big( |\boldxi_{\ell}| = i\big) \cdot (m + r \ell^2 + i) (N + 1 - \ell)^{-1/2} \ll \frac{1}{(i+1)!} (m + r \ell^2 + i) (N + 1 - \ell)^{-1/2} .\] (Note that this estimate does hold when $\ell = N$, which is why we included the $+1$ term.)
To complete the proof, we must show that the sum of the RHS over all $(\ell, i)$ is bounded by $\ll (m+r+1) e^{-r} N^{-1/2}$. 

The contribution from $\ell \le N/2$ is
\[ \ll N^{-1/2} \sum_{\ell \ge 1} \sum_{i \ge r \ell^{\kappa}} \frac{1}{(i+1)!}(m + r \ell^2 +i) \le N^{-1/2} \sum_{\ell \ge 1} \sum_{i \ge r \ell^{\kappa}} \frac{1}{(i+1)!} (m + r + 1) \ell^2 i.\] Using the bounds $\sum_{i \ge X} \frac{1}{i!} \ll e^{-X}$ and $\sum_{\ell \ge 1} \ell^2 e^{-r\ell^{\kappa}} \ll e^{-r}$, this is seen to be acceptable.

To estimate the contribution from $N/2 < \ell \le N$ we can use very crude bounds. We may bound this by
\begin{align*} \ll N \sum_{i \ge r (N/2)^{\kappa}}  \frac{1}{(i+1)!} & (m + r N^2 + i) \le (m + r + 1) N^3 \sum_{i \ge r(N/2)^{\kappa}} \frac{1}{i!} \\ & \ll (m + r + 1)N^3 e^{-r (N/2)^{\kappa}} \ll (m + r + 1) e^{-r} N^{-10}.\end{align*} This is (amply) acceptable and the proof is complete. The proof for $\boldbeta'$ is identical since $|\boldxi_{\ell}|$ and $|\boldxi'_{\ell}|$ have the same distribution.
\end{proof}

A consequence of this is the following bound on the $\omega$-measure of walks which are not $r$-bounded.

\begin{corollary}\label{omega-bdedness-est}
Suppose that $m, r \ge 0$ are integers. Then uniformly in $L \in \N$ we have 
\[ \omega_m \{  R_{\le L}(\beta) > r\}, \, \omega'_m \{  R_{\le L}(\beta') > r\} \ll (m + r+1) e^{-r}. \] Similar estimates hold for $\omega_m\{R(\beta) > r\}$ and $\omega'_m\{ R(\beta') > r\}$.
\end{corollary}
\begin{proof} The required bounds are immediate from \cref{limit-omega-m} and \cref{walks-bd-min-est}.\end{proof}

\subsection{Positivity estimates} In this section we show that an almost positive random walk almost surely has a parabolic shape. The following definition is very closely related to the definition of $T$-positive walk, which is \cref{good-walk-def}.
Recall that $\eta = \frac{1}{100}$ is a globally-defined small parameter.

\begin{definition}
Denote by $T(\beta)$ the smallest integer $T$ for which $-T + \ell^{1/2 - \eta} \le \beta(\ell) \le T + \ell^{1/2 + \eta}$ for all $\ell$, and set $T(\beta) := \infty$ if no such integer exists. More generally, define $T_{\le L}(\beta)$ to be the smallest integer $T$ for which this holds for all $\ell \le L$, that is to say for which $\beta$ is $T$-positive to length $L$.
\end{definition}

\begin{lemma}\label{pos-paths-cor}
Suppose that $m \ge 0$ and $T, N > 0$ are integers. Then
We have \[ \mb{P} \big( T_{\le N}(\boldbeta) > T ,\; \min_{1 \le i \le N} \boldbeta(i) \ge -m \big)  \ll (1 + m)^{O(1)} T^{-\eta/2} N^{-1/2}.\] A similar statement holds for $\boldbeta'$.
\end{lemma}
\begin{proof} When $T < 2m$, the result follows immediately from \cref{first-rand-walk-pos} by simply ignoring the condition $T_{\le N}(\boldbeta) > T$ on the LHS. Thus we assume that $T \ge 2m$ in what follows. Once again the same proof applies to both the dashed and undashed statements, so we prove only the former.
We first consider the `upper' failure of positivity in which there is some $\ell \le N$ such that $\boldbeta(\ell) \ge T + \ell^{1/2 + \eta}$. In fact we will handle the event $\boldbeta(\ell) \ge T + \ell^{1/2} \log \ell$. By \cref{first-rand-walk-pos} applied to the shifted walk $\boldgamma( \cdot ) := \boldbeta(\ell + \cdot) - \boldbeta(\ell)$ of length $N - \ell$, and by \cref{lemmaF.1}, we have
\begin{align*}
\sum_{\ell\le N} \mb{P} \big( \boldbeta(\ell) \ge T + \ell^{1/2} \log \ell,\; \min_{1 \le i \le N} \boldbeta(i) \ge -m\big) & \ll \sum_{\ell \le N}(N +1 - \ell)^{-1/2} \!\!\!\!\!\sum_{i \ge T + \ell^{1/2}\log \ell} (m + i)  \mb{P}(\boldbeta(\ell) = i)   \\ & \ll\sum_{\ell\le N} (N +1 - \ell)^{-1/2}\!\!\!\!\sum_{i \ge T + \ell^{1/2}\log \ell} (m + i) e^{-i^2/4\ell}  .
\end{align*}
 To bound this, set $i := T + \lceil \ell^{1/2} \log \ell\rceil + j$ and note that $i^2 \ge T^2 + j^2 + \ell \log^2 \ell$; the sum is therefore bounded by 
\begin{align*} \sum_{\ell \le N} (N +1 - \ell)^{-1/2}\ell^{-10} & \sum_{j \ge 0} \big( m + T + \ell^{1/2} \log \ell + j\big) e^{-T^2/4\ell} e^{-j^2/4\ell} \\ & \ll (1+m) T \sum_{\ell \le N} (N +1 - \ell)^{-1/2}\ell^{-9} e^{-T^2/4\ell}\sum_{j \ge 0} (j+1) e^{-j^2/4\ell} ,\end{align*} using here that $m + T + \ell^{1/2} \log \ell + j \le (1+m) T \ell (j+1)$. Since $\sum_j j e^{-j^2/4\ell} \ll \ell$, this is bounded by
\[ \ll (1+m) T \sum_{\ell\le N} (N +1 - \ell)^{-1/2}\ell^{-8} e^{-T^2/4\ell} .\] For $\ell \le T$ we use $e^{-T^2/4\ell} \le e^{-T/4}$ and $\sum_{\ell \le N} (N + 1 - \ell)^{-1/2}\ell^{-8} \ll N^{-1/2}$, and for $\ell \ge T$ we use $e^{-T^2/4\ell} \le 1$ and $\sum_{T \le \ell \le N} (N + 1 - \ell)^{-1/2}\ell^{-8} \ll T^{-7}N^{-1/2}$. Thus the probability of this `upper failure' is $\ll (m+1) T^{-6} N^{-1/2}$, which is substantially smaller than the stated bound.

We turn now to bounds for the more subtle `lower' failure of positivity. This fact, namely that almost positive walks almost surely have nearly squareroot growth, is critical for the whole paper. Our argument is closely related to that of Ritter \cite{Rit81}. Suppose that for some $\ell$, $1 \le \ell \le  N$, we have $\boldbeta(\ell) \le -T + \ell^{1/2 - \eta}$. Since $\boldbeta(\ell) \ge \min_{1 \le i \le N} \boldbeta(i) \ge -m$, we have $\ell \ge (T - m)^{2} \ge T^2/4$, so $T \le 2N^{1/2}$, and (obviously) $\boldbeta(\ell) \le \ell^{1/2 - \eta}$. We will bound the measure of walks for which there is an $\ell$ with these last two properties. The argument will be slightly different according to whether $\ell \le N/8$ (say) or not. \vspace*{8pt}

\emph{The case $\ell \le N/8$.} Suppose that $k \le N/4$ is a power of two (here $k$ is just a dummy variable, not the global variable $k$ from the main paper). Let $j$ be an integer with $j \in [k/2, 2k]$. Denote by $Z_{j,k}(\boldbeta)$ the indicator of the event that $\boldbeta(j) \le 2k^{1/2 - \eta}$.   Using \cref{first-rand-walk-pos} (applied to the shifted walk $\boldgamma(\cdot) := \boldbeta(j + \cdot) - \boldbeta(j)$, which has length $N - j \ge N/2$ and minimum $\ge -u - m$) and then \cref{est-72-geq} we have
\begin{align*} \mb{P} \big( & Z_{j,k}(\boldbeta) = 1, \; \min_{1 \le i \le N} \boldbeta(i) \ge -m\big)  \ll  N^{-1/2}\!\!\!\!\sum_{-m \le u \le 2 k^{1/2 - \eta}} (1 + m + u)\mb{P}\big(\boldbeta(j) = u, \,\min_{1 \le i \le j} \boldbeta(i) \ge -m\big)  \\ &  \ll (1 + m)^2 N^{-1/2}\sum_{-m \le u \le 2 k^{1/2 - \eta}} (m + u + 1)^2 j^{-3/2} \ll (1 + m)^5 N^{-1/2} k^{- 3\eta}. \end{align*} 

Denote
\[ Z_k(\boldbeta) := \sum_{j \in [k/2, 2k]} Z_{j,k}(\boldbeta).\]
It follows that for any $\delta \in (0,1)$ we have 
\begin{equation}\label{zk-bd} \mb{P} \big( Z_k(\boldbeta) \ge \delta k,\; \min_{1 \le i \le N} \boldbeta(i) \ge  -m \big)\ll (1 + m)^5 \delta^{-1} N^{-1/2}k^{ - 3\eta}. \end{equation}
We say that $\boldbeta$ is \emph{reasonable} at scale $k$ if $\boldbeta(j_2) - \boldbeta(j_1) \le k^{1/2 - \eta}$ and $\boldbeta(j_1) \le k$ whenever $j_1, j_2$ satisfy $k/2 \le j_1 < j_2 \le 2k$ and $j_2 - j_1 \le k^{1 - 2\eta} (\log k)^{-2}$. We claim the following estimate:
\begin{equation}\label{unreasonable}\mb{P} \big( \mbox{$\boldbeta$ unreasonable at scale $k$}, \; \min_{1 \le i \le N} \boldbeta(i) \ge - m \big) \ll (1 + m)^{2} k^{-10} N^{-1/2}. \end{equation}
To see this we use a union bound over the $\ll k^2$ choices of $j_1, j_2$. For each choice, we need to bound
\[ \sum_{\substack{-m \le i_1 \le k \\ i_2 - i_1 > k^{1/2 - \eta}}}  \mb{P} \big (\boldbeta(j_1) = i_1, \boldbeta(j_2) = i_2, \min_{1 \le i \le N} \boldbeta(i) \ge -m\big) + \sum_{i_1 > k} \mb{P} \big (\boldbeta(j_1) = i_1, \min_{1 \le i \le N} \boldbeta(i) \ge -m\big).\]
For this we apply \cref{first-rand-walk-pos} to the walk shifted by $j_2$ (for the first term) or $j_1$ (for the second term). Since $N- j_1, N - j_2 \gg N$ we see that this is
\begin{equation}\label{8330} \ll N^{-1/2}\!\!\!\!\!\!\sum_{\substack{-m \le i_1 \le k \\ i_2 - i_1 > k^{1/2 - \eta}}} \!\!\!\!(m + 1 + i_2) \mb{P} \big (\boldbeta(j_1) = i_1, \boldbeta(j_2) = i_2\big) + N^{-1/2} \sum_{i_1 > k} (m + 1 + i_1) \mb{P} \big (\boldbeta(j_1) = i_1\big).\end{equation}
By \cref{large-deviation-2} we have $\mb{P}(\boldbeta(j_1) = i_1) \ll e^{-i_1/10}$ uniformly for $j_1 \in [k/2, 2k]$ and $i_1 > k$, and so (after summing over the $O(k^2)$ choices of $j_1, j_2$) the contribution of the second term here is bounded as claimed in \cref{unreasonable}. To bound the first term in \cref{8330} we set $i = i_2 - i_1$ and sum out the $i_1$ variable, obtaining a bound
\[ \ll N^{-1/2}(k + m) \sum_{i > k^{1/2 - \eta}} (m + k + i) \mb{P}(\boldbeta (j_2 - j_1) = i) .\] By a standard large deviation estimate (see \cref{lemmaF.1}) this is \[ \ll N^{-1/2}(k+m) \sum_{i > k^{1/2 - \eta}} (m + k + i) e^{-i^2/4(j_2 - j_1)},\]  which in turn is bounded by\[ \ll N^{-1/2}(k + m)^2 \sum_{i > k^{1/2 - \eta}} i e^{-i^2(\log k)^2/4k^{1 - 2\eta}} < N^{-1/2}(1 + m)^{2} k^{-100}.\] Summing over the $O(k^2)$ choices for $j_1, j_2$, this completes the proof of \cref{unreasonable}. 

Now suppose that $\boldbeta(\ell) \le \ell^{1/2 - \eta}$ (with $\ell \le N/8$) and let $k$ be the unique power of two such that $\ell \in (k/2, k]$. Thus $k \le N/4$, and $\boldbeta(\ell) \le k^{1/2 - \eta}$. Suppose that $\boldbeta$ is reasonable at scale $k$. Then $\boldbeta(j) \le 2 k^{1/2 - \eta}$ for all $j$ with $0 \le j - \ell \le k^{1 - 2\eta} (\log k)^{-2}$, and so $Z_k(\boldbeta) \ge k^{1 - 2\eta} (\log k)^{-2}$. By \cref{zk-bd} with $\delta := k^{-2\eta} (\log k)^{-2}$, it follows that 
\begin{align*} \mb{P} \big( \min_{1 \le i \le N} \boldbeta(i) \ge -m, \; \boldbeta(\ell) \le \ell^{1/2 - \eta} & \, \mbox{for some $\ell \in (k/2, k]$}, \; \mbox{$\boldbeta$ reasonable at scale $k$}\big) \\ & \ll  (1 + m)^{5} N^{-1/2} k^{-\eta} (\log k)^2\end{align*} for all (powers of two) $k \le N/4$. Putting this together with \cref{unreasonable} gives
\[ \mb{P} \big( \min_{1 \le i \le N} \boldbeta(i) \ge -m,\; \boldbeta(\ell) \le \ell^{1/2 - \eta} \, \mbox{for some $\ell \in (k/2, k]$} \big) \ll  (1 + m)^{5} N^{-1/2} k^{-\eta} (\log k)^2.\] Summing $k$ over powers of two with $k \ge (T - m)^2$ yields
\[ \mb{P} \big( \min_{1 \le i \le N} \boldbeta(i) \ge -m, \; \boldbeta(\ell) \le \ell^{1/2 - \eta}\, \mbox{for some $\ell$, $N/8 \ge \ell \ge (T - m)^2$} \big) \ll  (1 + m)^{5} N^{-1/2} T^{-\eta/2}.\] 
This is the desired bound in the case $\ell \le N/8$.\vspace*{8pt}

\emph{The case $N/8 < \ell \le N$.} The proof in this case is somewhat similar. In this case, we say that $\boldbeta$ is \emph{reasonable} if $|\boldbeta(j_2) - \boldbeta(j_1)| \le N^{1/2 - \eta}$ whenever $j_1, j_2 \in [N/8, N]$ and $|j_2 - j_1| \le N^{1 - 5\eta/2}$. By another large deviation estimate (see \cref{cor-F2}), $\boldbeta$ is reasonable with probability $\ge 1 - N^{-10}$.

If $\boldbeta$ is reasonable, and if $\boldbeta(\ell) \le \ell^{1/2 - \eta}$ for some $\ell \in [N/8, N]$, then by appropriate rounding there is some $j = \lambda \lfloor N^{1 - 5\eta/2}\rfloor$, $\lambda \in \Z_{\ge 0}$, $N - j \in [N/8, N]$, with $\boldbeta(N - j) \le 2 N^{1/2 - \eta}$.

Now for a fixed such $j$ we have
\begin{align*} \mb{P} \big( \boldbeta(N - j) & \le 2N^{1/2 - \eta},\; \min_{1 \le i \le N} \boldbeta(i) \ge -m\big) = \sum_{-m \le u \le 2N^{1/2 - \eta}} \mb{P} \big( \boldbeta(N - j) = u, \;\min_{1 \le i \le N} \boldbeta(i) \ge -m\big)\\
& \ll  \sum_{-m \le u \le 2N^{1/2 - \eta}} \mb{P} \big( \boldbeta(N - j) = u, \,\min_{1 \le i \le N-j} \boldbeta(i) \ge -m\big) \cdot (m + u + 1) (1 + j)^{-1/2} \end{align*}
by applying \cref{first-rand-walk-pos} to the walk of length $j$ starting after $N - j$. (Here we used $(1 + j)^{-1/2}$ so that the result is true for $j = 0$). By \cref{est-72-geq}, and since $N - j \gg N$, this is
\[ \ll (1 + m)^{2} \sum_{-m \le u \le 2N^{1/2 - \eta}} (1 + u + m)^2 N^{-3/2} (1 + j)^{-1/2} \ll (1 + m)^5 N^{ - 3\eta} (1 +  j)^{-1/2}.\]
Summing over $j = \lambda \lfloor N^{1 - 5\eta/2} \rfloor$ with $\lambda \in \Z_{\ge 0}$ such that $N - j \in [N/8, N]$ (thus $\lambda \ll N^{5\eta/2}$) gives $\ll (1 + m)^{5} N^{-1/2 - \eta/2}$. Since $T \le 2 N^{1/2}$, this is $\ll (1 + m)^{5} N^{-1/2} T^{-\eta}$, which implies the desired bound in this case.\end{proof}

\begin{corollary}\label{lem13.10} Let $m \ge 0$ and $T \ge 1$ be integers.
We have \[ \omega_m \{  T(\beta) > T \},\;  \omega'_m \{  T(\beta') > T \} \ll (1 + m)^{O(1)} T^{-\eta/2}.\]
\end{corollary}
\begin{proof} Once again the proofs of the two statements are the same so we handle only the first one.
It follows immediately from \cref{pos-paths-cor} that if $L \le N$ then
\[ N^{1/2} \mb{P} \big( T_{\le L}(\boldbeta) > T, \; \min_{1 \le i \le N} \boldbeta(i) \ge -m\big) \ll (1 + m)^{O(1)} T^{-\eta/2} .\]
The result now follows from \cref{limit-omega-m} by letting $N \rightarrow \infty$, and then letting $L \rightarrow \infty$.
\end{proof}

\subsection{Constructing a jump step}

We now come to the most technical part of the section. We start with a definition. Here, $\eta = \kappa = \frac{1}{100}$ are the global constants involved in \cref{bounded-walk-def,good-walk-def}.
\begin{definition}\label{jump-step-7}
Let $V \in \N$. Let $\beta$ be a walk of length $N$ with increments $\xi_i = \beta(i) - \beta(i-1)$. Then we say that an index $t$ is a $V$-jump step if $1 + \xi_t = V$ and if, for all $\ell$ with $0 \le \ell \le N - t$, we have
\begin{equation}\label{jump-step-1}
\beta(t + \ell) - \beta(t) \ge \ell^{1/2 - 2\eta},
\end{equation}
and if $1 \le \ell \le N - t$ we have
\begin{equation}\label{jump-step-2}
\xi_{t + \ell} \le \ell^{2\kappa}.    
\end{equation}
\end{definition}

We will show that a random $\boldbeta$ almost surely has a jump step in a suitable range. This will be done by recursive application of the following result.

\begin{lemma}\label{first-step}
Suppose that $N, V$ are sufficiently large integers and let $\boldbeta$ be a random walk of length $N$ with $\Pois(1) - 1$ increments. Then we have
\[ \mb{P}\big(\mbox{index $1$ is a $V$-jump step for $\boldbeta$} \mid \min_{1 \le i \le N} \boldbeta(i) \ge 0\big) \ge V^{-V}.\]
\end{lemma}
\begin{proof} Choose absolute constants $R, T$ such that 
\begin{equation}\label{RT-bd} \mb{P} \big( \min_{1 \le i \le N} \boldbeta(i) \ge 0, \; \boldbeta \; \mbox{is $R$-bounded and $T$-positive up to $N$}\big) \gg  N^{-1/2},\end{equation} uniformly for $N$ sufficiently large. That such $R, T$ exist follows from \cref{walks-bd-min-est} and \cref{pos-paths-cor} with $m = 0$, together with the existence of $h(0)$, which implies that $\mb{P} \big(\min_{1 \le i \le N} \boldbeta(i) \ge 0\big) \gg N^{-1/2}$. Set $L := \lceil \max (10, 2T, R^{1/\kappa})\rceil$ (thus $L$ is still an absolute constant).

The walk with $\boldbeta(i) = V - 2 + i$ for $i = 1,\dots, N$ satisfies \cref{jump-step-1,jump-step-2} for $t = 1$. This walk occurs with probability $ \mb{P}( \Pois(1) = V) \cdot \mb{P}(\Pois(1) = 2)^{N-1}$ which, if $V$ is sufficiently large, is $\ge V^{-V}$ for any $N \le 2L$. Therefore we may assume $N > 2L$ in what follows.

Now by the existence of $h(0)$ we have $\mb{P}\big( \min_{1 \le i \le N} \boldbeta(i) \ge 0\big) \le C N^{-1/2}$ for some absolute $C$, so it suffices to show that
\begin{equation}\label{j-lem-1} \mb{P}\big(\mbox{index $1$ is a $V$-jump step for $\boldbeta$}, \; \min_{1 \le i \le N} \boldbeta(i) \ge 0\big) \ge C V^{-V} N^{-1/2}.\end{equation} 
Consider the shifted walk $\boldgamma(i) := \boldbeta(L + i) - \boldbeta(L)$ and let its increments be $\tilde\boldxi_i := \boldxi_{L + i}$. We claim that the event on the left in \cref{j-lem-1} contains the event $\mathcal{E}$ defined by
\[ \{ \boldxi_1 = V-1, \boldxi_2 = \cdots \boldxi_L = 1\}  \cap \{ \min_{1 \le i \le N-L} \boldgamma(i) \ge 0,\; \boldgamma \; \mbox{is $R$-bounded and $T$-positive up to $N-L$}\}.\] 
This is enough to complete the proof, since (using \cref{RT-bd} applied to $\boldgamma$) we have
\[ \mb{P}(\mathcal{E}) \gg \mb{P}( \Pois(1) = V) \cdot \mb{P}(\Pois(1) = 2)^{L - 1} (N - L)^{-1/2},\] so $\mb{P}(\mathcal{E}) \ge CV^{-V} N^{-1/2}$ provided $V$ is sufficiently large.

To prove the claim, suppose that $\boldbeta \in \mathcal{E}$. First note that we clearly have $1 + \boldxi_1 = V$. Now for $\ell$ with $1 \le \ell \le L - 1$, we have $\boldbeta(1 + \ell) - \boldbeta(1) = \ell \ge \ell^{1/2 - 2\eta}$. Thus \cref{jump-step-1} holds for $\ell \le L - 1$. Moreover for $i = 0,\dots, N - L$ we have, using the $T$-positive nature of $\boldgamma$ and the fact that $L \ge \max(10, 2T)$, that
\begin{align*} \boldbeta(L + i) - \boldbeta(1) & = \boldgamma(i) + (\boldbeta(L) - \boldbeta(1)) \ge  - T + i^{1/2 - \eta}  + (L-1)  \\ & \ge L^{1/2 - \eta} + i^{1/2 - \eta} \ge (L+ i)^{1/2 - \eta}.\end{align*} Thus \cref{jump-step-1} holds for $L \le \ell \le N - 1$ as well. 

The inequality \cref{jump-step-2} certainly holds for $1 \le \ell \le L - 1$. For $i = 0,\dots, N - L$ we have 
\[ \boldxi_{1 + L + i} = \tilde\boldxi_{1 + i} \le R (1 + i)^{\kappa} < (L + i)^{2\kappa},\]  and so \cref{jump-step-2} also holds for $L \le \ell \le N - 1$. This completes the proof of the claim, and hence of \cref{first-step}.
\end{proof}

Now we turn to the main result of this subsection.

\begin{lemma}\label{lemma7.19}
There exists an absolute constant $C$ such that the following holds. Let $N, V$ be sufficiently large integer parameters. Let $m \ge 0$ be an integer, and suppose that $\delta \in (0,\frac{1}{10})$. Let $t_{\min}, t_{\max}$ be two integer thresholds, and suppose that
\begin{equation}\label{tmin-threshold} N > t_{\max} > t_{\min} \ge \Big( \frac{8(m+2)^{C} V^V}{\delta} \Big)^{4/\eta}\end{equation} and
\begin{equation}\label{massive-scales} \log \log t_{\max} - \log \log t_{\min} \ge 20 V^{V} \log(1/\delta).\end{equation} Then 
\[ \mb{P}\big( \boldbeta \; \mbox{does not have a $V$-jump step in $[t_{\min}, t_{\max}]$} \mid \min_{1\le i \le N} \boldbeta(i) \ge -m\big) \ll \delta.\]
\end{lemma}
\begin{proof} For brevity in what follows set $\eps := V^{-V}$. 
Define thresholds $T_1 := t_{\min}$ and $T_{j+1} := T_j^{4}$, $j = 1,2,\dots$. We assume in what follows that any such threshold under discussion is $\le t_{\max}$; this will ultimately turn out to be the case by use of \cref{massive-scales}, as we shall show later. Set $I_j := [T_j, T_{j+1})$. We say that $\boldbeta$ has a \emph{local jump step in $I_j$} if there is some $t \in [T_j, T_{j}^2]$ such that $1 + \boldxi_t = V$ and such that \cref{jump-step-1,jump-step-2} are satisfied for all $\ell$ with $1 \le \ell < T_{j+1} - t$. (Note the case $\ell = 0$ of \cref{jump-step-1} is trivial.)

Denote by $\mathcal{J}_j$ the event that $\boldbeta$ has a local jump step in the interval $I_j$. Set
\[ \tau_j := \mb{P} \big( \neg \mathcal{J}_1 \cap \dots \cap \neg \mathcal{J}_{j} \mid \min_{1 \le i \le N} \boldbeta(i) \ge -m\big). \]
Observe that $\tau_1 \ge \tau_2 \ge \dots$ and that for $j \ge 2$ we have
\begin{align}\nonumber
\tau_{j-1} - \tau_j & = \mb{P} \big( \mathcal{J}_j \cap \neg \mathcal{J}_1 \cap \dots \cap  \neg\mathcal{J}_{j-1} \mid \min_{1 \le i \le N} \boldbeta(i) \ge -m\big) \\ & = \sum_{T_j \le s \le N} \mb{P} \big( \mathcal{J}_j \cap \neg \mathcal{J}_1 \cap \dots \cap \neg\mathcal{J}_{j-1} \cap \argmin_{T_j \le i \le N} \boldbeta(i) = s \mid \min_{1 \le i \le N} \boldbeta(i) \ge -m\big).\label{taujj}
\end{align}
(That is, $s$ is the first $s \in [T_j, N]$ for which $\boldbeta$ attains its minimum on that interval.) To estimate \cref{taujj} from below, we first consider the sum over $s \ge T_j^2$. For this, we simply ignore the $\mathcal{J}$ events, bounding the sum by
\begin{equation}\label{tail-s-tobound}  \mb{P} \big(\argmin_{T_j \le i \le N} \boldbeta(i) \in [T_j^2, N] \mid \min_{1 \le i \le N} \boldbeta(i) \ge -m\big) . \end{equation}
Suppose that $\boldbeta$ is $t_{\min}^{1/2}$-positive up to $N$ (see \cref{good-walk-def}). Then $\boldbeta(T_j) \le T_j^{1/2 + \eta} + t_{\min}^{1/2}$ and $\boldbeta(s) \ge (T_j^2)^{1/2 - \eta} - t_{\min}^{1/2}$ for all $s$ with $T_{j}^2 \le s \le N$. Since $T_j \ge t_{\min} > 2^8$ and $\eta \le \frac{1}{8}$, we have $T_j^{1/2 + \eta} + t_{\min}^{1/2} < (T_j^2)^{1/2 - \eta} - t_{\min}^{1/2}$, and so in this case we do not have $\argmin_{T_j \le i \le N} \boldbeta(i) \in [T_{j}^2, N]$. Thus we can bound \cref{tail-s-tobound} by
\[ \mb{P} \big(   \boldbeta \; \mbox{is not $t_{\min}^{1/2}$-positive up to $N$} \mid \min_{1 \le i \le N} \boldbeta(i) \ge -m\big).\]
Using \cref{first-rand-walk-pos} (3) and \cref{pos-paths-cor}, we see that the probability in \cref{tail-s-tobound} is $\le (m + 2)^{C} t_{\min}^{-\eta/4}$ for some absolute $C$. Thus we have shown that 
\begin{equation}\label{large-s-sum-jump} \sum_{T_j^2 \le s \le N} \mb{P} \big( \mathcal{J}_j \cap \neg \mathcal{J}_1 \cap \dots \cap \neg\mathcal{J}_{j-1} \cap \argmin_{T_j \le i \le N} \boldbeta(i) = s \mid \min_{1 \le i \le N} \boldbeta(i) \ge -m\big) \le (m + 2)^{C} t_{\min}^{-\eta/4}.\end{equation}

Consider now the portion of the sum in \cref{taujj} over the range $T_j \le s < T_{j}^2$.
We bound this below by
\begin{equation}\label{small-s-sum-jump} \sum_{T_j \le s < T_j^2} \mb{P} \big( \mathcal{J}^{(s)}_j \cap \neg \mathcal{J}_1 \cap \dots \cap \neg\mathcal{J}_{j-1} \cap \argmin_{T_j \le i \le N} \boldbeta(i) = s \mid \min_{1 \le i \le N} \boldbeta(i) \ge -m\big) ,\end{equation} where $\mathcal{J}^{(s)}_j$ is the event that $t = s+1$ is a local jump step; this event is obviously contained in $\mathcal{J}_j$.

We the apply the chain rule for conditional probability
\[ \mb{P}(A \cap B \cap E_s \mid M) = \mb{P}(B \mid M) \cdot \mb{P}(E_s \mid B \cap M) \cdot \mb{P}(A \mid B \cap E_s \cap M),\]
where $A = \mathcal{J}^{(s)}_j$, $B = \neg \mathcal{J}_1 \cap \dots \cap \neg\mathcal{J}_{j-1}$, $E_s = (\argmin_{T_j \le i \le N} \boldbeta(i) = s)$ and $M = (\min_{1 \le i \le N} \boldbeta(i) \ge -m)$. Observe that $\mb{P}(B \mid M) = \tau_{j-1}$. Denote $\mb{P}(E_s \mid B \cap M) := p(s)$. Then \cref{small-s-sum-jump} is
\begin{equation}\label{725-truncated} \tau_{j-1}\sum_{T_j \le s < T_{j}^2} p(s)\mb{P} \big( \mathcal{J}^{(s)}_j \mid \neg \mathcal{J}_1 \cap \dots \cap \neg\mathcal{J}_{j-1} ,\; \argmin_{T_j \le i \le N} \boldbeta(i) = s ,\;  \min_{1 \le i \le N} \boldbeta(i) \ge -m \big) .\end{equation}
The event we are conditioning on here has the form $\mathcal{E}_s(\boldbeta(1),\dots, \boldbeta(s)) \cap \big( \min_{1 \le i \le N - s} \boldgamma(i) \ge 0)$, where $\boldgamma_s(i) := \boldbeta(s + i) - \boldbeta(s)$ is the shifted walk starting at $s$ and 
\[ \mathcal{E}_s := \neg \mathcal{J}_1 \cap \cdots \cap \neg \mathcal{J}_{j-1} \cap (\min_{1 \le i \le s} \boldbeta(i) \ge -m ) \cap \big( \boldbeta(i) > \boldbeta(s) \; \mbox{for $T_j \le i < s$}\big);  \] note carefully that, by the definition of local jump step, $\mathcal{E}_s$ really does depend only on the values $\boldbeta(i)$ for $i \le s$. The event $\mathcal{J}^{(s)}_j$ depends only on $\boldgamma_s$ (the portion of $\boldbeta$ after $s$), and so \cref{725-truncated} is
\[ \tau_{j-1} \sum_{T_j \le s < T_{j}^2} p(s) \mb{P} \big( \mathcal{J}^{(s)}_j \mid \min_{1 \le i \le N - s} \boldgamma(i) \ge 0\big) \ge  \eps \tau_{j-1} \sum_{T_j \le s < T_{j}^2} p(s),\]  where the second step follows from \cref{first-step} applied to $\boldgamma$ (recalling here that $\eps = V^{-V}$).

Combining the upper bound \cref{large-s-sum-jump} with the lower bound we have just obtained for \cref{small-s-sum-jump} via \cref{725-truncated}, and substituting in to \cref{taujj}, we have shown that 
\begin{equation}\label{taujj-2} \tau_{j-1} - \tau_j  \ge \eps \tau_{j-1} \big( \sum_{T_j \le s < T_{j}^2} p(s)\big) - (m+2)^{C} t_{\min}^{-\eta/4}.\end{equation}
We must now give a lower bound for the sum over $p(s)$. Note first of all that 
\begin{equation}\label{sum-p-1} \sum_{T_j \le s \le N} p(s) = 1,\end{equation} since precisely one of the events $E_s$ occurs for any $\boldbeta$. Now for each $s$ we have
\begin{equation}\label{ps-individual} p(s) = \mb{P}(E_s \mid B \cap M) = \frac{\mb{P}(E_s \cap B \cap M)}{\mb{P}(B \cap M)} \le \frac{\mb{P}(E_s \cap M)}{\mb{P}(B \cap M)} = \frac{\mb{P}(E_s \mid M)}{\mb{P}(B \mid M)} = \frac{1}{\tau_{j-1}} \mb{P}(E_s \mid M).\end{equation} Furthermore, $\sum_{T_{j}^2 \le s \le N} \mb{P}(E_s \mid M)$ is precisely the probability in \cref{tail-s-tobound}, which we have already shown (in the analysis leading up to \cref{large-s-sum-jump}) to be at most $(m + 2)^{C} t_{\min}^{-\eta/4}$. Combining this with \cref{taujj-2,sum-p-1,ps-individual} gives 
\begin{equation*} \tau_{j-1} - \tau_j  \ge \eps\tau_{j-1}  -2(m+2)^{C+1} t_{\min}^{-\eta/4} \ge \eps\tau_{j-1} - \frac{\delta\eps}{4},\end{equation*} where this last inequality follows from \cref{tmin-threshold}. By induction, $\tau_j \le \frac{\delta}{4} + (1 - \frac{\delta}{4})(1 - \eps)^{j - 1}$, and so $\tau_j \le \delta/2$ for all $j \ge 10 \eps^{-1} \log(1/\delta)$, that is to say
\begin{equation}\label{loc-jump} \mb{P} \big(\boldbeta \; \mbox{has no local jump step in any $I_j$, $j \le 10 \eps^{-1} \log(1/\delta)$} \mid \min_{1 \le i \le N} \boldbeta(i) \ge -m\big) \le \delta/2.\end{equation}
Suppose now that $\boldbeta$ has a local jump step in $I_j$, but does not have an actual jump step in $I_j$. This means that either \cref{jump-step-1} or \cref{jump-step-2} fails for some $\ell$ with $T_{j+1} - t \le \ell \le N - t$ and for some $t \in [T_j, T_{j}^2]$. That is, either
\begin{equation}\label{failure-1} \boldbeta(t + \ell) \le \boldbeta(t) + \ell^{1/2 - 2\eta} \quad \mbox{or} \quad \boldxi_{t+\ell} > \ell^{2\kappa} \end{equation}
for some such $\ell$. We additionally observe that $\ell$ is much larger than $t$; certainly $\ell \ge t$ since $t \le T_j^2 \le \frac{1}{2} T_{j+1}$.

Suppose the first of the alternatives in \cref{failure-1} holds. We claim that either $\boldbeta(t + \ell) \le (t + \ell)^{1/2 - \eta} - t_{\min}$ or $\boldbeta(t) \ge t^{2/3} + t_{\min}$. To see this, first note that 
\begin{equation}\label{simple-jump-ineqs} \tfrac{1}{4}(t + \ell)^{1/2 - \eta} \ge \max \big(t^{2/3}, \ell^{1/2 - 2\eta} ,2t_{\min}\big).\end{equation} The first inequality here follows using $t + \ell \ge T_{j+1}$ and $t \le T_j^2$ and the fact that $T_{j+1} := T_j^4$. The second inequality follows from $\frac{1}{4}(t + \ell)^{1/2 - \eta} > \frac{1}{4}\ell^{1/2 - 2\eta}$ and the fact that $\ell$ is sufficiently large. The third inequality following using $t + \ell \ge T_{j+1}$ and $t_{\min} = T_1$. Adding the inequalities in \cref{simple-jump-ineqs} gives $(t + \ell)^{1/2 - \eta} - t_{\min} \ge t^{2/3}  + t_{\min}  + \ell^{1/2 - 2\eta}$, which implies the claim. The claim implies that $\boldbeta$ fails to be $t_{\min}$-positive  in the sense of \cref{good-walk-def}. By another application of \cref{first-rand-walk-pos} and \cref{pos-paths-cor}, the probability of this event conditional on $\min_{1 \le i \le N} \boldbeta(i) \ge -m$ is $\le (m + 2)^{C} t_{\min}^{-\eta/2} \le \delta/4$.

Now suppose the second option in \cref{failure-1} holds. Then $\boldbeta$ fails to be $t_{\min}^{\kappa/2}$-bounded in the sense of \cref{bounded-walk-def}; this is because $\ell^{2\kappa} > t_{\min}^{\kappa/2} (t + \ell)^{\kappa}$, since $\ell^{2/3} \ge (T_{j+1}/2)^{2/3} > T^{1/2}_1 = t^{1/2}_{\min}$ and $\ell^{4/3} > 2\ell \ge t + \ell$. By \cref{first-rand-walk-pos} (3) and \cref{walks-bd-min-est}, the probability of this conditional on $\min_{1 \le i \le N} \boldbeta(i) \ge -m$ is certainly $\le \delta/4$; the key point here is that the assumption \cref{tmin-threshold} implies that the term $e^{-t_{\min}^{\kappa/2}}$ arising from the application of \cref{walks-bd-min-est} is massively smaller than $\delta$, and also that $N$ is massively larger than $m$ and $1/\delta$.

The conclusion of the above analysis is that 
\[ \mb{P}\big(\boldbeta \; \mbox{has a local but not a global jump step in some $I_j$} \mid\min_{1 \le i \le N} \boldbeta(i) \ge -m\big) \le \delta/2.\]
Combining with \cref{loc-jump} proves that, conditioned on $\min_{1 \le i \le N} \boldbeta(i) \ge -m$, the probability that $\boldbeta$ has no (global) jump step in any $I_j$ with $j \le 10\eps^{-1} \log(1/\delta)$ is at most $\delta$. The proof is concluded by noting that, under condition \cref{massive-scales}, all such intervals $I_j$ are contained in $[t_{\min}, t_{\max}]$.
\end{proof}

\section{Bounds on \texorpdfstring{$\{0,1\}$}{}-matrices}\label{01-matrix-subsection}

This section assembles various ingredients that will be used in \cref{part3} of the paper, having to do with bounds on the number of $\{0,1\}$-matrices satisfying various linear-algebraic conditions.

Let $M$ be an $\ell \times N$ matrix. Write $r$ for its rank (thus $r \le \ell$), and for $d$ with $1 \le d \le r$ denote by $q_d$ the least index such that the first $q_d$ columns of $M$ span a subspace of $\Q^{\ell}$ of dimension $d$. Set $q_{r+1} := N+1$. We call $(r; q_1,\dots, q_r)$ the \emph{rank profile} of $M$ and denote it by $\rp(M)$.

\begin{lemma} \label{lem24}We have the following statements. Suppose that $\ell \in \N$.
\begin{enumerate}
    \item[\textup{(i)}] The number of $\ell$-by-$N$ matrices $M$ with $\{0,1\}$-entries and $\rp(M) = (r; q_1,\dots, q_{r})$ is bounded above by $2^{\ell^2} 2^{rN - \sum_{d = 1}^r q_d}$.
    \item[\textup{(ii)}] Suppose that $N \ge 10 \ell$ and that $1 \le r < \ell$. The number of $\ell$-by-$N$ matrices $M$ with $\{0,1\}$-entries and rank $r$, and at least $r+1$ different nonzero rows, is $\ll 2^{\ell^2} 2^{(r-\frac{1}{10}) N}$. \textup{(}By different rows we mean distinct as elements of $\Q^N$.\textup{)}
\end{enumerate}

\end{lemma}

\begin{proof}

(i) We uncover the columns of $M$ one-by-one. Consider the $q$th column. If $q$ is one of the $q_d$ then we bound the number of choices trivially by $2^{\ell}$. Otherwise, suppose that $q_d < q < q_{d+1}$; then the $q$th column lies in the linear span of the previous columns, which has dimension $d$. It is a folklore fact (see, for instance, \cite[Lemma 5.1]{FGK23}) that any subspace of $\Q^{\ell}$ of dimension $d$ intersects the cube $\{0,1\}^{\ell}$ in at most $2^d$ points, and so there are at most $2^d$ choices for the $q$th column. Therefore the quantity we seek is bounded above by
\[ 2^{\sum_{d = 1}^r d(q_{d+1} - q_d - 1) + \ell r} = 2^{rN - \sum_{d = 1}^r q_d - \sum_{d = 1}^r d + \ell r} \le 2^{\ell^2} 2^{rN - \sum_{d = 1}^r q_d} \] as claimed.

(ii) There is some set of $r$ rows which span the row space of $M$; this can be selected in $\binom{\ell}{r}$ ways. Fix such a set of rows. Consider the induced $r$-by-$N$ submatrix $M'$.

Say that an $r$-by-$N$ matrix is \emph{typical} if the intersection of the $\Q$-linear span of its rows with $\{0,1\}^N$ consists of only the rows themselves and zero. We claim that a random $r$-by-$N$ matrix with $\{0,1\}$-entries is typical with probability $\ge 1 - 2^{-N/10}$.  

Given the claim, we may conclude the proof of (ii) as follows. If $M'$ is typical, then $M$ can only have $r$ different nonzero rows, since all rows of $M$ are in the $\Q$-linear span of $M'$. If $M'$ is not typical, then by the claim there are $\le 2^{(r - \frac{1}{10}) N}$ choices for $M'$. Each of the remaining $\ell - r$ rows of $M$ is drawn from the intersection of $\{0,1\}^N$ with the row space of $M'$, which has size at most $2^r$ by the folklore fact above. This gives $\le (2^r)^{\ell - r}$ ways to complete $M'$ to $M$. Since $\binom{\ell}{r} (2^r)^{\ell - r} \le 2^{\ell^2}$, the result follows.

It remains to prove the claim. Let $Q$ be a random $r$-by-$N$ matrix with $\{0,1\}$-entries. Let $N'$ be an integer with $N/3 \le N' \le N/2$. We first reveal the $r$-by-$N'$ submatrix $Q'$ formed from the first $N'$ columns. The probability that $Q'$ does not have full rank is bounded above by $2^{r - N'}$, even over $\mb{F}_2$: each of the $2^r - 1$ possible annihilators of the column space occurs with probability $2^{-N'}$.

Suppose now that $Q'$ has been revealed and that it does have full rank. Fix a set $S$ of $r$ columns which span. We now reveal the remaining elements of $Q \setminus Q'$. Consider a hypothetical element of $\{0,1\}^N$ in the row span of $Q$, not zero or one of rows. Suppose that it is $\sum_{i = 1}^r \lambda_i v_i$ where $v_1,\dots, v_r$ are the rows of $Q$ and at least two of the $\lambda_i$ are not zero. Restricting to the columns in $S$, we see that there are at most $2^r$ choices for $(\lambda_1,\dots, \lambda_r)$. Fix such a choice. We now proceed to reveal the $N - N'$ columns in $Q \setminus Q'$. Each of these is a vector $\eps = (\eps_1,\dots, \eps_r) \in \{0,1\}^r$ such that $\sum_{i = 1}^r \lambda_i \eps_i \in \{0,1\}$. Observe that we have $\mb{P}_{\eps \in \{0,1\}^r} \big(\sum_{i = 1}^r \lambda_i \eps_i \in \{0,1\}\big) \le \frac{3}{4}$ since, if $\lambda_u,\lambda_v \ne 0$, then the set $\{0, \lambda_u ,\lambda_v, \lambda_u + \lambda_v\}$ has size at least $3$, hence for every $x = \sum_{i \ne u,v} \lambda_i \eps_i$ at least one of $x, x+ \lambda_u,x + \lambda_v, x + \lambda_u + \lambda_v$ is not in $\{0,1\}$. It follows that the probability that $\sum_{i = 1}^r \lambda_i v_i  \in \{0,1\}^N$ is $\le (3/4)^{N - N'}$. Summing over the possible choices of $(\lambda_1,\dots, \lambda_r)$ gives that the probability of $Q$ failing to be typical (for fixed $Q'$ of full rank) is $\le 2^r (3/4)^{N - N'}$.

Combining the above estimates, we see that the probability that $Q$ is not typical is $\le 2^{r - N'} +2^r (3/4)^{N - N'}$. Using the bounds $r \le N/10$ and the choice of $N'$, the result follows.
\end{proof}

\section{The function \texorpdfstring{$g$}{}}\label{g-function-sec}

In this section we give the basic properties of the function $g \in C^{\infty}(\R/\Z)$ featuring in our main theorem, and defined as in \cref{g-funct-def}. We have already noted the transformation law \cref{g-lambda-trans}. The other key properties we need are contained in the following lemma.

\begin{lemma}
The Fourier coefficients $\wh{g}(m) := \int_{\R/\Z} g(x) e^{-2\pi i m x} dx$ are given by
\begin{equation}\label{g-fourier-formula}
\wh{g}(m) = {-}\frac{1}{\log 2} \Gamma \big( \frac{\log \log 2}{\log 2} + \frac{2\pi m}{\log 2}i\big).
\end{equation}
In particular, $\wh{g}(m) \ne 0$ for all $m \in \Z$. We have the bounds
\begin{equation}\label{g-fourier-bound} |\wh{g}(m)| < \frac{6}{\log 2} e^{-\pi^2 |m|/\log 2}.\end{equation}
We have
\begin{equation}\label{g-ratio-bd}
\frac{\max g}{\min g} \le 1 + 2 \times 10^{-7}.
\end{equation}
\end{lemma}
\begin{remark}
Note that since the Fourier coefficients $|\wh{g}(m)|$ decay quicker than any fixed power of $m$, the function $g$ is indeed smooth with $g^{(j)}(\theta) = \sum_{m \in \Z} (2 \pi i m)^j \widehat{g}(m) e^{2\pi i \theta m}$ (see for instance \cite[Proposition 3.3.12]{grafakos}).
\end{remark}
\begin{proof} From the definition \cref{g-funct-def} and unfolding the $\int_{\R/\Z}$ and $\sum_{D \in \Z}$ into an integral over the real line, it follows that for $m \in \Z$ we have
\[
\wh{g}(m) = \int^{\infty}_{-\infty} (\log 2)^x (1 - e^{-2^{x}}) e^{2\pi i m x} \, \mathrm{d}x  = (\log 2)^{-1}\int^{\infty}_0 u^{\frac{\log \log 2}{\log 2} + \frac{2\pi i m}{\log 2} - 1} (1 - e^{-u})\, \mathrm{d}u    . 
\]
To investigate this, it makes sense to consider more generally the integral
 \begin{equation}\label{ist-def} I_{s} := \int_{0}^{\infty}u^{s - 1}(1-e^{-u})\, \mathrm{d}u\end{equation}
 for $\Re s \in (-1,0)$. It turns out that this can be explicitly evaluated in terms of the $\Gamma$-function, namely $I_{s} = {-}\Gamma(s)$ for $-1 < \Re s < 0$, and from this \cref{g-fourier-formula} follows immediately.
 
The evaluation of $I_s$ can be justified\footnote{
As an aside for the reader, the authors originally discovered that $g$ was almost constant via the use of GPT o4-mini. Upon being asked to prove this fact, the model suggested the broad strategy but miscomputed this integral (instead replacing it by the computation of the integral $\int_{0}^{\infty}\frac{x^{s-1}}{e^{x}-1}\, \mathrm{d}x = \zeta(s)\Gamma(s)$); this embarrassingly tricked the second author until clarification by the first author. Upon prompting the model with the relevant Wikipedia page on Mellin transforms, the model was able to complete the proof and give the counter integral argument to bound the $\Gamma(\cdot)$ (which is of course completely standard). }as follows:
\begin{align*}
\int^{\infty}_0 u^{s-1} (1 - e^{-u}) \, \mathrm{d}u &  = \int^{\infty}_0 u^{s-1} \int^u_0 e^{-t} \, \mathrm{d}t \mathrm{d}u   = \int^{\infty}_0 e^{-t} \int^{\infty}_t u^{s-1} \, \mathrm{d}u \mathrm{d}t \\ & = \int^{\infty}_0 e^{-t} \big(-\frac{t^s}{s}\big)\, \mathrm{d}t  = {-}\frac{1}{s}\Gamma(s+1) = {-}\Gamma(s).
\end{align*}  

Turning to \cref{g-fourier-bound}, presumably one can look up precise numerical asymptotics for $\Gamma$ in the relevant strip $-1 < \Re s < 0$. However, the integral representation \cref{ist-def} allows us to derive an appropriate bound from first principles with minimal effort via a contour integration argument, which we give now.

 Write $s = \sigma + it$ and assume $t > 0$ (the case $t < 0$ can be handled symmetrically or by taking complex conjugates). One may note that the integral here indeed converges for $\sigma \in (-1,0)$; as $u \rightarrow \infty$, the condition $\sigma < 0$ guarantees convergence, whilst as $u \rightarrow 0$ we have $1 - e^{-u} \ll u$ and the condition $\sigma > -1$ guarantees it.
To estimate $I_{\sigma+ it}$ we perform a contour integration in the upper right quadrant. Denote by $\mathcal{C}_r = \{ r e^{i \theta} : 0 \le \theta \le \pi/2\}$ the quarter circle of radius $r$. Let $0 < r < 1 < R$. Then applying Cauchy's theorem around the contour formed by $\mathcal{C}_r, \mathcal{C}_R$ and the portions of the real and imaginary axes between $r$ and $R$, we see that 
\begin{align*} \int^R_r & u^{\sigma - 1 + it}(1-e^{-u})\,\mathrm{d}u \\ & = i\int^R_r (u i)^{\sigma - 1 + it} (1 - e^{-u i})\,\mathrm{d}u - \int_{\mathcal{C}_R} u^{\sigma - 1 + it}(1-e^{-u})\,\mathrm{d}u + \int_{\mathcal{C}_r} u^{\sigma - 1 + it}(1-e^{-u})\,\mathrm{d}u . \end{align*}
We first estimate the integral over $\mathcal{C}_R$. Using that $|1 - e^{-z}| \le 2$ for $\Re z \ge 0$, we see that on this contour we have $|1 - e^{-u}| \le 2$, and at $u = R e^{i \theta}$ we have $|u^{\sigma - 1 + it}| = R^{\sigma-1} e^{-\theta t} \le R^{\sigma - 1}$ for $t \ge 0$. The length of the contour is $\pi R/2$, and so the total contribution is $\ll R^{\sigma - 1} \cdot R$ which tends to zero as $R \rightarrow \infty$, since $\sigma < 0$.

Next we estimate the integral over $\mathcal{C}_r$. Now we use $|1 - e^{-u}| \ll |u| = r$ and $|u^{\sigma - 1 + it}| \ll r^{\sigma-1}$. The length of the contour is $\pi r/2$, so the total contribution is $\ll r \cdot r^{\sigma-1} \cdot r = r^{\sigma + 1}$, which tends to zero as $r \rightarrow 0$ since $\sigma > -1$.

It follows that 
\[ I_{\sigma+it} = i \int^{\infty}_0 (u i)^{\sigma - 1 + it} (1 - e^{-u i})\, \mathrm{d}u = \int^{\infty}_0 u^{\sigma - 1 + it}e^{\sigma i\pi/2  - \pi t/2} (1-e^{-ui})\,\mathrm{d}u. \]

Taking absolute values,
\begin{align*} |I_{\sigma+it}| \le e^{-\pi t/2}  \int^{\infty}_0  u^{\sigma-1} |1 - e^{-iu}|\, \mathrm{d}u  & \le e^{-\pi t/2}\int_{0}^{\infty}u^{\sigma-1}(\min(u, 2))\,\mathrm{d}u\\ & =  2^{\sigma + 1}\big(\frac{1}{\sigma+1} - \frac{1}{\sigma}\big) e^{-\pi t/2} .\end{align*} 

In the case of interest to us, where $\sigma = \frac{\log \log 2}{\log 2}$, we get a numerical value of less than $6$, so $|\Gamma(\sigma + it)| = |I_{\sigma + it}| \le 6 e^{-\pi t/2}$ in this case. Since $\widehat{g}(m) = -\frac{1}{\log 2} I_{\frac{\log \log 2}{\log 2} + \frac{2\pi i m}{\log 2}}$, the bound \cref{g-fourier-bound} follows immediately.

Finally we establish \cref{g-ratio-bd}. Using \cref{g-fourier-bound} and simple numerics we have
\[
\sum_{|m| \ge 2} |\wh{g}(m)|  \le \frac{12}{\log 2} \sum_{m = 2}^{\infty} e^{-\pi^2 m/\log 2} < 7.5 \times 10^{-12}.
\]

By further numerical evaluations (using Wolfram Alpha) we have that 
\[
\wh{g}(0) = -\frac{1}{\log 2} \Gamma\big(\frac{\log \log 2}{\log 2}\big) \approx 5.1278218186\] and
\[ |\wh{g}(1)|  = |\wh{g}(-1)| = \Big|\frac{1}{\log 2}\Gamma\big(\frac{\log\log 2 + 2\pi i}{\log 2}\big)\Big| \approx 2.4479026947 \times 10^{-7}.
\]
Thus we conclude that
\[\frac{\max g}{\min g} \le \frac{\wh{g}(0) + \sum_{m\neq 0}|\wh{g}(m)|}{\wh{g}(0)-\sum_{m\neq 0}|\wh{g}(m)|}\le 1 + 2 \cdot 10^{-7},\] which is the desired bound \cref{g-ratio-bd}. 
\end{proof}

\part{The Poisson paradigm}\label{part3}

\section{Approximate \texorpdfstring{$\ell$}{}-wise independence}\label{sec3}

The main result of this section is \cref{lem:Poisson}. We will describe it informally below, but in order to do so we set up some notation.

We will consider a random multiset $(\mbf{A} \mid \beta, \beta')$ sampled conditioned on $(\boldbeta,\boldbeta') = (\beta, \beta')$ for some fixed pair of upper and lower walks $\beta,\beta'$. The general setup under discussion here is outlined in \cref{section5} and the reader may wish to review that now. We recall in particular from \cref{bi-defs,a-list,a-size-2} that
 $\mbf{A} = (\mbf{a}_{i,j})_{i \in [n], j \in [b_i]}$, where 
 \begin{equation}\label{bi-defs-rept} b_{i} := 1 - \xi'_i \; \;  (i = 1,\dots, \lfloor n/2\rfloor) \;\;\; \mbox{and}\;\;\;  b_{n + 1 - i} := 1 + \xi_i \; \; (i = 1,\dots, \lceil n/2\rceil).\end{equation}
 Throughout the section we will write
 \begin{equation}\label{disc-def} D = D(\beta,\beta') := \beta(\lceil n/2\rceil) - \beta'(\lfloor n/2\rfloor) = \sum_{i = 1}^n b_i - n.\end{equation}
 We occasionally refer to this quantity as the \emph{discrepancy} of $(\beta, \beta')$.

From now on we fix a truncation parameter $M$; the walks under consideration will have \[ \min_{1 \le i \le \lceil n/2\rceil} \beta(i), \; \min_{1 \le i \le \lfloor n/2\rfloor} \beta'(i) \ge -M,\] and later we will let $M \rightarrow \infty$ as $n \rightarrow \infty$. Associated to $M$ we will also associate the following key truncation thresholds.

\begin{definition}\label{truncation-threshold-defs} Given $M$ we define
\begin{itemize}
\item $R(M) := 20 M$ is a threshold for the $R$-boundedness of walks;
   \item $T(M) := e^{40M/\eta}$ is a threshold for the $T$-positivity of walks;
    \item $L(M)$ is an extremely large function of $M$. (How large it needs to be is reliant on the bounds in \cref{lemma7.19}, the result on the existence of jump steps; the point at which we need $L(M)$ to be very large is in the proof of \cref{exceptions-bd}.)
\end{itemize}
\end{definition}
These thresholds will be in place for the next several sections. We will also require some more technical parameters  namely $\ell_*(M), V(M), t_{\min}(M), t_{\max}(M)$. The parameter hierarchy will be

\begin{equation}\label{param-heir} T(M) \lll \ell_*(M)  \lll V(M) \lll t_{\min}(M) \lll t_{\max}(M) \lll L(M) < n/2.\end{equation}
and $V(M)$ will be an integer. For the rest of this section we will drop the explicit mention of $M$ when referring to $R$, $T$, $\ell_*$, $V$, $t_{\min}$, $t_{\max}$ and $L$ for notational brevity. We briefly discuss the purpose of the rest of these parameters.

The $\ell_*$ parameter will control the depth to which we manage to run a crucial inclusion-exclusion principle during the argument.

The $V, t_{\min}, t_{\max}$ parameters are connected with the (rather technical) notion of jump step (\cref{jump-step-7}). Throughout the next two sections we will assume that $\beta$ has a $V$-jump step at some index $t = t(\beta)$ which satisfies $t_{\min} \le t \le t_{\max}$; for definiteness we take $t$ minimal subject to these constraints. We informally remark that by \cref{lemma7.19} this will be a generic property of $\beta$ for appropriate choices of parameters. Note also that $b_{n + 1 - t} = 1 + \xi_t = V$. The jump step $t$ is associated with the random elements $\mbf{a}_{n + 1 -t,j}$, $j \in [V]$, of $\mbf{A}$, all of which have size $\asymp 2^{n - t}$; these will play a special role in what follows.

For technical reasons to do with calculations in the local limit theorem we also need to condition on the approximate distribution of these elements $\mbf{a}_{n + 1 - t,j}$. To set this up let us briefly recall from \cref{a-list} that $\mbf{a}_{n + 1 -t,j} = \tilde\phi(n - \mbf{u}_{t,j})$, where the $\mbf{u}_{t,j}$, $j \in [V]$, are independent uniform random variables on $(t-1,t]$. The definition of $\tilde\phi$ may be found at \cref{phis-def}; recall in particular that $\tilde\phi(x) \asymp 2^x$, so that indeed $\mbf{a}_{n + 1 - t} \asymp 2^{n - t}$.

For the rest of the section we abuse notation by writing $V^{1/3}$ for $\lfloor V^{1/3}\rfloor$ for brevity. Partition 
\begin{equation}\label{im-def} (t - 1,t] = \bigcup_{m = 1}^{V^{1/3}} I_m, \qquad I_m := \big( t - 1+ \frac{m - 1}{V^{1/3}}, t - 1 + \frac{m}{V^{1/3}}\big].\end{equation}
For each $\mbf{u}$ we record the `jump distribution' which is the random tuple $\mbf{f}  = (\mbf{f}_1,\dots, \mbf{f}_{V^{1/3}})$ where $\mbf{f}_m$ is the number (frequency) of the $\mbf{u}_{t,j}$ in the interval $I_m$. Thus, in particular we have $\sum_m \mbf{f}_m = V$ deterministically. The conditioned variables $(\mbf{u}_{t,j} \mid \mbf{f} = f)$ consist of $f_m$ independent uniform samples from $I_m$ for $m = 1,\dots, V^{1/3}$. 

For the remainder of this section and much of the next, we will condition not just to fixed choices $\boldbeta = \beta$, $\boldbeta' = \beta'$, but also to a fixed choice $\mbf{f} = f$. The precise point at which the additional conditioning on $\mbf{f}$ will be removed is after \cref{eq730d} below. For notational brevity, until that point we write $\mbf{A}$ rather than $(\mbf{A} \mid \boldbeta = \beta, \boldbeta' = \beta', \mbf{f} = f)$. The key result of the section is \cref{lem:Poisson} below. The main aim of this result (cf. \cref{key-a-beta}) is to understand the quantity $\mb{P}(k \in \Sigma(\mbf{A}))$ (with $\mbf{A}$ conditioned to $\boldbeta = \beta,\boldbeta' = \beta', \mbf{f} = f$ as discussed above).

The statement of \cref{lem:Poisson} is a little technical, and the proof lengthy, so we start by offering some motivation. We set 
\begin{equation}\label{as-def} \mbf{A}_{\sml} := (\mbf{a}_{i,j})_{1 \le i \le L, j \in [b_i]}, \quad  \mbf{A}_{\med} := (\mbf{a}_{i,j})_{L+1 \le i \le n + 1 - t, j \in [b_i]},\end{equation}
and
\begin{equation*}  \mbf{A}_{\lrg} := (\mbf{a}_{i,j})_{n +1 - t < i \le n, j \in [b_i]}.\end{equation*}
Note in particular that the elements associated to the jump step of $\beta$ are the elements of $\mbf{A}_{\med}$ with the highest index $i = n + 1 - t$.
We remind the reader that $\Sigma (\mbf{A})$ is the set of sums $\{\sum_{1 \le i \le n, j \in [b_i]} \eps_{i,j} \mbf{a}_{i,j} : \eps_{i,j} \in \{0,1\}\}$, and (for example) $\Sigma(\mbf{A}_{\med})$ is defined analogously.

We consider the event $(k \in \Sigma(\mbf{A}))$ as the union of the events $X_{\eps} = (k - \sum_{n + 1 - t < i \le n, j \in [b_i]} \eps_{i,j} \mbf{a}_{i,j} \in \Sigma(\mbf{A}_{\med}) + \Sigma(\mbf{A}_{\sml}))$, over all choices of $\eps_{\lrg} := (\eps_{i,j})_{n + 1 - t < i \le n, j \in [b_i]}$. The idea will be to compute $\mb{P}(k \in \Sigma(\mbf{A}))$ by using inclusion-exclusion on these events. To enable this, we need to show that they are approximately $\ell$-wise independent for $\ell \rightarrow \infty$.

We relabel the elements of $\mbf{A}$ associated to the jump step $t$ as $(\mbf{a}_{n + 1 -t,j})_{j \in [V]}$ to be consistent with the conditioning to $\mbf{f} = f$. Thus we write $\mbf{a}_{n + 1 - t,j} = \tilde\phi (n - \mbf{U}_m)$ where $m = m(j)$ is the unique value of $m \in \{1,\dots V^{1/3}\}$ for which $f_1 + \cdots f_{m-1} < j \le f_1 + \cdots + f_m$, $\mbf{U}_m$ is uniform on $I_m$ and $\tilde\phi$ is defined in \cref{phis-def}. Unravelling the definitions, we thus have 
\[ \mb{P}(\mbf{a}_{n + 1 - t,j} = x) = V^{1/3}\big| (n - I_m) \cap \big[ \frac{n}{H_k} H_{x-1}, \frac{n}{H_k} H_x\big)\big|.\] Thus, if we set $\alpha := V^{1/3}n/H_k \approx V^{1/3}/\log 2$ then the distribution of $\mbf{a}_{n + 1 - t,j}$ is as follows. There are integers $x_0, x_1 \asymp 2^{n - t}$ and $\alpha_0, \alpha_1$ with $0 \le \alpha_0, \alpha_1 \le \alpha$ (all potentially depending on $j$) such that:
\begin{equation}\label{a-jump-dist} \mb{P}(\mbf{a}_{n +1 - t,j} = x) = \left\{ \begin{array}{ll} 0 & \mbox{$x < x_0$ or $x > x_1$};\\ \alpha_0/x_0 & x = x_0; \\ \alpha/x & x_0 < x < x_1;\\ \alpha_1/x_1 & x = x_1.  \end{array}\right.\end{equation}
Thus, $\mbf{a}_{n + 1 - t, j}$ takes values in an interval of length $\asymp V^{-1/3} 2^{n - t}$. A particular point to note is that we must replace the anticoncentration bound \cref{anti-concentration-aij} for $i = n + 1 -t$ by the weaker bound
 \begin{equation}\label{anti-concentration-jump-scale} \sup_x \mb{P}(\mbf{a}_{n + 1- t,j} = x) \ll V^{1/3} 2^{t - n}\end{equation} (which follows immediately from \cref{a-jump-dist}); for other values of $i$ it is unchanged.

Define $\mu_j = \mb{E} \mbf{a}_{n + 1 - t,j}$ and write
\begin{equation}\label{mu-sig-def} \mu = \frac{1}{2}\sum_{j = 1}^V \mu_j \quad \mbox{and} \quad  \sigma^2 = \frac{1}{4}\sum_{j = 1}^V \mb{E} \mbf{a}_{n + 1 - t,j}^2.\end{equation}
In various places in the analysis we will use that (deterministically, and uniformly in $f$)
\begin{equation}\label{basic-a-bds} \mbf{a}_{n + 1 - t,j} , \mu_j \asymp 2^{n - t}, \qquad |\mbf{a}_{n + 1 - t,j} - \mu_j| \ll V^{-1/3} 2^{n - t}, \qquad \sigma \asymp V^{1/2} 2^{n - t},\end{equation} all of which are clear from the definitions. Finally let $g$ be the Gaussian density with mean $\mu$ and variance $\sigma^2$,
\begin{equation}\label{gaussian-def} g(x) := \frac{1}{\sigma \sqrt{2 \pi}} e^{-\frac{1}{2}\big( \frac{x - \mu}{\sigma}\big)^2},\end{equation} noting that once again this depends on $f$. 

Finally we come to the statement of \cref{lem:Poisson}. Recall \cref{bounded-walk-def,good-walk-def}.

\begin{lemma}\label{lem:Poisson} Let $\beta, \beta'$ be upper and lower walks that are each $R$-bounded and $T$-positive up to lengths $\lceil n/2\rceil, \lfloor n/2\rfloor$ respectively. Set $L := L(M) < n/2$. Suppose that $\min_{1 \le i \le \lceil n/2\rceil} \beta(i)$, $\min_{1 \le i \le \lfloor n/2\rfloor} \beta'(i) \ge - M$, and that $|D| \le M/10$, where the discrepancy $D$ is given by \cref{disc-def}.
Suppose that $\beta$ has a $V$-jump step. Let $n'$ be an integer satisfying $L \le n' \le n - L$. Suppose that $S \subset \mbf{N}$ is a set of size $\tau 2^{\sum_{i=1}^{n'} b_{i}}$ with $\max(S) \le 2^{n - L/2}$ and where $V^{-1} \le \tau \le 1$. Let $\ell \le \ell_*$, and let $x_i$, $i = 1,\dots, \ell$, be positive integers such that 
\begin{equation}\label{xi-assump} |x_i - \mu|\le (\log V)^{1/4} \sigma.\end{equation} Furthermore suppose that we have the separation condition \begin{equation}\label{lem:Poisson:sep} |x_i-x_{i'}| > 2^{n - L/2} \end{equation} for $i \ne i'$. Write $\mbf{A}_{\med}^{> n'} := (\mbf{a}_{i,j})_{n'+1 \le i \le n + 1 - t, j \in [b_i]}$. Then we have that 
\begin{equation}\label{lem31-main} \mb{P} \big(x_1,\dots, x_{\ell} \in S + \Sigma(\mbf{A}^{> n'}_{\med}) \big) =\big(1 + O( V^{-1/4})\big) \Big( \tau 2^{\sum_{i =1}^{n+1 - t} b_i}\Big)^{\ell} \prod_{i = 1}^{\ell} g(x_i).\end{equation}
\end{lemma}

\begin{remarks} Note that $\mbf{A}_{\med}^{> L} = \mbf{A}_{\med}$. Note that by \cref{mu-sig-def,basic-a-bds,xi-assump} we have
\begin{equation} \label{xi-rough} x_1,\dots, x_{\ell} \asymp V 2^{n - t},\end{equation} and so one typically expects the separation condition to be comfortably satisfied; see \cref{7lem-claim-2} for more details on this.
\end{remarks}

The following simple lemma will feature a few times in the proof.
\begin{lemma}
Suppose that $\beta, \beta'$ are upper and lower walks. Suppose that $\beta$ is $T$-positive up to length $\lceil n/2\rceil$, and that the discrepancy $D = D(\beta,\beta')$ satisfies $|D| \le M/10$. Let notation be as in \cref{bi-defs-rept}. Then
\begin{equation}\label{bnt} -2t_{\max} \le \sum_{i=1}^{n + 1 - t} b_{i} - (n - t)  \le - \tfrac{1}{2}t^{1/2 - \eta},\end{equation}
and
\begin{equation}\label{bnL} \sum_{i = 1}^{n-L} b_{i} < n - L.\end{equation}
\end{lemma}
\begin{proof} 
From \cref{bi-defs-rept,disc-def} we have that if $m < \lceil n/2\rceil$ then \begin{equation}\label{bbeta}\sum_{i = 1}^{n - m} b_{i} =  n + D - (m + \beta(m)).\end{equation}
 We first prove \cref{bnt}. We certainly have $t \le t_{\max} \lll n/2$, and so by \cref{bbeta} we have $\sum_{i = 1}^{n + 1 - t} b_i - (n - t) = D + 1 - \beta(t - 1)$. 
 
 For the lower bound, by the $T$-positive nature of the walk $\beta$ it follows that $D - \beta(t-1) \ge D - T - (t - 1)^{1/2 + \eta}$, and the result follows from the parameter hierarchy \cref{param-heir}, noting in particular that $t \le t_{\max}$.

 For the upper bound, by the $T$-positive nature of $\beta$ it follows that $D - \beta(t-1) \le D + T - (t - 1)^{1/2 - \eta}$, and the result again follows from the parameter hierarchy, specifically that $t \ge t_{\min} \ggg T$, and all these parameters are sufficiently large in terms of $M$. For \cref{bnL}, an analogous proof works since $L \ggg T$.
\end{proof}

For the proof of \cref{lem:Poisson} (which is rather lengthy), the plan is to first establish a weighted version of \cref{lem31-main} in which the number of representations in $S + \Sigma(\mbf{A}_{\med}^{\ge n'})$ is counted. This is stated as \cref{lem:Poisson-weights} below. However, in practice the number of representations is essentially always either 0 or 1, so we can later remove the weights. To do this rigorously a related technical result (\cref{lem:Poisson-upper} below) is required. The deduction of \cref{lem:Poisson} from the weighted version is quite short and is given at the end of the section.

Fix $S$ and $n'$ and define $R_{\mbf{A}}(x)$ to be the number of representations of $x$ in $S + \Sigma(\mbf{A}_{\med}^{> n'})$, that is to say as a sum of a $\{0,1\}$-combination of the $\mbf{a}_{i,j}$, $n'+1 \le i \le n + 1 - t, j \in [b_i]$, and an element of $S$.

\begin{lemma}\label{lem:Poisson-weights} With the notation and assumptions of \cref{lem:Poisson}, and with $R_{\mbf{A}}$ as just defined, we have 
\begin{equation}\label{lem31-main-weighted} \mb{E} R_{\mbf{A}}(x_1) \cdots R_{\mbf{A}}(x_{\ell}) =\big(1 + O( V^{-1/4})\big) \Big( \tau 2^{\sum_{i =1}^{n + 1 - t} b_i}\Big)^{\ell} g(x_1) \cdots g(x_{\ell}).\end{equation}
\end{lemma}
\begin{proof}  Denote by $\Xi$ the set of pairs $(i,j)$ with $n'+1 \le i \le n + 1 - t$ and $j \in [b_i]$. Thus
\begin{equation}\label{xi-size} |\Xi| = \sum_{i = n'+1}^{n + 1 - t} b_i.\end{equation}
Note for future reference that, since $n' \le n - L$ and $b_i \ge 0$ for all $i$, 
\begin{equation}\label{xi-fr} |\Xi| \ge \sum_{i = n+1-L}^{n + 1 - t}b_i = (L + \beta(L)) - \big((t-1) +  \beta(t-1)\big) > L/2\end{equation} by the $T$-positive property \cref{good-walk-def} and the parameter hierarchy \cref{param-heir}, and similarly, using the parameter hierarchy and that $n' \ge L$, $L < n/2$ and $|D| < M$,
\begin{equation}\label{xi-upper-mat} |\Xi| \le \sum_{i = L+1}^{n + 1 - t} b_i = n + D - \big((t - 1) + \beta(t - 1)\big) - (L - \beta'(L)) < n - L/2.\end{equation}
Recall that by definition $R_{\mbf{A}}(x)$ is the number of pairs $(\eps, s) \in \{0,1\}^{\Xi} \times S$ with 
\[ \sum_{(i,j) \in \Xi} \eps_{i,j} \mbf{a}_{i,j} + s = x.\] Thus 
\begin{equation}\label{counting-cor} \mb{E} R_{\mbf{A}}(x_1) \cdots R_{\mbf{A}}(x_{\ell}) = \sum_{\eps^{(1)}, \dots, \eps^{(\ell)} \in \{0,1\}^{\Xi}}\sum_{s_1,\dots, s_{\ell} \in S} \mb{P}\Big(\sum_{(i,j) \in \Xi} \eps^{(u)}_{i,j} \mbf{a}_{i,j} = x_u - s_u, u \in  [\ell]\Big).\end{equation}
For each choice of $\eps^{(1)},\dots, \eps^{(\ell)}$ consider the $\ell$-by-$|\Xi|$ matrix $M(\eps) := (\eps^{(u)}_{i,j})_{u \in [\ell], (i,j) \in \Xi}$ formed with the $\eps^{(u)}$ as rows. We will consider the columns $(\eps_{i,j}^{(u)})_{u \in [\ell]}$ to be ordered by decreasing $i$ and increasing $j$; thus the first $V$ columns of $M(\eps)$ correspond to $(i,j) = (n + 1 - t, j)$, $j \in [b_{n+1 - t}] = [V]$, these being the ones associated to the jump step.

To estimate the inner probability in \cref{counting-cor}, we will need to pay attention to the linear-algebraic structure of this matrix $M(\eps)$. To this end, recall the definition of rank profile from \cref{01-matrix-subsection} and split the sum according to the rank profile. Write $E_{(r; q_1,\dots, q_r)}$ for the contribution of the terms with $\rp(\eps^{(1)},\dots, \eps^{(\ell)})= (r; q_1,\dots, q_r)$ to \cref{counting-cor}. 

Our first main task (which is already somewhat lengthy) will be to show that only terms with $r = \ell$ and $q_{\ell} \le V$ contribute importantly; to this end will show that 
\begin{equation}\label{main-first-to-show} \sum_{\substack{(r; q_1,\dots, q_r) \\ r < \ell \mbox{\,\scriptsize or} \\ r = \ell \mbox{\, \scriptsize and \,} q_{\ell} > V}} E_{(r; q_1, \dots, q_r)} \ll 2^{-\Omega(V)} 2^{\ell \big(\sum_{i =1}^{n + 1 - t} b_i - (n - t) \big)}.\end{equation}
To see that this is indeed negligible compared to the stated bound in \cref{lem:Poisson} we use the assumption that $\tau \ge V^{-1}$, the fact that $g(x_i) \gg \sigma^{-1} e^{-C (\log V)^{1/2}}$ (here we use the assumption \cref{xi-assump}) and finally that $\sigma \asymp V^{1/2} 2^{n - t}$ (see \cref{basic-a-bds}) and that $V \ggg \ell$ (from \cref{param-heir}).

The proof of the claim \cref{main-first-to-show} will occupy the next few pages. 
Consider to begin with a fixed choice of $\eps^{(1)},\dots, \eps^{(\ell)}$ with $\rp(\eps^{(1)},\dots, \eps^{(\ell)}) = (r; q_1,\dots, q_r)$. To bound the inner sum in \cref{counting-cor}, first note that 
\begin{equation}\label{remove-sum-s} \sum_{s_1,\dots, s_{\ell} \in S} \mb{P} \Big(\sum_{(i,j) \in \Xi} \eps^{(u)}_{i,j} \mbf{a}_{i,j} = x_u - s_u, u \in [\ell]\Big) \le |S|^r \sup_{x'_u} \mb{P}\Big(\sum_{(i,j) \in \Xi} \eps^{(u)}_{i,j} \mbf{a}_{i,j} = x'_u, u \in [\ell]\Big), \end{equation} since (due to the rank of $M(\eps)$ being $r$) for fixed $x_u$ the probability on the left can only be nonzero for at most $|S|^r$ choices of $(s_1,\dots, s_{\ell})$, since all $s_u$s are determined by some subset of $r$ of them.

Consider the columns $(\eps_{i,j}^{(u)})_{u \in [\ell]}$ of $M(\eps)$ for $(i,j) \in \Xi$. We remind the reader that we order these first by decreasing $i$ and then by increasing $j$. Let the columns corresponding to increments in the rank be $(n +1 - i_1,j_1),\dots, (n +1 - i_r, j_r)$ where $t \le i_1 \le i_2 \le \cdots \le i_r$ (note that these are completely determined by the rank profile). In particular observe that 
\begin{equation}\label{qd-id}  \sum^{ n + 1 - t}_{n + 2- i_d }b_i < q_d \le \sum^{n + 1 - t}_{ n +1 -i_d} b_i,\end{equation} a relation we will use several times. We claim that 
\begin{equation}\label{single-prob} \mb{P} \Big(\sum_{(i,j) \in \Xi} \eps^{(u)}_{i,j} \mbf{a}_{i,j} = x'_u, u \in [\ell]\Big) \ll  V^{O(\ell)} \prod_{d = 1}^r 2^{i_d - n}.\end{equation}
To see this, first reveal all $\mbf{a}_{i,j}$ with $(i,j)$ not one of the $(n +1 -i_d, j_d)$, $d \in [r]$. The remaining elements $\mbf{a}_{n +1 -i_1,j_1},\dots, \mbf{a}_{n + 1 - i_r, j_r}$ are then determined (with the $\eps$ and $x'_u$ fixed). The claim \cref{single-prob} then follows from the anticoncentration bound \cref{anti-concentration-aij}, unless $n +1 - i_d = t$ in which case we must instead use \cref{anti-concentration-jump-scale}, which potentially gives rise to the $V^{O(\ell)}$ factors.

Combining \cref{remove-sum-s,single-prob} gives
\begin{equation}\label{1124-1126} \sum_{s_1,\dots, s_{\ell} \in S} \mb{P} \Big(\sum_{(i,j) \in \Xi} \eps^{(u)}_{i,j} \mbf{a}_{i,j} = x_u - s_u, u \in [\ell]\Big) \le V^{O(\ell)}|S|^r \prod_{d = 1}^r 2^{i_d - n}.  \end{equation}
From this, \cref{lem24} (i) and the definition of $E_{(r; q_1,\dots, q_r)}$ (just after \cref{counting-cor}) we have
\begin{align} \nonumber E_{(r; q_1,\dots, q_r)}  & \ll 2^{\ell^2} 2^{r|\Xi| - \sum_{d = 1}^r q_d}|S|^r V^{O(\ell)} \prod_{d = 1}^r 2^{i_d-n} \\ &  \ll V^{O(\ell)} 2^{r \big( \sum_{i=1}^{n + 1 - t} b_i - (n - t)\big) - \sum_{d = 1}^r (q_d - i_d + t)}.\label{bd-fixed-rank-profile} \end{align} 
(Here, the $2^{\ell^2}$ has been absorbed into the $V^{O(\ell)}$ term due to the assumption \cref{param-heir} on parameters, and for the final step we used \cref{xi-size} and the crude bound $|S| \le 2^{\sum_{i=1}^{n'} b_i}$, which follows from the assumption that $\tau \le 1$.) The bound \cref{bd-fixed-rank-profile} will be our key tool in the following analysis, the aim of which is to establish \cref{main-first-to-show}. Below we will use it examine the contributions to \cref{main-first-to-show} from various rank profiles. 

Suppose that $d \in \{1,\dots, r\}$. Using \cref{qd-id} we have
\begin{equation}\label{qd-low} q_d \ge \sum_{n + 2- i_d}^{n +1 - t} b_i = \sum_{n + 2- i_d}^{n} b_i - (t - 1 + \beta(t - 1)) \ge \sum_{n + 2- i_d }^{n} b_i - 2t_{\max},\end{equation}
 using the $T$-positive nature of $\beta$ (see \cref{good-walk-def}) and that $t \le t_{\max}$, $T \lll t_{\max}$.

Suppose first that $i_d \le 1 + \lceil n/2\rceil$. Then \cref{qd-low} gives $q_d \ge (i_d-1)  +\beta(i_d-1)  - 2t_{\max}$. By the $T$-positive nature of $\beta$ (and since $T \lll t_{\max}$) it follows that $q_d \ge i_d + i_d^{1/2 - \eta} - 3t_{\max}$, and therefore by \cref{calc-1} we have $q_d \ge i_d + \frac{1}{2}q_d^{1/2 - \eta} - 3t_{\max}$.

Alternatively, suppose that $i_d > 1 + \lceil n/2\rceil$. Then \cref{qd-low}, the $T$-positive nature of $\beta'$ and the assumption that $|D| \le M/10 \lll t_{\max}$ give $q_d \ge D -1 + i_d + \beta'(n - i_d) - 2t_{\max} \ge i_d + (n - i_d)^{1/2 - \eta} - 3t_{\max}$ and so $(n - q_d) \le (n - i_d) - (n - i_d)^{1/2 - \eta} + 3t_{\max}$. Note that by \cref{xi-upper-mat} and since $n + 1 - i_d \ge n' + 1$, we have $n - q_d, n-i_d > 0$. By \cref{calc-2} it follows that $(n - q_d) \le (n - i_d) - \frac{1}{2}(n - q_d)^{1/2 - \eta} + 6t_{\max}$.

In both cases we have
\begin{equation}\label{iq-bd}  q_d - i_d  \ge  \frac{1}{2}\min(q_d, n - q_d)^{1/2 - \eta} - O(t_{\max}).\end{equation}

We now turn to the promised analysis of the contributions to \cref{main-first-to-show} from various rank profiles.\vspace*{8pt}

Step 1: the contribution from $q_r \ge t_{\max}^{10}$. Using the parameter hierarchy \cref{param-heir} it follows from \cref{bd-fixed-rank-profile,iq-bd} that
\[ E_{(r; q_1,\dots, q_r)} \ll 2^{O(\ell t_{\max}) - \frac{1}{2}\sum_{d = 1}^r \min(q_d, n - q_d)^{1/2 - \eta}}2^{r \big(\sum_{i=1}^{n + 1 - t} b_i - (n - t) \big)} .\]
By \cref{bnt} we may replace the preceding with 
\[ E_{(r; q_1,\dots, q_r)} \ll 2^{O(\ell t_{\max}) - \frac{1}{2}\sum_{d = 1}^r \min(q_d, n - q_d)^{1/2 - \eta}}2^{\ell \big(\sum_{i=1}^{n + 1 - t} b_i - (n - t) \big)} .\]

We sum this over all rank profiles $(r; q_1, \dots, q_r)$ with $r \le \ell$ and $\max_d \min(q_d, n - q_d) = q$, i.e. all $q_d$ are either $\le q$ or $\ge n - q$, with at least one equality.  The number of choices for these is $\le \ell(2q)^{\ell}$, so the contribution of them to \cref{main-first-to-show} is 
\[ \ll \ell (2q)^{\ell}  2^{O(\ell t_{\max}) - \frac{1}{2}q^{1/2 - \eta}}2^{\ell \big(\sum_{i = 1}^{n + 1 - t} b_i - (n - t) \big)}.\] Taking into account the parameter hierarchy \cref{param-heir}, the sum of this over $q \ge t_{\max}^{10}$ is 
\[ \ll 2^{-t_{\max}^4} 2^{\ell \big(\sum_{i =1}^{n + 1 - t} b_i - (n - t) \big)}
,\]
which is completely negligible in comparison to the desired bound \cref{main-first-to-show} since $t_{\max}$ is much bigger than $V$. 

To complete the analysis of Step 1, it suffices to show that if $q < t_{\max}^{10}$ then $q_r < t_{\max}^{10}$. By the definition of $q$, it suffices to show that $q_r < n - t_{\max}^{10}$. This follows immediately from \cref{xi-upper-mat} and the parameter hierarchy, since $q_r \le |\Xi| \le n - L/2 < n - t_{\max}^{10}$. \vspace*{8pt} 

Step 2: The contribution to \cref{main-first-to-show} from $r < \ell$ and $q_r \le t_{\max}^{10}$.  We consider further the structure of the $\eps^{(u)}$. Suppose first that there are only $r$ distinct $\eps^{(u)}$. Then in order for the system of equations $\sum_{(i,j) \in \Xi} \eps_{i,j}^{(u)}\mbf{a}_{i,j}  = x_u - s_u$, $u \in [\ell]$ to have any solutions in the $\mbf{a}_{i,j}$ at all, we must have $x_u - s_u = x_{u'} - s_{u'}$ for some $u \ne u'$, and therefore $|x_u - x_{u'}| \le \operatorname{diam}(S) \le \max(S) \le 2^{n - L/2}$ by assumption. This is contrary to the separation assumption \cref{lem:Poisson:sep}.
 Thus the contribution of these $(\eps^{(1)},\dots, \eps^{(\ell)})$ to \cref{counting-cor} is zero.

Therefore there are at least $r+1$ different $\eps^{(u)}$. Suppose first that one of these is the zero vector. Then we have $x_u = s_u$ and so $x_u \le \max(S) \le 2^{n - L/2}$. However, as noted in \cref{xi-rough}, we have $x_u \asymp V 2^{n - t}$, so there is a contradiction since $L \ggg t$.

Suppose, then, that there are at least $r+1$ different nonzero $\eps^{(u)}$. However, we showed in \cref{lem24} (ii) that such matrices are relatively rare; there are at most $2^{\ell^2} 2^{(r - \frac{1}{10}) |\Xi|}$ of them, noting that the condition $|\Xi| \ge 10\ell$ required in \cref{lem24} (ii) is satisfied by \cref{xi-fr} and the parameter hierarchy.

Using \cref{1124-1126} (and recalling \cref{counting-cor}), we see that the contribution of this case to the LHS of \cref{main-first-to-show} is
\begin{equation}\label{633a} \ll V^{O(\ell)} 2^{(r - \frac{1}{10})|\Xi|}|S|^r \prod_{d = 1}^r 2^{i_d - n} \ll 2^{-|\Xi|/10}V^{O(\ell)} 2^{r \big( \sum_{i=1}^{n + 1 - t} b_i - (n - t) \big)} 2^{-rt + \sum_{d = 1}^r i_d}, \end{equation}
where in the last step we used \cref{xi-size} and the crude bound $|S| \le 2^{\sum_{i =1}^{n'} b_i}$.

To bound this we use \cref{iq-bd}, which certainly implies that $i_d \le q_d + O(t_{\max}) \le 2t_{\max}^{10}$, as well as $\sum_{i =1}^{n + 1 - t} b_i - (n - t) \ge - 2t_{\max}$, which is the left-hand inequality in \cref{bnt}. This gives that \cref{633a} is at most 
\[ \ll 2^{-|\Xi|/10} V^{O(\ell)} 2^{O(\ell t_{\max}^{10})} 2^{\ell \big(\sum_{i \le n + 1 - t} b_i - (n - t)\big)} . \]
This is bounded by the RHS of \cref{main-first-to-show} due to the fact that $|\Xi| > L/2$ (which is \cref{xi-fr}) and the parameter hierachy, which has $L$ vastly larger than $V, \ell$ and $t_{\max}$, and so the analysis of Step 2 is complete.\vspace*{8pt}

Step 3: The contribution to \cref{main-first-to-show} from $r = \ell$ and $V < q_{\ell} \le t_{\max}^{10}$. Here we use that $t$ is a $V$-jump step (\cref{jump-step-7}), which has not previously come in to the analysis. From \cref{bd-fixed-rank-profile} we see that it suffices to show that \begin{equation}\label{step3-main-to-show} \sum_{ \substack{q_1,\dots, q_s \le V \\  V < q_{s + 1}, \dots, q_{\ell} \le t_{\max}^{10}}} 2^{-\ell t + \sum_{d = 1}^{\ell}  (i_d - q_d)} \ll 2^{-V}.\end{equation} for every $s < \ell$ (and then sum over the $\ell$ choices for $s$ and the $V^{O(\ell)}$ choices for $q_1,\dots, q_s \le V$, always bearing in mind that $\ell$ is much smaller than $V$ so that these factors contribute only $2^{o(V)})$. Now for $d \le s$ we have $q_d \le V = b_{n + 1 - t}$ and so $i_d = t$ and hence $-t + (i_d - q_d) \le 0$, and so it suffices to show
\begin{equation}\label{suff-93} \sum_{V < q_d \le t_{\max}^{10}} 2^{-t + (i_d - q_d)} \ll 2^{-V}\end{equation} for all $d = s+1,\dots, \ell$, whereupon the desired bound \cref{step3-main-to-show} follows by taking the product over $d$ (since by the assumption of Step 3 there is at least one $d$ in the range). Fix such a $d \in \{s+1,\dots, \ell\}$. We have $i_d > t$. By \cref{iq-bd} we also have (comfortably) $i_d < n/2$, and so by \cref{qd-id} and \cref{jump-step-1} we have
\begin{align*}
q_d > \sum_{n +2 - i_d}^{n + 1 - t} b_i  & = V + \sum_{n +2 -  i_d}^{n - t} b_i = V + (i_d-1) + \beta(i_d -1) - t - \beta(t) \\ & \ge V + (i_d -1 - t) + (i_d -1 - t)^{1/2 - 2\eta}.
\end{align*}
The LHS of \cref{suff-93} is therefore bounded above by
\begin{equation}\label{suff94} \ll 2^{-V}\sum_{V < q_{d} \le t_{\max}^{10}} 2^{-(i_{d} - 1- t)^{1/2 - 2\eta}} \ll 2^{-V}\sum_{V < q_{d} \le t_{\max}^{10}} 2^{-(i_{d} - t)^{1/2 - 2\eta}}.\end{equation}
To bound this we additionally use \cref{jump-step-2} (and \cref{qd-id} again). This gives that 
\[ q_{d} \le \sum_{n+1  - i_d}^{n + 1 - t} b_i \le V + O \big(\sum_{u = 1}^{i_d - t} u^{2 \kappa}\big) \le V + O\big((i_d - t)^{1 + 2\kappa}\big),\]
and so
\[ (i_d - t)^{1/2 - 2\eta} \gg  (q_d - V)^{(1/2 - 2\eta)/(1 + 2\kappa)} \ge (q_d - V)^{1/3} \] (since $\eta = \kappa = \frac{1}{100}$). Therefore \cref{suff94} is bounded by 
\[ 2^{-V}\sum_{q_d > V} 2^{- \Omega((q_{d} - V)^{1/3})} \ll 2^{-V}\]
and so the required bound \cref{suff-93} holds.\vspace*{8pt}

Steps 1,2 and 3 cover all cases in \cref{main-first-to-show} and so the proof of that statement is now complete. The remaining argument required for the proof of \cref{lem:Poisson} concerns the contribution to \cref{counting-cor} from $\rp(\eps^{(1)},\dots, \eps^{(\ell)}) = (\ell; q_1,\dots, q_{\ell})$ with $q_{\ell} \le V$. Equivalently, the matrix $(\eps^{(u)}_{n + 1 -t,j})_{u \in [\ell], j \in [V]}$ formed from the first $V$ columns of $M(\eps)$ has full rank $\ell$. Beyond this point we will use only this property; the notion of rank profile will have no further role.

Before turning to the heart of the argument, we make one further reduction, which is to the case where at least $4^{-\ell} V$ of the $V$ columns $(\eps_{n + 1 - t,j}^{(u)})_{u \in [\ell]}$, $j = 1,\dots, V$ are the basis vector $e_i \in \mbf{Q}^{\ell}$, for every $i = 1,\dots, \ell$. In this situation we say that $(\eps^{(1)},\dots, \eps^{(\ell)})$ is \emph{good}. We claim that (with the $\eps$ sampled independently at random)
\begin{equation}\label{good-prob} \mb{P}_{\eps} ((\eps^{(1)},\dots, \eps^{(\ell)}) \; \mbox{bad}) \le e^{-\Omega(2^{-\ell} V)}.\end{equation}
For fixed $i \in \{1,\dots, \ell\}$ denote by $X_j$ the event that $(\eps_{n + 1 - t,j}^{(u)})_{u \in [\ell]} =  e_i \in \mbf{Q}^{\ell}$, thus $X_j$ is a Bernoulli random variable with mean $p := 2^{-\ell}$ and $X := \sum_{j = 1}^V X_j$ has the binomial $\operatorname{Bi}(V,p)$ distribution with mean $pV$. By \cref{binom-ld} (and since $p \le \frac{1}{2}$) we have $\mb{P}(X\le p^2 V) \le  \mb{P}( |X - \mb{E} X| \ge \frac{1}{2} \mb{E} X) \le 2 e^{-pV/12}$. Summing over $i = 1,\dots, \ell$ gives 
\[ \mb{P}_{\eps} ((\eps^{(1)},\dots, \eps^{(\ell)}) \; \mbox{bad}) \le  2\ell e^{-pV/12} \le e^{-\Omega(2^{-\ell} V)},\] which is the desired claim \cref{good-prob}.

Thus the \emph{number} of $(\eps^{(1)},\dots, \eps^{(\ell)})$ which are bad is $\le e^{-\Omega(2^{-\ell} V)} 2^{\ell |\Xi|}$. For each such bad matrix, we bound the inner sum in \cref{counting-cor} using \cref{1124-1126}, as before. Now we have $r = \ell$ and $i_1 = \cdots = i_{\ell} = t$ since all the rank jumps of $M(\eps)$ occur in the first $V$ columns (with the ordering of columns described above, that is to say these first $V$ columns are the ones indexed by $(n + 1 -t,1),\dots, (n + 1 - t,V)$). Thus we see that the contribution of these terms to \cref{counting-cor} is
\begin{equation}\label{eq339pre} \ll e^{-\Omega(2^{-\ell} V)} 2^{\ell |\Xi|} \cdot |S|^{\ell} V^{O(\ell)} 2^{\ell(t - n)} .\end{equation}
By the parameter hierarchy we have $e^{-\Omega(2^{-\ell} V)} V^{O(\ell)} < e^{-V^{1/2}}$. Using the trivial bound $|S| \le 2^{\sum_{i =1}^{n'} b_i}$ and \cref{xi-size}, we see that \cref{eq339pre} is
\begin{equation*} \ll e^{-V^{1/2}} 2^{\ell \big(\sum_{i =1}^{n + 1 - t} b_i - (n - t)\big)}.\end{equation*} 
This is negligible compared to the stated bound in \cref{lem:Poisson} by the discussion immediately following \cref{main-first-to-show}.\vspace*{8pt}

It remains to analyse the contribution to \cref{counting-cor} from good tuples $(\eps^{(1)},\dots, \eps^{(\ell)})$. To spell it out, we wish to show that 
\begin{equation}\label{counting-cor-2} \sum_{\substack{(\eps^{(1)}, \dots, \eps^{(\ell)}) \\  \operatorname{good}\\ s_1,\dots, s_{\ell} \in S}} \mb{P} \Big(\sum_{(i,j) \in \Xi} \eps^{(u)}_{i,j} \mbf{a}_{i,j} = x_u - s_u, u \in [\ell]\Big) = \big(1 + O(V^{-1/4})\big) \Big( \tau 2^{\sum_{i =1}^{n + 1 - t} b_i}\Big)^{\ell}\prod_{i = 1}^{\ell} g(x_i).\end{equation}
We now reduce to handling only the indices with $i = n + 1 -t$, that is to say `at the jump step'. To this end write 
\[ h(\mbf{A},\eps^{(u)}) = \sum_{\substack{(i,j) \in \Xi \\ i < n + 1 - t}} \eps_{i,j}^{(u)} \mbf{a}_{i,j}.\] 
Observe that using \cref{dyadic-aij,jump-step-2} that we have (for all choices of the $\eps^{(u)}$, and deterministically in $\mbf{A}$)
\begin{equation}\label{h-bound} |h(\mbf{A},\eps^{(u)})| \ll 2^{n - t}\sum_{u \ge 1} b_{n + 1 - t - u} 2^{- u}  \ll 2^{n - t} \sum_{u \ge 1} u^{2\kappa}2^{-u} \ll 2^{n - t}.\end{equation} 
Write $\eps_{\jump}^{(u)} = (\eps_{n + 1 - t,j}^{(u)})_{j \in [V]} \in \{0,1\}^V$ and also $\mbf{a}_{\jump, j} := \mbf{a}_{n+ 1 - t,j}$. In the light of \cref{h-bound}, and since $|S| = \tau 2^{\sum_{i =1}^{n'} b_i}$ and $|\Xi| = \sum_{i = n' + 1}^{n + 1 - t} b_i = V + \sum_{i = n' + 1}^{n -t} b_i$, in order to establish \cref{counting-cor-2} it suffices to show 
\begin{equation}\label{counting-cor-3} \sum_{(\eps_{\jump}^{(1)}, \dots, \eps_{\jump}^{(\ell)})  \operatorname{good}}\mb{P}\Big( \sum_{j = 1}^V \eps_{\jump,j}^{(u)} \mbf{a}_{\jump,j} = x'_u, u \in [\ell]  \Big) = \big(1 + O(V^{-1/4})\big) 2^{V\ell} \prod_{i=1}^{\ell} g(x_i) \end{equation} uniformly for $|x'_u - x_u| \ll  2^{n - t}$, applying this bound with the choice $x'_u = x_u - s_u - h(\mbf{A},\eps^{(u)})$ then summing over all $s_u$, all $\eps_{i,j}^{(u)}$, $i < n + 1 - t, j \in [b_i]$, and revealing the remaining elements $(\mbf{a}_{i,j})_{i < n + 1 - t}$ of $(\mbf{a}_{i,j})_{(i,j) \in \Xi}$. (Here there is a slight abuse of notation in talking about $(\eps_{\jump}^{(1)}, \dots ,\eps_{\jump}^{(\ell)})$ being good, a notion which was originally defined with reference to the full vectors $\eps^{(u)}$ but in fact only depends on the $\eps_{\jump}^{(u)}$.)

Now we observe that 
\begin{equation}\label{g-lipschitz-claim} \prod_{i = 1}^{\ell} g(x_i) = (1 + O(V^{-1/4})) \prod_{i = 1}^{\ell} g(x'_i).\end{equation}
This follows fairly simply from the analytic properties of $g( \cdot)$. Indeed, using $|x'_i - x_i| \ll 2^{n - t}$ and that $\sigma \asymp V^{1/2} 2^{n - t}$, we have $|\sigma^{-1}(x_i - \mu) - \sigma^{-1} (x'_i -\mu)| \ll V^{- 1/2}$, and so from the uniformly Lipschitz nature of $e^{-x^2/2}$ we have $|g(x_i) - g(x'_i)| \ll V^{- 1/2} \sigma^{-1}$. Therefore $|\prod_{i = 1}^{\ell} g(x_i) - \prod_{i = 1}^{\ell} g(x'_i)| \ll 2^{\ell} \sigma^{-\ell} V^{- 1/2} \ll V^{-1/3} \sigma^{-\ell}$. Finally, using that $\sigma^{-1}|x'_i - \mu| \le \sigma^{-1}|x_i - \mu| + O(V^{- 1/2}) \le 2(\log V)^{1/4}$ we have $g(x'_i) \ge (2\pi)^{-1/2}\sigma^{-1} e^{-2(\log V)^{1/2}}$, and so $V^{-1/3} \sigma^{-\ell} \ll V^{-1/4} \prod_{i = 1}^{\ell} g(x'_i)$. This completes the proof of \cref{g-lipschitz-claim}.
Thus \cref{counting-cor-3} follows from 
\begin{equation}\label{counting-cor-4} \sum_{(\eps_{\jump}^{(1)}, \dots, \eps_{\jump}^{(\ell)}) \operatorname{good}}\mb{P}\Big( \sum_{j = 1}^V \eps_{\jump,j}^{(u)} \mbf{a}_{\jump,j} = x'_u, u \in [\ell]  \Big) = (1 + O(V^{-1/4}))  2^{V\ell} \prod_{i = 1}^{\ell} g(x'_i), \end{equation} 
where we may assume $\sigma^{-1} |x'_i - \mu| \le 2(\log V)^{1/4}$.
We phrase this task probabilistically, thus \cref{counting-cor-4} is the statement that 
\begin{equation}\label{counting-cor-4a}
\mb{P}\Big( \sum_{j = 1}^V \boldeps_{\jump,j}^{(u)} \mbf{a}_{\jump,j} = x'_u, u \in [\ell]  \Big) =  \big(1 + O(V^{-1/4})\big) \prod_{i = 1}^{\ell} g(x'_i),
\end{equation}
where the tuple $(\boldeps_{\jump}^{(1)},\cdots,\boldeps_{\jump}^{(\ell)})$ is sampled uniformly from all good tuples, and the probability is taken over the random choices of the $\boldeps$ and the $\mbf{a}$. Here, we have noted that by our earlier estimate \cref{good-prob} on the number of bad tuples, and the parameter hierarchy \cref{param-heir}, the sample space of $\boldeps$ tuples has size $\big( 1 - O(e^{-\Omega( 2^{-\ell} V)})\big)^{\ell} 2^{V\ell} = (1 + O(V^{-1/2})) 2^{V\ell}$.

As before we have $g(x'_i) \ge (2\pi)^{-1/2}\sigma^{-1} e^{-2(\log V)^{1/2}}$ and so in order to establish \cref{counting-cor-4a} it is enough to prove
\begin{equation}\label{counting-cor-4b}
\mb{P}\Big( \sum_{j = 1}^V \boldeps_{\jump,j}^{(u)} \mbf{a}_{\jump,j} = x'_u, u \in [\ell]  \Big) =   \prod_{i=1}^{\ell} g(x'_i) + O(V^{-1/4} \sigma^{-\ell}).
\end{equation}
We will prove this for \emph{all} $(x'_u)_{u \in [\ell]} \in \mbf{R}^{\ell}$. For notational brevity, until the end of the proof (of \cref{lem:Poisson-weights}) we drop the jump subscripts; thus we write simply $\boldeps_j := \boldeps_{\jump, j}$, $\mbf{a}_{j} := \mbf{a}_{\jump, j}$. Define the random $\mbf{R}^{\ell}$-valued random variable
\begin{equation}\label{x-rv-def} X := \frac{1}{\sigma}\big( \sum_{j = 1}^V \boldeps_{j}^{(u)} \mbf{a}_{j} - \mu\big)_{u \in [\ell]}.\end{equation}  Let 
\[ \gamma(x) :=  \sigma^{-\ell} (2\pi)^{-\ell/2} e^{-\Vert x \Vert_2^2/2}.\] Then (dropping the dashes on the $x'_i$) the desired bound \cref{counting-cor-4b} can be rephrased as the following local limit bound:
\begin{equation}\label{counting-cor-5} \sup_{x \in ((\mbf{Z} - \mu)/\sigma)^{\ell}}\big| \mb{P}(X = x) - \gamma(x) \big| \ll V^{-1/4} \sigma^{-\ell}.\end{equation} As usual with statements of local limit type, we approach this task with a Fourier argument.
By orthogonality of characters we have that 
\begin{equation*} \mb{P}(X = x) =  (2\pi \sigma)^{-\ell}\int_{[-\pi\sigma,\pi \sigma]^{\ell}}e^{-i\theta \cdot x}\mb{E} (e^{i\theta \cdot X})\, \mathrm{d}\theta.\end{equation*}
Also, by direct calculation we have 
\begin{equation*} \gamma(x)  = (2\pi \sigma)^{-\ell}\int_{\mathbf{R}^{\ell}} e^{-i \theta \cdot x} e^{-\Vert \theta\Vert_2^2/2} \, \mathrm{d}\theta.
\end{equation*}
Therefore the required bound \cref{counting-cor-5} follows from 
\begin{equation}\label{counting-cor-6} \Big| \int_{\mathbf{R}^{\ell}} e^{-i \theta \cdot x} e^{-\Vert \theta\Vert_2^2/2} \, \mathrm{d}\theta  - \int_{[-\pi\sigma,\pi \sigma]^{\ell}}e^{-i\theta \cdot x}\mb{E} (e^{i\theta \cdot X})\, \mathrm{d}\theta \Big| \ll V^{-1/4}.\end{equation}
We prove this by establishing the following three estimates:
\begin{equation}\label{est-1} \sup_{\theta \in [-\log V, \log V]^{\ell}} \big|\mb{E}(e^{i \theta \cdot X}) - e^{-\Vert \theta \Vert_2^2/2} \big|  \ll V^{-2/7}  ;\end{equation}
\begin{equation}\label{est-2} \int_{\mbf{R}^{\ell} \setminus [-\log V, \log V]^{\ell} } e^{-\Vert \theta \Vert_2^2/2} \, \mathrm{d}\theta  \ll V^{-10};\end{equation} and
\begin{equation}\label{est-3} \int_{[-\pi \sigma, \pi \sigma]^{\ell} \setminus [-\log V, \log V]^{\ell} } |\mb{E}(e^{i \theta \cdot X})|  \, \mathrm{d}\theta  \ll V^{-10}.\end{equation}
These three estimates are easily seen to imply \cref{counting-cor-6}, using here that $(2\log V)^{\ell} = V^{o(1)}$ to handle the contribution from $\Vert \theta \Vert_{\infty} \le \log V$.\vspace*{8pt}

\emph{Proof of \cref{est-1}.} This is essentially a standard computation associated with proving a multivariate central limit theorem. Since the proportion $(\eps_{\jump}^{(1)}, \dots ,\eps_{\jump}^{(\ell)})$ of tuples in $(\{0,1\}^V)^{\ell}$ that are bad is exponentially small in $V$ (see \cref{good-prob}), in this range we may for the purposes of proving \cref{est-1} drop the `good' condition; that is to say it suffices to prove \cref{est-1} with $X$ replaced by the variable $\tilde X$, defined as in \cref{x-rv-def} but with the $\boldeps_{j}^{(u)}$ sampled independently at random from $\{0,1\}$. Let $\theta = (\theta_1,\dots, \theta_{\ell})$ and for notational convenience in the rest of the proof write $e(\xi) := e^{i\xi}$. Assume (as in the estimate \cref{est-1} we are trying to prove) that $|\theta_u| \le \log V$ for all $u$. 

From this definition and \cref{mu-sig-def} we have
\begin{align}\nonumber \mb{E}(e^{i \tilde X \cdot \theta}) &  = \mb{E}_{\boldeps,\mbf{a}} e\big( \frac{1}{\sigma} \sum_{u = 1}^{\ell} \sum_{j = 1}^V (\boldeps_j^{(u)} \mbf{a}_j - \tfrac{1}{2}\mu_j)\theta_u \big) \\ &  = \mb{E}_{\mbf{a}}\prod_{j=1}^{V}\prod_{u = 1}^{\ell}e\big( \frac{1}{2\sigma}  (\mbf{a}_j -\mu_j) \theta_u \big)\mb{E}_{\boldeps} \prod_{j=1}^{V}\prod_{u = 1}^{\ell}e\big( \frac{1}{\sigma}  (\boldeps_j^{(u)}-\tfrac{1}{2}) \mbf{a}_j \theta_u \big). \label{first-328}\end{align}

To estimate the inner average over the $\boldeps_j^{(u)}$ we use that 
\begin{equation}\label{taylor-eit} \frac{1}{2}\big(e(\xi)+e(-\xi)\big) = e^{-\xi^2/2} (1 + O(|\xi|^3))\end{equation} for $|\xi| \le 1$. We apply this with $\xi = \mbf{a}_j \theta_u/2\sigma$. Note (see \cref{basic-a-bds}) that $\mbf{a}_j \asymp 2^{n - t}$, $|\theta_u| \le \log V$ and $\sigma \asymp V^{1/2} 2^{n - t}$, so $|\xi| \ll V^{-1/2}\log V$. Therefore substituting \cref{taylor-eit} into \cref{first-328} gives
\[ \mb{E}_{\mbf{a}}(e^{i \tilde X \cdot \theta}) =   \mb{E}(1 + \eta(\mbf{a}))\prod_{j=1}^{V}\prod_{u = 1}^{\ell}  e\big( \frac{1}{2\sigma}  (\mbf{a}_j -\mu_j) \theta_u \big) \cdot \exp\Big(-\sum_{j=1}^{V}\sum_{u=1}^{\ell}\frac{\mbf{a}_j^2\theta_u^2}{8\sigma^2}\Big) \] where $\Vert \eta \Vert_{\infty} \ll \ell V ( V^{-1/2} \log V)^3 \ll V^{-1/2 + o(1)}$. 
Now from \cref{mu-sig-def,basic-a-bds} we have $\mbf{a}_j^2 = (1 + O(V^{-1/3})) \mu_j^2 = (1 + O(V^{-1/3}) )\mb{E} \mbf{a}_j^2$ and so $\sum_{j = 1}^V \mbf{a}_j^2 = 4 \sigma^2 (1 + O(V^{-1/3}))$; since $\sum_{u = 1}^{\ell} \theta_u^2 = \Vert \theta \Vert_2^2$, this gives
\[ \mb{E}(e^{i \tilde X \cdot \theta}) =  e^{-\Vert \theta \Vert_2^2/2}  \mb{E}_{\mbf{a}}(1 + \eta'(\mbf{a}))\prod_{j=1}^{V}\prod_{u = 1}^{\ell}  e\big( \frac{1}{2\sigma}  (\mbf{a}_j -\mu_j) \theta_u \big)  \] where $\Vert \eta' \Vert_{\infty} \ll \ell V^{-1/3}\log V \ll V^{-2/7}$. Consequently
\[ \mb{E}(e^{i \tilde X \cdot \theta}) =  O(V^{-2/7}) + e^{-\Vert \theta \Vert_2^2/2}  \mb{E}_{\mbf{a}}\prod_{j=1}^{V}\prod_{u = 1}^{\ell}  e\big( \frac{1}{2\sigma}  (\mbf{a}_j -\mu_j) \theta_u \big) . \] 
Hence, in order to prove the desired result \cref{est-1} it suffices to prove that 
\begin{equation}\label{328-suff} \mb{E}\prod_{j=1}^{V}\prod_{u = 1}^{\ell}e\big( \frac{1}{2\sigma}  (\mbf{a}_j -\mu_j) \theta_u \big) = 1 + O(V^{-1/2}).\end{equation}
Note that the LHS here is
\[ \prod_{j = 1}^V \mb{E}_{\mbf{a}_j} e\big( \lambda(\mbf{a}_j - \mu_j)\big) , \] where $\lambda := \frac{1}{2\sigma} \sum_{u = 1}^{\ell} \theta_u$. By \cref{basic-a-bds} and the fact that $|\theta_u| \le \log V$ we have $|\lambda (\mbf{a}_j - \mu_j)| \ll (\ell \log V) V^{-5/6} \ll V^{-3/4} \lll 1$ and so  
\[e\big( \lambda(\mbf{a}_j - \mu_j)\big) = 1 + i\lambda (\mbf{a}_j - \mu_j) + O(\lambda^2 (\mbf{a}_j - \mu_j)^2); \] taking expectations gives
\[\mb{E}_{\mbf{a}_j} e\big( \lambda(\mbf{a}_j - \mu_j)\big) = 1  + O(\lambda^2 (\mbf{a}_j - \mu_j)^2) = 1 + O(V^{-3/2}), \] and \cref{328-suff} follows by taking the product over $j = 1,\dots, V$.\vspace*{8pt}

\emph{Proof of \cref{est-2}.} This is a straightforward exercise using $\int_{t > r} e^{-t^2/2} dt \le r^{-1} \int_{t > r} t e^{-t^2/2} dt = r^{-1} e^{-r^2/2}$.\vspace*{8pt}

\emph{Proof of \cref{est-3}.} We split into dyadic ranges $K \le \Vert \theta \Vert_{\infty} < 2K$, $K \in \{ 2^i \log V : i = 0,1,\dots\}$. The region of integration in \cref{est-3} corresponding to such a range has volume $\le (4K)^{\ell}$ and so it suffices to show the pointwise bound
\begin{equation} \label{eq331} |\mb{E}(e^{i X \cdot \theta})| \ll (8 K)^{-\ell} V^{-10}\end{equation}
 uniformly for $K \le \Vert \theta \Vert_{\infty} < 2K$. 

We first consider the smaller ranges $\log V \le K \ll V^{1/2}/\log V$ (say). As in the analysis of \cref{est-1} it suffices to prove \cref{eq331} with $X$ replaced by $\tilde X$, defined as in \cref{x-rv-def} but without the `good' condition, that is to say with the $\boldeps_j^{(u)}$ sampled independently from $\{0,1\}$. From the definition \cref{x-rv-def} and the triangle inequality we have 
\[ \big|\mb{E}(e^{i \tilde X \cdot \theta})\big|  \le  \mb{E}_{\mbf{a}}\Big| \mb{E}_{\boldeps} e\big( \frac{1}{\sigma} \sum_{u = 1}^{\ell} \sum_{j = 1}^V \boldeps_j^{(u)} \mbf{a}_j \theta_u \big) \Big|  = \mb{E}_{\mbf{a}} \prod_{u = 1}^{\ell} \prod_{j = 1}^V \Big| \frac{1 + e(\theta_u \mbf{a}_j/\sigma)}{2} \Big|.
 \]
Note that, since $\mbf{a}_j \asymp 2^{n - t}$, $\sigma \asymp V^{1/2} 2^{n - t}$ and $|\theta_u| = o(V^{1/2})$, the arguments of the exponentials here are $o(1)$ and we may use the inequality $\frac{1}{2}|1 + e(\xi)| \le e^{-|\xi|^2/8}$ to give $|\mb{E}(e^{i \tilde X \cdot \theta})| \le e^{-\Omega(K^2)}$. This is bounded as in \cref{eq331} for $K \ge \log V$, using here that $\ell$ is much smaller than $V$ by the parameter hierarchy.

Now we turn to the larger values $K \ge V^{1/2}/\log V$. Here, we apply the triangle inequality the other way around, obtaining
\begin{equation*} \big|\mb{E}(e^{i X \cdot \theta})\big|  \le \mb{E}_{(\boldeps^{(1)},\dots, \boldeps^{(\ell)}) \operatorname{good}} \Big| \mb{E}_{\mbf{a}} e\big( \frac{1}{\sigma} \sum_{u = 1}^{\ell} \sum_{j = 1}^V \boldeps_j^{(u)} \mbf{a}_j \theta_u \big) \Big|.\end{equation*}
For $j \in [V]$ denote by $\phi_j$ the characteristic function $\phi_j(\xi) := \mb{E} e(\xi \mbf{a}_j)$. Then the preceding may be rewritten as 
\[ \big|\mb{E}(e^{i X \cdot \theta})\big| \le \mb{E}_{(\boldeps^{(1)},\dots, \boldeps^{(\ell)}) \operatorname{good}} \prod_{j = 1}^V \Big| \phi_j\big(\frac{1}{\sigma}\sum_{u = 1}^{\ell} \boldeps_j^{(u)} \theta_u\big)\Big|.\]
Consequently, 
\begin{equation}\label{exp-avg-bd} \big|\mb{E}(e^{i X \cdot \theta})\big| \le \prod_{i= 1}^{\ell} \sup_j\big|\phi_j(\theta_i/\sigma)\big|^{4^{-\ell} V} \le \sup_j \big|\phi_j(\Vert \theta \Vert_{\infty}/\sigma)\big|^{4^{-\ell} V}:\end{equation} 
this follows by picking out just those indices $j$ for which $(\boldeps_j^{(u)})_{u \in [\ell]} = e_i \in \mbf{Q}^{\ell}$, for $i = 1,\dots, \ell$; by the definition of `good' there are at least $4^{-\ell} V$ such $j$ for each value of $i$. We may use the trivial bound $|\phi_j(\theta_u/\sigma)| \le 1$ everywhere else.

To obtain estimates for the quantities $\phi_j(\Vert \theta \Vert_{\infty}/\sigma)$, first recall from \cref{a-jump-dist} how $\mbf{a}_j$ is distributed. With the notation used there (recalling that $\mbf{a}_j = \mbf{a}_{\jump, j} = \mbf{a}_{n + 1 - t, j}$), it follows that 
\[ \phi_j(\xi) = \frac{\alpha_0}{x_0} e(\xi x_0) + \frac{\alpha_1}{x_1} e(\xi x_1) + \big(1 - \frac{\alpha_0}{x_0} - \frac{\alpha_1}{x_1}\big) \mb{E} (e^{i \xi X}),\] where $X$ is the random variable taking values in $(x_0, x_1) \subset \Z$, with $\mb{P} (X = x) \propto 1/x$. In our context $\alpha_0/x_0$, $\alpha_1/x_1$ are minuscule and so we certainly have
\begin{equation*} |\phi_j(\xi)| \le \tfrac{1}{2}\big(1 + |\mb{E}(e^{i \xi X})|\big).\end{equation*}
We can estimate $|\mb{E}(e^{i \xi X})|$ using \cref{a1-lem}, which is designed for this purpose. We take $N = x_0$, $R = x_1 - x_0 + 1$ and $\xi = \Vert \theta \Vert_{\infty}/\sigma$. One may easily check that $N \asymp 2^{n - t}$ and $R \asymp V^{-1/3} 2^{n - t}$. Since $\sigma \asymp V^{1/2} 2^{n - t}$, we have $\xi R \asymp \Vert \theta \Vert_{\infty} V^{-5/6} \asymp K V^{-5/6}$.
\cref{a1-lem} then gives
\begin{equation} \label{eq332}  \big|\phi_j(\Vert \theta \Vert_{\infty}/\sigma)\big| \le \left\{ \begin{array}{ll} 1 - c K^2 V^{-5/3}  &  K \le CV^{5/6}; \\ \frac{3}{4} & K \ge CV^{5/6}.\end{array} \right.\end{equation}
Here, $c, C$ are absolute constants with $0 < c < 1 < C$ (independent of $j$, and $c$ may be smaller than the constant in \cref{a1-lem}).

For $V^{1/2}/\log V \le K \le CV^{5/6}$ we use the first bound in \cref{eq332}, obtaining from \cref{exp-avg-bd} that $|\mb{E}(e^{i X \cdot \theta})|\ll (1 - c \frac{V^{-2/3}}{\log^2 V})^{4^{-\ell} V} \ll e^{-V^{1/4}}$ which is certainly a bound of strength \cref{eq331}. For $K \ge CV^{5/6}$ we instead use the second bound in \cref{eq332} getting $|\mb{E}(e^{i X \cdot \theta})| \ll (\frac{3}{4})^{4^{-\ell} V}$, which is again (massively) of strength \cref{eq331}.

We have now completed the proof of the three statements \cref{est-1,est-2,est-3}, and hence \cref{counting-cor-6}. Tracing back through the successive reductions \cref{counting-cor-5}, \cref{counting-cor-4}, \cref{counting-cor-3} and \cref{counting-cor-2}, we see that the contribution to \cref{counting-cor} from the good tuples $(\eps^{(1)},\dots, \eps^{(\ell)})$ provides the main term in \cref{lem31-main-weighted}, and the proof of \cref{lem:Poisson-weights} is (at last) complete. \end{proof}

As previously stated, we will also require the following technical estimate, which will be used to deduce the unweighted \cref{lem:Poisson} from the version \cref{lem:Poisson-weights} with weights.

\begin{lemma}\label{lem:Poisson-upper} Let the notation and setup be as in \cref{lem:Poisson} and, as in \cref{lem:Poisson-weights}, let $R_{\mbf{A}}(x)$ be the number of representations of $x$ in $S + \Sigma(\mbf{A}^{> n'}_{\med})$. Then we have 
\[\mb{E} R_{\mbf{A}}(x_1)  R_{\mbf{A}}(x_2)\cdots R_{\mbf{A}}(x_{\ell})(R_{\mbf{A}}(x_\ell)-1)\le V^{O(\ell)} \Bigg(\frac{2^{\sum_{i =1}^{n + 1 - t} b_i}}{\sigma}\Bigg)^{\ell+1}.\] 
\end{lemma}
\begin{proof}
The proof goes along similar lines to (part of) the proof of \cref{lem:Poisson-weights}. In particular, define $\Xi$ in the same way. We have the following modification of \cref{counting-cor}:
\begin{align}\nonumber \mb{E} R_{\mbf{A}}(x_1) & \cdots R_{\mbf{A}}(x_{\ell})(R_{\mbf{A}}(x_{\ell}) - 1) \\ & = \sum_{\substack{\eps^{(1)}, \dots, \eps^{(\ell+1)} \in \{0,1\}^{\Xi} \\ \eps^{(\ell)} \ne \eps^{(\ell+1)}}}\sum_{s_1,\dots, s_{\ell+1} \in S} \mb{P}\Big(\sum_{(i,j) \in \Xi} \eps^{(u)}_{i,j} \mbf{a}_{i,j} = x_u - s_u, u \in  [\ell+1]\Big),\label{counting-cor-mod}\end{align} where $x_{\ell+1} := x_{\ell}$. This follows by observing that $R_{\mbf{A}}(x_{\ell}) (R_{\mbf{A}}(x_{\ell}) - 1)$ counts pairs of distinct representations of $x_{\ell}$ in $S + \Sigma(\mbf{A}^{> n'}_{\med})$.

We now proceed completely analogously to the initial steps in the proof of \cref{lem:Poisson-weights}, simply changing $\ell$ to $\ell + 1$ in many places. In particular the first main task is to show 
\begin{equation}\label{main-first-to-show-rpt} \sum_{\substack{(r; q_1,\dots, q_r) \\ r < \ell+1 \mbox{\,\scriptsize or} \\ r = \ell+1 \mbox{\, \scriptsize and \,} q_{\ell+1} > V}} E_{(r; q_1, \dots, q_r)} \ll 2^{-\Omega(V)} 2^{(\ell+1) \big(\sum_{i =1}^{n + 1 - t} b_i - (n - t) \big)}.\end{equation}
which is \cref{main-first-to-show} with $\ell$ replaced by $\ell + 1$. Key tools for doing this will be \cref{bd-fixed-rank-profile,iq-bd}, which hold verbatim.

The analysis of Steps 1 and 3 proceeds essentially identically, changing $\ell$ to $\ell +1$ a few times.

The initial part of Step 2 requires a very small modification, in that if $r < \ell + 1$ and there are only $r$ distinct $\eps^{(u)}$, then we must have $\eps^{(u)} = \eps^{(u')}$ for some pair $(u,u')$ with $u < u'$, and moreover we cannot have $(u,u') = (\ell, \ell + 1)$ due to the stipulation that $\eps^{(\ell)} \ne \eps^{(\ell + 1)}$ in the range of summation in \cref{counting-cor-mod}. The rest of that argument proceeds as before, once again switching some instances of $\ell$ to $\ell + 1$. 

This concludes the proof of \cref{main-first-to-show-rpt}. Recalling that $\sigma \asymp V^{1/2} 2^{n - t}$, to conclude the proof of \cref{lem:Poisson-upper} we need only show that 
\begin{equation*}
 \sum_{q_1,\dots, q_{\ell +1} \le V} E_{(\ell+1; q_1,\dots, q_{\ell+1})} \ll V^{O(\ell)} 2^{(\ell+1) \big(\sum_{i =1}^{n + 1 - t} b_i - (n - t) \big)}.
\end{equation*}
Since the number of possibilities for $q_1,\dots, q_{\ell + 1}$ is $V^{O(\ell)}$, it suffices to prove the same bound for each $E_{(\ell+1; q_1,\dots, q_{\ell+1})}$ individually.
For this we use \cref{bd-fixed-rank-profile}, where we can take all $i_d$ to be equal to $t$ since $q_1,\dots, q_{\ell+1} \le V$. This gives exactly the desired bound (in fact with an extra factor $2^{-\sum_{d = 1}^{\ell + 1} q_d} \le 1$). This concludes the proof of \cref{lem:Poisson-upper}.\end{proof}

The final task of this section is to establish the key result \cref{lem:Poisson}.

\begin{proof}[Proof of \cref{lem:Poisson}]
The upper bound is immediate from \cref{lem:Poisson-weights} since $1_{x \in S + \Sigma(\mbf{A}^{> n'}_{\med})} \le R_{\mbf{A}}(x)$ pointwise. 
For the lower bound, we use that $1_{x \in S + \Sigma(\mbf{A}^{> n'}_{\med})} \ge R_{\mbf{A}}(x)  \mathbf{1}_{R_{\mbf{A}}(x)\in \{0,1\}}$, and so 
\[ \mb{P}\big(x_1, \dots, x_{\ell} \in S + \Sigma(\mbf{A}^{> n'}_{\med})\big) \ge \mb{E} R_{\mbf{A}}(x_1) \cdots R_{\mbf{A}}(x_{\ell}) - \mb{E} R_{\mbf{A}}(x_1) \cdots R_{\mbf{A}}(x_{\ell}) \big(1 - \prod_{i = 1}^{\ell} \mathbf{1}_{R_{\mbf{A}}(x_i) \in \{0,1\}}\big).\]
However
\begin{align*}
 \mb{E} R_{\mbf{A}}(x_1) \cdots R_{\mbf{A}}(x_{\ell}) \big(1 - \prod_{i = 1}^{\ell} \mathbf{1}_{R_{\mbf{A}}(x_i) \in \{0,1\}}\big) &\le  \mb{E} R_\mbf{A}(x_1) \cdots R_{\mbf{A}}(x_{\ell}) \sum_{i = 1}^{\ell} \mathbf{1}_{R_{\mbf{A}}(x_i) > 1} \\
&\le \sum_{i=1}^{\ell}\mb{E} R_{\mbf{A}}(x_1) \cdots R_{\mbf{A}}(x_{\ell}) (R_{\mbf{A}}(x_i)-1).
\end{align*}
Using \cref{lem:Poisson-upper}, and noting that $g(x_i) \ge V^{-o(1)}/\sigma$ due to the hypotheses on the $x_i$, comparison with \cref{lem:Poisson-weights} shows that it is enough to prove that 
\begin{equation*} V^{O(\ell)} \frac{2^{\sum_{i=1}^{n + 1 - t} b_i}}{\sigma} \ll V^{-1/4}.\end{equation*}
This follows using $\sigma \asymp V^{1/2} 2^{n - t}$ and the upper bound in \cref{bnt}, together with the parameter hierarchy (in particular, that $t \ge t_{\min} \ggg V$).
\end{proof}

\section{Inclusion--exclusion}\label{sec4}

In this section we carry out the inclusion-exclusion analysis hinted at in the introduction to \cref{sec3}, using \cref{lem:Poisson} as our main tool. The following is the main result of the section.

As in the previous section, we have a main truncation parameter $M$, key further thresholds $R(M)$, $T(M)$ and $L(M)$ as described in \cref{truncation-threshold-defs}, and technical auxiliary thresholds $\ell_*(M)$, $V(M)$, $t_{\min}(M)$, $t_{\max}(M)$ as in the hierarchy \cref{param-heir}.

The reader may also wish to recall the definitions of the random variables $\varrho^*_{\mbf{u} \mid \beta}(\ell)$, $\tau^*_{\mbf{x} \mid \beta'}(\ell)$, which may be found in \cref{xbeta-ybeta-defs,y-beta-def}. Here is the main result of the section.

\begin{proposition}\label{sec6-main}
Fix upper and lower walks $\beta, \beta'$ which are $R$-bounded and $T$-positive up to lengths $\lceil n/2\rceil$, $\lfloor n/2\rfloor$, and which have $\min_{1 \le i \le \lceil n/2\rceil} \beta(i), \min_{1 \le i \le \lfloor n/2\rfloor} \beta'(i) \ge -M$. Write $D  = D(\beta, \beta') := \beta(\lceil n/2\rceil) - \beta'(\lfloor n/2\rfloor)$ and suppose that $|D| \le M/10$. Suppose that $\beta$ has a $V(M)$-jump step at some $t \in [t_{\min}(M), t_{\max}(M)]$. Then
\[ \mb{P}\big(k \in \Sigma(\mbf{A} \mid \beta, \beta')\big) = 1 - \mb{E} \exp \big( -2^{D - \xi}  \varrho^*_{\mbf{u} \mid \beta}(L(M))\tau^*_{\mbf{x} \mid \beta'}(L(M))\big) + O(\ell_*(M)^{-1}).\] 
\end{proposition}
\begin{remark}
Note that the final expression is independent of the location of the jump step $t$.    
\end{remark}
To avoid excessively cumbersome notation, we drop the explicit mention of $M$ from the parameters $T, t_{\min}, t_{\max}, V, L, \ell_*$, as we did in the last section. For the rest of the section we will assume that $\beta, \beta'$ satisfy the hypotheses of \cref{sec6-main} without further comment; this applies in particular to the subsidiary \cref{key-61-tech}.

For the initial analysis notation will be as in \cref{sec3}. Specifically, the reader should recall \cref{bi-defs-rept}, the parameter hierarchy \cref{param-heir} and the fact that we are conditioning to $\boldbeta = \beta, \boldbeta' = \beta'$ and also to a fixed choice of the jump distribution $\mbf{f} = f$. In particular, from this point onwards in the section we write $\mbf{A}$ as shorthand for $(\mbf{A} \mid \boldbeta = \beta, \boldbeta' = \beta', \mathbf{f} = f)$.

Associated to this data are the parameters $\mu, \sigma^2$ defined in \cref{mu-sig-def}, and the Gaussian density $g(x)$ with mean $\mu$ and variance $\sigma^2$ as in \cref{gaussian-def}.

As in \cref{as-def} we will divide $\mbf{A}$ as $\mbf{A}_{\sml} \cup \mbf{A}_{\med} \cup  \mbf{A}_{\lrg}$. For much of the section we will additionally condition to fixed choices of $A_{\sml} = (a_{i,j})_{1 \le i \le L;j \in [b_i]}$ and $A_{\lrg} = (a_{i,j})_{n + 1 - t < i \le n;j \in [b_i]}$ and sample only $\mbf{A}_{\med}$ at random, with the randomly selected object being denoted by bold font as usual. Since \cref{dyadic-aij} holds deterministically, we always assume
\begin{equation*} a_{i,j} \asymp 2^{i}.\end{equation*}

\subsection{Consequences of \cref{sec3}} 
In the proof of \cref{sec6-main} we will make crucial use of the main result of the last section, \cref{lem:Poisson}. We record the consequences we need here.  For notational clarity in what follows we will write $E \approx E'$ as shorthand for $E' = (1 + O(V^{-\Omega(1)}))E$. Recall that $\Omega(1)$ denotes a small positive constant, which may change from line to line.

First (in \cref{lem:Poisson}) take $n' = L$ and $S := \Sigma(A_{\sml})$. As remarked after the statement of \cref{lem:Poisson}, in this case we have $\mbf{A}_{\med}^{> n'} = \mbf{A}_{\med}$. The condition $\max(S) \le 2^{n - L/2}$ is satisfied by a large margin. Write $\tau(A_{\sml})$ for the parameter $\tau$ in this application of \cref{lem:Poisson}, that is to say
\begin{equation}\label{tau-small} \tau(A_{\sml}) = 2^{-\sum_{i =1}^L b_i} |\Sigma(A_{\sml})|.\end{equation}
Here the notation $\Sigma(S)$ is as defined in \cref{sig-tau-def}.
Assume in what follows that $\tau(A_{\sml}) \ge V^{-1}$. Suppose that $\ell \le \ell_*$ and that $x_1,\dots, x_{\ell}$ satisfy \cref{xi-assump} and the separation condition \cref{lem:Poisson:sep}. Then all the hypotheses of \cref{lem:Poisson} hold and the conclusion \cref{lem31-main} is that
\begin{equation}\label{poss-asmall}
\mb{P}\big(x_1,\dots, x_{\ell} \in \Sigma(\mbf{A}_{\med}) + \Sigma(A_{\sml})\big) \approx \big(\tau(A_{\sml}) 2^{\sum_{i=1}^{n + 1 - t} b_i} \big)^{\ell} \prod_{i = 1}^{\ell} g(x_i)
\end{equation}
(where $g$ is a Gaussian density with mean $\mu$ and variance $\sigma^2$ as in \cref{gaussian-def}.)
Comparing this with the case $\ell = 1$, we immediately obtain the key almost-independence statement
\begin{equation}\label{almost-indep-statement}
\mb{P}\big(x_1,\dots, x_{\ell} \in \Sigma(\mbf{A}_{\med}) + \Sigma(A_{\sml})\big) \approx \prod_{i = 1}^{\ell} \mb{P}\big(x_i \in \Sigma(\mbf{A}_{\med}) + \Sigma(A_{\sml})\big),
\end{equation} where here we used the fact that $(1 + V^{-\Omega(1)})^{\ell} = 1 + O(V^{-\Omega(1)})$ due to the hierarchy of parameters. This will be the first key consequence of \cref{lem:Poisson} that we will use later.

The second important consequence will be an alternative expression for the probabilities on the right in \cref{almost-indep-statement}. For this, note first that the case $\ell = 1$ of \cref{poss-asmall} gives that
\begin{equation} \label{prescore-2} \mb{P}\big(x \in \Sigma(\mbf{A}_{\med}) + \Sigma(A_{\sml})\big) \approx \tau(A_{\sml}) 2^{\sum_{i =1}^{n + 1 - t} b_i} g(x),\end{equation} still under the assumption that $\tau(A_{\sml}) \ge V^{-1}$ and also that 
\begin{equation}\label{x-assump} |x - \mu| \le (\log V)^{1/4}\sigma.\end{equation} (The separation condition \cref{lem:Poisson:sep} is redundant.) Leaving this aside temporarily, we now use \cref{lem:Poisson} differently in order to obtain an alternative expression for the right-hand side of \cref{prescore-2}.
Now take $n' = n - L$ and $S$ to be any subset of $[2^{- L}k]$ of size $2^{\sum_{i =1}^{n'} b_i}$. That such a set exists follows from \cref{bnL} (recalling that $k \ge 2^n$). The condition $\max(S) \le 2^{n - L/2}$ is again satisfied quite comfortably (noting that $k \asymp 2^n$), and in this case $\tau = 1$. Take $\ell = 1$ in \cref{lem31-main}, obtaining (for any $S$) 
\begin{equation*}  \mb{P}(x \in S + \Sigma(\mbf{A}_{\med}^{> n - L})) \approx 2^{\sum_{i=1}^{n + 1 - t} b_i} g(x).\end{equation*} Comparing with \cref{prescore-2} gives
\begin{equation}\label{prescore-3c} 
 \mb{P}\big(x \in \Sigma(\mbf{A}_{\med}) + \Sigma(A_{\sml})\big) \approx \tau(A_{\sml}) \mb{P}(x \in S + \Sigma(\mbf{A}_{\med}^{> n - L})) .
\end{equation}

We claim that 
\begin{equation}\label{12.7claim}
 \mb{P}(x \in S + \Sigma(\mbf{A}_{\med}^{> n - L})) \approx \mb{E} R_{\mbf{A}}(x).
\end{equation}
where $R_{\mbf{A}}$ denotes the number of representations of $x$ in $S + \Sigma(\mbf{A}_{\med}^{> n - L})$ (as defined just prior to \cref{lem:Poisson-weights}). First note that the LHS of \cref{12.7claim} is simply $\mb{E} 1_{R_{\mbf{A}}(x) \ge 1}$, so clearly $\mb{P}(x \in S + \Sigma(\mbf{A}_{\med}^{> n - L})) \le \mb{E} R_{\mbf{A}}(x)$. Obtaining a bound in the other direction is only slightly more involved. By employing the inequality $R - R(R - 1) \le 1_{R \ge 1}$, it suffices to show that 
\begin{equation}\label{sec12-suff}
\mb{E} R_{\mbf{A}}(x) (R_{\mbf{A}}(x) - 1) \ll V^{-\Omega(1)}\mb{E} R_{\mbf{A}}(x).
\end{equation}
To see that this is so, we note that \cref{lem:Poisson-weights} (together with the lower bound $\tau(A_{\sml}) \ge V^{-1}$, \cref{basic-a-bds,x-assump}) tells us that $\mb{E} R_{\mbf{A}}(x) \gg V^{-1}  2^{\sum_{i =1}^{n + 1 - t} b_i} g(x) \gg V^{-O(1)} 2^{\sum_{i =1}^{n + 1 - t} b_i - (n - t)}$, whilst \cref{lem:Poisson-upper} gives $\mb{E} R_{\mbf{A}}(x) (R_{\mbf{A}}(x) - 1) \ll V^{O(1)} 2^{2\big(\sum_{i=1}^{n + 1 - t} b_i - (n - t)\big)}$. The bound \cref{sec12-suff} then follows using \cref{bnt} and the parameter hierarchy, and this concludes the proof of the claim \cref{12.7claim}.

Combining \cref{prescore-3c,12.7claim} evidently gives
\[ \mb{P}\big(x \in \Sigma(\mbf{A}_{\med}) + \Sigma(A_{\sml})\big) \approx \tau(A_{\sml})\mb{E} R_{\mbf{A}}(x) .\] Expanding out the expectation term on the RHS we obtain
\begin{equation}\label{prescore-3d} 
 \mb{P}\big(x \in \Sigma(\mbf{A}_{\med}) + \Sigma(A_{\sml})\big) \approx \tau(A_{\sml}) \sum_{\eps'} \mb{P}\big(x \in S + \sum_{\substack{n - L < i \le n +1 - t\\ j \in [b_i]}} \eps'_{i,j} \mbf{a}_{i,j} \big) ,
\end{equation}
where $\eps' = (\eps'_{i,j})_{n - L < i \le n +1 - t, j \in [b_i]}$ ranges over all tuples with $\eps'_{i,j} \in \{0,1\}$, and the use of $\eps'$ as a dummy variable is to avoid conflict later.
Recall that $S$ is an arbitrary subset of $[2^{-L} k]$ of size $2^{\sum_{i =1}^{n - L} b_i}$. Noting that the LHS of \cref{prescore-3d} is independent of $S$, we average over all such $S$ and thereby obtain (assuming $\tau(A_{\sml}) \ge 1/V$ and \cref{x-assump}) that
\begin{align}\nonumber
\mb{P}\big(x \in \Sigma(\mbf{A}_{\med}) & + \Sigma(A_{\sml})\big) \approx \\ &  
\tau(A_{\sml}) 2^{\sum_{i \le n - L} b_i} (2^{-L} k)^{-1}\sum_{\eps'} \mb{P}(x \in [2^{-L} k] + \sum_{\substack{n - L < i \le n +1 - t\\ j \in [b_i]}}\eps'_{i,j} \mbf{a}_{i,j} ).\label{prescore-3e} \end{align} 
This is the second key result that we will require in the proof of \cref{sec6-main}.

\subsection{The Poisson paradigm} In this section we establish the key ingredient in the proof of \cref{sec6-main}, which we state as \cref{key-61-tech} below. We retain the notation in force so far, in particular that we are conditioning to fixed choices of $A_{\sml}, A_{\lrg}$. In what follows we write $\eps:= (\eps_{i,j})_{n +1 - t < i \le n;j \in [b_i]}$ and denote 
\begin{equation}\label{e-alarge-def} \mathcal{E} = \mathcal{E}(A_{\lrg}) := \big\{ \eps : \big| k - \sum_{n +1 - t < i \le n;j \in [b_i]} \eps_{i,j} a_{i,j} - \mu\big| \le (\log V)^{1/4} \sigma \big\} \end{equation} and
\[ \Sigma'(A_{\lrg}) := \{ \sum_{n + 1 - t < i \le n;j \in [b_i]} \eps_{i,j} a_{i,j} : \eps  \in \mathcal{E}(A_{\lrg})\}.\]
Here, $\mu, \sigma$ are defined in \cref{mu-sig-def}. In particular, it is important to note that the definition of $\mathcal{E}$ depends on the jump distribution vector $f$ that we conditioned to at the beginning of the last section, and not just on the walks $\beta,\beta'$.

\begin{lemma}\label{key-61-tech} Let $\beta, \beta'$ be as in the statement of \cref{sec6-main}. Let $A_{\sml} = (a_{i,j})_{1 \le i \le L;j \in [b_i]}$ and $A_{\lrg} = (a_{i,j})_{n + 1 - t < i \le n;j \in [b_i]}$ be multisets. Define $\mathcal{E} = \mathcal{E}(A_{\lrg})$ as in \cref{e-alarge-def}.
Suppose that $\tau(A_{\sml}) \ge 1/V$ and that $A_{\lrg}$ satisfies the following separation condition: as $\eps$ ranges over $\{0,1\}^{\sum_{n +1 - t < i \le n} b_i}$, the numbers $\sum_{n +1 - t < i \le n ;j \in [b_i]} \eps_{i,j} a_{i,j}$ are $2^{n - L/2}$-separated. Then, writing $\mbf{A}_{\med} := (\mbf{a}_{i,j})_{L+1 \le i \le  n + 1 - t; j \in [b_i]}$ \textup{(}chosen at random\textup{)}, we have
\[ \mb{P}(k \in \Sigma'(A_{\lrg}) + \Sigma(\mbf{A}_{\med}) + \Sigma(A_{\sml})) = 1 - e^{-S(A_{\sml}, A_{\lrg})} + O(\ell_*^{-1}),\] where 
\begin{equation}\label{score-def-1} S(A_{\sml}, A_{\lrg}) := 
\sum_{\eps \in \mathcal{E}} \mb{P}(X_{\eps}), \end{equation} with
$X_{\eps}$ the event that 
\begin{equation}\label{score-def-2}  k - \sum_{n + 1 - t < i \le n;j\in [b_i]} \eps_{i,j} a_{i,j} \in \Sigma(\mbf{A}_{\med}) + \Sigma(A_{\sml}) .\end{equation}
\end{lemma}
\begin{proof} In this proof all instances of $\mb{P}$ are probabilities with respect to the random choice of $\mbf{A}_{\med}$.  We retain the convention that $E \approx E'$ is shorthand for $E' = (1 + O(V^{-\Omega(1)}))E$.
Write $S :=S(A_{\sml}, A_{\lrg})$ for brevity. We have
\begin{equation}\label{p-xe-union} \mb{P}\big(k \in \Sigma'(A_{\lrg}) + \Sigma(\mbf{A}_{\med}) + \Sigma(A_{\sml})\big) = \mb{P} \big(\bigcup_{\eps \in \mathcal{E}} X_{\eps} \big).\end{equation} 
Let $\eps^{(1)},\dots, \eps^{(\ell)} \in \mathcal{E}$ be distinct. We claim that \cref{almost-indep-statement} holds with the $x_u$ being the quantities $k - \sum_{n +1 - t < i \le n;j \in [b_i]} \eps^{(u)}_{i,j} a_{i,j}$. Indeed, the latter statement was proven assuming that the $x_u$ satisfy \cref{xi-assump} and the separation condition \cref{lem:Poisson:sep}. In the present context the first of these conditions is a consequence of the restriction to $\eps \in \mathcal{E}$, and the separation condition is one of the assumptions of \cref{key-61-tech}. The claim follows.

This instance of \cref{almost-indep-statement} gives the key almost-independence statement
\begin{equation}\label{key-Pois} \mb{P}\big( \bigcap_{u = 1}^{\ell} X_{\eps^{(u)}}\big) \approx \prod_{u = 1}^{\ell} \mb{P}(X_{\eps^{(u)}}),\end{equation} provided that $\ell \le \ell_*$. 

We pause to show that the $\mb{P}(X_{\eps})$ are individually rather small. To do this we first note by the definition \cref{score-def-2} of $X_{\eps}$ and \cref{prescore-3e} we have
\begin{equation}\label{p-eps-eval} \mb{P}(X_{\eps}) \approx \tau(A_{\sml}) 2^{\sum_{i =1}^{n -L}  b_i}(2^{-L}k)^{-1} \sum_{\eps'}  \mb{P} \big(k - \sum_{\substack{i > n + 1 - t \\ j \in [b_i]}} \eps_{i,j} a_{i,j}  - \sum_{\substack{i = n +1 - L\\ j \in [b_i]}}^{n +1 - t} \eps'_{i,j} \mathbf{a}_{i,j}  \in [2^{- L}k] \big),  \end{equation}  where (as before) $\eps'= (\eps'_{i,j})_{n +1 - L \le i \le n + 1 - t;j \in [b_i]}$.  Now if $\eps \in \mathcal{E}$ and if, for some $\eps'$, the inner probability is not zero then for some instance $(a_{i,j})_{n +1 - L \le i \le n +1 - t;j \in [b_i]}$ of $(\mbf{a}_{i,j})_{n +1 - L\le i \le n +1 - t;j \in [b_i]}$ we have
\begin{equation}\label{77a}  \big|\sum_{\substack{i = n +1 - L \\ j \in [b_i]}}^{n +1 - t}\eps'_{i,j} a_{i,j}\big| \ge \mu - 2^{-L} k - (\log V)^{1/4} \sigma.\end{equation}  
We claim that this forces at least one of the $\eps'_{n + 1 - t, u}$ (that is, associated to the jump step) to not be zero. To see this claim, we may use \cref{dyadic-aij,jump-step-2} to bound \begin{equation} \label{eq77b} \big| \sum_{\substack{i = n +1 - L \\ j \in [b_i]}}^{n - t} \eps'_{i,j} a_{i,j}\big| \ll \sum_{i = n +1 - L}^{n - t} b_{i} 2^{i} \ll 2^{n - t} \sum_{i = 1}^L  (1 + i^{2\kappa} )2^{-i} \ll 2^{n - t}.\end{equation} 
However, recalling that $\mu \asymp  V2^{n - t}$, $k \asymp 2^n$, $\sigma \asymp V^{1/2} 2^{n - t}$ and $L \ggg t$, the RHS of \cref{77a} is $\gg V2^{n - t}$, which is much larger than \cref{eq77b}; this establishes the claim.

Using the anticoncentration estimate \cref{anti-concentration-jump-scale} (revealing all $\mbf{a}_{i,j}$ other than $\mathbf{a}_{n + 1 -t,u}$), we have that \[ \mb{P} \Big(k - \sum_{\substack{i > n +1 - t \\ j \in [b_i]}} \eps_{i,j} a_{i,j}  - \sum_{\substack{i = n +1 - L\\ j \in [b_i]}}^{n +1 - t}  \eps'_{i,j} \mathbf{a}_{i,j}  = x\Big) \ll V^{1/3}2^{t - n},\] uniformly for $x \in [2^{-L} k]$, $\eps'$, and $\eps \in \mathcal{E}$.

Applying this to \cref{p-eps-eval}, summing over the $2^{\sum_{i = n +1 - L}^{n + 1 - t} b_i}$ choices of $\eps'$ and using the trivial bound $\tau(A_{\sml}) \le 1$ and the upper bound in \cref{bnt} we obtain
\begin{equation}\label{p-xeps-upper} \mb{P}(X_{\eps}) \ll 2^{\sum_{i =1}^{n -L} b_i} \cdot 2^{\sum_{i = n +1 - L}^{n +1 - t} b_i} \cdot  V^{1/3}2^{t - n}  \ll V^{1/3}2^{-\frac{1}{2} t^{1/4}} \le 2^{- \frac{1}{3}t_{\min}^{1/4}}.\end{equation}
We set this aside for later use, and resume the main line of argument. We divide into two cases according to the value of $S = S(A_{\sml}, A_{\lrg})$. \vspace*{8pt}

\emph{Case 1:} $S \le \ell_*/100$. This is the main case. To estimate the RHS of \cref{p-xe-union} we use the inclusion-exclusion principle (``Poisson paradigm''). Let $\ell_0 \in [\ell_*/2, \ell_*]$ be even. Then by the Bonferroni inequalities we have
\begin{equation*} 1 - \mb{P} \big( \bigcup_{\eps \in \mathcal{E}} X_{\eps}\big) \le \sum_{0 \le \ell \le \ell_0} \sum_{\substack{\eps^{(1)} < \cdots < \eps^{(\ell)} \\ \eps^{(1)},\dots,\eps^{(\ell)} \in \mathcal{E}}} (-1)^{\ell} \mb{P}(\bigcap_{u = 1}^{\ell} X_{\eps^{(u)}}).\end{equation*} By \cref{key-Pois}
it follows that 
\begin{align}\nonumber 1  - & \mb{P} \big( \bigcup_{\eps \in \mathcal{E}} X_{\eps}\big)  \le \sum_{0 \le \ell \le \ell_0} (-1)^{\ell}\sum_{\substack{\eps^{(1)} < \cdots < \eps^{(\ell)} \\ \eps^{(1)},\dots,\eps^{(\ell)} \in \mathcal{E}}}  (1 + O(V^{-\Omega(1)})) \prod_{u = 1}^{\ell} \mb{P}(X_{\eps^{(u)}}) \\ &  \le \sum_{0 \le \ell \le \ell_0} (-1)^{\ell} \sum_{\substack{\eps^{(1)} < \cdots < \eps^{(\ell)} \\ \eps^{(1)},\dots,\eps^{(\ell)} \in \mathcal{E}}}  \prod_{u = 1}^{\ell} \mb{P}(X_{\eps^{(u)}}) + O(V^{-\Omega(1)})\sum_{0 \le \ell \le \ell_0} \sum_{\substack{\eps^{(1)} < \cdots < \eps^{(\ell)} \\ \eps^{(1)},\dots,\eps^{(\ell)} \in \mathcal{E}}} \prod_{u = 1}^{\ell} \mb{P}(X_{\eps^{(u)}}). \label{bonf}\end{align}

The error term here is bounded by $\ll V^{-\Omega(1)} \sum_{0 \le \ell \le \ell_0} \frac{1}{\ell!} \big( \sum_{\eps} \mb{P}(X_{\eps}) \big)^{\ell} < V^{-\Omega(1)} e^{S} < e^{-\ell_*}$ (provided $\ell_*$ is small enough in terms of $V$), and so \cref{bonf} implies that 
\begin{equation}\label{1223a} 1  -  \mb{P} \big( \bigcup_{\eps \in \mathcal{E}} X_{\eps}\big) \le  \sum_{0 \le \ell \le \ell_0} (-1)^{\ell} \sum_{\substack{\eps^{(1)} < \cdots < \eps^{(\ell)} \\ \eps^{(1)},\dots,\eps^{(\ell)} \in \mathcal{E}}}  \prod_{u = 1}^{\ell} \mb{P}(X_{\eps^{(u)}}) + O(e^{-\ell_*}).\end{equation}
We also have
\begin{equation}\label{2-sum} \big| e^{-S} - \sum_{\ell = 0}^{\ell_0} (-1)^{\ell}\frac{1}{\ell!} S^{\ell} \big| \le \sum_{\ell > \ell_0} \frac{1}{\ell!} S^{\ell} \ll \big( \frac{eS}{\ell_0}\big)^{\ell_0} \ll \big(\frac{e}{50}\big)^{\frac{1}{2}\ell_*} \ll e^{-\ell_*}.\end{equation}
Also, \cref{p-xeps-upper,lemmaB1} give that 
\begin{equation}\label{1224a} \Big| S^{\ell} - \ell! \sum_{\substack{\eps^{(1)} < \cdots < \eps^{(\ell)} \\ \eps^{(1)},\dots,\eps^{(\ell)} \in \mathcal{E}}}\prod_{u = 1}^{\ell} \mb{P}(X_{\eps^{(u)}})\Big| \le \ell(\ell-1) 2^{-\frac{1}{3}t_{\min}^{1/4}} S^{\ell - 1} < 2^{-\frac{1}{4}t_{\min}^{1/4}} < e^{-10\ell_*}.\end{equation} 
Combining \cref{1223a,1224a,2-sum} with \cref{p-xe-union} gives
\[ \mb{P}\big( k \in \Sigma'(A_{\lrg}) + \Sigma(\mbf{A}_{\med}) + \Sigma(A_{\sml})\big) \ge 1 - e^{-S(A_{\sml},A_{\lrg})} + O(e^{-\ell_*}).\]
A corresponding upper bound follows using the lower bound Bonferroni inequalities (taking $\ell_0$ odd) in an entirely analogous manner. This completes the proof of \cref{key-61-tech} in Case 1.\vspace*{8pt}

\emph{Case 2:} $S > \ell_*/100$. For this we use a second moment method. We have
\begin{equation}\label{sec-mom}
\mb{E} \big| \sum_{\eps \in \mathcal{E}} ( 1_{X_{\eps}} - \mb{P}(X_{\eps})) \big|^2  = \sum_{\eps, \eps' \in \mathcal{E}} \mb{P}(X_{\eps} \cap X_{\eps'}) - S^2.
\end{equation}
Now by \cref{key-Pois} with $\ell =2$ we have when $\eps \neq \eps'$ that
\[ \mb{P}(X_{\eps} \cap X_{\eps'}) = \big(1 + O(V^{-\Omega(1)})\big)  \mb{P}(X_{\eps}) \mb{P}(X_{\eps'}).\] Summing over $\eps, \eps'$ gives
\[ \sum_{\eps, \eps' \in \mathcal{E}}\mb{P}(X_{\eps} \cap X_{\eps'}) = \big(1 + O(V^{-\Omega(1)})\big) \Big( S^2  -  \sum_{\eps \in \mathcal{E}} \mb{P}(X_{\eps})^2\Big) + S.\]
Therefore from \cref{sec-mom} we have
\[ \mb{E} \big| \sum_{\eps \in \mathcal{E}} ( 1_{X_{\eps}} - \mb{P}(X_{\eps})) \big|^2 \ll S+ V^{-\Omega(1)} S^2.\] It follows by Markov's inequality that the complement of $\bigcup_{\eps} X_{\eps}$, that is to say the event $\sum_{\eps} 1_{X_{\eps}} = 0$, has probability $\ll S^{-1} + V^{-\Omega(1)} \ll \ell_*^{-1}$. 

This concludes the analysis of Case 2 and hence the proof of \cref{key-61-tech}.
\end{proof}

\subsection{Proof of the main result}

We now turn to the task of proving the main result of the section, \cref{sec6-main}.  To do this, we will use \cref{key-61-tech} and undo the conditioning to $\mbf{A}_{\sml} = A_{\sml}$, $\mbf{A}_{\lrg} = A_{\lrg}$.
We first need to establish that the application of \cref{key-61-tech} is generically valid. That this is so is the content of \cref{7lem-claim-1,7lem-claim-2} below.

\begin{lemma}\label{7lem-claim-1} Define $\tau(A_{\sml})$ as in \cref{tau-small}. Then, with the setup as in \cref{sec6-main}, the probability that $\tau(\mbf{A}_{\sml}) \le 1/V$ is at most $2^T/V$.\end{lemma}
\begin{proof}
Writing $r_{\mbf{A}}(x)$ for the multiplicity of $x$ in $\Sigma(\mbf{A}_{\sml})$, we have $\tau(\mbf{A}_{\sml}) = 2^{-\sum_{i =1}^{L} b_i} \Supp(r_{\mbf{A}})$. By Cauchy--Schwarz it follows that if $|\tau(\mbf{A}_{\sml})|\le 1/V$ then 
\begin{equation}\label{cauchy-est} \sum_{x} r_{\mbf{A}}(x)^2 \ge 2^{\sum_{i =1}^L b_i} V. \end{equation}
On the other hand, we have
\[ \mb{E}\sum_x  r_{\mbf{A}}(x)^2 = \sum_{\eps, \eps' } \mb{P}\big( \sum_{1 \le i \le L;j \in [b_i]} \eps_{i,j} \mathbf{a}_{i,j} =  \sum_{1 \le i \le L;j \in [b_i]} \eps'_{i,j} \mathbf{a}_{i,j}\big),\] where now $\eps = (\eps_{i,j})_{1 \le i \le L;j \in [b_i]}$, $\eps' = (\eps'_{i,j})_{1 \le i \le L;j \in [b_i]}$.
For each pair $(\eps, \eps')$ in the above sum, consider the largest $\ell$ for which there is some $m$ (which we take minimal) such that $\eps_{\ell,m} \neq \eps'_{\ell,m}$. Then by revealing all $\mathbf{a}_{i,j}$ except $\mathbf{a}_{\ell,m}$, we see from \cref{anti-concentration-aij} that the inner probability is $\ll 2^{-\ell}$.

The number of pairs $(\eps, \eps')$ corresponding to a given $(\ell, m)$ is $2^{ 2\sum_{i =1}^L b_i - \sum_{i = \ell + 1}^Lb_i - (m - 1)}$. 
Thus we have 
\begin{align*} 
\mb{E}\sum_x r_{\mbf{A}}(x)^2 \ll \sum_{\ell = 1}^L & 2^{-\ell} \sum_{m \in [b_{\ell}]} 2^{ 2\sum_{i =1}^L b_i - \sum_{i = \ell + 1}^L b_i - (m - 1)} \\ \nonumber  & \ll  \sum_{\ell =1}^L 2^{-\ell + 2\sum_{i=1}^L b_i - \sum_{i = \ell + 1}^L b_i}   = 2^{\sum_{i =1}^L b_i} \sum_{\ell =1}^L 2^{- \beta'(\ell)}.\end{align*}
Since $\beta'$ is a $T$-positive walk, $\beta'(\ell) \ge -T + \ell^{1/4}$ and so it follows that 
\[ \mb{E} \sum_x r_{\mbf{A}}(x)^2 \ll 2^{T + \sum_{i=1}^L b_i} .\] By Markov, 
\[ \mb{P} \big( \sum_x r_{\mbf{A}}(x)^2 \ge 2^{\sum_{i=1}^L  b_i} V\big) \ll 2^T/V,\] and so the lemma follows from the remarks leading to \cref{cauchy-est}.\end{proof}

The next lemma controls the separation property required in \cref{key-61-tech}. Here, very crude estimates suffice.

\begin{lemma}\label{7lem-claim-2} With the setup in \cref{sec6-main}, the probability that there exist distinct $\eps = (\eps_{i,j})_{n +1 - t < i \le n;j \in [b_i]}$ and $\eps' = (\eps'_{i,j})_{n +1 - t < i \le n;j \in [b_i]}$ for which $\big| \sum_{i,j} (\eps_{i,j}  - \eps'_{i,j}) \mbf{a}_{i,j}\big| \le 2^{n - L/2}$ is $\le 2^{-L/8}$.
\end{lemma}
\begin{proof}
The number of pairs $\eps, \eps'$ is $< 2^{2\sum_{i = n + 2 - t}^n b_i} < 2^{L/4}$, using the $R$-boundedness of $\beta$ and the parameter hierarchy for the second bound. For each pair, choose some index with $\eps_{u,v} \neq \eps'_{u,v}$, and reveal all $\mbf{a}_{i,j}$ except for $\mbf{a}_{u,v}$. By \cref{anti-concentration-aij}, $\mb{P} \big( \big| \sum_{n +1 - t < i \le n;j \in [b_i]} (\eps_{i,j}  - \eps'_{i,j}) \mbf{a}_{i,j}\big| \le  2^{n - L/2} \big) \ll 2^{t - n} \cdot 2^{n - L/2}$. Summing over all pairs $\eps, \eps'$ and using the parameter hierarchy again gives the result.
\end{proof}

Now we begin the proof of \cref{sec6-main} in earnest.
From \cref{key-61-tech} we have 
\begin{align} \nonumber \mb{P}\big(k \in \Sigma'(A_{\lrg}) + \Sigma(\mbf{A}_{\med}) + \Sigma(A_{\sml})\big) & = (1 - e^{-S(A_{\sml}, A_{\lrg})})1_{\tau(A_{\sml}) \ge 1/V,\, A_{\lrg} \in \mathscr{S}} \\  & + O\big(1_{\tau(A_{\sml}) \le 1/V} + 1_{A_{\lrg} \notin \mathscr{S}} + \ell_*^{-1}\big),\label{first-k-est}\end{align} where $\mathscr{S}$ denotes the separation condition in the statement of \cref{key-61-tech}.
Now from the definition \cref{score-def-1} of $S(A_{\sml}, A_{\lrg})$ and \cref{score-def-2} we have
\[ S(A_{\sml}, A_{\lrg})\ = \sum_{\eps \in \mathcal{E}} \mb{P} \big( k - \sum_{\substack{n +1 - t < i \le n \\ j \in [b_i]}} \eps_{i,j} a_{i,j} \in \Sigma(\mbf{A}_{\med}) + \Sigma(A_{\sml}) \big).\]
If $\tau(A_{\sml}) \ge 1/V$ we now apply \cref{prescore-3e} with $x = k - \sum_{n +1 - t < i \le n, j \in [b_i]} \eps_{i,j} a_{i,j}$, $\eps \in \mathcal{E}$; the relevant condition \cref{x-assump} holds precisely because $\eps \in \mathcal{E}(A_{\lrg})$ (recall the definition \cref{e-alarge-def}).  This gives that if $\tau(A_{\sml}) \ge 1/V$ then
\begin{equation*}  S(A_{\sml}, A_{\lrg}) = (1 + O(V^{-\Omega(1)})) S'(A_{\sml}, A_{\lrg}) \end{equation*} where $S'(A_{\sml}, A_{\lrg})$ equals
\begin{equation}\label{s-prime-def} \tau(A_{\sml})2^{\sum_{i=1}^{n - L} b_i}(2^{-L}k)^{-1} \sum_{\substack{\eps \in \mathcal{E}(A_{\lrg}) \\  \eps'}}  \mb{P} \Big(k - \sum_{\substack{i > n +1 - t \\ j \in [b_i]}} \eps_{i,j} a_{i,j}  - \sum_{\substack{i = n +1 - L  \\ j \in [b_i]}}^{n + 1 - t} \eps'_{i,j} \mathbf{a}_{i,j}  \in [2^{- L}k] \Big)  .\end{equation}
where $\eps = (\eps_{i,j})_{n +1 - t < i \le n;j \in [b_i]}$ and $\eps' = (\eps'_{i,j})_{n +1 - L \le i \le n +1 - t;j \in [b_i]}$ range over all tuples of elements from $\{0,1\}$ (subject to the restriction $\eps \in \mathcal{E}(A_{\lrg})$).
Substituting into \cref{first-k-est}, removing the conditioning on $A_{\lrg}, A_{\sml}$ and applying \cref{7lem-claim-1,7lem-claim-2} (noting that all error probabilities, as well as the $V^{-\Omega(1)}$ term, are much less than $\ell_*^{-1}$ by the parameter hierarchy) we obtain
\begin{equation}\label{preprop-72} \mb{P} \big(k \in \Sigma'(\mbf{A}_{\lrg}) + \Sigma(\mbf{A}_{\med}) + \Sigma(\mbf{A}_{\sml})\big) = 1 - \mb{E}e^{-S'(\mbf{A}_{\sml}, \mbf{A}_{\lrg})} +  O(\ell_*^{-1}).\end{equation}

Our next task is to relate the LHS here to $\mb{P}(k \in \Sigma(\mbf{A}))$. We pause to recall here that $\mbf{A}$ still means that we are conditioning to $\boldbeta = \beta, \boldbeta' = \beta'$ and $\mbf{f} = f$. 

In carrying out the analysis (and also later on) we will use the following lemma.
\begin{lemma}\label{claim3-lem}
Suppose that $s$ is an integer with $0 \le s \le n +1 - L$. Write $\eps = (\eps_{i,j})_{s \le i \le n; j \in [b_i]}$, and $\eps_{\lrg} := (\eps_{i,j})_{n + 1 - t <i \le n; j \in [b_i]}$. Then
\[ \sum_{\eps} \mb{P} \big(k - \sum_{s \le i \le n; j \in [b_i]} \eps_{i,j} \mathbf{a}_{i,j} \in [2^{s - n - 1} k], \, \eps_{\lrg} \notin \mathcal{E}(\mbf{A}_{\lrg})\big) \ll R^22^{\sum_{i = s}^n b_i - (n - s)} e^{-\Omega((\log V)^{1/2})}. \]
\end{lemma}
\begin{proof} Suppose $k \in \sum_{s \le i \le n; j \in [b_i]} \eps_{i,j} \mbf{a}_{i,j}  + [2^{s -n - 1} k]$ with $\eps_{\lrg} \notin \mathcal{E}(\mbf{A}_{\lrg})$, that is to say \[ \Big|k - \sum_{n + 1 - t < i \le n;j \in [b_i]} \eps_{i,j} \mathbf{a}_{i,j} - \mu\Big| \ge  (\log V)^{1/4}\sigma.\] Then \[ \Big|\sum_{s \le i \le n +1 - t;j \in [b_i]} \eps_{i,j} \mbf{a}_{i,j} - \mu\Big| \ge (\log V)^{1/4}\sigma - 2^{s - n - 1} k > \tfrac{1}{2} (\log V)^{1/4}\sigma,\] where the last step follows using that $\sigma \asymp V^{1/2} 2^{n - t}$ whilst $2^{s - n} k \ll 2^s \le 2^{n - L} \lll 2^{n - t}$.

By \cref{eq77b} (truncated to $i = s$ at the lower end of the summation) it follows that the contribution from $i < n +1 - t$ is negligible and hence that
\[ \Big|\sum_{j=1}^V \eps_{\jump,j} \textbf{a}_{\jump,j} - \mu\big| = \Big|\sum_{j=1}^V \eps_{n +1 - t,j} \mbf{a}_{n +1 - t,j} - \mu\big| \ge \tfrac{1}{4} (\log V)^{1/4}\sigma,\] where here and henceforth we write $\eps_{\jump, j} := \eps_{n + 1 - t,j}$ and $\mbf{a}_{\jump, j} := \mbf{a}_{n +1 - t,j}$ for $j \in [V]$. 

Next we observe that if some $\eps$ contributes nontrivially to the LHS in the lemma then it must lie in the set $\mathcal{E}_0$ consisting of those $\eps$ for which at least one $\eps_{i,j}$, $i \ge  n - 2 \log_2 R$, is not zero. This follows immediately from the $R$-boundedness of $\beta$ and \cref{R-bded-small-1}, and the fact that $\mbf{a}_{i,j} \ll 2^{i}$ deterministically (see \cref{dyadic-aij}).

Putting these observations together, the LHS in the lemma is bounded above by
\[ \sum_{\eps \in \mathcal{E}_0} \mb{P}\Big(k - \sum_{s \le i \le n;j \in [b_i]} \eps_{i,j} \mathbf{a}_{i,j} \in [2^{s - n - 1} k],\, \big|\sum_{j=1}^V \eps_{\jump,j} \textbf{a}_{\jump,j} - \mu\big| \ge \tfrac{1}{4} (\log V)^{1/4}\sigma\Big). \]
Denoting $\tilde\eps := (\eps_{i,j})_{s \le i \le n, i \neq n +1 - t; j \in [b_i]}$ (that is, all the $\eps_{i,j}$ other than the $\eps_{\jump, j}$), this is 
\[ \ll 2^{s - n} k\sum_{\tilde\eps: \tilde\eps \in \mathcal{E}_0}  \sup_x \mb{P} \big( \sum_{\substack{s \le i \le n \\ i \neq n +1 - t \\ j \in [b_i]}} \eps_{i,j} \mathbf{a}_{i,j} = x\big) \cdot \sum_{\eps_{\jump}} \mb{P} \Big( \big|\sum_{j=1}^V \eps_{\jump,j} \textbf{a}_{\jump,j} - \mu\big| \ge \tfrac{1}{4} (\log V)^{1/4}\sigma\Big). \]
Here, $\tilde\eps$ ranges over the set $\{0,1\}^{\sum_{s \le i \le n, i \ne n + 1 - t} b_i}$, and $\eps_{\jump} = (\eps_{\jump, j})_{j \in [V]}$ ranges over $\{0,1\}^V = \{0,1\}^{b_{n + 1 - t}}$. Additionally, we have implicitly used that the parameter hierarchy \cref{param-heir} implies that $t \ggg 2\log_2 R$, so the conditions $\eps \in \mathcal{E}_0$ and $\tilde\eps \in \mathcal{E}_0$ are the same.
By revealing all $\mathbf{a}_{i,j}$ except for one corresponding to an index $(i,j)$ with $i \ge n -  2 \log_2 R$ for which $\tilde\eps_{i,j} \ne 0$, we see from \cref{anti-concentration-aij} that the $\sup_x \mb{P}( \sum \eps_{i,j} \mbf{a}_{i,j} = x)$ term is $\ll R^2 2^{-n}$.

Thus, writing $\boldeps_{\jump,j}$, $j = 1,\dots, V$, for i.i.d. random $\{0,1\}$ variables (independent of the $\mbf{a}$ variables) and using that $k \asymp 2^n$, the LHS in the lemma is bounded above by
\begin{align}\nonumber 2^{s - n}k \cdot & 2^{\sum_{s \le i \le n, i \ne n +1 - t} b_i}  R^2 2^{-n} \cdot 2^V \mb{P} \big( \big| \sum_{j = 1}^V \boldeps_{\jump,j} \mathbf{a}_{\jump,j} - \mu\big| > \tfrac{1}{4}\sigma (\log V)^{1/4} \big) \\ &  = R^2 2^{\sum_{i = s}^n b_i - (n - s)} \mb{P} \big( \big| \sum_{j = 1}^V \mbf{X}_j - \mu\big| > \tfrac{1}{4} (\log V)^{1/4} \sigma\big),\label{63-a1} \end{align} where $\mbf{X}_j := \boldeps_{\jump, j} \mbf{a}_{\jump, j}$.
Now recall that by definition \cref{mu-sig-def}, $\mu = \frac{1}{2}\sum_{j = 1}^V \mb{E} \mbf{a}_{\jump, j}$, and so $\sum_{j = 1}^V \mb{E}\mbf{X}_j = \mu$. Now by Hoeffding's inequality (see \cref{hoeffding-inequality}), if $H$ is an upper bound for $\max_j |\mbf{X}_j|$ we have
\[ \mb{P} \big( \big|\sum_{j = 1}^V \mbf{X}_j - \mu \big| \ge \tfrac{1}{4}(\log V)^{1/4} \sigma\big) \le 2 \exp \big(-\sigma^2(\log V)^{1/2}/(32VH^2)\big) \ll e^{-\Omega((\log V)^{1/2})},\]
where in the last step we used the fact that we can take $H \asymp 2^{n - t}$, and that $\sigma \asymp V^{1/2} 2^{n - t}$. The proof of \cref{claim3-lem} is complete.
\end{proof}

Finally we complete the proof of \cref{sec6-main}. 

\begin{proof}[Proof of \cref{sec6-main}] From the case $s = 1$ of \cref{claim3-lem}, noting that in this case we have $\sum_{s \le i \le n}  b_i - n = D$, we see that the probability that $k \in \Sigma(\mbf{A})$, but $k \notin \Sigma'(\mbf{A}_{\lrg}) + \Sigma(\mbf{A}_{\med}) + \Sigma(\mbf{A}_{\sml})$, is $\ll R^2 2^{D}e^{-\Omega((\log V)^{1/2})} \ll \ell_*^{-1}$. By this and \cref{preprop-72} we have
\begin{equation}\label{eq730a} \mb{P}(k \in \Sigma(\mbf{A})) =  1 - \mb{E} e^{-S'(\mbf{A}_{\sml}, \mbf{A}_{\lrg})}  + O(\ell_*^{-1}).\end{equation} Recall that $S'(\mbf{A}_{\sml}, \mbf{A}_{\lrg})$ is defined in \cref{s-prime-def}. Define $S''(A_{\sml}, A_{\lrg})$ as for $S'(A_{\sml}, A_{\lrg})$, but without the constraint $(\eps_{i,j})_{n + 1 - t < i \le n;j \in [b_i]} \in \mathcal{E}(A_{\lrg})$. That is, 
\begin{align}\nonumber S''(A_{\sml}, & A_{\lrg}) := \tau(A_{\sml})2^{\sum_{i = 1}^{n - L} b_i} (2^{-L}k)^{-1} \times \\ & \times \sum_{(\eps_{i,j})_{ n +1 - L \le i \le n ; j \in [b_i]}}\mb{P} \big(k - \sum_{\substack{n + 1 - t < i \le n \\ j \in [b_i]}} \eps_{i,j} a_{i,j}  - \sum_{\substack{i = n+1 - L\\ j \in [b_i]}}^{n +1 - t} \eps_{i,j} \mathbf{a}_{i,j}  \in [2^{- L}k] \big)  .\label{s-prime-double-def}\end{align}
Using the trivial bound $\tau(A_{\sml}) \le 1$, we have that $\mb{E}\big( S''(\mbf{A}_{\sml},  \mbf{A}_{\lrg})  - S'(\mbf{A}_{\sml}, \mbf{A}_{\lrg})\big)$ is bounded by
\[ \ll 2^{\sum_{i=1}^{n - L} b_i - (n - L)}\sum_{\eps_{i,j}} \mb{P} \big( k  - \sum_{\substack{n + 1 - L \le i \le n \\j \in [b_i]}} \eps_{i,j} \mathbf{a}_{i,j}  \in [2^{- L}k],\, (\eps_{i,j})_{n +1 - t < i \le n; j \in [b_i]} \notin \mathcal{E}(\mbf{A}_{\lrg}) \big) \] where the first sum is over all $(\eps_{i,j})_{ n +1 - L \le i \le n ; j \in [b_i]}$ with components in $\{0,1\}$.

Applying \cref{claim3-lem} with $s = n -  L + 1$, and using that $\sum_{i = 1}^n b_i = n + D$, it follows using the parameter hierarchy \cref{param-heir} that
\begin{equation}\label{eq730b} \mb{E}\big( S''(\mbf{A}_{\sml},  \mbf{A}_{\lrg})  - S'(\mbf{A}_{\sml}, \mbf{A}_{\lrg})\big) \ll R^2 2^{D}  e^{-\Omega((\log V)^{1/2})} \ll \ell_*^{-1}.\end{equation} 
 Since $x \mapsto e^{-x}$ is Lipschitz with constant $1$ on $[0, \infty)$, it follows from \cref{eq730a,eq730b} that we have
\begin{equation}\label{eq730c} \mb{P}(k \in \Sigma(\mbf{A})) =  1 - \mb{E} e^{-S''(\mbf{A}_{\sml}, \mbf{A}_{\lrg})} + O(\ell_*^{-1}).\end{equation}

Recall once more that, throughout the preceeding analysis, $\mbf{A}$ means $(\mbf{A} \mid \boldbeta = \beta, \boldbeta' = \beta', \mbf{f} = f)$.

From the definition \cref{s-prime-double-def}, using again that $\sum_{i = 1}^n b_i = n + D$, recalling that $\sum_{i > n - L} b_i = L + \beta(L)$ and $k = 2^{n + \xi}$, and rescaling the $a$-variables by dividing by $k$, we have that $S''(A_{\sml}, A_{\lrg})$ equals
\[  2^{D - \xi} \tau(A_{\sml}) 2^{-\beta(L)}\sum_{\eps = (\eps_{i,j})_{i \ge n +1 - L; j}} \mb{P} \big(1 - \sum_{\substack{i > n +1 - t \\ j \in [b_i]}} \eps_{i,j} \frac{a_{i,j}}{k}  - \sum_{\substack{i = n +1 - L \\ j \in [b_i]}}^{n +1 - t} \eps_{i,j} \frac{\mbf{a}_{i,j}}{k} \in [2^{- L}] \big). \]
Therefore
\[ S''(A_{\sml}, \mbf{A}_{\lrg})  = 2^{D - \xi} \tau(A_{\sml}) 2^{-\beta(L)}\sum_{\eps = (\eps_{i,j})_{i \ge n +1 - L; j}} \mb{P}_{\mbf{A}_{\med}} \big(1 -  \sum_{\substack{i = n +1 - L \\ j \in [b_i]}}^{n} \eps_{i,j} \frac{\mbf{a}_{i,j}}{k} \in [2^{- L}] \big). \] In this expression, note carefully that the probability is taken only over random choices of $\mbf{A}_{\med} = (\mbf{a}_{i,j})_{L +1 \le i \le n + 1 - t; j \in [b_i]}$, which means that $S''(A_{\sml},\mbf{A}_{\lrg})$ still depends nontrivially on the random variables $\mbf{A}_{\lrg} = (\mbf{a}_{i,j})_{n +1 - t < i \le n; j \in [b_i]}$.
Recalling \cref{xbeta-ybeta-defs}, this may be written
\begin{equation*} S''(A_{\sml}, \mbf{A}_{\lrg})  = 2^{D - \xi} \tau(A_{\sml})  \mb{E}_{\mbf{A}_{\med}} \varrho^*_{\tilde{\mbf{u}} \mid \beta}(L),\end{equation*} where $\tilde{\mbf{u}} = (\tilde{\mbf{u}}_{i,j})_{1 \le i \le L, j \in [1 + \xi_i]}$ is defined by \begin{equation}\label{utilde-def} 2^{-\tilde{\mbf{u}}_{i,j}} = \frac{1}{k}\mbf{a}_{n +1 - i,j}\end{equation} for $1 \le i \le L$ and $j \in [1 + \xi_i]$.  Recall here that $b_{n +1 - i} = 1 + \xi_i$ for $i \le \lceil n/2\rceil$, where the $\xi_i$ are the increments of $\beta$, and note that $\tilde{\mbf{u}}$ depends on $\mbf{A}_{\med}$ and $\mbf{A}_{\lrg}$.
Set 
\begin{equation}\label{s-triple-prime} S'''(A_{\sml}, \mbf{A}_{\lrg}) := 2^{D - \xi} \tau(A_{\sml}) \varrho^*_{\tilde{\mbf{u}} \mid \beta}(t-1);
\end{equation} note that $\varrho^*_{\tilde{\mbf{u}} \mid \beta}(t-1)$ depends only on $\mbf{A}_{\lrg}$ and not on $\mbf{A}_{\med}$, so this is well-defined. By the Lipschitz nature of $e^{-x}$ and the triangle inequality, and since $|D| \le M/10$, we have 
\begin{align*}\nonumber
& \mb{E}_{\mbf{A}_{\lrg}} \big|   e^{-S''(A_{\sml},\mbf{A}_{\lrg})}   - e^{-S'''(A_{\sml},\mbf{A}_{\lrg})}\big|  \ll \mb{E}_{\mbf{A}_{\lrg}} \big|  S''(A_{\sml}, \mbf{A}_{\lrg})  - S'''(A_{\sml}, \mbf{A}_{\lrg})\big| \\ & \ll 2^{M/10} \mb{E}_{\mbf{A}_{\lrg}} \big| \varrho^*_{\tilde{\mbf{u}} \mid \beta}(t-1) - \mb{E}_{\mbf{A}_{\med}} 
\varrho^*_{\tilde{\mbf{u}} \mid \beta}(L) \big| \le 
2^{M/10} \mb{E}_{\mbf{A}_{\lrg}, \mbf{A}_{\med}} \big| \varrho^*_{\tilde{\mbf{u}} \mid \beta}(t-1) - \varrho^*_{\tilde{\mbf{u}} \mid \beta}(L)\big|,\nonumber \end{align*}
and so
\begin{equation*}\mb{E} \big|   e^{-S''(\mbf{A}_{\sml}, \mbf{A}_{\lrg})}   - e^{-S'''(\mbf{A}_{\sml}, \mbf{A}_{\lrg})}\big| \ll 
2^{M/10} \mb{E} \big| \varrho^*_{\tilde{\mbf{u}} \mid \beta}(t-1) - \varrho^*_{\tilde{\mbf{u}} \mid \beta}(L)\big| ,\end{equation*} where here averages are taken over the whole random choice of $\mbf{A} = \mbf{A}_{\sml} \cup \mbf{A}_{\med} \cup \mbf{A}_{\lrg}$, conditioned to $\boldbeta = \beta$, $\boldbeta' = \beta'$ and $\mathbf{f} = f$. Therefore by \cref{eq730c} we have 
\begin{equation}\label{eq730d} \mb{P}(k \in \Sigma(\mbf{A})) =  1 - \mb{E} e^{-S'''(\mbf{A}_{\sml}, \mbf{A}_{\lrg})} + O(\ell_*^{-1}) + 
O \big( 2^{M/10} \big)\mb{E} \big| \varrho^*_{\tilde{\mbf{u}} \mid \beta}(t-1) - \varrho^*_{\tilde{\mbf{u}} \mid \beta}(L)\big|.\end{equation}
Recall that $\mbf{f}$ is the `jump distribution' from the previous section; specifically, $f_m$ is the number of the $\mbf{u}_{t,j}$ in the interval $I_m$ (see \cref{im-def}). However, the term $S'''(\mbf{A}_{\sml}, \mbf{A}_{\lrg})$ is independent of $\mbf{f}$, using \cref{a-list} and that the jump indices with $\mbf{a}_{n + 1 - t,j}$ are part of $\mbf{A}_{\med}$. Thus it is tempting to remove the conditioning on $\mbf{f}$ at this juncture. To this end we note the simple comparison estimate
\[
\mb{E} \big| \varrho^*_{\tilde{\mbf{u}} \mid \beta}(t-1) - \varrho^*_{\tilde{\mbf{u}} \mid \beta}(L)\big| \le \mb{P}(\mbf{f} = f)^{-1} \tilde{\mb{E}} \big| \varrho^*_{\tilde{\mbf{u}} \mid \beta}(t-1) - \varrho^*_{\tilde{\mbf{u}} \mid \beta}(L)\big| \le V^{V/3}\tilde{\mb{E}} \big| \varrho^*_{\tilde{\mbf{u}} \mid \beta}(t-1) - \varrho^*_{\tilde{\mbf{u}} \mid \beta}(L)\big|
\]
where (as above) $\mb{E}$ denotes averaging with the conditioning to $\mbf{f} = f$ in place, and $\tilde{\mb{E}}$ is the average with no condition on $\mbf{f}$ (but still conditioning to $\boldbeta = \beta$ and $\boldbeta' = \beta'$). The second computation here comes from
\[ \mb{P}(\mbf{f} = (f_1,\dots, f_{V^{1/3}})) = \frac{V!}{f_1! \cdots f_{V^{1/3}}!} V^{-V/3} \ge V^{-V/3}.\] Thus, at the expense of a slightly worse error term, we can remove the conditioning on $\mbf{f}$ in \cref{eq730d} to conclude that 
\begin{equation}\label{eq730dd} \mb{P}(k \in \Sigma(\mbf{A})) =  1 - \mb{E} e^{-S'''(\mbf{A}_{\sml}, \mbf{A}_{\lrg})} + O(\ell_*^{-1}) + 
O \big( 2^{M/10}V^{V/3} \big)\mb{E} \big| \varrho^*_{\tilde{\mbf{u}} \mid \beta}(t-1) - \varrho^*_{\tilde{\mbf{u}} \mid \beta}(L)\big|.\end{equation}
Here, and from now on in this section, we are conditioning only to $\boldbeta = \beta, \boldbeta' = \beta'$, thus $\mbf{A}$ denotes $(\mbf{A} \mid \beta, \beta')$ and the averages on the right are with no conditioning on $\mbf{f}$ (we called this averaging $\tilde{\mb{E}}$ above), with the $\tilde{\mbf{u}}_{i,j}$ still defined as in \cref{utilde-def}.

From the definition \cref{s-triple-prime} we have
\[ \mb{E} e^{-S'''(\mbf{A}_{\sml}, \mbf{A}_{\lrg})} = \mb{E} \exp\big({-}2^{D- \xi}\tau(\mbf{A}_{\sml}) \varrho^*_{\tilde{\mbf{u}} \mid \beta}(t-1)\big).\]
Since $e^{-2^D x}$ is Lipschitz on $[0,\infty)$ with constant at most $2^D \le 2^{M/10}$, we have
\begin{align} \nonumber \big| \mb{E} e^{-S'''(\mbf{A}_{\sml}, \mbf{A}_{\lrg})} & - \mb{E} \exp \big(-2^{D - \xi} \tau^*_{\mbf{x} \mid \beta'}(L) \varrho^*_{\mbf{u} \mid \beta}(L)\big)\big| \\ &  \ll 2^{M/10}  \mb{E} \big| \tau^*_{\mbf{x} \mid \beta'}(L) \varrho^*_{\mbf{u} \mid \beta}(L) - \tau(\mbf{A}_{\sml})\varrho^*_{\tilde{\mbf{u}} \mid \beta}(t-1) \big| \label{eq730f}.\end{align} 
Applying the inequality  $|x_1 x_2 - x_3 x_4| \le x_3 |x_4 - x_5| + x_3 |x_5 - x_2| + x_2 |x_3 - x_1|$ (trivially valid for non-negative reals $x_1,\dots, x_5$) with $x_1 := \tau^*_{\mbf{x} \mid \beta'}(L)$, $x_2 := \varrho^*_{\mbf{u} \mid \beta}(L)$, $x_3 := \tau(\mbf{A}_{\sml})$, $x_4 := \varrho^*_{\tilde{\mbf{u}} \mid \beta}(t-1)$ and $x_5 := \varrho^*_{\mbf{u} \mid \beta}(t-1)$, we may bound the RHS of \cref{eq730f} by $2^{M/10} (E_1 + E_2 + E_3)$, where
\[  E_1 := \mb{E}  \tau(\mbf{A}_{\sml}) \cdot \mb{E} \big| \varrho^*_{\tilde{\mbf{u}} \mid \beta}(t-1) - \varrho^*_{\mbf{u} \mid \beta}(t-1)\big|, \qquad
 E_2 := \mb{E}  \tau(\mbf{A}_{\sml}) \cdot \mb{E} \big| \varrho^*_{\mbf{u} \mid \beta}(t-1) - \varrho^*_{\mbf{u} \mid \beta}(L)\big|,\] and
\[ E_3 := \mb{E} \varrho^*_{\mbf{u} \mid \beta}(L) \cdot \mb{E} | \tau(\mbf{A}_{\sml}) - \tau^*_{\mbf{x} \mid \beta'}(L) |.\]
We also introduce 
\[ E_4 := \mb{E} \big| \varrho^*_{\tilde{\mbf{u}} \mid \beta}(t-1) - \varrho^*_{\tilde{\mbf{u}} \mid \beta}(L)\big| \] (which appears in \cref{eq730dd}). We now proceed to bound $E_1,\dots, E_4$ in turn.\vspace*{8pt}

\emph{Bounding $E_1$.}
To bound $E_1$, we first use that $\tau(\mbf{A}_{\sml}) \le 1$ deterministically, and so
\begin{equation}\label{e1-first-bd}  E_1 \le \mb{E} \big| \varrho^*_{\tilde{\mbf{u}} \mid \beta}(t-1) - \varrho^*_{\mbf{u} \mid \beta}(t-1)\big|.\end{equation}
To bound this we use \cref{u-couple} and \cref{utilde-def}, which give that $|2^{-\mbf{u}_{i,j}} - 2^{-\tilde{\mbf{u}}_{i,j}}| \ll 2^{-i}i/n \ll \frac{1}{n}$ deterministically for $1 \le i \le L$ and $j \in [1 + \xi_i]$. Given $\eps = (\eps_{i,j})_{1 \le i \le t-1, j \in [1 + \xi_i]}$, set $\sigma_{\eps} := \sum_{1 \le i \le t - 1, j \in [1 + \xi_i]} \eps_{i,j} 2^{-\mbf{u}_{i,j}}$ and $\tilde{\sigma}_{\eps} := \sum_{1 \le i \le t-1, j \in [1 + \xi_i]} \eps_{i,j} 2^{-\tilde{\mbf{u}}_{i,j}}$. Then we have
\begin{equation}\label{sig-sig-prime} \big| \sigma_{\eps} - \tilde\sigma_{\eps}\big| \ll \frac{1}{n} \sum_{i = 1}^{t-1} (1 + \xi_i) \lll n^{-1/2}, \end{equation} with the final bound following very comfortably from the parameter hierarchy and since $\beta$ is $R$-bounded. Now we have by definition

\[ \varrho^*_{\mbf{u} \mid \beta}(t-1 ) = 2^{-\beta(t-1)} \sum_{\eps} 1\big(\sigma_{\eps} \in [1 - 2^{1-t},1]), \quad \varrho^*_{\tilde{\mbf{u}} \mid \beta}(t-1 ) = 2^{-\beta(t-1)} \sum_{\eps} 1\big(\tilde{\sigma}_{\eps} \in [1 - 2^{1-t},1]),\]
and so using \cref{e1-first-bd,sig-sig-prime} we obtain
\[ E_1 \le 2^{-\beta(t-1)} \sum_{\eps} \mb{P}(\sigma_{\eps} \in [1 - O( n^{-1/2}), 1 + O(n^{-1/2})] \cup [1 - 2^{1 - t} - O(n^{-1/2}), 1 - 2^{1 - t} + O(n^{-1/2})]\big).\]

As in the proof of \cref{X-upper-gen}, using the $R$-boundedness of $\beta$ we see that if $\sigma_{\eps} > \frac{1}{2}$ then $\eps_{a,b} \ne 0$ for some $a \le 2 \log_2 R$. In particular, the probability term above is zero unless this is the case.
Revealing all $\mbf{u}_{i,j}$ except $\mbf{u}_{a,b}$, it follows that, for each $\eps$, the probability term above is bounded by $\mb{P}(\sigma_{\eps} \in \cdots\big) \ll 2^{a} n^{-1/2}$. Since the number of $\eps$ is $2^{\sum_{i =1}^{t-1} (1 + \xi_i)} < 2^{R t_{\max}^2} < n^{1/10}$ and since $2^a < n^{1/10}$, we have the bound \begin{equation}\label{e1-bd} E_1 \ll n^{-1/4}.\end{equation} 

\emph{Bounding $E_2$.} We again use that $\tau(\mbf{A}_{\sml}) \le 1$, and for the second term we use \cref{lem73-first} and a telescoping sum. This gives a bound of \begin{equation}\label{e2-bound} E_2 \ll  2^{3T - \Omega(t_{\min}^{1/4})}.\end{equation}

\emph{Bounding $E_3$.} By \cref{X-upper-simple} we have $\mb{E} \varrho^*_{\mbf{u} \mid \beta}(L) \ll R^2$. To bound the $\mb{E} | \tau(\mbf{A}_{\sml}) - \tau^*_{\mbf{x} \mid \beta'}(L)|$ term, define $\tilde{\mbf{x}}$ by $\tilde{\mbf{x}}_{i,j} := \mbf{a}_{i, j}$ for $i \in [L]$ and $j \in [1 - \xi'_i]$. Thus by the definitions of the various quantities (see \cref{tau-small} and \cref{y-beta-def}) we have
\begin{equation}\label{form-11} \tau(\mbf{A}_{\sml}) = \tau^*_{\tilde{\mbf{x}} \mid \beta'}(L). \end{equation}
We showed in \cref{s-couple} that $\mb{P}(\mbf{x}_{i,j} \ne \tilde{\mbf{x}}_{i,j}) \ll L/n$, uniformly for $i \in [L]$ and $j \in [1 - \xi'_i]$; thus with probability $1 - O(L^3/n)$ we have $\mbf{x}_{i,j} = \tilde{\mbf{x}}_{i,j}$ for all $i \in [L]$ and $j \in [1 - \xi'_i]$, using here the crude bound $\sum_{i =1}^L (1 - \xi'_i) \ll L^2$, which follows from the $R$-boundedness of $\beta'$ and the parameter hierarchy.

In the event that $\mbf{x}_{i,j} = \tilde{\mbf{x}}_{i,j}$ for all $i \in [L]$ and $j \in [1 - \xi'_i]$, it follows from \cref{form-11} that $\tau(\mbf{A}_{\sml}) = \tau^*_{\mbf{x} \mid \beta'}(L)$. Since all $\tau^*$ quantities are $\le 1$ deterministically, we obtain from this analysis that
\begin{equation}\label{e3-bound} E_3 \ll R^2 \frac{L^3}{n} .\end{equation}

\emph{Bounding $E_4$.} Finally, to bound $E_4$, we use
\[ E_4 \le \mb{E}\big| \varrho^*_{\tilde{\mbf{u}} \mid \beta}(t-1) - \varrho^*_{\mbf{u} \mid \beta}(t-1)\big| + \mb{E} \big| \varrho^*_{\mbf{u} \mid \beta}(t-1) - \varrho^*_{\mbf{u} \mid \beta}(L)\big| + \mb{E}\big| \varrho^*_{\tilde{\mbf{u}} \mid \beta}(L) - \varrho^*_{\mbf{u} \mid \beta}(L)\big|. \] The first term we already bounded during the analysis of $E_1$. The third term may be bounded analogously, replacing $t-1$ by $L$ throughout. Due to the parameter hierarchy we have the bound $n^{-1/4}$ for both of these terms. Finally, as we already saw during the analysis of $E_2$, the second term may be bounded by $\ll 2^{3T - \Omega(t_{\min}^{1/4})}$ by use of \cref{lem73-first}. Thus
\begin{equation}\label{e4-bound} E_4 \ll 2^{3T - \Omega(t_{\min}^{1/4})}.\end{equation}

Putting the bounds \cref{e1-bd,e2-bound,e3-bound,e4-bound} together with \cref{eq730d,eq730f} (recalling that the RHS of \cref{eq730f} is bounded by $2^{M/10}(E_1 + E_2 + E_3)$) and using the parameter hierarchy \cref{param-heir} concludes the proof of \cref{sec6-main}.\end{proof}

\part{Putting everything together}\label{part4}

\section{Expressing \texorpdfstring{$p(k)$}{} as an average over walks}\label{sec:reduc-to-crit}

We turn now to the proof of our main results. At this point we drop the convention (in force for the preceding two sections) that $\mbf{A}$ is shorthand for either $(\mbf{A} \mid \boldbeta = \beta,\, \boldbeta' = \beta',\, \mathbf{f} = f)$ or $(\mbf{A} \mid \beta, \beta')$. From now on, $\mbf{A}$ will denote the unconditioned multiset as defined in \cref{a-law}, and we will explicitly write $(\mbf{A} \mid \beta, \beta')$ for the conditioned sets as described in \cref{subsec-53}. There should be little danger of confusion since the only result we will need from the preceding two sections is \cref{sec6-main}, which is stated with the longhand notation. Our starting point will be the expression \cref{key-measure-change-cond} for $p(k)$.

As in the previous two sections, we have a main truncation parameter $M$ and various further key further thresholds $R(M)$, $T(M)$, $L(M)$ as described in \cref{truncation-threshold-defs}, as well as technical auxiliary thresholds $\ell_*(M)$, $V(M)$, $t_{\min}(M)$, $t_{\max}(M)$ as in the hierarchy \cref{param-heir}. We will always assume that $L(M)$ is much less than $n = \lfloor \log_2 k\rfloor$; later we will let $M \rightarrow \infty$ as $k \rightarrow \infty$ sufficiently slowly so that this holds.

To shorten some expressions it is convenient to set $N := \lceil n/2\rceil$ and $N' := \lfloor n/2\rfloor$ througout this section.

\subsection{Reducing to walks with large minima and small discrepancy} 
As before the \emph{discrepancy} of a pair of walks $(\beta,\beta')$, which we denote by $D(\beta,\beta')$, is $\beta(N) - \beta'(N')$. For $D \in \Z$ and $m, m' \in \Z_{\ge 0}$ denote
\begin{equation} \mathscr{B}_{D, m, m'}  := \{ (\beta, \beta') \in \Z^{\N} \times \Z^{\N}  : D(\beta,\beta') = D,  \min_{0 \le i \le N} \beta(i) = -m,\, \; \min_{0 \le i \le N'} \beta'(i) = -m'\},\label{bdm-def}\end{equation} where we declare $\beta(0) = \beta'(0) = 0$. It follows from \cref{lem13.7a} with $\mathcal{E}$ being the trivial (always satisfied) event that
\begin{equation}\label{bdm-prob} \mb{P} \big( (\boldbeta,\boldbeta') \in \mathscr{B}_{D,m,m'} \big) \ll (1 + m + m')^{3} (\log k)^{-3/2}.\end{equation}

 Write
\begin{equation}\label{edm-def} E_{D,m,m'} := \sum_{(\beta,\beta') \in \mathscr{B}_{D,m,m'}} \mb{P}\big((\boldbeta,\boldbeta') = (\beta,\beta')\big)  \mb{P}\big(k \in \Sigma(\mbf{A} \mid \beta,\beta')\big).\end{equation} 
Then \cref{key-measure-change-cond} gives

\begin{equation}\label{pk-walks-2} p(k) = (1 + o(1)) e^{\gamma(\frac{1}{\log 2} - 1)} k^{-\delta}  \sum_{|D| \le 20 \log n}(\log 2)^{D - \xi} \sum_{m, m' \ge 0} E_{D,m,m'}  +  O(k^{-\delta} (\log k)^{-2}).\end{equation}

We now wish to show that the most important contribution to this sum comes from $|D|, m,m' = O(1)$. In this direction, our first result shows that the contribution from large $\max(m,m')$ is negligible.

\begin{lemma}\label{lem102}
We have 
\[ \sum_{|D| \le 20 \log n}(\log 2)^D \sum_{\max(m, m') \ge (\log n)^2} E_{D,m,m'} \ll (\log k)^{-100}.\]
\end{lemma}
\begin{proof} 
We first bound the contribution from the (extremely rare) pairs of walks for which $\max_{i \le n} b_i$ is greater than or equal to $n$. Since $b_i \samedist \Pois(1)$, the probability that $(\boldbeta,\boldbeta')$ have this property is $\ll n \mb{P}(\Pois(1) \ge n) < e^{-n}$.  Bounding $\mb{P}(k \in \Sigma(\mbf{A} \mid \beta,\beta'))$ trivially by $1$, and $(\log 2)^{D}$ by $(\log 2)^{- 20 \log n} \ll n^{O(1)}$, we see that the contribution from these walks is negligible.

Henceforth, we restrict to pairs of walks with $\max_{i \le n} b_i \le n$. Fix a pair of walks $(\beta,\beta') \in \mathscr{B}_{D,m,m'}$, where $|D| \le 20 \log n$. We consider the contributions from $m \ge (\log n)^2$ and from $m' \ge (\log n)^2$ separately.

Suppose first that $m \ge (\log n)^2$. Let $q$, $0 \le q \le N$, be the smallest value such that $\beta(q) = -m$. Since $\xi_i \ge -1$ for all $i$, we have $\beta(q) \ge -q$, and so $q \ge m \ge (\log n)^2$. Writing $\mbf{a}_{i,j}$ for the elements of $(\mbf{A} \mid \beta,\beta')$ as usual (see \cref{a-list}), it follows using \cref{anti-concentration-aij} that 
\begin{equation*}\max \{ \sum_{i \le n - q; j \in [b_i]} \eps_{i,j} \mbf{a}_{i,j} : \eps_{i,j} \in \{0,1\}\} \ll n 2^{n - q} . \end{equation*} Since $q \ge (\log n)^2$, this is certainly $o(k)$.
Therefore the event $(k \in \Sigma(\mbf{A} \mid \beta,\beta'))$ is contained in the union of the events $(\sum_{n  - q < i \le n; j \in [b_i]} \eps_{i,j} \mbf{a}_{i,j} = k - x)$, over all $x \ll n 2^{n - q} = o(k)$ and all $2^{q + \beta(q)}$ choices of $(\eps_{i,j})_{n - q < i \le n; j \in [b_i]}$. Note moreover that for such an event to be nontrivial, we must have $\eps_{i,j} \ne 0$ for some $i \ge n - 2 \log_2 n$; otherwise, 
\[ \sum_{n - q < i \le n;j \in [b_i]} \eps_{i,j} \mbf{a}_{i,j} \le \sum_{i \le n - 2 \log_2 n; j \in [b_i] } \eps_{i,j} \mbf{a}_{i,j} \ll n2^{n - 2 \log_2 n} < k/2 < k - x.\] By revealing all $\mbf{a}_{i,j}$ except for such a pair of indices, we see that $\mb{P} (\sum_{n - q < i \le n; j \in [b_i]} \eps_{i,j} \mbf{a}_{i,j} = k - x) \ll 2^{2 \log_2 n - n} = n^2 2^{-n}$ for each choice of $x$ and the $\eps_{i,j}$. Summing over these choices gives
\[ \mb{P}(k \in \Sigma(\mbf{A} \mid \beta,\beta')) \ll n 2^{n - q} \cdot 2^{q + \beta(q)} \cdot n^2 2^{-n} \ll n^3 2^{-(\log n)^2},\] and so $(\log 2)^D \mb{P}(k \in \Sigma(\mbf{A} \mid \beta,\beta')) \ll n^{O(1)} 2^{-(\log n)^2} \ll (\log k)^{-100}$.
Summing over all $(\beta,\beta')$ lying in some $\mathscr{B}_{D, m, m'}$ with $|D| \le 20 \log n$ and $m \ge (\log n)^2$, weighted by $\mb{P}\big((\boldbeta, \boldbeta') = (\beta, \beta')\big)$, gives a bound of the strength claimed in \cref{lem102}.

Now we consider the possibility that $m' \ge (\log n)^2$; the analysis here is very similar. Let $q'$, $0 \le q' \le N'$, be the smallest value such that $\beta'(q') = -m'$. Similarly to before, it follows using \cref{anti-concentration-aij} that
\begin{equation*} \max \{ \sum_{i \le q' ; j \in [b_i]} \eps_{i,j} \mbf{a}_{i,j}: \eps_{i,j} \in \{0,1\} \} \ll n 2^{q'}.
\end{equation*}
Since $q' \le N'$, this is vastly smaller than $k$. The event $(k \in \Sigma(\mbf{A} \mid \beta,\beta'))$ is contained in the union of the events $\big(\sum_{q' < i \le n;j \in [b_i]} \eps_{i,j} \mbf{a}_{i,j} = k - x\big)$, over all $x \ll n 2^{q'}$ and all $2^{n + D - q' + \beta'(q')}$ choices of $(\eps_{i,j})_{q' < i \le n; j \in [b_i]}$, noting here that $\sum_{q' < i \le n} b_i = \sum_{i = 1}^n b_i - \sum_{i = 1}^{q'} b_i = n + D - (q' - \beta'(q'))$. As before, $\mb{P}(\sum_{q' < i \le n;  j \in [b_i]} \eps_{i,j} \mbf{a}_{i,j} = k - x) \ll n^2 2^{-n}$ for each choice of $x$ and the $\eps_{i,j}$. Summing over these choices gives
\[ \mb{P}(k \in \Sigma(\mbf{A} \mid \beta,\beta')) \ll n 2^{q'} \cdot 2^{n +D - q' + \beta'(q')} \cdot n^2 2^{-n} \ll n^{O(1)} 2^{-(\log n)^2},\] using here that $|D| \le 20 \log n$. Summing over all $(\beta, \beta')$ lying in some $\mathscr{B}_{D,m,m'}$ with $|D| \le 20 \log n$ and $m' \ge (\log n)^2$, weighted by $\mb{P}\big((\boldbeta,\boldbeta') = (\beta,\beta')\big)$, again gives a bound of the strength claimed in \cref{lem102}. This completes the proof.
\end{proof}

To prepare for our study of the more modest values of $m$, $m'$, we establish the following lemma, which should be thought of as very roughly suggesting that $\mb{P}(k \in \Sigma(\mbf{A} \mid \beta,\beta'))$ behaves like $\min( 2^{-m}, 2^{-m'})$ when $(\beta, \beta') \in \mathscr{B}_{D,m,m'}$.

\begin{lemma}\label{lem10.3} Suppose that $D \in \Z$ and that $m, m' \in \Z_{\ge 0}$. Let $\beta, \beta'$ be walks such that $(\beta,\beta') \in \mathscr{B}_{D,m,m'}$. Let $q$, $0 \le q \le N$, be minimal such that $\beta(q) = -m$, and $q'$ be minimal such that $\beta'(q') = -m'$. Let $r$ be the least positive integer such that $\beta$ is $r$-bounded to length $N$ \textup{(}see \cref{r-beta-def}\textup{)}, and let $r'$ be defined similarly for $\beta'$. Then we have \begin{equation}\label{lem93-1} \mb{P}(k \in \Sigma(\mbf{A} \mid \beta,\beta')) \ll r^{O(1)} (1 + q)^{\kappa} 2^{-m},\end{equation} and 
\begin{equation}\label{lem93-2} \mb{P}(k \in \Sigma(\mbf{A} \mid \beta,\beta')) \ll (r')^{O(1)} (1 + q')^{\kappa} 2^{D-m'}.\end{equation} 
\end{lemma}
\begin{proof} By replacing $r$ by $\max(100, r)$ if necessary (and similarly for $r'$) we may assume $r, r' \ge 100$. We begin with the first bound \cref{lem93-1}. 
If $m \le 2 \log_2 r$, the bound is trivial, so assume henceforth that $m \ge 2 \log_2 r$. We may also assume that $m$ is sufficiently large. In particular $q \ge 1$. The argument is very similar to that in the previous lemma. Since $\beta$ is $r$-bounded, we have $|b_{n + 1 - i}| \le  r i^{\kappa} + 1$ and so using \cref{dyadic-aij} we have
\begin{equation}\label{crude-104} \max \{ \sum_{i \le n - q; j \in [b_i]} \eps_{i,j} \mbf{a}_{i,j} : \eps_{i,j} \in \{0,1\}\} \ll r q^{\kappa} 2^{n - q} . \end{equation}
As in the previous lemma, we have $q \ge m$, and so $q \ge 2 \log_2 m$; therefore (since also $q$ is sufficiently large) the bound in \cref{crude-104} is $\le k/2$. Now we proceed as in the previous lemma, writing the event $(k \in \Sigma(\mbf{A} \mid \beta,\beta'))$ as the union of the events $(\sum_{n - q < i \le n; j \in [b_i]} \eps_{i,j} \mbf{a}_{i,j} = k - x)$, over all $x \ll r q^{\kappa} 2^{n - q}  < k/2$ and all $2^{q + \beta(q)}$ choices of $(\eps_{i,j})_{n - q < i \le n; j \in [b_i]}$.

For such an event to be nontrivial, we must have $\eps_{i,j} \ne 0$ for some $i \ge n -  2 \log_2 r$; otherwise, 
\[ \sum_{n - q < i \le n;j \in [b_i]} \eps_{i,j} \mbf{a}_{i,j} \le \sum_{i \le n -2 \log_2 r; j \in [b_i] } \eps_{i,j} \mbf{a}_{i,j} \ll  \sum_{i \le n -  2 \log_2 r} (r i^{\kappa} + 1) 2^{i}  < k/2 < k - x.\]  
By revealing all $\mbf{a}_{i,j}$ except for such a pair of indices, we see that $\mb{P} (\sum_{n - q < i \le n; j \in [b_i]} \eps_{i,j} \mbf{a}_{i,j} = k - x) \ll 2^{2 \log_2 r - n} \ll r^{O(1)} 2^{-n}$ for each choice of $x$ and the $\eps_{i,j}$. Summing over these choices gives
\[ \mb{P}(k \in \Sigma(\mbf{A} \mid \beta,\beta')) \ll r q^{\kappa} 2^{n - q} \cdot 2^{q + \beta(q)} \cdot r^{O(1)} 2^{-n} \ll r^{O(1)} q^{\kappa} 2^{-m} ,\] which is the required result.

Now we turn to the second bound, whose proof is again very similar. The $r'$-boundedness of $\beta'$ gives $|b_{i}| \le 1 + r'i^{\kappa}$, and so using \cref{anti-concentration-aij} we have
\begin{equation*} \max \{ \sum_{i \le q'; j \in [b_i]} \eps_{i,j} \mbf{a}_{i,j} : \eps_{i,j} \in \{0,1\}\} \ll r'(q')^{\kappa} 2^{q'} . \end{equation*} The rest of the argument proceeds as before; as in the proof of \cref{lem102} the number of choices for $(\eps_{i,j})_{q' < i \le n; j \in [b_i]}$ is $2^{n + D - q' + \beta'(q')}$. We obtain
\[ \mb{P}(k \in \Sigma(\mbf{A} \mid \beta,\beta')) \ll r' (q')^{\kappa} 2^{q'} \cdot 2^{n + D - q' + \beta'(q')} \cdot (r')^{O(1)} 2^{-n} = (r')^{O(1)} (q')^{\kappa} 2^{D - m'},\] which is the required result.
\end{proof}

It follows from \cref{lem93-1} and the definition \cref{edm-def} that
\begin{align*} E_{D,m,m'} & \ll 2^{-m}\sum_{(\beta,\beta') \in \mathscr{B}_{D,m,m'}} \mb{P}\big((\boldbeta,\boldbeta') = (\beta,\beta')\big) r_{\beta}^{O(1)} (1 + q_{\beta})^{\kappa}  \\ & = 2^{-m}\sum_{\substack{0 \le q \le N \\ r \ge 1}} r^{O(1)} (1 +q)^{\kappa} \mb{P}(\min_{0 \le i \le N} \boldbeta(i) = -m,\; \min_{0 \le i \le N} \boldbeta'(i) = -m', \, D(\boldbeta, \boldbeta') = D,\\[-8pt] & \hspace{8cm} R_{\le N}(\boldbeta) = r, \, Q_m(\boldbeta) = q) ,\end{align*} where  $q_{\beta},r_{\beta}$ are the $q,r$ associated to $\beta$ as in the statement of \cref{lem10.3}, $R_{\le N}(\beta)$ is the least positive integer $r$ such that $\beta$ is $r$-bounded to length $N$, and $Q_m(\boldbeta)$ is the least non-negative integer such that $\boldbeta(q) = -m$. To estimate this, we apply the first estimate in \cref{lem13.7a} with the event $\mathcal{E}$ being that $R_{\le N}(\boldbeta) = r$ and that $Q_m(\boldbeta) = q$. This gives (recalling that $n \sim \log_2 k$)
\begin{align*} E_{D,m,m'}   \ll (1+ m + m')^2 2^{-m} (\log k)^{-1} \sum_{\substack{0 \le q \le N \\ r \ge 1}} r^{O(1)} & (1 + q)^{\kappa}\mb{P}\big(R_{\le N}(\boldbeta) = r,\\[-8pt] & Q_m(\boldbeta) = q, \,  \min_{0 \le i \le N} \boldbeta(i)= -m\big).\end{align*}
Applying the AM-GM inequality $r^{O(1)} (1 + q)^{\kappa} \le r^{O(1)} + (1 + q)^{2\kappa}$ and summing out the redundant variables gives
\begin{align}\nonumber E_{D, m, m'} & \ll (1 + m + m')^2 2^{-m} (\log k)^{-1} \sum_{0 \le q \le N} (1 + q)^{2\kappa} \mb{P}\big(Q_m(\boldbeta) = q, \, \min_{0 \le i \le N} \boldbeta(i) = -m\big) \\ & + (1 + m + m')^2 2^{-m} (\log k)^{-1}\sum_{r \ge 1} r^{O(1)} \mb{P}\big(R_{\le N}(\boldbeta) = r, \, \min_{0 \le i \le N} \boldbeta(i) = -m\big).\label{edm-est-sec13}\end{align}
By \cref{min-pos-lem} (relaxing the condition $Q_m(\boldbeta) = q$ to $\boldbeta(q) = -m$, and the condition $\min_{0 \le i \le N} \boldbeta(i) = -m$ to $\min_{1 \le i \le N} \boldbeta(i) \ge -m$), we have 
\begin{align} \nonumber \sum_{0 \le q \le N} (1 + q)^{2\kappa} \mb{P}\big(Q_m(\boldbeta) = q, \, \min_{0 \le i \le N} \boldbeta(i) = -m\big) & \ll (1 + m) \sum_{0 \le q \le N} (1 + q)^{2\kappa - 3/2} (N + 1 - q)^{-1/2} \\ & \label{q-estimate-sec13} \ll (1 + m) N^{-1/2} \ll (1 + m) (\log k)^{-1/2}. \end{align}
By \cref{walks-bd-min-est} (relaxing the condition $R_{\le N}(\boldbeta) = r$ to $R_{\le N}(\boldbeta) > r - 1$) we have
\begin{align}\nonumber \sum_{r \ge 1} r^{O(1)}  \mb{P}\big( R_{\le N}(\boldbeta) = r, \, & \min_{0 \le i \le N} \boldbeta(i) = -m\big) \ll N^{-1/2} \sum_{r \ge 1} r^{O(1)}(m + r) e^{-r} \\ & \ll (1 + m) N^{-1/2} \ll (1 + m) (\log k)^{-1/2}.\label{r-est-sec13}\end{align}
It follows from \cref{edm-est-sec13,q-estimate-sec13,r-est-sec13} that 
\begin{equation}\label{edm1a} E_{D,m,m'}   \ll (1 + m + m')^{3} 2^{-m} (\log k)^{-3/2} .\end{equation}

An essentially identical analysis starting from \cref{lem93-2} yields
\begin{equation}\label{edm2} E_{D,m,m'}   \ll (1+ m + m')^3 2^{D-m'} (\log k)^{-3/2} .\end{equation}
From this we deduce the following estimate concerning the contribution to \cref{pk-walks-2} from walks with small minima or large discrepancy.
\begin{lemma}
The sum of $(\log 2)^D E_{D,m,m'}$ over all $D, m,m'$ with $|D| \le 20 \log n$, $0 \le m, m' \le (\log n)^2$ and either $|D| \ge M/10$ or $\max(m,m') \ge M$ is bounded by $\ll  e^{-\Omega(M)} (\log k)^{-3/2}$.
\end{lemma}
\begin{proof} Given the estimates \cref{edm1a,edm2}, the task is to show that 
\begin{equation} \label{des-edmm}\sum_{D,m,m'} (\log 2)^D (1 + m + m')^3 \min (2^{-m}, 2^{D - m'}) \ll e^{-\Omega(M)},\end{equation}
where the sum is over triples $D,m, m'$, $D \in \Z$, $m, m' \in \Z_{\ge 0}$, satisfying one of (1) $|D| \ge M/10$, (2) $|D| \le M/10$ but $\max(m,m') \ge M$. \vspace*{8pt}

\emph{Contribution from} (1). First we consider the contribution from $D \ge M/10$. For this we use the estimate $\sum_{m,m' \ge 0} (1 + m + m')^{3} \min( 2^{-m}, 2^{D - m'}) \ll D^{4}$. (To verify this, first note that the contribution from $m' \le 2D$ is $\ll D\sum_{m \ge 0} (D + m)^{3} 2^{-m}$, which is acceptable; for $m' \ge 2D$ we have $\min (2^{-m}, 2^{D - m'}) \le \min (2^{-m/2}, 2^{-m'/2})$, so by symmetry the sum is $\ll \sum_{m \ge m'} (1 + m + m')^{3} 2^{-m/2} \ll \sum_{m} (1 + m)^{4} 2^{-m/2}$, which is again acceptable.) Including the factor $(\log 2)^D$ and summing this estimate over $D \ge M/10$ gives $\sum_{D \ge M/10} (\log 2)^D D^{4} \ll (M/10)^{4} (\log 2)^{M/10}$, which is acceptable.

Next we consider the contribution from $D \le -M/10$. For this we use the estimate $\sum_{m, m' \ge 0} (1 + m + m')^{3} \min(2^{-m}, 2^{D - m'}) \ll |D|^{4} 2^{-|D|}$. (To verify this, first note that the contribution from $m \le 2|D|$ is $\ll |D|\sum_{m' \ge 0} (|D| + m')^{3} 2^{-|D| - m'}$, which is acceptable; for $m \ge 2|D|$ we have $\min(2^{-m}, 2^{D - m'}) \le 2^{-|D|} \min(2^{-m/2}, 2^{-m'/2})$, and we can conclude as before.) Including the factor $(\log 2)^D$ and summing this estimate over $D \le -M/10$ gives $\ll \sum_{D \le -M/10} |D|^{4}(\log 2)^{-|D|} 2^{-|D|} \ll (M/10)^{4}(2 \log 2)^{-M/10}$, so this contribution is also acceptable.\vspace*{8pt}

\emph{Contribution from} (2). The contribution to \cref{des-edmm} from $|D| \le M/10$ and $\max(m,m') \ge M$ is
\[ \ll M (2/\log 2)^{M/10} \sum_{\max(m,m') \ge M}(1 + m + m')^{3} \min(2^{-m}, 2^{- m'}).\]
By symmetry it is enough to sum this over $m \ge \max(m', M)$. We then note, in turn, that 
\[ \sum_{m' \le m} (1 + m + m')^{3} \min(2^{-m}, 2^{-m'}) \ll (1 + m)^{4} 2^{-m},\] then that $\sum_{m \ge M}  (1 + m)^{4} 2^{-m} \ll M^{4} 2^{-M}$.  The result follows.
\end{proof}

It follows, substituting into \cref{pk-walks-2}, that 
\begin{align} \nonumber p(k) & =  (1 + o(1)) e^{\gamma(\frac{1}{\log 2} - 1)} k^{-\delta}  \sum_{|D| \le M/10} (\log 2)^{D - \xi}\sum_{m,m' \le M} E_{D,m,m'} + \\[5pt] & \hspace{8cm}  +  O \big( e^{-\Omega(M)}k^{-\delta} (\log k)^{-3/2}\big).\label{pk-walks-3}\end{align}
This is the desired reduction to the consideration of walks with large minimum and small discrepancy.

\subsection{Passing to `good' walks}

We now continue with the analysis, the next aim being to refine \cref{pk-walks-3} so that $E_{D,m,m'}$ (see \cref{edm-def}) only sums over walks which are $R(M)$-bounded, $T(M)$-positive and for which $\beta$ has a jump step, at which point we will be able to apply \cref{sec6-main}.

Recall the definition \cref{bdm-def} of $\mathscr{B}_{D, m, m'}$. We now define two slightly different subsets of $\mathscr{B}_{D,m,m'}$. Let $C_0$ be a (sufficiently large) constant to be specified later. Set
\[ \mathscr{B}^{(1)}_{D,m,m'} := \{(\beta, \beta') \in \mathscr{B}_{D,m,m'} : \beta,\beta' \; \mbox{are $R(M)$-bounded and $T(M)$-positive to length $L(M)$ and}\]
\begin{equation}\label{b1-def} \mbox{there are $q, q' \le e^{C_0 M}$ with $\beta(q) = -m$, $\beta'(q') = -m'$}\}.\end{equation} 
By a \emph{jump step at scale $M$} we mean a $V(M)$-jump step in the sense of \cref{jump-step-7} which lies in the interval $[t_{\min}(M), t_{\max}(M)]$, with the parameters $V(M), t_{\min}(M), t_{\max}(M)$ as in the hierarchy \cref{param-heir}. Define $\mathscr{B}^{(2)}_{D,m,m'}$ to be the set of all $(\beta, \beta') \in \mathscr{B}_{D,m,m'}$ such that $\beta,\beta'$ are $R(M)$-bounded and $T(M)$-positive to lengths $N$, $N'$ respectively, there are $q, q' \le e^{C_0 M}$ with $\beta(q) = -m$, $\beta'(q') = -m'$, and such that $\beta$ has a jump step at scale $M$.
We clearly have
\begin{equation}\label{nesting} \mathscr{B}^{(2)}_{D,m,m'} \subseteq \mathscr{B}^{(1)}_{D,m,m'} \subseteq \mathscr{B}_{D,m,m'}.\end{equation} 
We claim that both these sets are `generic' in the following sense.

\begin{lemma} \label{exceptions-bd} Suppose that $m , m' \le M$ and that $|D| \le M/10$. We have for $i = 1,2$ that
\[ \mb{P} \big((\boldbeta, \boldbeta') \in \mathscr{B}_{D,m,m'} \setminus \mathscr{B}^{(i)}_{D,m,m'}\big) \ll e^{-10M} (\log k)^{-3/2}.\]
\end{lemma}
\begin{proof}
By \cref{nesting} it is enough to handle the case $i = 2$. We use \cref{lem13.7a} with $\mathcal{E} = \bigcup_{i = 1}^4 \mathcal{E}^{(i)}$ being the event that at least one of the following occurs: (1) $\boldbeta$ is not $R(M)$-bounded to length $N$; (2) $\boldbeta$ is not $T(M)$-positive to length $N$; (3) there is $q$ with $e^{C_0 M} \le q\le N$ such $\beta(q) = -m$ and (4) $\boldbeta$ has no jump step at scale $M$. There are also (except for (4)) symmetric events involving $\boldbeta'$ to be considered, which may be handled in exactly the same way and about which we shall make no further comment. Applying \cref{lem13.7a}, we see that it is enough to show that 
\begin{equation}\label{except-prob} \mb{P}\big( \boldbeta \in \mathcal{E}^{(i)} ,\, \min_{1 \le i \le N} \boldbeta(i) \ge -m\big) \ll e^{-10M} N^{-1/2}.\end{equation}
We handle the case $i = 1,2,3,4$ in turn. \vspace*{8pt}

\emph{Case $i = 1$.} By \cref{walks-bd-min-est} it follows that \cref{except-prob} is $\ll R(M) e^{-R(M)} N^{-1/2}$. Recalling from \cref{truncation-threshold-defs} that $R(M) = 20 M$, this is acceptable.\vspace*{8pt}

\emph{Case $i = 2$.} By \cref{pos-paths-cor} it follows that \cref{except-prob} is $\ll M^{O(1)} T(M)^{-\eta/2}N^{-1/2}$. Recalling from \cref{truncation-threshold-defs} that $T(M) = e^{40M/\eta}$, this is acceptable.\vspace*{8pt}

\emph{Case $i = 3$.} By \cref{min-pos-lem}, \cref{except-prob} is 
\[ \ll M \sum_{e^{C_0 M} \le q \le N} q^{-3/2}(N + 1 - q)^{-1/2} \ll M e^{-C_0 M/2}N^{-1/2},\] which is acceptable if $C_0 > 20$.
\vspace*{8pt}

\emph{Case $i = 4$.} In this case the bound follows from  \cref{lemma7.19} with $\delta = e^{-10M}$ and \cref{prop81-3-ii}, provided parameters are chosen appropriately. Note here that for any choice of $V(M)$ there are appropriate choices of $t_{\min}(M)$ and $t_{\max}(M)$ for which \cref{tmin-threshold,massive-scales} both hold.
\end{proof}

For $i = 1,2$ define $E^{(i)}_{D,m,m'}$ analogously to $E_{D,m,m'}$ (cf. \cref{edm-def}), that is to say 
write
\begin{equation*}E^{(i)}_{D,m,m'} := \sum_{(\beta,\beta') \in \mathscr{B}^{(i)}_{D,m,m'}} \mb{P}\big((\boldbeta,\boldbeta') = (\beta,\beta')\big)  \mb{P}(k \in \Sigma(\mbf{A} \mid \beta,\beta')).\end{equation*}

If we replace $E_{D,m,m'}$ by $E^{(i)}_{D,m,m'}$ in \cref{pk-walks-3}, the error is bounded by 
\begin{equation}\label{be-bet-tobound} \ll k^{-\delta} \sum_{|D| \le M/10} (\log 2)^D \sum_{m,m' \le M}  \mb{P} ((\boldbeta, \boldbeta') \in \mathscr{B}_{D,m,m'} \setminus \mathscr{B}^{(i)}_{D,m,m'}) \ll e^{-5M} k^{-\delta} (\log k)^{-3/2}\end{equation}
by \cref{exceptions-bd}. Thus we may replace $E_{D,m,m'}$ by $E^{(i)}_{D,m,m'}$ in \cref{pk-walks-3} with changes only in the implied constant in the error term. 

Define also
\begin{equation*} \tilde E^{(i)}_{D,m,m'} := \sum_{(\beta,\beta') \in \mathscr{B}^{(i)}_{D,m,m'}} \mb{P}\big((\boldbeta,\boldbeta') = (\beta,\beta')\big)  \Big( 1 - \mb{E} \exp\big({-}2^{D - \xi}  \varrho^*_{\mbf{u} \mid \beta}(L(M))\tau^*_{\mbf{x} \mid \beta'}(L(M))\big)\Big).\end{equation*}
The purpose of this definition is that we have replaced $\mb{P}(k \in \Sigma(\mbf{A} \mid \beta,\beta'))$ by the main term in the formula obtained in \cref{sec6-main}. Indeed, by \cref{sec6-main} it follows that 
\[ E^{(2)}_{D,m,m'} = \tilde E^{(2)}_{D,m,m'} + O(\ell_*(M)^{-1}  \mb{P}((\boldbeta,\boldbeta') \in \mathscr{B}_{D,m,m'} )). \]
 Thus, using \cref{bdm-prob}, we see that the error if we further replace $E^{(2)}_{D,m,m'}$ by $\tilde E^{(2)}_{D,m,m'}$ in \cref{pk-walks-3} is bounded by 
\begin{align*}
& \ll \ell_*(M)^{-1} k^{-\delta}\sum_{|D| \le M/10, m, m' \le M} (\log 2)^D \mb{P}\big( (\boldbeta,\boldbeta') \in \mathscr{B}_{D,m,m'}\big) \\ & 
  \ll \ell_*(M)^{-1} k^{-\delta}\sum_{|D| \le M/10, m, m' \le M} (\log 2)^D (1 + m + m')^{3} (\log k)^{-3/2} \ll e^{-10M} k^{-\delta}(\log k)^{-3/2}   
\end{align*}
if an appropriate choice of $\ell_*(M)$ is made in the parameter hierarchy \cref{param-heir}. Finally, we replace $\tilde E^{(2)}_{D,m,m'}$ by $\tilde E^{(1)}_{D,m,m'}$ in \cref{pk-walks-3}. The error in doing this is again bounded as in \cref{be-bet-tobound} by a very similar argument using \cref{exceptions-bd}.

Let us state the final conclusion of the above analysis, and the only one we will need subsequently, which is that with negligible error we can replace $E_{D,m,m'}$ by $\tilde E^{(1)}_{D,m,m'}$ in \cref{pk-walks-3}, and therefore we have
\begin{equation}\label{pk-walks-4}  p(k) =  (c_0 + o(1))  k^{-\delta} (\log k)^{-3/2}  \big( f_{M,n}(\xi) +  O ( e^{-\Omega(M)}) \big),
\end{equation}
where $c_0$ is the constant in \cref{thm:main}, and for $0 \le \theta \le 1$ we define $f_{M,n}(\theta)$ to be
\begin{equation}\label{fmn-def} 
\big(\frac{2}{\pi}\big)^{1/2} n^{3/2}\!\!\!\!\!\sum_{\substack{|D| \le M/10  \\ m,m' \le M \\ (\beta,\beta') \in \mathscr{B}^{(1)}_{D,m,m'}} }\hspace{-0.3cm}(\log 2)^{D - \theta}  \mb{P} \big( (\boldbeta,\boldbeta') = (\beta,\beta') \big) \big( 1 - \mb{E} \exp\big({-}2^{D - \theta}  \varrho^*_{\mbf{u} \mid \beta}(L) \tau^*_{\mbf{x} \mid \beta'}(L)\big)\big) ,\end{equation}
with $L = L(M)$ and $\mathscr{B}^{(1)}_{D,m,m'}$ defined in \cref{b1-def}. (Here we have used that $(\log k)^{3/2} = (1 + o(1)) (\log 2)^{3/2} n^{3/2}$, but it is natural to write things in terms of $\log k$ in \cref{pk-walks-4} and in terms of $n$ in \cref{fmn-def}.)

\section{Decoupling the walks}\label{decouple-sec}

The notation in this section follows that of the previous three sections and in particular we continue to write $L = L(M)$ and $N = \lceil n/2\rceil$, $N' = \lfloor n/2\rfloor$ for brevity. It is also convenient to write $\beta_{\le L}$ for $(\beta(i))_{1 \le i \le L}$, and similarly for $\beta'$. We always have $\beta(0) = \beta'(0) = 0$. Our aim now is to show that the function $f_{M,n}$ in \cref{fmn-def} essentially does not depend on $n$. The average in \cref{fmn-def} is over pairs of walks $(\beta,\beta')$ satisfying the conditions \begin{equation}\label{couple-relation} D(\beta,\beta') := \beta(N) - \beta'(N') = D, \quad \min_{0 \le i \le N} \beta(i) = -m, \quad \mbox{and} \quad \min_{0 \le i \le N'} \beta'(i) = -m'.\end{equation}
Note that membership of $(\beta, \beta')$ in $\mathscr{B}^{(1)}_{D,m,m'}$ is defined by the conditions \cref{couple-relation} and the  following four conditions which depend only on each of $\beta,\beta'$ individually and in fact only on the values $\beta_{\le L}, \beta'_{\le L}$, namely 
\begin{enumerate}
    \item $\beta$ is $R(M)$-bounded to length $L$;
    \item $\beta$ is $T(M)$-positive to length $L$;
    \item $\min_{0 \le i \le L} \beta(i) = -m$;
    \item There is some $q \le e^{C_0 M}$ such that $\beta(q) = -m$,
\end{enumerate}
and similar conditions on $\beta'$. Note that the redundant condition (3) is deliberately retained here as it will be relevant to a later definition. It is also the case that the random variables $ \varrho^*_{\mbf{u} \mid \beta}(L)$, $\tau^*_{\mbf{x} \mid \beta'}(L)$ depend only $\beta_{\le L}$ and $\beta'_{\le L}$ respectively.

Define $\mathscr{C}_{m,M}$ to be the set of all tuples $\vec{c} = (c_1,\dots, c_L) \in \Z^L$ such that $\beta$ satisfies (1) to (4) above iff $\beta_{\le L} \in \mathscr{C}_{m,M}$. Define $\mathscr{W}_{m,M}$ to be the set of all walks $\beta$ satisfying (1) to (4) above; thus $\beta \in \mathscr{W}_{m,M}$ if and only if $\beta_{\le L} \in \mathscr{C}_{m,M}$.

We may then write \cref{fmn-def} as
\begin{align*}  & f_{M,n}(\theta) = \big(\frac{2}{\pi}\big)^{1/2}n^{3/2}\sum_{\substack{|D| \le M/10 \\ m,m' \le M}}\sum_{\substack{\vec{c} \in \mathscr{C}_{m, M} \\ \vec{c}' \in \mathscr{C}_{m',M}}}   (\log 2)^{
D - \theta}  \mb{E} \big(1 -  \exp \big({-}2^{D - \theta}  \varrho^*_{\mbf{u} \mid \vec{c}}(L) \tau^*_{\mbf{x} \mid \vec{c}'}(L)\big)\big)  \times \\ & \times \mb{P} \big(\min_{0 \le i \le N} \boldbeta(i) = -m,\;  \min_{0 \le i \le N'} \boldbeta'(i) = -m',\;  \boldbeta(N) - \boldbeta'(N') = D,\, \boldbeta_{\le L} = \vec{c}, \, \boldbeta'_{\le L} = \vec{c}'\big).\end{align*} Applying \cref{decoupling}, we obtain
\begin{align} \nonumber f_{M,n}(\theta) & = \sum_{\substack{|D| \le M/10 \\ m,m' \le M}} \sum_{\substack{\vec{c} \in \mathscr{C}_{m, M} \\ \vec{c}' \in \mathscr{C}_{m',M}}}   (\log 2)^{
D - \theta}   \mb{E} \big(1 -  \exp \big({-}2^{D - \theta}  \varrho^*_{\mbf{u} \mid \vec{c}}(L) \tau^*_{\mbf{x} \mid \vec{c}'}(L)\big)\big)  \times \\[-0.5cm] & \hspace{4cm} \times \big(  h(m + c_L) h'(m' + c'_L) + O(n^{- \eps_1})\big) \mb{P}\big( \boldbeta_{ \le L} = \vec{c}, \, \boldbeta'_{\le L} = \vec{c}'\big) \nonumber \\ &  = f_M(\theta) + O(n^{-\eps_1/2}),\label{fmn-fm}\end{align}
where $\eps_1$ is the (positive) exponent in \cref{decoupling} and 
\begin{equation} f_M(\theta) := \sum_{\substack{|D| \le M/10 \\ m,m' \le M }} (\log 2)^{
D - \theta}  \iint h(m + \beta(L)) h'(m' + \beta'(L)) \Psi_{m,m',M,D,\theta}(\beta,\beta')\, \mathrm{d}\mb{P}(\beta) \, \mathrm{d}\mb{P}'(\beta') \label{FM-form}\end{equation}
with
\begin{equation}\label{psi-def} \Psi_{m,m', M,D,\theta}(\beta,\beta') := \mb{E} \big(1 - \exp\big({-}2^{D - \theta}  \varrho^*_{\mbf{u} \mid \beta}(L) \tau^*_{\mbf{x} \mid \beta'}(L)\big)\big) 1_{\mathscr{W}_{m,M}}(\beta)1_{\mathscr{W}_{m',M}}(\beta'),\end{equation} and the integrations being with respect to the path measures of $\boldbeta$, $\boldbeta'$ respectively.
We comment briefly on the application of \cref{decoupling}. The required condition $\max c_i, \max c'_i \le N^{\eps_1}$ is comfortably implied by the definition of $\mathscr{W}_{m, M}$, in particular by the fact that $\vec{c}$ must be the initial segment of an $R(M)$-bounded walk. We also note that the $O(n^{-\eps_1/2})$ error term in \cref{fmn-fm} follows using the fact that $(\log 2)^{-D}$ is much smaller than $n^{\eps_1/2}$ (and that the $\mathscr{W}_{m,M}$ are disjoint as $m$ varies). 

The crucial point to note is that $f_M(\theta)$ no longer depends on $n$; in other words $f_M(\theta)$ is a `universal' object which does not see $k, n , \xi$.

Combining \cref{pk-walks-4,fmn-def,fmn-fm} gives
\begin{equation}\label{pk-walks-5}  p(k) =  (c_0 + o(1))  k^{-\delta} (\log k)^{-3/2}  \big( f_{M}(\xi) +  O ( e^{-\Omega(M)}) \big),
\end{equation}
where $\xi = \{\log_2 k\}$ as usual, the $o(1)$ denotes a quantity tending to zero as $k \rightarrow \infty$, and $M$ is any quantity satisfying $M \le M_0(k)$ for some fixed function $M_0$ with $\lim_{k \rightarrow \infty} M_0(k)=\infty$. The function $M_0$ is specified by the requirement that it must be possible to choose parameters as in \cref{param-heir}.

Let us pause to remark on how \cref{powers-two-cor} already follows from this and the existing literature.

\begin{proof}[Proof of \cref{powers-two-cor}]
If $k$ is a power of two then $\xi = \{ \log_2 k\} = 0$. By applying \cref{pk-walks-5} for fixed $M$ and $k \rightarrow \infty$, we see that $|f_M(0) - f_{M+1}(0)| \ll e^{-\Omega(M)}$, and therefore we have $f_M(0) = c_1 + O(e^{-\Omega(M)})$ for some $c_1$. Substituting in to \cref{pk-walks-5}, the corollary follows (with $C = c_0 c_1$) except we have not shown that $C > 0$. This does not follow directly from this argument, but is a consequence of \cite{EFG16} and will also be shown in a self-contained fashion in \cref{mu-not-zero,mu-prime-not-zero} below.
\end{proof}

\section{Proof of the main theorem}\label{sec15-mainthm-proof}

In this section we prove \cref{thm:main}. The main task is to show that the measures $\mu, \mu'$ exist and are not zero. The more detailed characterisations described in \cref{prop16.2,sec17-main} must wait until later.

Recall that $\mathscr{W}_{m,M}$ denotes the set of walks $\beta$ satisfying conditions (1) -- (4) at the start of \cref{decouple-sec}. Throughout the section $c := -\frac{\log \log 2}{\log 2}$.

\subsection{Approximating $f_M$ by a convolution} The starting point for this section is \cref{pk-walks-5}, with $f_M$ defined in \cref{FM-form}. Our aim is to show convergence of $f_M$ to a function $f$, which we will ultimately describe as a convolution $g \ast \mu \ast \mu'$. As a stepping stone to this result, our main goal in this subsection is to approximate $f_M$ by a convolution; in particular, the function $g$ (defined in \cref{g-funct-def}) will enter the picture for the first time. Specifically, we will prove the following.

\begin{lemma}
Let $M$ be a sufficiently large positive integer. There is a function $\tilde f_M \in C^{\infty}(\R/\Z)$ with 
\begin{equation}\label{fm-tildefm-approx} \Vert f_M - \tilde f_M \Vert_{\infty} \ll e^{-\Omega(M)}\end{equation}
and such that $\tilde f_M$ may be written as
\begin{equation}\label{fm-conv} \tilde f_M = g \ast \mu_M \ast \mu'_M,\end{equation}
$g \in C^{\infty}(\R/\Z)$ the function in \cref{g-funct-def}, $\mu_M$ is the Borel measure on $\R/\Z$ given by
\begin{equation}\label{mu-m-def} \int \psi \, \mathrm{d}\mu_{M} =  \sum_{m \le M}\int_{\Z^\N_{\ge -m}} \mb{E} \varrho^*_{\mbf{u} \mid \beta}(L(M))^c \psi\big(\log_2 \varrho^*_{\mbf{u} \mid \beta}(L(M))\big)1_{\mathscr{W}_{m,M}}(\beta)\, \mathrm{d}\omega_m(\beta)\end{equation} for test functions $\psi \in C(\R/\Z)$, and $\mu'_M$ is defined by
\begin{equation}\label{mu-m-prime-def} \int \psi \, \mathrm{d}\mu'_{M} =  \sum_{m \le M}\int_{\Z^\N_{\ge -m}} \mb{E} \tau^*_{\mbf{u} \mid \beta'}(L(M))^c \psi\big(\log_2 \tau^*_{\mbf{u} \mid \beta'}(L(M))\big)1_{\mathscr{W}_{m,M}}(\beta')\, \mathrm{d}\omega'_m(\beta').\end{equation} Here, the convention is that $x^c \psi(\log_2 x) = 0$ if $x = 0$.
\end{lemma}
\begin{proof}
We again write $L = L(M)$ for brevity.
With \cref{FM-form} in mind, consider first any expression of the form
\[ \iint h(\beta(L) + m) h'(\beta'(L) + m') \Psi(\beta, \beta')\, \mathrm{d}\mb{P}(\beta)\, \mathrm{d}\mb{P}'(\beta')\] where $\Psi : \Z^\N_{\ge -m} \times \Z^\N_{\ge -m'} \rightarrow \R$ is such that $\Psi(\beta,\beta')$ depends only on the initial segments $\beta_{\le L}, \beta'_{\le L}$. (Note that the functions $\Psi_{m,m',M,D,\theta}$ in \cref{psi-def} are of this type due to the nature of the $1_{\mathscr{W}_{m,M}}$ cutoffs.) By \cref{doob-path} this expression can be written 
\begin{equation*}  \int_{\Z^\N_{\ge -m}}\int_{\Z^\N_{\ge -m'}} \Psi(\beta,\beta')  \, \mathrm{d}\omega_m (\beta)\, \mathrm{d}\omega'_{m'}(\beta').\end{equation*} 
Returning to \cref{FM-form}, we therefore have

\begin{equation}\label{fm-theta-exp} f_M(\theta) =  \sum_{|D| \le M/10} (\log 2)^{D - \theta} \sum_{m, m' \le M} \int_{\Z^\N_{\ge -m}}\int_{\Z^\N_{\ge -m'}}  \Psi_{m,m',M,D,\theta}(\beta, \beta') \, \mathrm{d}\omega_m(\beta) \, \mathrm{d}\omega'_{m'}(\beta').
\end{equation}

Suppose in all that follows that $\theta \in [0,1]$. The next part of our analysis is to show that we can extend the sum over $D$ in \cref{fm-theta-exp} to all of $\Z$ with small error.  First note that, immediately from the definition \cref{psi-def}, we have $\Psi_{m,m',M,D,\theta}(\beta, \beta') \le 1$ always and so by \cref{first-rand-walk-pos} the integral in \cref{fm-theta-exp} is $\ll \omega_m(\Z^{\N}_{\ge -m}) \omega'_{m'}(\Z^{\N}_{\ge -m'}) = h(m) h(m') \ll M^2$.

We immediately see that the error incurred by adding in the terms $D > M/10$ to \cref{fm-theta-exp} is $\ll e^{-\Omega(M)}$.

When $D$ is very negative we need an alternative argument. By the inequality $1 - e^{-x} \le x$ we have $1 - \exp\big({-}2^{D - \theta}  \varrho^*_{\mbf{u} \mid \beta}(L) \tau^*_{\mbf{x} \mid \beta'}(L)\big) \ll 2^{D}  \varrho^*_{\mbf{u} \mid \beta}(L) \tau^*_{\mbf{x} \mid \beta'}(L)$. Thus, since $\tau^*_{\mbf{x} \mid \beta'}(L) \le 1$, it follows from \cref{X-upper-gen} and the definition \cref{psi-def} that we additionally have $\Psi_{m,m',M,D,\theta}(\beta, \beta') \ll 2^D r^2$ if $\beta$ is $r$-bounded up to $L$, and hence certainly if $\beta$ is $r$-bounded. Thus, defining $R(\beta)$ to be the minimal integer $r$ for which $\beta$ is $r$-bounded (as in \cref{r-beta-def}), we have using \cref{omega-bdedness-est} that
\begin{align*} \int_{\Z^\N_{\ge -m}}\int_{\Z^\N_{\ge -m'}} \Psi_{m,m',M,D,\theta} & \, \mathrm{d}\omega_m \, \mathrm{d}\omega'_{m'} \ll 2^D \omega'_{m'}(\Z^{\N}_{\ge -m'}) \sum_r r^2 \omega_m(R(\beta) = r) \\ & \ll 2^D M \sum_{r \in \Z_{\ge 0}} (M + r + 1)e^{-r} \ll M^{2} 2^D. \end{align*}
From this (and since $2 \log 2 > 1$) we see that the error from adding the terms $D < -M/10$ to \cref{fm-theta-exp} is also $\ll e^{-\Omega(M)}$.

Thus, up to error $e^{-\Omega(M)}$, we can replace $f_M(\theta)$ by the completed sum
\begin{equation*} \tilde f_M(\theta) =  \sum_{D \in \Z} (\log 2)^{D - \theta} \sum_{m, m' \le M} \int_{\Z^\N_{\ge -m} \times \Z^\N_{\ge -m'}}  \Psi_{m,m',M,D,\theta}(\beta, \beta') \, \mathrm{d}\omega_m \, \mathrm{d}\omega'_{m'}.
\end{equation*}
Recalling the definition \cref{psi-def}, and temporarily writing $\varrho^*_{\mbf{u} \mid \beta} = \varrho^*_{\mbf{u} \mid \beta}(L)$, $\tau^*_{\mbf{x} \mid \beta'} = \tau^*_{\mbf{x} \mid \beta'}(L)$ for brevity, it follows that $\tilde f_M(\theta)$ equals
\[  \sum_{D \in \Z} (\log 2)^{D - \theta}  \sum_{m,m' \le M} \int_{\Z^\N_{\ge -m}}\int_{\Z^\N_{\ge -m'}} \mb{E} \big(1 - e^{-2^{D - \theta}\varrho^*_{\mbf{u} \mid \beta}\tau^*_{\mbf{x} \mid \beta'}}\big)  1_{\mathscr{W}_{m,M}}(\beta) 1_{\mathscr{W}_{m',M}}(\beta') \, \mathrm{d}\omega_m(\beta) \, \mathrm{d}\omega'_{m'}(\beta').
\]
Taking the sum over $D$ on the inside and using \cref{g-lambda-trans} with $\xi= \theta$ and $\lambda = \varrho^*_{\mbf{u} \mid \beta}\tau^*_{\mbf{x} \mid \beta'}$, we see that $\tilde f_M(\theta)$ equals
\[  \sum_{m,m' \le M}\int_{\Z^\N_{\ge -m}}\int_{\Z^\N_{\ge -m'}} \mb{E} \big(\varrho^*_{\mbf{u} \mid \beta}\tau^*_{\mbf{x} \mid \beta'} \big)^c g\big(\theta - \log_2( \varrho^*_{\mbf{u} \mid \beta} \tau^*_{\mbf{x} \mid \beta'})\big) 1_{\mathscr{W}_{m,M}}(\beta) 1_{\mathscr{W}_{m',M}}(\beta') \, \mathrm{d}\omega_m(\beta) \, \mathrm{d}\omega'_{m'}(\beta') . \]  Recalling the definitions \cref{mu-m-def} and \cref{mu-m-prime-def}, we see that \cref{fm-conv} indeed holds, and this concludes the proof.\end{proof}

\subsection{Convergence of the measures $\mu_M$} 

The remaining task is to examine the convergence of the sequences of measures $(\mu_M)_{M = 1}^{\infty}$ and $(\mu'_M)_{M = 1}^{\infty}$. In this section we handle the first of these sequences in detail. Recall that $\mu_M$ is defined in \cref{mu-m-def}, which we can write as 
\begin{equation*}
\int \psi \, \mathrm{d}\mu_{M} =  \sum_{m \ge 0}\int_{\Z^\N_{\ge -m}} \mb{E} \varrho^*_{\mbf{u} \mid \beta}(L(M))^c \psi\big(\log_2 \varrho^*_{\mbf{u} \mid \beta}(L(M))\big)1_{m \le M} 1_{\mathscr{W}_{m,M}}(\beta)\, \mathrm{d}\omega_m(\beta).
\end{equation*}
Set
\begin{equation}\label{mu-m-ti} 
\int \psi \, \mathrm{d}\tilde\mu_{M} =  \sum_{m \ge 0}\int_{\Z^\N_{\ge -m}} \mb{E} \varrho^*_{\mbf{u} \mid \beta}(L(M))^c \psi\big(\log_2 \varrho^*_{\mbf{u} \mid \beta}(L(M))\big)\, 1_{\min \beta = -m} 1_{\argmin \beta \le L(M)}\mathrm{d}\omega_m(\beta),
\end{equation} where $\min \beta = \min_{i \in \{0\} \cup \N} \beta(i)$ and $\argmin\beta$ denotes the smallest index $q \ge 0$ for which $\beta(q) = \min \beta$. Our main result is the following. Here, $\dist{\mu}{\nu}$ denotes the bounded Lipschitz distance between two measures; for a full discussion of this concept, see \cref{bl-app}.

\begin{lemma}\label{mu-mutilde-compar}
We have $\dist{\mu_M}{\tilde\mu_M} \ll e^{-\Omega(M)}$.
\end{lemma}

In the proof of this lemma we will use the following auxiliary bounds.
\begin{lemma}\label{x-ell-bounds} 
Uniformly for $m \in \Z_{\ge 0}$, $\ell \in \N$, and $N \in [\ell, \infty) \cup \{\infty\}$ we have the following bounds:
\begin{equation}\label{exc-bd} \int_{\substack{\min \beta = -m \\ \argmin \beta \le \ell}} \mb{E} \varrho^*_{\mbf{u} \mid \beta}(\ell)^c\, \mathrm{d}\omega_m(\beta) \ll 2^{-m/4};\end{equation}
\begin{equation}\label{exc-rm-bd} \int_{\substack{\min \beta = -m \\ R_{\le N}(\beta) = r }} \mb{E} \varrho^*_{\mbf{u} \mid \beta}(\ell)^c\, \mathrm{d}\omega_m(\beta) \ll r^2(m + r + 1) e^{-r};\end{equation}
\begin{equation}\label{exc-qrm-bd} \int_{\substack{ \min \beta = -m \\ R_{\le N}(\beta) = r \\ \argmin \beta = q}} \mb{E} \varrho^*_{\mbf{u} \mid \beta}(\ell)^c\, \mathrm{d}\omega_m(\beta) \ll r^3 (1 + q)^{\kappa - 3/2} \end{equation} for $q \le \ell$, and
\begin{equation}\label{exc-mt-bd} \int_{\substack{\min \beta = -m \\ R_{\le N}(\beta) = r \\ T(\beta) > T}} \mb{E} \varrho^*_{\mbf{u} \mid \beta}(\ell)^c\, \mathrm{d}\omega_m(\beta) \ll r^{2}(1 + m)^{O(1)} T^{-\eta/2}.\end{equation}
Here we define $R_{\le \infty}(\beta)  := R(\beta)$, the smallest integer $r$ for which $\beta$ is $r$-bounded.\end{lemma}
\begin{proof} The key estimate will be \cref{X-upper-gen}. From this, H\"older's inequality and AM-GM we have that, if $\min \beta = -m$ and $\argmin \beta \le \ell$, \begin{align*} \mb{E} \varrho^*_{\mbf{u} \mid \beta}(\ell)^{c}   \le \big( \mb{E}\varrho^*_{\mbf{u} \mid \beta}(\ell)\big)^c \mb{P}(\varrho^*_{\mbf{u} \mid \beta}(\ell) \ne 0)^{1 - c} &  \ll R_{\le \ell}(\beta)^{3} (1 + \argmin \beta)^{\kappa} 2^{-m/3} \\ & \le (R_{\le \ell}(\beta)^{6} + (1+ \argmin \beta)^{2\kappa}) 2^{-m/3}. \end{align*} Therefore 
\begin{align*} \int_{\substack{\min \beta = -m \\ \argmin \beta \le \ell}} & \mb{E} \varrho^*_{\mbf{u} \mid \beta}(\ell)^c\, \mathrm{d}\omega_m(\beta) \\ &  \ll 2^{-m/3} \sum_{q \ge 0} (1 + q)^{2\kappa} \omega_m \{\argmin \beta = q\}  +  2^{-m/3} \sum_r r^{6} \omega_m\{ R_{\le \ell}(\beta) = r\}.   \end{align*}
To estimate the two terms on the right we use \cref{qm-beta-bd,omega-bdedness-est} respectively, obtaining that the above is 
\[ \ll (1 + m) 2^{-m/3} \big( \sum_q (1 + q)^{2\kappa - 3/2} + \sum_r r^{6}(r +1)e^{-r} \big) \ll (1 + m) 2^{-m/3} \ll 2^{-m/4}, \] which gives \cref{exc-bd}.
For \cref{exc-rm-bd} we use H\"older as above, but now using the trivial bound $\mb{P}(\varrho^*_{\mbf{u} \mid \beta}(\ell) \ne 0) \le 1$. This gives $\mb{E} \varrho^*_{\mbf{u} \mid \beta}(\ell)^c \ll R_{\le \ell}(\beta)^{2} \le R_{\le N}(\beta)^2$, and so 
\[\int_{\substack{\min \beta = -m \\ R_{\le N}(\beta) = r}} \mb{E} \varrho^*_{\mbf{u} \mid \beta}(\ell)^c\, \mathrm{d}\omega_m(\beta) \ll r^{2}\omega_m \{  R_{\le N}(\beta) = r \}.\]
The bound \cref{exc-rm-bd} follows using \cref{omega-bdedness-est}.

For \cref{exc-qrm-bd} we proceed as for \cref{exc-bd}, except we replace $R_{\le \ell}(\beta)$ by $R_{\le N}(\beta)$ and avoid the temptation to apply AM-GM. We then have
\[ \int_{\substack{\min \beta = -m \\ R_{\le N}(\beta) = r \\ \argmin \beta = q}} \mb{E} \varrho^*_{\mbf{u} \mid \beta}(\ell)^c\, \mathrm{d}\omega_m(\beta) \ll r^{3} (1 + q)^{\kappa}2^{-m/3} \omega_m \{ \argmin \beta = q\},\] and the result follows from \cref{argmin-omega-lem} and the crude estimate $(1 + m) 2^{-m/3} \ll 1$.

Finally, for \cref{exc-mt-bd} we proceed as for \cref{exc-rm-bd}, obtaining
\[ \int_{\substack{\min \beta = -m \\ R_{\le N}(\beta) = r \\ T(\beta) > T}} \mb{E} \varrho^*_{\mbf{u} \mid \beta}(\ell)^c\, \mathrm{d}\omega_m(\beta) \ll r^{2}\omega\{ R_{\le N}(\beta) = r, T(\beta) > T\} \le r^{2} \omega_m\{ T(\beta) > T\}, \] which gives the required bound using \cref{lem13.10}.
\end{proof}

\begin{proof}[Proof of \cref{mu-mutilde-compar}]
Let $\psi$ be a function on $\R/\Z$ with $\Vert \psi \Vert_{\Lip} \le 1$. In particular, $\Vert \psi \Vert_{\infty} \le 1$ and so we have
\begin{align}\nonumber \big|\int \psi \, & \mathrm{d}\mu_M - \int \psi \, \mathrm{d} \tilde\mu_M \big| \\ &  \le  \sum_{m \ge 0}\int_{\Z^\N_{\ge -m}} \mb{E} \varrho^*_{\mbf{u} \mid \beta}(L(M))^c \big| 1_{\min \beta = -m, \argmin \beta \le L(M)} - 1_{m \le M}1_{\mathscr{W}_{m,M}}(\beta)\big| \, \mathrm{d}\omega_m(\beta).\label{psi-psi-tilde-diff}\end{align}
Recall that $\mathscr{W}_{m,M}$ is the set of walks $\beta$ satisfying items (1) -- (4) in \cref{decouple-sec}, and so we may write
\[ 1_{\mathscr{W}_{m,M}}(\beta) = 1_{R_{\le L(M)}(\beta) \le R(M)} 1_{T_{\le L(M)}(\beta) \le T(M)} 1_{\min_{0 \le i \le L(M)} \beta(i) = -m} 1_{\min \beta^{-1}(-m) \le e^{C_0 M}}.\] Restricted to walks $\beta$ in $\Z^{\N}_{\ge -m}$, for some $m \ge 0$, we have
\[ 1_{m \le M} 1_{\mathscr{W}_{m,M}}(\beta) = 1_{m \le M} 1_{R_{\le L(M)}(\beta) \le R(M)} 1_{T_{\le L(M)}(\beta) \le T(M)} 1_{\min \beta = -m} 1_{\argmin \beta \le e^{C_0 M}}.\] Note that, when written this way, it is clear that $1_{m \le M} 1_{\mathscr{W}_{m,M}}(\beta) \le 1_{\min \beta = -m} 1_{\argmin \beta \le L(M)}$ pointwise for $\beta \in \Z_{\ge -m}^{\N}$ (and so the absolute value signs in \cref{psi-psi-tilde-diff} are unnecessary). Writing $L = L(M)$ for brevity, by a telescoping sum we may bound \cref{psi-psi-tilde-diff} as a sum of four terms
\begin{equation}\label{15-term-1}
\sum_{m \ge 0} \int_{\Z_{\ge -m}^{\N}} \mb{E} \varrho^*_{\mbf{u} \mid \beta}(L)^c 1_{m > M} 1_{\min \beta = -m} 1_{\argmin \beta \le L} \, \mathrm{d}\omega_m(\beta),
\end{equation}
\begin{equation}\label{15-term-2}
\sum_{m \ge 0} \int_{\Z_{\ge -m}^{\N}} \mb{E} \varrho^*_{\mbf{u} \mid \beta}(L)^c  1_{m \le M} 1_{\min \beta = -m} 1_{\argmin \beta \le L} 1_{R_{\le L}(\beta) > R(M)} \, \mathrm{d}\omega_m(\beta),
\end{equation}
\begin{equation}\label{15-term-3}
\sum_{m \ge 0} \int_{\Z_{\ge -m}^{\N}} \mb{E} \varrho^*_{\mbf{u} \mid \beta}(L)^c 1_{m \le M} 1_{\min \beta = -m} 1_{e^{C_0 M} < \argmin \beta \le L} 1_{R_{\le L}(\beta) \le R(M)} \, \mathrm{d}\omega_m(\beta),
\end{equation}
and
\begin{equation}\label{15-term-4}
\sum_{m \ge 0} \int_{\Z_{\ge -m}^{\N}} \mb{E} \varrho^*_{\mbf{u} \mid \beta}(L)^c  1_{m \le M} 1_{\min \beta = -m} 1_{\argmin \beta \le e^{C_0 M}} 1_{R_{\le L}(\beta) \le R(M)} 1_{T_{\le L}(\beta) > T(M)} \, \mathrm{d}\omega_m(\beta).
\end{equation}
We bound the contributions from these pieces using \cref{x-ell-bounds}, showing that in each case we get a bound $\ll e^{-\Omega(M)}$. In what follows, recall from \cref{truncation-threshold-defs} that $R(M) := 20 M$ and that $T(M) := e^{40M/\eta}$.

To bound \cref{15-term-1} we use \cref{exc-bd} and sum over $m > M$. This clearly gives a bound $\ll e^{-\Omega(M)}$.

To bound \cref{15-term-2} we use \cref{exc-rm-bd}. Summing over $m \le M$ and $r > R(M)$, this gives a bound
\[ \ll \sum_{m \le M} \sum_{r > R(M)} r^2 (m + r + 1) e^{-r} \ll e^{-\Omega(M)}.\] 

To bound \cref{15-term-3} we use \cref{exc-qrm-bd}. Summing this over $m \le M$, $r \le R(M)$, and $e^{C_0 M} < q \le L$ gives
\[ \ll M \sum_{r \le R(M)} \sum_{e^{C_0 M} < q \le L} r^3 (1 + q)^{\kappa - 3/2} \ll M^5 (e^{C_0 M})^{\kappa - 1/2} \ll e^{-\Omega(M)}.\]

Finally, to bound \cref{15-term-4} we use \cref{exc-mt-bd}, noting that $T(\beta) \ge T_{\le L}(\beta)$. Summing this over $m \le M$ and $r \le R(M)$ gives a bound of $M^{O(1)} T(M)^{-\eta/2} \ll e^{-\Omega(M)}$. This completes the proof of \cref{mu-mutilde-compar}.\end{proof}

We now define measures $\mu^{(\ell)}$ by
\begin{equation}\label{mu-ell-def} \int \psi \, \mathrm{d} \mu^{(\ell)} := \sum_{m \ge 0}\int_{\substack{\min \beta = -m \\ \argmin \beta \le \ell}} \mb{E} \varrho^*_{\mbf{u} \mid \beta}(\ell)^c \psi (\log_2 \varrho^*_{\mbf{u} \mid \beta}(\ell)) \, \mathrm{d}\omega_m(\beta).\end{equation}
Thus from the definition \cref{mu-m-ti} we see that
\begin{equation}\label{mu-ell-mu-m} \tilde\mu_M = \mu^{(L(M))}.\end{equation}
We now claim the following.
\begin{proposition}\label{prop124}
There is a Borel measure $\mu$ on $\R/\Z$ with $\Vert \mu \Vert_{\BL} = \dist{\mu}{0} < \infty$ such that $\dist{\mu^{(\ell)}}{\mu} \rightarrow 0$ as $\ell \rightarrow \infty$.   
\end{proposition}
\begin{proof} Since the space of measures is complete with respect to the bounded Lipschitz metric (\cref{completeness-measures}) it is enough to show that the $\mu^{(\ell)}$ are a Cauchy sequence in the bounded Lipschitz norm. For the rest of the section suppose that $\psi \in C(\R/\Z)$ is a test function with $\Vert \psi \Vert_{\Lip} \le 1$.
We first observe the inequality 
\begin{equation}\label{lip-xy} \big| x^c \psi(\log_2 x) - y^c \psi(\log_2 y) \big| \ll |x - y|^c\end{equation} for $x,y \in [0,\infty)$. To prove this, suppose wlog that $y \ge x$ and write $y = x + h$. If $x = O(h)$ then the result is trivial. Otherwise, we have
\[ x^c \big( \psi (\log_2 x) - \psi(\log_2(x + h))\big) = x^c \big( \psi(\log_2 x) - \psi(\log_2 x + O(\frac{h}{x})) \big) \ll x^c \frac{h}{x} \ll h^c\] by the Lipschitz property. Also
\[ \big( (x + h)^c - x^c\big) \psi(\log_2(x + h)) \le \big| (x + h)^c - x^c\big| \ll x^c \big(1 + O(\frac{h}{x}) - 1\big) \ll h^c \big( \frac{h}{x}\big)^{1 - c} \ll h^c,\] and together these statements give \cref{lip-xy}. 
We note also that
\begin{equation}\label{simple-real-var} |x - y|^c \le x^c + y^c.\end{equation}

Let $\ell < \ell'$. Then by the definition \cref{mu-ell-def} and from \cref{lip-xy} we have
\begin{align}\nonumber\big| & \int \psi \, \mathrm{d}\mu^{(\ell)} - \int \psi \, \mathrm{d}\mu^{(\ell')} \big| \\\nonumber & \le \sum_m \int_{\substack{\min \beta = -m \\ \argmin \beta \le \ell}}\Big|  \mb{E} \varrho^*_{\mbf{u} \mid \beta}(\ell)^c \psi (\log_2 \varrho^*_{\mbf{u} \mid \beta}(\ell)) -  \mb{E} \varrho^*_{\mbf{u} \mid \beta}(\ell')^c \psi (\log_2 \varrho^*_{\mbf{u} \mid \beta}(\ell')) \Big| \, \mathrm{d}\omega_m(\beta) \\ \nonumber & \qquad\qquad\qquad\qquad\qquad\qquad + \sum_{m}\int_{\substack{\min \beta = -m \\ \ell < \argmin \beta \le \ell'}} \mb{E}\varrho^*_{\mbf{u} \mid \beta}(\ell')^c \, \mathrm{d}\omega_m(\beta)\\  
& \ll \sum_m \int_{\substack{\min \beta = -m \\ \argmin \beta \le \ell}}  \mb{E} \big|\varrho^*_{\mbf{u} \mid \beta}(\ell) -  \varrho^*_{\mbf{u} \mid \beta}(\ell')\big|^c \, \mathrm{d}\omega_m(\beta) + \sum_{m}\int_{\substack{\min \beta = -m \\ \ell < \argmin \beta \le \ell'}} \mb{E}\varrho^*_{\mbf{u} \mid \beta}(\ell')^c \, \mathrm{d}\omega_m(\beta).\label{key-integral}
\end{align}
To bound the first term we split into two parts. As usual, $R(\beta)$ denotes the least integer so that $\beta$ is $R(\beta)$-bounded, and $T(\beta)$ the least integer such that $\beta$ is $T(\beta)$-bounded. The `main' region of integration is over $m \le \ell^{\eta_0}$ and $\beta$ for which $R(\beta) \le \ell^{\eta_0}$ and $T(\beta) \le \ell^{1/8}$, where here $\eta_0\in (0,\eta)$ is an absolute constant to be specified. On this region, it follows from \cref{lem73-first} with a telescoping sum that $\mb{E} |\varrho^*_{\mbf{u} \mid \beta}(\ell) - \varrho^*_{\mbf{u} \mid \beta}(\ell')| \ll e^{-\Omega(\ell^{1/4})}$, and hence $\mb{E} |\varrho^*_{\mbf{u} \mid \beta}(\ell) - \varrho^*_{\mbf{u} \mid \beta}(\ell')|^c \ll e^{-\Omega(\ell^{1/4})}$ by H\"older's inequality. Since $\omega_m (\Z^{\N}_{\ge -m}) \ll m$, the contribution of the main region of integration to the first term in \cref{key-integral} is $\ll e^{-\Omega(\ell^{1/4})}$.

To estimate the contribution to the first term from the remaining portions, we first apply \cref{simple-real-var}, thereby reducing matters to estimating
\begin{equation}\label{remaining-portion-1} \sum_{m >\ell^{\eta_0}}\int_{\substack{\min \beta = -m \\\argmin \beta  \le \ell }}
\mb{E} \varrho^*_{\mbf{u} \mid \beta}(\tilde\ell)^c  \, \mathrm{d}\omega_m(\beta),\qquad \sum_{m \le \ell^{\eta_0}}\int_{\substack{\min \beta = -m \\
R(\beta) > \ell^{\eta_0}}}
\mb{E} \varrho^*_{\mbf{u} \mid \beta}(\tilde\ell)^c  \, \mathrm{d}\omega_m(\beta),\end{equation} 
and
\begin{equation}\label{remaining-portion-3} \sum_{m \le \ell^{\eta_0}}\int_{\substack{R(\beta) \le \ell^{\eta_0} \\ T(\beta) > \ell^{1/8}}}
\mb{E} \varrho^*_{\mbf{u} \mid \beta}(\tilde\ell)^c 1_{\min \beta = -m} \, \mathrm{d}\omega_m(\beta)\end{equation} for $\tilde\ell \in \{ \ell, \ell'\}$. (We dropped the $\argmin$ conditions in the second estimate in \cref{remaining-portion-1} and in \cref{remaining-portion-3} since we will not need these in the estimations.)

The contribution from the first integral in \cref{remaining-portion-1} is $\ll e^{-\Omega(\ell^{\eta_0})} < \ell^{-10}$ by \cref{exc-bd}.

The contribution from the second integral in \cref{remaining-portion-1} can be estimated using \cref{exc-rm-bd} with $N = \infty$, which gives a bound of 
\begin{equation}\label{mr-bd} \ll \sum_{m \le \ell^{\eta_0}} \sum_{r > \ell^{\eta_0}} r^2 (m + r + 1) e^{-r} \ll e^{-\Omega(\ell^{\eta_0/2})} < \ell^{-10}. \end{equation}

Finally, the contribution from \cref{remaining-portion-3} can be estimated using \cref{exc-mt-bd}, which gives a bound of
\[ \ll \ell^{-\eta/16} \sum_{m \le \ell^{\eta_0}} \sum_{r \le \ell^{\eta_0}} (r(1 + m))^{O(1)} \ll \ell^{-\eta/32},\]
provided that $\eta_0$ is chosen appropriately. This will be the weakest of the various estimates we obtain.

We also need to estimate the second term in \cref{key-integral}. Using \cref{exc-bd}, the contribution from $m > \ell^{\eta_0}$ is $< \ell^{-10}$ similarly to before.

The contribution from $m \le \ell^{\eta_0}$ and $r \ge \ell^{\eta_0}$ is $< \ell^{-10}$ as in the previous estimation of the second integral in \cref{remaining-portion-1}.

The final remaining contribution is at most
\[ \sum_{m \le \ell^{\eta_0}} \sum_{r \le \ell^{\eta_0}} \int_{\substack{\min \beta = -m \\ \ell < \argmin \beta \le \ell'}} \mb{E} \varrho^*_{\mbf{u} \mid \beta}(\ell')^c \, \mathrm{d}\omega_{m} (\beta).\] By \cref{exc-qrm-bd} with $N = \infty$ this is
\[ \ll \ell^{5\eta_0} \sum_{q > \ell} (1 + q)^{\kappa - 3/2} \ll \ell^{5\eta_0 + \kappa - 1/2} < \ell^{-1/10}.\]

Putting everything together, it follows that 
\[ \big| \int \psi \, \mathrm{d}\mu^{(\ell)} - \int \psi \, \mathrm{d}\mu^{(\ell')} \big| \ll \ell^{-\eta/32}.\] Since $\psi$ was an arbitrary Lipschitz test function, it follows that $\dist{\mu^{(\ell)}}{\mu^{(\ell')}} \ll \ell^{-\eta/32}$, which establishes that the $\mu^{(\ell)}$ are indeed a Cauchy sequence.
\end{proof}

The following proposition summarises what we have shown. 

\begin{proposition}\label{main-mu-conv-prop} 
Let measures $\mu_M$ be defined as in \cref{mu-m-def}. Then there is a Borel measure $\mu$ on $\R/\Z$ with $\Vert \mu \Vert_{\BL} = \dist{\mu}{0} < \infty$ such that $\dist{\mu_M}{\mu} \rightarrow 0$ as $M \rightarrow \infty$. Moreover, $\dist{\mu^{(\ell)}}{\mu} \rightarrow 0$ as $\ell \rightarrow \infty$, where the measures $\mu^{(\ell)}$ are defined as in \cref{mu-ell-def}.
\end{proposition}
\begin{proof}
This follows by combining \cref{mu-mutilde-compar}, \cref{mu-ell-mu-m} and \cref{prop124}.
\end{proof}

\subsection{Convergence of the measures $\mu'_M$}
In this section we examine the convergence of the sequence of measures $(\mu'_M)_{M = 1}^{\infty}$, where the $\mu'_M$ are defined in \cref{mu-m-prime-def}. The analysis is very similar to that in the previous section, so we restrict ourselves to stating the main results and indicating the points at which the arguments differ in minor details. Analogously to \cref{main-mu-conv-prop}, the main statement is the following.

\begin{proposition}\label{main-mu-prime-conv-prop} 
Let measures $\mu'_M$ be defined as in \cref{mu-m-prime-def}. Then there is a Borel measure $\mu'$ on $\R/\Z$ such that $\Vert \mu' \Vert_{\BL} = \dist{\mu'}{0} < \infty$ and $\dist{\mu'_M}{\mu'} \rightarrow 0$ as $M \rightarrow \infty$. Moreover, $\dist{\mu^{\prime (\ell)}}{\mu'} \rightarrow 0$ as $\ell \rightarrow \infty$, where the measures $\mu^{\prime (\ell)}$ are defined by
\begin{equation}\label{mu-prime-ell-def} \int \psi d \mu^{\prime (\ell)} := \sum_{m \ge 0}\int_{\substack{\min \beta' = -m \\ \argmin \beta' \le \ell}} \mb{E} \tau^*_{\mbf{x} \mid \beta'}(\ell)^c \psi (\log_2 \tau^*_{\mbf{x} \mid \beta'}(\ell)) \, \mathrm{d}\omega'_m(\beta'),\end{equation} where $c = -\frac{\log \log 2}{\log 2}$.
\end{proposition}
\begin{proof} (Sketch.)
For the most part the argument is the same as the one in the previous section, replacing $\beta$ by $\beta'$, $\varrho^*_{\mbf{u} \mid \beta}$ by $\tau^*_{\mbf{x} \mid \beta'}$, and $\omega_m$ by $\omega'_m$ throughout. In place of the estimates \cref{X-upper-simple} and \cref{X-non-vanishing} from \cref{X-upper-gen}, we instead use (respectively) the trivial bound $\tau_{\mbf{x} \mid \beta'}(\ell) \le 1$ and \cref{Y-minm-bd}. With very minor modifications this allows us to recover analogues of all the statements in \cref{x-ell-bounds}. Later on in the proof, at the point where we applied a telescoping sum to \cref{lem73-first}, we instead apply \cref{lem23}, the $T$-positivity of $\beta'$, and a telescoping sum, which gives a similar bound $\mb{E} |\tau^*_{\mbf{x} \mid \beta'}(\ell) - \tau^*_{\mbf{x} \mid \beta'}(\ell')| \ll e^{-\Omega(\ell^{1/4})}$ at the appropriate point.
\end{proof}

\subsection{$\mu$ is not zero}\label{mu-not-zero}
In this section we prove that the measure $\mu$ is not zero. Whilst this follows from \cref{thm:main} and the main result of Ford \cite{For08}, it does not take long to provide a direct proof using the machinery we have set up.

Our starting point is \cref{mu-ell-def}. We take $\psi = 1$ and we restrict attention to walks $\beta$ which satisfy $\xi_1 = C$, $\xi_2 = \cdots = \xi_{C+1} = -1$, and which are $R$-bounded and $T$-positive, where $R := C^2$ and $T := e^{C^2}$.  Here $C$ is a sufficiently large constant to be specified later. Denote by $\mathscr{B}_C$ the collection of such walks which additionally satisfy $\min \beta = 0$ (and hence $\argmin \beta = 0$). Ignoring all terms in \cref{mu-ell-def} with $m > 0$, it suffices to prove that there is some choice of $C$ (independent of $\ell$) such that 
\begin{equation}\label{large-bc} \omega_0(\mathscr{B}_C) > 0\end{equation}
and that 
\begin{equation}\label{main-rho-lower} \mb{E} \varrho^*_{\mbf{u} \mid \beta}(\ell)^c \gg 1\end{equation} uniformly for $\beta \in \mathscr{B}_C$ and all $\ell$ (the implied constant may depend on $C$).

We begin with the first statement \cref{large-bc}. Denote by $\mathscr{B}_C^*$ the event $\mathscr{B}_C$ without the $R$-bounded and $T$-positive conditions. By \cref{lemma7.2} we have, for all sufficiently large $N$,
\[ \omega_0(\mathscr{B}^*_C) \gg N^{1/2} \mb{P}\big( \boldbeta \in \mathscr{B}^*_C ,  \min_{1 \le i \le N} \boldbeta(i) \ge 0\big) .\]   The event that $\boldbeta \in \mathscr{B}^*_C$ and that $\min_{1 \le i \le N} \boldbeta(i) \ge 0$ has probability $\frac{e^{-C-1}}{(C+1)!} \mb{P}(\min_{1 \le i \le N - C - 1} \boldgamma(i) \ge 0)$, where $\boldgamma$ is another $\Pois(1) - 1$ walk. This is $\sim \frac{e^{-C-1}}{(C+1)!}h(0) (N - C - 1)^{-1/2} \gg C^{-C - 3/2} N^{-1/2}$. It follows that $\omega_0(\mathscr{B}_C^*) \gg C^{-C-3/2}$. The desired result \cref{large-bc} is immediate from this and \cref{omega-bdedness-est,pos-paths-cor} and the choice of $R$ and $T$.

For \cref{main-rho-lower}, it suffices by H\"older's inequality and \cref{rho-ell-2-bound} to show that 
\begin{equation}\label{main-rho-lower-1} \mb{E} \varrho^*_{\mbf{u} \mid \beta}(\ell) \gg 1\end{equation} and that 
uniformly for $\beta \in \mathscr{B}_C$ and all $\ell$ sufficiently large. Again the implied constant can depend on $C$. Let us recall the definition \cref{x-beta-form}, which is \begin{equation*}\varrho^*_{\mbf{u} \mid \beta}(\ell)  = 2^{-\beta(\ell)} \sum_{\eps_{\le \ell}} 1 \Big( \sum_{i \le \ell; j \in [1 + \xi_i]} \eps_{i,j}2^{-\mbf{u}_{i, j}} \in [1 - 2^{-\ell},1]\Big).\end{equation*} 
We can now prove \cref{main-rho-lower-1}. We restrict to $\eps_{\le \ell}$ satisfying \begin{equation}\label{eps-constraint} \eps_{1,1} = 1,\, \eps_{1,2} = \cdots = \eps_{1, C+1} = 0,\end{equation} and the other $\eps_{i,j}$ can be arbitrary. Note there are no $\eps_{i,j}$ for $2 \le i \le C + 1$ since $1 + \xi_i = 0$ for these $i$. Deterministically, we have that
\[ 0 \le  \sum_{i \ge C+2} \eps_{i,j} 2^{-\mbf{u}_{i,j}} \le \sum_{i \ge C+2} (1 + \xi_i) 2^{1 - i} < \tfrac{1}{10}\] for $C$ large enough, using here the fact that $\boldbeta$ is $C^2$-bounded and so $|\xi_i| \le C^2 i^{\kappa}$. Therefore for any such $\eps$ we have
\begin{equation}\label{individ-lower} \mb{P}\big( \sum_{i \le \ell; j \in [1 + \xi_i]} \eps_{i,j}2^{-\mbf{u}_{i, j}} \in [1 - 2^{-\ell},1]\big) \ge \inf_{x \in [0,\frac{1}{10}]} \mb{P} (2^{-\mbf{u}_{1,1}}  \in 1 - x- [2^{-\ell}]\big) \gg 2^{-\ell}.\end{equation}
(To justify this last step observe that if $f(x)$ is the density function of $2^{-\mbf{u}_{1,1}}$ (where, recall, $\mbf{u}_{1,1}$ is uniform on $[0,1]$) and $g(x)$ is the density function of a uniform $\operatorname{U}[\frac{1}{2},1]$ variable, then a short calculation gives that $f(x)/g(x) \ge 1/2\log 2 > \frac{1}{2}$.) The number of choices for $\eps_{\le \ell}$ satisfying \cref{eps-constraint} is $2^{-C - 1} 2^{\sum_{i \le \ell}(1 + \xi_i)} = 2^{-C - 1} 2^{\ell + \beta(\ell)}$. Summing \cref{individ-lower} over these choices immediately gives \cref{main-rho-lower-1}.\vspace*{8pt}

\subsection{$\mu'$ is not zero}\label{mu-prime-not-zero}

In this subsection we show that $\mu'$ is not zero.

Our starting point is \cref{mu-prime-ell-def}. We take $\psi$ to be the constant function $1$, and we remove all terms from the sum except $m = 0$, noting that they are all non-negative. It thus suffices to prove that

\begin{equation}\label{suff-mu-prime-pos} \int \mb{E} \tau^*_{\mbf{x} \mid \beta'}(\ell)^c \, \mathrm{d}\omega'_0(\beta') \gg 1\end{equation} uniformly in $\ell$. Note that both the $\min \beta' = 0$ condition and the $\argmin$ condition are trivial here since $\omega'_0$ is supported on $\Z_{\ge 0}^{\N}$ and the minimum of $\beta'$ occurs at $0$. In order to prove this we introduce the representation function 
\begin{equation*} r_{\mbf{x} \mid \beta', \ell}(t) := \# \{ \eps  = (\eps_{i,j})_{i \le \ell; j \in [1 - \xi'_i]}: \sum_{i \le \ell; j \in [1 - \xi'_i]} \eps_{i,j} \mbf{x}_{i,j} = t\}.\end{equation*} Thus by definition \cref{y-beta-form} we have
\begin{equation}\label{15-4-1} \tau^*_{\mbf{x} \mid \beta'}(\ell) = 2^{\beta'(\ell) - \ell} | \Supp(r_{\mbf{x} \mid \beta',\ell})|.\end{equation}
Note also that 
\begin{equation}\label{first-moment-15} \sum_t r_{\mbf{x} \mid \beta',\ell}(t) = 2^{\ell - \beta'(\ell)}.\end{equation}
We proceed via a second moment argument essentially identical to the one in the proof of \cref{7lem-claim-1} (which is unsurprising since our task here is closely analogous to the one there). By opening up the second moment we have
\begin{equation}\label{second-moment-15-expand} \mb{E} \sum_t r_{\mbf{x} \mid \beta',\ell}(t)^2 = \sum_{\eps,\eps'} \mb{P}\big( \sum_{i \le \ell;j \in [1 - \xi'_i]} (\eps_{i,j} - \eps'_{i,j}) \mbf{x}_{i,j} = 0 \big).\end{equation}
We foliate the sum over $\eps, \eps'$ according to the largest index $r$ for which there is some $s$ such that $\eps_{r,s} \neq \eps'_{r,s}$. For any such pair $\eps, \eps'$, the inner probability is $\ll 2^{-r}$; this follows by revealing all $\mbf{x}_{i,j}$ except $\mbf{x}_{r,s}$ and applying the anti-concentration estimate \cref{non-conc-n}.  The number of pairs $\eps, \eps'$ with a given value $r$ is $\le 2^{2\sum_{i \le r} (1 - \xi'_i) + \sum_{r < i \le \ell} (1 - \xi'_i)} = 2^{\ell - \beta'(\ell) + r - \beta'(r)}$. Therefore from \cref{second-moment-15-expand} we obtain
\[\mb{E} \sum_t r_{\mbf{x} \mid \beta',\ell}(t)^2 \ll 2^{\ell - \beta'(\ell)} \sum_{r \le \ell} 2^{-\beta'(r)}.\]
If we assume that $\beta'$ is $T$-positive (for some $T$) then $\beta'(r) \ge -T + r^{1/4}$ and so 
\[\mb{E} \sum_t r_{\mbf{x} \mid \beta',\ell}(t)^2 \ll 2^{\ell - \beta'(\ell) + T}.\] In particular, with probability at least $\frac{1}{2}$ in $(\mbf{x} \mid \beta')$ we have
\[ \sum_t r_{\mbf{x} \mid \beta',\ell}(t)^2 \ll 2^{\ell - \beta'(\ell) + T}.\]
For any such $(\mbf{x} \mid \beta')$ it follows from this, \cref{15-4-1,first-moment-15} that 
$\tau^*_{\mbf{x} \mid \beta'}(\ell) \gg 2^{-T}$, and so 
\[ \mb{E}\tau^*_{\mbf{x} \mid \beta'}(\ell)^c \gg 2^{-cT}. \]
The desired result \cref{suff-mu-prime-pos} now follows from the fact that 
\[ \omega'_0 \{ \beta' : \beta' \mbox{ is $T$-positive} \} > 0\] for sufficiently large $T$, which is an immediate consequence of \cref{lem13.10}.

\subsection{Conclusion of the proof of \cref{thm:main}}

It is now a fairly straightforward matter to complete the proof of \cref{thm:main}. We begin with \cref{pk-walks-5}, which we repeat here:

\begin{equation}\label{pk-walks-6}  p(k) =  (c_0 + o(1))  k^{-\delta} (\log k)^{-3/2}  \big( f_{M}(\xi) +  O ( e^{-\Omega(M)}) \big).
\end{equation}
This was established under the assumption that $M \le M_0(k)$ for some fixed function $M_0$ with $\lim_{k \rightarrow \infty} M_0(k) = \infty$.
We showed in \cref{fm-tildefm-approx,fm-conv} that 
\begin{equation}\label{fm-fm-tilde} \Vert f_M - \tilde f_M \Vert_{\infty} \ll e^{-\Omega(M)}\end{equation}
where
\begin{equation}\label{g-fm}
 \tilde f_M = g \ast \mu_M \ast \mu'_M\end{equation}
with $g \in C^{\infty}(\R/\Z)$ being the function in \cref{g-funct-def}. We have also shown that $\mu_M \rightarrow \mu$ and $\mu'_M \rightarrow \mu'$ in bounded Lipschitz norm. To complete the proof we set $f := c_0 g \ast \mu \ast \mu'$. Then $f \in C^{\infty}(\R/\Z)$ by two applications of \cref{smooth-conv-meas}. By \cref{pk-walks-6,fm-fm-tilde,g-fm}, it suffices to show that 
\begin{equation*} \Vert g \ast \mu_M \ast \mu'_M - g \ast \mu \ast \mu' \Vert_{\infty} \rightarrow 0.\end{equation*}
To show this, we use the triangle inequality and \cref{simple-lip-conv} to bound
\begin{align*}
\Vert g \ast \mu_M \ast \mu'_M - g \ast \mu \ast \mu'\Vert_{\infty}   & \le \Vert  g \ast \mu_M \ast (\mu'_M - \mu') \Vert_{\infty} +  \Vert g \ast (\mu - \mu_M) \ast \mu' \Vert_{\infty} \\ & \le \dist{\mu_M'}{\mu'} \Vert g \ast \mu_M \Vert_{\Lip} +  \dist{\mu_M}{\mu} \Vert g \ast \mu'\Vert_{\Lip} .
\end{align*}
We have $\dist{\mu_M'}{\mu'}, \dist{\mu_M}{\mu} \rightarrow 0$, whilst $\Vert g \ast \mu_M \Vert_{\Lip}$ and $\Vert g \ast \mu' \Vert_{\Lip}$ are bounded uniformly in $M$ by \cref{smooth-conv-meas}. This concludes the proof of \cref{thm:main}.

\part{Description of the measures}\label{part5}
We have completed the proof of \cref{thm:main}, but the measures $\mu, \mu'$ have only been described as rather complicated limits over random walks satisfying certain technical conditions. Our task in this, the final part of the main paper, is to remove these conditions (and the random walks!) and prove the characterisations of these measures in \cref{prop16.2,sec17-main}.

\section{Description of \texorpdfstring{$\mu$}{}}\label{sec16}

The proof of \cref{prop16.2} is the somewhat more difficult of these two tasks. At the moment, we have characterised $\mu$ as the (bounded Lipschitz) limit of measures $\mu^{(\ell)}$, defined as in \cref{mu-ell-def}. Once again, in this section $\eta_0 \in (0,\eta)$ denotes a sufficiently small absolute constant.

\subsection{From walks to sequences}
Recall the definition \cref{hm-def} of $h(m)$ and that (see \cref{first-rand-walk-pos} (2)) $h(m) = (2/\pi)^{1/2} (m + 1)$ for $m \ge 0$. By a slight abuse of notation we now write $h(x) := (2/\pi)^{1/2}(x + 1)^+$ for all $x \in \R$. (That is, $h(x) = (2/\pi)^{1/2}(x + 1)$ for $x > -1$, and $0$ otherwise.)
We observe the relation
\begin{equation}\label{recur-relat-2} e^{-a} h(s - a) = \int^{s+1}_a e^{-u} h(s + 1 - u) \, \mathrm{d}u\end{equation}
for $a \le s+1$.

\begin{definition}
Denote by $\Omega_{m}$ the space of all sequences $x = (x_i)_{i = 1}^{\infty}$ of real numbers with $\lim_{i \rightarrow \infty} x_i = \infty$ and satisfying $x_i \le i + m$ for all $i$. 
\end{definition}
We consider the $\sigma$-algebra on $\Omega_{m}$ generated by sets of the form $\{ x\in \Omega_{m}: (x_1,\dots, x_r) \in B\}$ for some $r$ and for some Borel measurable $B \subset \R^r$, which we can and do assume is contained in $\{ (x_1,\dots, x_r) : x_i \le i + m \; \mbox{for $i = 1,\dots, r$}\}$. We define
\begin{equation}\label{nu-m-def} \nu_m \{ x : (x_1,\dots, x_r) \in B\} := \int_{B} h(m + r - x_r) e^{-x_r} 1_{0 < x_1 < \cdots < x_r}\, \mathrm{d}x_1 \cdots \mathrm{d}x_r \end{equation} 
We remark that without the factor $h(m + r - x_r)$, this would be the usual density functional for the Poisson process. It turns out that $\nu_m$ extends to give a well-defined measure on $\Omega_m$ with the stated $\sigma$-algebra. This is a continuous variant of the Doob conditioning construction given in \cref{72lemma}.

\begin{lemma}
The measure $\nu_m$ is well-defined and $\nu_m (\Omega_{m}) := h(m)$.
\end{lemma}
\begin{proof} The essential point is to check compatibility of \cref{nu-m-def} across different values of $r$, that is to say we must show that the measures of $\{ x : (x_1,\dots, x_{r-1}) \in B\}$ and $\{ x : (x_1,\dots, x_{r}) \in B \times [0,r + m]\}$ coincide whenever $B \subset \prod_{i = 1}^{r-1} [0, i + m] \subset \R^{r - 1}$. To this end, if $r \ge 2$ we have
\begin{align*}
\nu_m  \{ x : & (x_1,\dots, x_{r}) \in B \times [0, r + m]\} \\ & = \int_{\R^{r-1}} 1_{B}(x_1,\dots, x_{r-1})\big( \int_{x_{r-1}}^{r + m} h(m+r - x_r) e^{-x_r} \, \mathrm{d}x_r \big) \, \mathrm{d}x_1 \dots\mathrm{d}x_{r-1} \\ & = \int_{\R^{r-1}} 1_{B}(x_1 ,\dots, x_{r-1}) h(m + r - 1 - x_{r-1}) e^{-x_{r-1}}\,\mathrm{d}x_1 \dots \mathrm{d}x_{r-1} \\ & = \nu_m \{ x : (x_1,\dots, x_{r-1}) \in B\}, \end{align*} where the middle step follows from \cref{recur-relat-2} with $a = 0$ and $s = m + r - 1 - x_{r-1}$ upon making the substitution $u = x_r - x_{r - 1}$. If $r = 1$ we instead take the lower limit of the inner integral down to $0$; this confirms that $\nu_m(\Omega_m) = \nu_m \{ x : x_1 \in [0, m + 1]\} = h(m)$.
\end{proof}

Given $x = (x_i)_{i = 1}^{\infty} \in \Omega_m$, we may define the walk $\beta_{x}$ by $\beta_{x}(j) = \# \{ i : x_i \le  j\} - j$. Thus, the number of elements of $x$ in $(j - 1, j]$ is $\beta(j) - \beta(j-1) + 1$. Observe that the condition that an increasing sequence $x$ lies in $\Omega_m$ is equivalent to $\min \beta_{x} \ge -m$. Indeed, if $x \in \Omega_m$ then $x_i \le i + m$ for all $i$, and so $x_{j-m} \le j$ for $j > m$, which means that $\# \{ i : x_i \le j\} \ge j - m$ and so $\beta_{x}(j) \ge -m$. Conversely if $x \notin \Omega_m$ then $x_j > j + m$ for some $j$, which means that $\# \{ i : x_i \le j + m\} \le j - 1$ and so $\beta_x(j+m) \le -m - 1$.

\begin{lemma}\label{lem125} Let $F : \Omega_m \rightarrow \C$ be an integrable function. Then
\[ \int_{\Omega_m} F(x) \, \mathrm{d}\nu_m(x) = \int_{\Z^{\N}_{\ge -m}} \mb{E} F(\mbf{u} \mid \beta) \, \mathrm{d}\omega_m(\beta).\] Here, $(\mbf{u} \mid \beta)$ is the upper Poisson process $\mbf{u}$ conditioned to $\boldbeta = \beta$ as described in \cref{subsec-53}.
\end{lemma}
\begin{proof} By a limiting argument we may assume that $F$ depends only on the elements of $x$ in $[0,L]$ for some $L$ (the functions we consider will in any case have this property). We evaluate the LHS by splitting up according to the values of $\beta_{x}(j)$, $j = 1,\dots, L$. Note that $\beta_{x}(j)= c_j$ for $j = 1,\dots, L$ is the condition that $x_{i} \in (j-1, j]$ for $j-1 + c_{j-1} < i \le j + c_j$, $j \le L$, and also that $x_{L + c_L + 1} > L$. Therefore, writing $r := c_L + L$ for brevity, we have
\begin{align*} \int_{\Omega_m} & F(x) \prod_{j = 1}^L 1_{\beta_{x}(j) = c_j} \, \mathrm{d}\nu_m(x) \\ & =  \int_{\substack{x_i \in (j-1,j] \; \mbox{\scriptsize for} \\ j - 1 + c_{j-1} < i \le j + c_j ,\\ L < x_{r+1} \le m + r + 1}} F(x_1,\dots, x_{r})  h(m + r + 1 - x_{r + 1}) 1_{x_1 < \dots < x_r} e^{-x_{r + 1}} \, \mathrm{d}x_1 \dots \mathrm{d}x_{r+1}\\ &  = e^{-L} h(m + c_L )\int_{\substack{x_i \in (j-1,j] \; \mbox{\scriptsize for} \\ j - 1 + c_{j-1} < i \le j + c_j } } F(x_1,\dots, x_{r}) 1_{x_1 < \dots < x_r} \, \mathrm{d}x_1 \dots \mathrm{d}x_{r} \\ & = e^{-L} h(m + c_L ) \prod_{j = 1}^L \frac{1}{(c_j - c_{j-1} + 1)!} \mb{E}F(\mbf{u} \mid \beta),\end{align*} where $\beta$ is any walk with $\beta(j) = c_j$ for $j \le L$ (it does not matter which). Here, to get the middle step we integrated out the $x_{r+1}$ variable and applied \cref{recur-relat-2}. This is exactly \[\omega_m\{ \beta : \beta(j) = c_j \; \mbox{for $j = 1,\dots, L$}\} \mb{E}F(\mbf{u} \mid \beta),\]  where here we used the definition \cref{omega-meas-def}, recalling that the increments of $\boldbeta$ are $\Pois(1) - 1$. Summing over all choices of $c_1,\dots, c_L$ gives the result.
\end{proof}

\subsection{Description of the measure $\mu$}

We turn now to the main analysis and proof of \cref{prop16.2}. The argument will proceed as follows. In \cref{main-mu-conv-prop}, we have already characterised $\mu$ as the BL-limit of the measures $\mu^{(\ell)}$ defined in \cref{mu-ell-def}. We have
\begin{equation}\label{mu-ell-def-rpt} \int \psi \, \mathrm{d} \mu^{(\ell)} := \sum_{m \ge 0}\int_{\min_{0 \le i \le \ell} \beta(i) = -m} \mb{E} \varrho^*_{\mbf{u} \mid \beta}(\ell)^c \psi (\log_2 \varrho^*_{\mbf{u} \mid \beta}(\ell)) \, \mathrm{d}\omega_m(\beta).\end{equation} 
Here, and for the rest of the section, $c = -\frac{\log \log 2}{\log 2}$. Note that \cref{mu-ell-def-rpt} differs slightly in appearance from \cref{mu-ell-def}; we have replaced the condition $(\min \beta = -m) \cap (\argmin \beta \le \ell)$ by $\min_{0 \le i \le \ell} \beta(i) = -m$, which is a seemingly weaker condition. However there is no contribution to \cref{mu-ell-def-rpt} from any walks with $\min \beta < -m$, since $\omega_m$ is supported on walks with $\min \beta \ge -m$, so the two definitions are equivalent.

Define $\mu^{(\ell)}_*$ by
\begin{equation}\label{mu-ell-star-def} \int \psi \, \mathrm{d}\mu^{(\ell)}_* :=  \big(\frac{2}{\pi}\big)^{1/2}\mb{E} (\ell - \mbf{u}_{\ell})^+\varrho_{\mbf{u}}(\ell)^c \psi(\log_2 \varrho_{\mbf{u}}(\ell)),\end{equation} where $\varrho_{\mbf{u}}(\ell)$ is defined as in \cref{x-stat-def}. To prove \cref{prop16.2}, it is enough to show that $\dist{\mu^{(\ell)}}{\mu^{(\ell)}_{*}} = o_{\ell \rightarrow \infty}(1)$. We will do this in steps, specifying intermediate measures $\mu_i^{(\ell)}$, $i = 0,1,\dots, 5$, with $\mu_0^{(\ell)} := \mu^{(\ell)}$ and $\mu_5^{(\ell)} := \mu_{*}^{(\ell)}$, and show that $\dist{\mu^{(\ell)}_{i}}{\mu^{(\ell)}_{i -1}} = o_{\ell \rightarrow \infty}(1)$ for $i = 1,\dots, 5$; an application of the triangle inequality then concludes the argument. For the remainder of the analysis, $\psi$ denotes an arbitrary function on $\R/\Z$ with $\Vert \psi \Vert_{\Lip} \le 1$. Then our task is to prove that 
\begin{equation}\label{mu-mu-circ-close}  \big| \int \psi \, \mathrm{d}\mu_i^{(\ell)} - \int \psi \, \mathrm{d}\mu_{i-1}^{(\ell)}\big| = o_{\ell \rightarrow \infty}(1),\end{equation} uniformly in $i \in \{1,\dots, 5\}$ and $\psi$.\vspace*{10pt}

\emph{Step 1.}  First we compare $\mu^{(\ell)}$ to the measure $\mu_1^{(\ell)}$ in which we restrict the integral in \cref{mu-ell-def-rpt} to walks $\beta$ which are suitably bounded, positive and for which $m \le \ell^{\eta_0}$. More precisely, define $\mu_1^{(\ell)}$ by
\[ \int \psi \, \mathrm{d} \mu_1^{(\ell)} := \sum_{m \le \ell^{\eta_0}}\int_{\substack{\min_{0 \le i \le \ell} \beta(i) = -m \\ R(\beta) \le \ell^{\eta_0} \\ T(\beta) \le \ell^{1/8}}} \mb{E} \varrho^*_{\mbf{u} \mid \beta}(\ell)^c \psi (\log_2 \varrho^*_{\mbf{u} \mid \beta}(\ell)) \, \mathrm{d}\omega_m(\beta).\]
We claim that \cref{mu-mu-circ-close} holds for $i = 1$, with a bound of $\ll \ell^{-\eta/32}$ on the RHS. To prove this claim, we first restrict \cref{mu-ell-def-rpt} to $m \le \ell^{\eta_0}$. The error in doing this has already been handled via the first estimate in \cref{remaining-portion-1}. The rest of the error comes from the contribution of $\beta$ with
$\min \beta = -m$, $m \le \ell^{\eta_0}$, and $\argmin \beta \le \ell$ but either $R(\beta) > \ell^{\eta_0}$ or $T(\beta) > \ell^{1/8}$; that these contributions are appropriately small follows in exactly the same way we bounded the second sum in \cref{remaining-portion-1} and the sum in \cref{remaining-portion-3}. 

\vspace*{10pt}

\emph{Step 2.} To obtain $\mu_2^{(\ell)}$ from $\mu_1^{(\ell)}$ we replace the two copies of $\varrho^*_{\mbf{u} \mid \beta}(\ell)$ by $\varrho_{\mbf{u} \mid \beta}(\ell)$ (which is defined in \cref{rho-alt}). Note that for $\varrho_{\mbf{u} \mid \beta}(\ell)$ to be defined we need \cref{lim-betaupper-j-assump} to be satisfied, that is to say $\lim_{i \rightarrow \infty}(i + \beta(i)) = \infty$. This will be the case for all walks under consideration here, which (since we can restrict to the support of $\omega_m$) have $\min \beta = -m$. Moreover, by the definition of $T$-positive walk and the fact that $T(\beta) \le \ell^{1/8}$, we have $- \ell^{1/8} + i^{1/4} \le \beta(i) \le \ell^{1/8} + i^{3/4}$. Therefore the quantity $\ell'$, defined to be minimal such that $\ell' + \beta(\ell') \ge \ell$, will satisfy
\begin{equation}\label{ell-prime-ell} \ell' \asymp \ell .\end{equation}
We thus define $\mu_2^{(\ell)}$ by
\[ \int \psi \, \mathrm{d} \mu_2^{(\ell)} := \sum_{m \le \ell^{\eta_0}}\int_{\substack{\min_{0 \le i \le \ell} \beta(i) = -m \\ R(\beta) \le \ell^{\eta_0} \\ T(\beta) \le \ell^{1/8}}} \mb{E} \varrho_{\mbf{u} \mid \beta}(\ell)^c \psi (\log_2 \varrho_{\mbf{u} \mid \beta}(\ell)) \, \mathrm{d}\omega_m(\beta).\]
By \cref{lip-xy}, and since $\omega_m(\Z^{\N}_{\ge -m}) = h(m) \ll 1 + m$, in order to establish \cref{mu-mu-circ-close} with $i = 2$ it is sufficient to show
\begin{equation}\label{suff-step2} \mb{E}\big| \varrho_{\mbf{u} \mid \beta}(\ell) - \varrho^*_{\mbf{u} \mid \beta}(\ell) \big|^c \ll \ell^{-10},\end{equation} 
uniformly for $\beta$ with $T(\beta) \le \ell^{1/8}$. However from \cref{lem73-first,lem73-second}, Cauchy-Schwarz, and the triangle inequality, for such walks $\beta$ we have 
\[ \mb{E}\big| \varrho_{\mbf{u} \mid \beta}(\ell) - \varrho^*_{\mbf{u} \mid \beta}(\ell) \big| \ll e^{-\Omega(\ell^{1/4})}.\]
Here we used \cref{ell-prime-ell} in the application of \cref{lem73-second}. The required statement \cref{suff-step2} then follows using H\"older's inequality.\vspace*{10pt}

\emph{Step 3.} To obtain $\mu_3^{(\ell)}$ from $\mu_2^{(\ell)}$ we remove the constraints $R(\beta) \le \ell^{\eta_0}$ and $T(\beta) \le \ell^{1/8}$ again. However we retain the constraint $\beta(\ell) \ge 0$ (which is implied by $T(\beta) \le \ell^{1/8}$, since $\beta(\ell) \ge -\ell^{1/8} + \ell^{1/2 - \eta} > 0$). Thus we define
\begin{equation}\label{mu-ell-3} \int \psi \, \mathrm{d} \mu_3^{(\ell)} := \sum_{m \le \ell^{\eta_0}}\int_{\substack{\min_{0 \le i \le \ell} \beta(i) = -m\\ \beta(\ell) \ge 0}} \mb{E} \varrho_{\mbf{u} \mid \beta}(\ell)^c \psi (\log_2 \varrho_{\mbf{u} \mid \beta}(\ell)) \, \mathrm{d}\omega_m(\beta).\end{equation} 
To establish \cref{mu-mu-circ-close} with $i = 3$, the argument is similar to the case $i = 1$. Bounding $\psi$ pointwise by $1$, it suffices to bound
\begin{equation}\label{undo-1} \sum_{m \le \ell^{\eta_0}} \int_{\substack{\min_{0 \le i \le \ell} \beta(i) = -m \\ R(\beta) > \ell^{\eta_0}}} \mb{E}\varrho_{\mbf{u} \mid \beta}(\ell)^c \, \mathrm{d}\omega_m(\beta)\end{equation}
and
\begin{equation}\label{undo-2} \sum_{m \le \ell^{\eta_0}} \int_{\substack{\min_{0 \le i \le \ell} \beta(i) = -m \\ R(\beta) \le  \ell^{\eta_0} \\ T(\beta) > \ell^{1/8}}} \mb{E}\varrho_{\mbf{u} \mid \beta}(\ell)^c \, \mathrm{d}\omega_m(\beta).\end{equation}
For \cref{undo-1} we use \cref{rho-moment-bd} and \cref{omega-bdedness-est}; \cref{undo-1} may then be bounded exactly as in \cref{mr-bd}, and is seen to be $\ll \ell^{-10}$. For \cref{undo-2}, we note that we have an exact analogue of \cref{exc-mt-bd} for $\varrho_{\mbf{u} \mid \beta}(\ell)$, namely 
\[ \int_{\substack{\min \beta = -m \\ R(\beta) = r \\ T(\beta) > T}} \mb{E}\varrho_{\mbf{u} \mid \beta}(\ell)^c\, \mathrm{d}\omega_m(\beta) \ll r^{O(1)}(1 + m)^{O(1)} T^{-\eta/2}.\] The proof is the same, using the other estimates in \cref{X-upper-gen}. We can then proceed exactly as in the estimation of \cref{remaining-portion-3}. This concludes the proof that \cref{mu-mu-circ-close} holds for $i = 3$.\vspace*{8pt}

Using \cref{lem125}, we can rewrite the definition \cref{mu-ell-3} of $\mu_3^{(\ell)}$. For a sequence $x$, denote by $S_{m}$ the condition that $\# (x \cap [0,i]) \ge i - m$ for $i \le \ell$, or equivalently $x_{i - m} \le i$ for all $m < i \le \ell$, and $\tilde S$ the condition that $\# (x \cap [0,\ell]) \ge \ell$, or equivalently $x_{\ell} \le \ell$. (These definitions also depend on $\ell$, but this is fixed through the argument so we suppress this.) We include $S_{-1}$ in the definition, with $S_{-1}$ being the empty set.

Now (deterministically in $\mbf{u}$ conditioned to $\boldbeta = \beta$) we have that $\min_{1 \le i \le \ell} \beta(i) \ge -m$ iff $\# ((\mbf{u} \mid \beta) \cap [0,i]) \ge i - m$ for $i \le \ell$, which is so if and only if $(\mbf{u} \mid \beta) \in S_{m}$. Therefore we have (deterministically) \begin{equation}\label{careful-support} 1_{S_{m} \setminus S_{m - 1}}(\mbf{u} \mid \beta) = 1_{\min_{0 \le i \le \ell} \beta(i) = -m}.\end{equation} 
Also (again deterministically) we have that $\beta(\ell) \ge 0$ iff $\# ((\mbf{u} \mid \beta) \cap [0, \ell]) \ge \ell$, which is so if and only if $(\mbf{u} \mid \beta) \in \tilde S$. Therefore
\begin{equation}\label{careful-support-2} 1_{\tilde S}(\mbf{u} \mid \beta) = 1_{\beta(\ell) \ge 0}.\end{equation} 
Thus, if for $m \ge 0$ we define
\begin{equation}\label{fm-defn} F_m(x) := \varrho_x(\ell)^c \psi (\log_2 \varrho_x(\ell))1_{S_{m} \setminus S_{m - 1}}(x)1_{\tilde S}(x)\end{equation} (with $\varrho_x(\ell)$ defined as in \cref{x-stat-def}) and apply \cref{lem125}, then it follows using \cref{mu-ell-3} and \cref{careful-support,careful-support-2} that we have

\[ \int \psi \, \mathrm{d} \mu_3^{(\ell)} = \sum_{m \le \ell^{\eta_0}}\int F_m(x)  \, \mathrm{d}\nu_m(x).\]

Since $F_m$ depends only on $x_1,\dots, x_{\ell}$ we can write in the definition \cref{nu-m-def} of $\nu_m$, obtaining
\begin{align} \nonumber \int \psi \, \mathrm{d} \mu_3^{(\ell)} & = \sum_{m \le \ell^{\eta_0}}\int F_m(x) h(m + \ell - x_{\ell}) e^{-x_{\ell}} 1_{0 < x_1 < \cdots < x_{\ell}}   \, \mathrm{d}x_1 \cdots \mathrm{d}x_{\ell} \\ & = \sum_{m \le \ell^{\eta_0}} \mb{E} F_m(\mbf{u}) h(m + \ell - \mbf{u}_{\ell}).\label{another-mu}\end{align}
Here, as usual, $\mbf{u}$ is the upper process (a rate $1$ Poisson process on $[0,\infty)$).\vspace*{10pt}

\emph{Step 4.} For this step, we replace $h(m + \ell - \mbf{u}_{\ell})$ in \cref{another-mu} by $(\frac{2}{\pi})^{1/2} (\ell - \mbf{u}_{\ell})^+$, that is we define
\begin{equation}\label{mu-ell-4} \int \psi \, \mathrm{d} \mu_4^{(\ell)} = \big(\frac{2}{\pi}\big)^{1/2}\sum_{m \le \ell^{\eta_0}}\mb{E} (\ell - \mbf{u}_{\ell})^+ F_m(\mbf{u})  .\end{equation}

For this we recall that $h(x) = (2/\pi)^{1/2} (x + 1)^+$, which implies that  
\[ h(m + \ell - \mbf{u}_{\ell}) = O(m+1) +  \big(\frac{2}{\pi}\big)^{1/2} (\ell - \mbf{u}_{\ell})^+.\]  (Here consider the cases $\ell - \mbf{u}_\ell \ge 0$ and $\ell - \mbf{u}_{\ell} < 0$ separately; in the latter case, the LHS is zero unless $m + \ell - \mbf{u}_{\ell} \in (-1,m]$). 

Thus, to show \cref{mu-mu-circ-close} with $i = 4$, it is enough to show that the contribution from the $O(m+1)$ term is small. Since $\eta_0 \lll \frac{1}{10}$, it therefore suffices to show
\begin{equation}\label{suff-i4} \mb{E} |F_m(\mbf{u})| \ll (1 + m)^{2}\ell^{-1/2}\end{equation} for all $m \le \ell^{\eta_0}$.
We prove this by conditioning on the embedded walk $\boldbeta$ of $\mbf{u}$. The average on the LHS of \cref{suff-i4} is 
\[ \int \mb{E} |F_m(\mbf{u} \mid \beta)|\, \mathrm{d}\mb{P}(\beta),\] where $\mb{P}$ denotes measure on the path space of $\boldbeta$. By the definition \cref{fm-defn} of $F_m$ and \cref{careful-support,careful-support-2} (and since $|\psi| \le 1$), this is
\[ \ll \int_{\min_{0 \le i \le \ell} \beta(i) = -m, \beta(\ell) \ge 0}  \mb{E} \varrho_{\mbf{u} \mid \beta}(\ell)^c\, \mathrm{d}\mb{P}(\beta).\] By \cref{rho-moment-bd} (noting here that $\beta(\ell) \ge 0$ gives $\ell' \le \ell$), this is 
\[ \ll \int_{\min_{0 \le i \le \ell} \beta(i) = -m}  R_{\le \ell}(\beta)^2 \, \mathrm{d} \mb{P}(\beta)= \sum_r r^2 \mb{P}(R_{\le \ell}(\boldbeta) = r, \; \min_{0 \le i \le \ell} \boldbeta(i) = -m).\] 
By \cref{walks-bd-min-est}, this is 
\[ \ll \ell^{-1/2} \sum_{r} r^2 (m + r + 1) e^{-r} \ll (1 + m)^{2} \ell^{-1/2},\]
which is exactly \cref{suff-i4}. This completes the analysis of the case $i = 4$.\vspace*{10pt}

\emph{Step 5.} For the final step, we begin with the claim that if we extend the sum over $m$ in \cref{mu-ell-4} to all $m \ge 0$ then we obtain $\int \psi \, \mathrm{d} \mu_5^{(\ell)} = \int \psi \, \mathrm{d} \mu_*^{(\ell)}$, where $\mu_*^{(\ell)}$ is defined in \cref{mu-ell-star-def}. Indeed, from the definition \cref{fm-defn} of $F_m$ and a telescoping sum we have
\[
 \sum_{m \ge 0}\mb{E} (\ell - \mbf{u}_{\ell})^+ F_m(\mbf{u}) =    \mb{E} (\ell - \mbf{u}_{\ell})^+\varrho_{\mbf{u}}(\ell)^c \psi(\log_2 \varrho_{\mbf{u}}(\ell)) 1_{\tilde{S}}(\mbf{u}).
\]
However, the cutoff $1_{\tilde S}(\mbf{u})$ is equivalent to $1_{\mbf{u}_{\ell} \le \ell}$ and so is redundant, and this establishes the claim.

Therefore to prove \cref{mu-mu-circ-close} when $i = 5$, it is enough to show that
\begin{equation*} \sum_{m > \ell^{\eta_0} }\mb{E} (\ell - \mbf{u}_{\ell})^+ F_m(\mbf{u}) = o_{\ell \rightarrow \infty}(1).\end{equation*}
Once again we sample using the embedded walk. By the definition \cref{fm-defn} of $F_m$ and \cref{careful-support,careful-support-2}, it is enough to bound
\begin{equation}\label{to-bound-16} \sum_{m > \ell^{\eta_0}} \int_{\substack{\min_{0 \le i \le \ell} \beta(i) = -m \\ \beta(\ell) \ge 0}}  \mb{E}(\ell - (\mbf{u} \mid \beta)_{\ell})^+\varrho_{\mbf{u} \mid \beta}(\ell)^c\, \mathrm{d}\mb{P}(\beta) .\end{equation}
Here, for clarity we note that $(\mbf{u} \mid \beta)_{\ell}$ is the $\ell$th element of $(\mbf{u} \mid \beta)$ when listed in increasing order.
Recall that, in the context of \cref{X-upper-gen}, $\ell'$ is defined to be minimal so that $\ell' + \beta(\ell') \ge \ell$. Since $\beta(\ell) \ge 0$, we have $\ell' \le \ell$. Since all steps of $\beta$ are $\ge -1$, for any $s \ge 0$ we have $\beta(\ell' + s) \ge \ell - \ell' - s$, and so in particular $\min_{\ell' \le i \le \ell} \beta(i) \ge 0$, and hence $\min_{0 \le i \le \ell} \beta(i) = \min_{0 \le i \le \ell'} \beta(i)$.

Denote by $R_{\le \ell}(\beta)$ the least integer $R$ for which $\beta$ is $R$-bounded to length $\ell$. We use the bounds \cref{rho-moment-bd,rho-nonvanishing-bd}. Since $\ell' \le \ell$, these give $\mb{E} \varrho_{\mbf{u} \mid \beta}(\ell) \ll R_{\le \ell}(\beta)^2$ and  $\mb{P}(\varrho_{\mbf{u} \mid \beta}(\ell) \ne 0) \ll \ell^3 2^{-m}$ respectively. By H\"older's inequality, 
\[ \mb{E} \varrho_{\mbf{u} \mid \beta}(\ell)^c \le R_{\le \ell}(\beta)^{2c} \mb{P}(\varrho_{\mbf{u} \mid \beta}(\ell) \ne 0)^{1 - c} \le R_{\le \ell}(\beta)^2 \ell^2 2^{-m/3}.\]
Using the trivial bound $(\ell - (\mbf{u} \mid \beta)_{\ell})^+ \le \ell$, \cref{to-bound-16} is bounded by
\[ \ell^3 \sum_{m > \ell^{\eta_0}} 2^{-m/3} \sum_{r} r^2 \mb{P} \big( R_{\le \ell}(\boldbeta) = r, \, \min_{0 \le i \le \ell} \boldbeta(i) = -m\big).\]
By \cref{walks-bd-min-est}, this is bounded by
\[ \ll \ell^3 \sum_{m \ge \ell^{\eta_0}} 2^{-m/3} \sum_r r^2 (m + r + 1)e^{-r}  \ll e^{-\Omega(\ell^{\eta_0})}.\]
This completes the proof of \cref{mu-mu-circ-close} for $i = 5$, and thus the whole proof of \cref{prop16.2}.

\section{Description of \texorpdfstring{$\mu'$}{}}\label{sec17}

In this section we establish \cref{sec17-main}, which is a description of the measure $\mu'$. Recall that we have currently identified $\mu'$ as the bounded Lipschitz limit of the measures $\mu^{\prime (\ell)}$, defined in \cref{mu-prime-ell-def}, as $\ell \rightarrow \infty$. Throughout this section, $\ell$ will denote a large integer parameter.  The reader may also want to recall the definition of the lower process $\mbf{x}$, defined just before \cref{sec17-main}, and the conditioned lower process $(\mbf{x} \mid \beta')$, which is described in \cref{subsec-53}. Throughout the section we will assume that $\beta'$ is a walk with increments $\xi'_i$ bounded above by $1$ for which \cref{lim-beta-j-assump} holds. Given $\ell$ and $\beta'$, we denote by $\ell'$ the least integer such that $\ell' - \beta'(\ell') \ge \ell$, as in the statement of \cref{Y-upper-nonstar}. In some cases where $\beta'$ is not clear from context, we will write this as $\ell'(\beta')$ (which of course still depends on $\ell$). 

\subsection{The first $\ell$ arrivals} In this subsection we assemble some ingredients for the main argument, having to do with the distribution of the first $\ell$ arrivals of $(\mbf{x} \mid \beta')$. 

\begin{lemma} \label{lem172} We have the following statements.
\begin{enumerate}
\item $\beta'(\ell) > 0$ if and only if $\ell' > \ell$;
\item The distribution of $\tau^*_{\mbf{x} \mid \beta'}(\ell)$ depends only on the values of $\beta'(1),\dots, \beta'(\ell')$;
\item We have $\log_2 (\mbf{x} \mid \beta')_{\ell}   = \ell +  \beta'(\ell') + O(1 +  |\xi'_{\ell'}|)$.
\end{enumerate}
\end{lemma}
\begin{proof}
(1) Indeed if $\beta'(\ell) > 0$ then $\ell - \beta'(\ell) < \ell$ and so, since $i - \beta'(i)$ is a non-decreasing function of $i$, we have $\ell' > \ell$. Conversely if $\ell' > \ell$ then $\ell$ does not satisfy $\ell - \beta'(\ell) \ge \ell$, which means $\beta'(\ell) > 0$.

To prepare for (2) and (3) we begin by reviewing the coupling of the lower process $\mbf{x}$ to the lower walk $\boldbeta'$, as described in \cref{subsec-53}. The key features of the coupling are as follows. There is an auxiliary rate $1$ Poisson process $\mbf{s}$ such that $\boldbeta'(i) = i - \#(\mbf{s} \cap [0,i])$, and the conditioned process $(\mbf{s} \mid \beta') := (\mbf{s} \mid \boldbeta' = \beta')$ consists of sampling $1 - \xi'_i$ uniform random elements from $(i-1,i]$, which are then ordered. We then define $\mbf{x} = \phi(\mbf{s})$, where $\phi : [0,\infty) \rightarrow \N$ is defined in \cref{phis-def}. In particular to establish (2) it is enough to show that the distribution of the first $\ell$ elements $(\mbf{s} \mid \beta')_i$, $i = 1,\dots, \ell$, depends only on the values $\beta'(1),\dots, \beta'(\ell')$. This statement is essentially the purpose of $\ell'$. Indeed, since $\# ((\mbf{s} \mid \beta') \cap [0,i]) = i - \beta'(i)$, we see that (deterministically)
\begin{equation}\label{s-bounds-1} \ell'-1 < (\mbf{s} \mid \beta')_{\ell} \le \ell',\end{equation} and (2) then follows.

Finally we turn to (3). Recall that by \cref{phi-tilde-x-size-log} we have $\log_2 \mbf{x} = \mbf{s} + O(1)$. Thus to prove (3) it suffices to show that 
\begin{equation}\label{s-suff-x} (\mbf{s} \mid \beta')_{\ell}   =  \ell + \beta'(\ell') + O(1 + |\xi'_{\ell'}|).\end{equation}
From the definition of $\ell'$ we have \begin{equation}\label{second-ell} (\ell' - 1) - \beta'(\ell' - 1) \le \ell - 1.\end{equation} Using \cref{s-bounds-1,second-ell} (and $\ell' - \beta'(\ell') \ge \ell$) one may check that 
\begin{equation*} \beta(\ell') - 1\le (\mbf{s} \mid \beta')_{\ell} - \ell \le \beta'(\ell' - 1),\end{equation*} and so (since $\beta'(\ell) = \beta(\ell'-1) + \xi'_{\ell'}$) the desired statement \cref{s-suff-x} follows immediately.
\end{proof}

Recall that the embedded walk $\boldbeta'$ has $1 - \Pois(1)$ steps. The following lemma records that $\ell'(\boldbeta')$ is usually $O(\ell)$.

\begin{lemma}\label{dyadic-ell-dash} Suppose that $j \ge 1$ is an integer. Then we have 
\[ \mb{P}(2^j \ell \le \ell'(\boldbeta') < 2^{j+1} \ell) \ll 2^j \ell e^{-2^{j - 4} \ell}.\]
\end{lemma}
\begin{proof} Suppose that $\beta'(\ell'-1) <  \ell'/2$. Then, since $\ell - 1 \ge (\ell'-1) - \beta'(\ell'-1) > \ell'/2 - 1$, we have $\ell' < 2\ell$. Thus if $\ell'(\beta') \ge 2^{j} \ell$ then $\beta'(\ell'-1) \ge  \ell'/2 > (\ell' - 1)/2$. 

Now for any fixed $t$ we have $\mb{P}(\boldbeta'(t) \ge t/2) = \mb{P}(\Pois(t) \le t/2) \le e^{-t/8}$ by the first estimate in \cite[Theorem A.1.15]{alon-spencer}. Thus if $j \ge 1$ then
\[ \mb{P}(2^j\ell \le \ell'(\boldbeta') < 2^{j + 1} \ell) \le \sum_{2^j\ell - 1 \le t < 2^{j+1}\ell - 1} \mb{P}(\boldbeta'(t) \ge t/2)  \ll 2^j \ell e^{-2^{j-4} \ell}.\qedhere
\]
\end{proof}

\subsection{The main argument}

Recall from \cref{main-mu-prime-conv-prop} that $\mu'$ is the bounded Lipschitz limit of the measures $\mu^{\prime (\ell)}$ defined by
\begin{equation}\label{mu-prime-ell-def-rpt} \int \psi \, \mathrm{d} \mu^{\prime (\ell)} := \sum_{m \ge 0}\int_{\min_{0 \le i \le \ell} \beta'(i) = -m} \mb{E} \tau^*_{\mbf{x} \mid \beta'}(\ell)^c \psi (\log_2 \tau^*_{\mbf{x} \mid \beta'}(\ell)) \, \mathrm{d}\omega'_m(\beta').\end{equation} 
Here, $\omega'_m$ is defined as in \cref{sec7.1}, and $c = -\frac{\log \log 2}{\log 2}$ as usual. Note that walks $\beta'$ in the support of $\omega'_m$ are lower, that is to say they have steps bounded above by $1$ rather than below by $-1$. 

Define a measure $\mu^{\prime(\ell)}_*$ by
\begin{equation}\label{mu-prime-ell-star-def} \int \psi \, \mathrm{d}\mu^{\prime(\ell)}_* := (\frac{2}{\pi})^{1/2}  \mb{E} (\log_2 \mbf{x}_{\ell} - \ell)^+ \tau_{\mbf{x}}(\ell)^c \psi(\log_2 \tau_{\mbf{x}}(\ell)).\end{equation} \cref{sec17-main} is equivalent to the statement that $\mu^{\prime (\ell)}_* \rightarrow \mu'$ in bounded Lipschitz norm.

As in the last section, we will prove this by specifying intermediate measures $\mu_i^{\prime (\ell)}$, $i = 0,1,\dots, 5$, with $\mu_0^{\prime (\ell)} := \mu^{\prime (\ell)}$ and $\mu_5^{\prime (\ell)} := \mu_{*}^{\prime (\ell)}$, and show that 
\begin{equation}\label{mu-mu-prime-circ-close} \dist{\mu^{\prime (\ell)}_{i}}{\mu^{\prime (\ell)}_{i -1}} = o_{\ell \rightarrow \infty}(1)\end{equation} for $i = 1,\dots, 5$, and this is enough to complete the proof.\vspace*{10pt}

\emph{Step 1.} In what follows we let $\psi$ be an arbitrary function with $\Vert \psi \Vert_{\Lip} \le 1$. First we replace $\mu^{\prime (\ell)}$ by $\mu_1^{\prime (\ell)}$ in which we restrict the integral in \cref{mu-prime-ell-def-rpt} to walks $\beta'$ which are suitably bounded, positive and for which $m \le \ell^{\eta_0}$. More precisely, define $\mu_1^{\prime (\ell)}$ by
\[ \int \psi \, \mathrm{d} \mu_1^{\prime (\ell)} := \sum_{0 \le m \le \ell^{\eta_0}}\int_{\substack{\min_{0 \le i \le \ell} \beta'(i) = -m \\ R(\beta') \le \ell^{\eta_0} \\ T(\beta') \le \ell^{1/8}}} \mb{E} \tau^*_{\mbf{x} \mid \beta'}(\ell)^c \psi (\log_2 \tau^*_{\mbf{x} \mid \beta'}(\ell)) \, \mathrm{d}\omega'_m(\beta').\]
That \cref{mu-mu-prime-circ-close} holds for $i = 1$ follows in essentially the same way as before.\vspace*{10pt}

\emph{Step 2.} To obtain $\mu_2^{\prime (\ell)}$ from $\mu_1^{\prime (\ell)}$ we replace the two copies of $\tau^*_{\mbf{x} \mid \beta'}(\ell)$ by $\tau_{\mbf{x} \mid \beta'}(\ell)$, where here the relevant definitions are \cref{y-beta-form,tau-unstar-def}. Thus we define $\mu_2^{\prime (\ell)}$ by
\begin{equation} \int \psi \, \mathrm{d} \mu_2^{\prime (\ell)} := \sum_{m \le \ell^{\eta_0}}\int_{\substack{\min_{0 \le i \le \ell} \beta'(i) = -m \\ R(\beta') \le \ell^{\eta_0} \\ T(\beta') \le \ell^{1/8}}} \mb{E} \tau_{\mbf{x} \mid \beta'}(\ell)^c \psi (\log_2 \tau_{\mbf{x} \mid \beta'}(\ell)) \, \mathrm{d}\omega'_m(\beta').\end{equation}
 By \cref{lip-xy}, and since $\omega'_m(\Z^{\N}_{\ge -m}) = h'(m) \ll 1 + m$, it is sufficient to show
\begin{equation}\label{suff-step2-y} \mb{E}\big|  \tau_{\mbf{x}\mid \beta'}(\ell) - \tau^*_{\mbf{x} \mid \beta'}(\ell) \big|^c \ll \ell^{-10}.\end{equation}

Now since $\beta'$ is $\ell^{1/8}$-positive, we have $\beta'(\ell) \ge \ell^{1/2 - \eta} - \ell^{1/8} > 0$. Therefore by \cref{tau-y-relation} and the monotonicity property \cref{y-unstar-monotonicity} we have $\tau^*_{\mbf{x} \mid \beta'}(\ell) = \tau_{\mbf{x} \mid \beta'}(\ell - \beta'(\ell)) \ge \tau_{\mbf{x} \mid \beta'}(\ell)$. We also have $\beta'(2\ell) \le (2\ell)^{1/8} + (2\ell)^{1/2 + \eta} < \ell$, and so by further applications of \cref{tau-y-relation} and \cref{y-unstar-monotonicity} we have $\tau^*_{\mbf{x} \mid \beta'}(2\ell) = \tau_{\mbf{x} \mid \beta'}(2 \ell - \beta'(2\ell)) \le \tau_{\mbf{x} \mid \beta'}(\ell)$. Combining these facts and then applying \cref{lem23} and a telescoping sum, we have
\[ \mb{E} |\tau_{\mbf{x}\mid \beta'}(\ell)  -\tau^*_{\mbf{x} \mid \beta'}(\ell)| \le \mb{E} |\tau^*_{\mbf{x} \mid \beta'}(2\ell) -\tau^*_{\mbf{x} \mid \beta'}(\ell)|  \ll e^{-\Omega(\ell^{1/4})}.\] The required bound \cref{suff-step2-y} then follows from H\"older's inequality.
\vspace*{10pt}

\emph{Step 3.} To obtain $\mu_3^{\prime (\ell)}$ from $\mu_2^{\prime (\ell)}$ we remove the constraints $R(\beta') \le \ell^{\eta_0}$ and $T(\beta') \le \ell^{1/8}$ again. However we retain the constraint $\beta'(\ell) > 0$ (which is implied by $T(\beta') \le \ell^{1/8}$ as noted above).

Thus we define
\begin{equation}\label{mu-dash-ell-3} \int \psi \, \mathrm{d} \mu_3^{\prime (\ell)} := \sum_{0 \le m \le \ell^{\eta_0}}\int_{\substack{\min_{0 \le i \le \ell} \beta'(i) = -m\\ \beta'(\ell) > 0}} \mb{E} \tau_{\mbf{x} \mid \beta'}(\ell)^c \psi (\log_2 \tau_{\mbf{x} \mid \beta'}(\ell)) \, \mathrm{d}\omega'_m(\beta').\end{equation} 
The proof that \cref{mu-mu-prime-circ-close} holds with $i = 3$ is very similar to, but easier than, the handling of Step 3 in the last section. Bounding $\psi$ and $\tau_{\mbf{x} \mid \beta'}(\ell)^c$ pointwise by $1$, it suffices to bound

\begin{equation*} \sum_{m \le \ell^{\eta_0}} \omega'_m \{ \beta': R(\beta') > \ell^{\eta_0}\} \quad \mbox{and} \sum_{m \le \ell^{\eta_0}} \omega'_m \{ \beta': T(\beta') > \ell^{1/8}\}\end{equation*}
That these are bounded by $o_{\ell \rightarrow \infty}(1)$ follows using \cref{omega-bdedness-est} and \cref{lem13.10} respectively, assuming that $\eta_0$ is sufficiently small.
\vspace*{8pt}

Now, as explained at the start of the proof of \cref{lem172}, the constraint $\beta'(\ell) > 0$ is equivalent to $\ell' > \ell$. Thus we may replace \cref{mu-dash-ell-3} by
\[ \int \psi \, \mathrm{d} \mu_3^{\prime (\ell)} := \sum_{m \le \ell^{\eta_0}}\sum_{t > \ell} \int_{\substack{\min_{0 \le i \le \ell} \beta'(i) = -m \\ \ell'(\beta') = t}} \mb{E} \tau_{\mbf{x} \mid \beta'}(\ell)^c \psi (\log_2 \tau_{\mbf{x} \mid \beta'}(\ell)) \, \mathrm{d}\omega'_m(\beta').\] Here, recall that we have written $\ell'(\beta')$ for $\ell'$ to emphasise that this quantity depends on $\beta'$.
Observe that the conditions $\min_{0 \le i \le \ell} \beta'(i) = -m$ and $\ell'(\beta') = t$ depend only on $\beta'(1),\dots, \beta'(t)$ (since $t \ge \ell$), and so too does the integrand (since $\tau_{\mbf{x} \mid \beta'}(\ell)$ is involves only the first $\ell$ elements of $\mbf{x}$, and the number of elements of $\mbf{x}$ up to $t$ is $t - \beta'(t) \ge \ell$). Recalling the definition (see \cref{sec7.1}) of $\omega'_m$, and in particular \cref{doob-path}, it follows that 
\begin{equation}  \int \psi \, \mathrm{d} \mu_3^{\prime (\ell)} = \sum_{m \le \ell^{\eta_0}}\sum_{t > \ell} \int_{\substack{ \min_{0 \le i \le \ell} \beta'(i) = -m \\ \min_{0 \le i \le t} \beta'(i) \ge -m \\ \ell'(\beta') = t}} h'(m + \beta'(t))\mb{E} \tau_{\mbf{x} \mid \beta'}(\ell)^c \psi (\log_2 \tau_{\mbf{x} \mid \beta'}(\ell))\, \mathrm{d} \mb{P}'(\beta') ,\label{mu-dash-ell-3-again}\end{equation} where here $\mb{P}'$ denotes measure on the path space of $\boldbeta'$.
Note that from \cref{second-ell} we have $\beta'(t - 1) \ge t - \ell$. Since all steps of $\beta'$ are bounded above by $1$, it follows that $\beta'(t - i) \ge t - \ell - (i - 1)$ for all $i \ge 1$, and in particular $\min_{\ell \le i \le t - 1} \beta'(i) \ge 1$. Therefore we may replace \cref{mu-dash-ell-3-again} by
\begin{equation} \int \psi \, \mathrm{d} \mu_3^{\prime (\ell)} = \sum_{m \le \ell^{\eta_0}}\sum_{t} \int_{\substack{\min_{0 \le i \le \ell} \beta'(i) = -m \\ \ell'(\beta') = t}}  1_{t > \ell}1_{\beta'(t) \ge -m} h'(m + \beta'(t))\mb{E} \tau_{\mbf{x} \mid \beta'}(\ell)^c \psi (\log_2 \tau_{\mbf{x} \mid \beta'}(\ell))\, \mathrm{d}\mb{P}'(\beta') .\label{mu-dash-ell-3-once-again}\end{equation} 

\vspace*{10pt}

\emph{Step 4.} We use the fact that 
\begin{equation}\label{h-prime-limit} h'(r) = (\frac{2}{\pi})^{1/2} r + O(1)
\end{equation} for $r \ge 0$, which is \cref{first-rand-walk-pos} (2).
Suppose that $\ell'(\beta') = t$ and $\beta'(t) \ge -m$. Then from \cref{h-prime-limit} and \cref{lem172} (3) we have
\begin{equation}\label{h-linearise-pre} h'(m + \beta'(t)) =  (\frac{2}{\pi})^{1/2} \big( \log_2(\mbf{x} \mid \beta')_{\ell} - \ell\big) + O(m +1 + |\xi'_{t}|).   \end{equation}
Suppose that $\log_2(\mbf{x} \mid \beta')_{\ell} < \ell$. Then from \cref{lem172} we have $\beta'(t) \le O(1 + |\xi'_{t}|)$ and so by \cref{h-prime-limit} $h'(m + \beta'(t)) = O(1 + m + |\xi'_{t}|)$.
Thus we may replace \cref{h-linearise-pre} by \begin{equation*} h'(m + \beta'(t))  =   (\frac{2}{\pi})^{1/2} (\log_2 (\mbf{x} \mid \beta')_{\ell} - \ell)^+ + O(1 + m + |\xi'_{t}|) .\end{equation*} 
Next we multiply by the cutoff $1_{t > \ell}1_{\beta'(t) \ge -m}$ appearing in \cref{mu-dash-ell-3-once-again}; we claim that (still assuming $\ell'(\beta') = t$)
\begin{equation}\label{h-linearise-1a} 1_{t > \ell}1_{\beta'(t) \ge -m} h'(m + \beta'(t))  =   (\frac{2}{\pi})^{1/2} (\log_2 (\mbf{x} \mid \beta')_{\ell} - \ell)^+ + O(1 + m + |\xi'_{t}|) .\end{equation} This is trivial if $t > \ell$ and $\beta'(t) \ge -m$, so suppose that either $t \le \ell$, or that $\beta'(t) < -m$. In either case, the LHS is zero. In both cases, $\beta'(t) \le 0$ (in the first case this is since $\ell'(\beta') = t$ and so $t - \beta'(t) \ge \ell$ by the definition of $\ell'$). It follows from \cref{lem172} (3) that $((\log_2 \mbf{x} \mid \beta')_{\ell} - \ell)^+ = O(1 + |\xi'_t|)$, so the claim follows in this case also. 

Now we define $\mu^{\prime (\ell)}_{4}$ itself. We set

\begin{equation}\label{mu-dash-ell-4} \int \psi \, \mathrm{d}\mu^{\prime (\ell)}_{4}   = (\frac{2}{\pi})^{1/2}\sum_{m \le \ell^{\eta_0}}\int_{\min_{0 \le i \le \ell} \beta'(i) = -m}\mb{E} (\log_2 (\mbf{x} \mid \beta')_{\ell} - \ell)^+ \tau_{\mbf{x} \mid \beta'}(\ell)^c \psi (\log_2 \tau_{\mbf{x} \mid \beta'}(\ell))\, \mathrm{d}\mb{P}'(\beta') .\end{equation} 

This is the same as \cref{mu-dash-ell-3-once-again} except that we have replaced $1_{t > \ell} 1_{\beta'(t) \ge -m} h'(m + \beta'(t))$ by the main term in \cref{h-linearise-1a}; note that we have removed the sum over $t$, since nothing on the right in \cref{mu-dash-ell-4} references this variable.

To prove \cref{mu-mu-prime-circ-close} in the case $i = 4$, we must handle the contribution from the error term in \cref{h-linearise-1a}. We split this error term into two parts: the contribution from $O(1+m)$ and the contribution from $O(|\xi'_{t}|)$. These terms are, by definition, 

\begin{equation}\label{first-error} \ll  \sum_{m \le \ell^{\eta_0}} (1 + m)\sum_{t} \int_{\substack{\min_{0 \le i \le \ell} \beta'(i) = -m \\ \ell'(\beta') = t}}\mb{E} \tau_{\mbf{x} \mid \beta'}(\ell)^c \psi (\log_2 \tau_{\mbf{x} \mid \beta'}(\ell)) \, \mathrm{d}\mb{P}'(\beta')\end{equation}  and
\begin{equation}\label{second-error} \ll  \sum_{m \le \ell^{\eta_0}} \sum_{t} \int_{\substack{\min_{0 \le i \le \ell} \beta'(i) = -m \\ \ell'(\beta') = t}}|\xi'_{t}| \mb{E} \tau_{\mbf{x} \mid \beta'}(\ell)^c \psi (\log_2 \tau_{\mbf{x} \mid \beta'}(\ell)) \, \mathrm{d}\mb{P}'(\beta'). \end{equation}

\emph{Bounding \cref{first-error}.} We bound the $\tau$ and $\psi$ terms trivially by $1$; thus this contribution is
\[ \ll \sum_{m \le \ell^{\eta_0}} (1+m) \sum_{t} \int_{\substack{\min_{0 \le i \le \ell} \beta'(i) = -m \\ \ell'(\beta') = t}} \, \mathrm{d}\mb{P}'(\beta') \ll  \sum_{m \le \ell^{\eta_0}} (1+m) \mb{P}(\min_{1 \le i \le \ell} \boldbeta'(i) \ge -m) \] From the definition \cref{hm-def-2} and the fact that $h'(m) \ll 1 + m$, this is $\ll \ell^{3\eta_0 - 1/2}$, which is acceptable.

\emph{Bounding \cref{second-error}.} Again we bound the $\tau$ and $\psi$ terms trivially by $1$, so the contribution is 
\begin{equation}  \ll\sum_{m \le \ell^{\eta_0}} \sum_{t} \int_{\substack{\min_{0 \le i \le \ell} \beta'(i) = -m \\ \ell'(\beta') = t}}|\xi'_{t}|\, \mathrm{d}\mb{P}'(\beta')   \ll \sum_{m \le \ell^{\eta_0}} \sum_{t} \sum_{s = 0}^{\infty} \mb{P} \big( \min_{0 \le i \le \ell} \boldbeta'(i) = -m, \,\ell'(\boldbeta') = t, \, |\boldxi'_{t}| > s\big). \label{to-ell-dash}\end{equation}
We divide our sum over $t$ dyadically. For the sum over $t \le 2\ell$ we have the upper bound
\[ \ll \sum_{m \le \ell^{\eta_0}} \sum_{s = 0}^{\infty} \mb{P} \big( \min_{0 \le i \le \ell} \boldbeta'(i) = -m,\; \sup_{i \le 2\ell} |\boldxi'_{i}| > s).\] The contribution from $s \le \ell^{\eta_0}$ is $\ll \ell^{\eta_0} \sum_{m \le \ell^{\eta_0}}\mb{P} \big( \min_{0 \le i \le \ell} \boldbeta'(i) = -m) \ll \ell^{3\eta_0 - 1/2}$, as before. The contribution from $s > \ell^{\eta_0}$ is $\ll \ell^{O(1)}\sum_{s > \ell^{\eta_0}} \mb{P}(|1 - \Pois(1)| > s) < \ell^{-10}$.

Turning now to the higher dyadic ranges of $t$ in \cref{to-ell-dash}, we drop the $\min$ condition completely and bound this above by
\[ \sum_{m \le \ell^{\eta_0}} \sum_{j \ge 1} \sum_{2^j \ell \le t < 2^{j+1} \ell} \sum_{s = 0}^{\infty} \mb{P} \big( \ell'(\boldbeta') = t, |\boldxi'_{t}| > s\big). \] The contribution from $s > 2^j\ell$ is bounded by
\[ \ll \ell^{O(1)} \sum_{j \ge 1} 2^j \sum_{s > 2^j \ell} \mb{P} (|1 - \Pois(1)| > s), \] which is negligible. The contribution from $s \le 2^j \ell$ is 
\[ \ll \ell^{O(1)} \sum_{j \ge 1} 2^j \mb{P}( 2^j\ell \le \ell'(\boldbeta') \le 2^{j + 1} \ell),\] which is negligible by \cref{dyadic-ell-dash}.

Putting these estimates together completes the proof of \cref{mu-mu-prime-circ-close} in the case $i = 4$.

\vspace*{10pt}

\emph{Step 5.} Note that to obtain $\mu^{\prime (\ell)}_5 = \mu_*^{\prime (\ell)}$ from $\mu^{\prime (\ell)}_4$ we just need to drop the restriction $m \le \ell^{\eta_0}$ in \cref{mu-dash-ell-4}. Indeed, this then gives an unrestricted integral

\[ (\frac{2}{\pi})^{1/2}\int \mb{E} (\log_2 (\mbf{x} \mid \beta')_{\ell} - \ell)^+ \tau_{\mbf{x} \mid \beta'}(\ell)^c \psi (\log_2 \tau_{\mbf{x} \mid \beta'}(\ell))\, \mathrm{d}\mb{P}'(\beta') ,\] which is equal to $\int_{\R/\Z} \psi \, \mathrm{d} \mu^{\prime (\ell)}_*$ as defined in \cref{mu-prime-ell-star-def} by removing the conditioning on $\beta'$.
To estimate the difference between $\int_{\R/\Z} \psi \, \mathrm{d}\mu_4^{\prime (\ell)}$ and $\int_{\R/\Z} \psi \, \mathrm{d}\mu_5^{\prime (\ell)}$, by bounding the $\psi$ term trivially by $1$ it is sufficient to bound
\begin{equation}\label{step-5-tobound} \sum_{m > \ell^{\eta_0}}\int_{\min_{0 \le i \le \ell} \beta'(i) = -m} \mb{E} (\log_2 (\mbf{x} \mid \beta')_{\ell} - \ell)^+ \tau_{\mbf{x} \mid \beta'}(\ell)^c \, \mathrm{d}\mb{P}'(\beta') . \end{equation}
Clearly, we may restrict the integral to those $\beta'$ for which $\log_2 (\mbf{x} \mid \beta')_{\ell} \ge \ell$ for some instance of $\mbf{x}$. By \cref{lem172} (3) this implies that $\beta'(\ell') \ge -O(1 + |\xi'_{\ell'}|)$. From the definition of $\ell'$ it follows that \begin{equation}\label{ell-dash-and-ell} \ell' \ge \ell + \beta'(\ell') \ge \ell - O(1 + |\xi'_{\ell'}|).\end{equation} We now distinguish two cases: Case 1 (the typical case) is that $\min_{i \le \ell'-1} \beta'(i) \le - m/2$. Case 2 (the exceptional case) is that $\min_{i \le \ell'-1} \beta'(i) > - m/2$ (but still $\min_{i \le \ell} \beta'(i) = -m$).\vspace*{8pt}

\emph{Case 1.}  For $\tau_{\mbf{x} \mid \beta'}(\ell)$ we use the bound $\ll R \cdot (\ell')^{\kappa} 2^{-m/2}$, which follows from \cref{Y-upper-nonstar}. Here $R$ is a (minimal) positive integer parameter for which $|\xi'_i| \le R i^{\kappa}$ for all $i \le \ell' -1$, where (as usual) $\ell'$ is minimal so that $\ell' - \beta'(\ell') \ge \ell$.  Suppose that $\min_{i \le \ell'-1} \xi'_i = -t$. Then, since all $\xi'_i$ are $\le 1$, we have $\beta'(\ell'-1) \le \ell' - 2 - t$. By \cref{second-ell}, it follows that $t \le \ell - 2$, and so we can take $R := \ell$. We thus have the bound
\begin{equation}\label{tau-bd-17} \tau_{\mbf{x} \mid \beta'}(\ell) \ll \ell \ell' 2^{-m/2}. \end{equation}
For the $(\log_2 (\mbf{x} \mid \beta')_{\ell} - \ell)^+$ term we can use crude bounds. By \cref{lem172} (3)  and the fact that the steps of $\beta'$ are bounded above by $1$ we have
\begin{equation}\label{phi-ss} (\log_2 (\mbf{x} \mid \beta')_{\ell} - \ell)^+  \ll \ell' + |\xi'_{\ell'}|  . \end{equation} Combining \cref{tau-bd-17,phi-ss} (crudely) gives
\begin{equation*}
 (\log_2 (\mbf{x} \mid \beta'))_{\ell} - \ell)^+ \tau_{\mbf{x} \mid \beta'}(\ell)^c  \ll \ell 2^{-cm/2} ( (\ell')^2 + \ell'|\xi'_{\ell'}|) \ll \ell 2^{-cm/2} ( (\ell')^2 + |\xi'_{\ell'}|^2).
\end{equation*}
It follows that the contribution of Case 1 to \cref{step-5-tobound} is bounded by 
\begin{equation}\label{step-5-almost} \ll e^{-\Omega(\ell^{\eta_0})} \int \big( \ell'(\beta')^2 + |\xi'_{\ell'(\beta')}|^2 \big)\, \mathrm{d} \mb{P}'(\beta'). \end{equation}
We have $\int \ell'(\beta')^2\, \mathrm{d}\mb{P}'(\beta') \ll \ell^2$; to prove this we partition into dyadic ranges according to the value of $\ell'(\beta')$, using the trivial bound $\ll \ell^2$ for the range $\ell' \le 2\ell$, and \cref{dyadic-ell-dash} to show that the contribution from the higher dyadic ranges $2^j \ell \le \ell' < 2^{j+1} \ell$, $j \ge 1$, is negligible.

For the second term in \cref{step-5-almost}, we have
\[ \int |\xi'_{\ell'(\beta')}|^2\, \mathrm{d}\mb{P}'(\beta') \le  \ell + \sum_{j \ge 1} \sum_{2^j\ell \le t < 2^{j+1} \ell} \sum_s s \mb{P}(\ell'(\boldbeta') = t, |\boldxi'_{t}| > s),\] where the first term is the contribution from $\ell'(\beta') \le 2 \ell$, which is bounded above by $\sum_{i \le 2\ell} \mb{E}|\boldxi'_i|^2  \ll \ell$. Consider the contribution to the sum from some fixed $j \ge 1$. The contribution from $s > 2^j \ell$ is
\[ \ll 2^j \ell \sum_{s > 2^j \ell} s \mb{P}(|1 - \Pois(1)| > s) \ll 2^j \ell e^{-2^j \ell}.\] The contribution from $s \le 2^j \ell$ is
\[ \ll (2^j \ell)^2 \mb{P} (2^j \ell \le \ell'(\boldbeta') < 2^{j+1} \ell) \ll \ell^{3} 2^{3j} e^{-2^{j-4} \ell},\] by \cref{dyadic-ell-dash}. Summed over $j$, we get a tiny contribution of $\ll e^{-\Omega(\ell)}$ in both cases.

It follows from this analysis that \cref{step-5-almost} is bounded by $\ll e^{-\Omega(\ell^{\eta_0})}$, and hence that the contribution of Case 1 to \cref{step-5-tobound} is similarly bounded.\vspace*{8pt}

\emph{Case 2.} We remind the reader that this is the case that $\min_{i \le \ell'-1} \beta'(i) > - m/2$ and $\min_{i \le \ell} \beta'(i) = -m$. This implies that $\ell' \le \ell$ and that $\sum_{i = \ell'}^{\ell} (1 + |\xi'_i|) > m/2$. Therefore either (i) $\ell - \ell' \gg \sqrt{m}$, or (ii) there is some $i$, $\ell' \le i \le \ell$, with $|\xi'_i| \gg \sqrt{m}$. If (i) holds then, by \cref{ell-dash-and-ell}, we have $|\xi'_{\ell'}| \gg \sqrt{m}$, and so in both cases we have $\max_{\ell' \le i \le \ell} |\xi'_i| \gg \sqrt{m}$. Since the sum in \cref{step-5-tobound} is restricted to $m > \ell^{\eta_0}$, this implies $\max_{i \le \ell} |\xi'_i| \gg \ell^{\eta_0/2}$.

To bound the contribution of Case 2 to \cref{step-5-tobound}, we use \cref{phi-ss}, the fact that $\ell' \le \ell$ in Case 2, and the trivial bound $\tau_{\mbf{x} \mid \beta'}(\ell) \le 1$, to bound the contribution by 
\[ \ll \sum_{u\gg \ell^{\eta_0/2}}(\ell + u) \mb{P}\big(\max_{i \le \ell} |\xi'_i| = u\big) \le \ell \sum_{u\gg \ell^{\eta_0/2}}(\ell + u) \mb{P}( \Pois(1)  = u + 1) < \ell^{-10}. \]
Thus, the contribution to \cref{step-5-tobound} from Case 2 is minuscule.

Cases 1 and 2 together establish \cref{mu-mu-prime-circ-close} in the case $i = 5$, and this completes the proof of \cref{sec17-main}.

\vspace*{10pt}

\vspace*{10pt}

\part*{Appendices}

\appendix

\section{Almost positive random walks}\label{poisson-app} 

In this appendix we give the proofs of some results in \cref{random-walk-sec}. The main references for this material are various parts of Feller \cite{feller2}. We also benefitted from consulting \cite[Section 2]{aidekon-shi}.

As in the main part of the paper, we consider a random walk $\boldbeta(N) = \boldxi_1 + \cdots + \boldxi_N$ and a complementary random walk $\boldbeta'(N) = \boldxi'_1 + \cdots + \boldxi'_N$, where $\boldxi_i \samedist \Pois(1) - 1$ and $\boldxi'_i \samedist 1 - \Pois(1)$. Much of the material under discussion here is, however, valid much more generally.

\subsection{Positive random walks} \label{a1-sec} In this subsection $n$ is a dummy variable. The first task is to understand $h(m)$ and $h'(m)$ for $m \in \{-1,0\}$.  Important ingredients in the analysis are the following identities of Sparre Andersen \cite{sparre-andersen} (subsequently generalised by Spitzer \cite{spitzer}). Denote $p_n := \mb{P}(\min_{1 \le i \le n} \boldbeta(i) > 0)$ and $\tilde{p}_n := \mb{P}(\min_{1 \le i \le n} \boldbeta(i) \ge 0)$. Then 
\begin{equation}\label{strict-generating} f(z) := 1 + \sum_{n = 1}^{\infty} p_n z^n = \exp \big( \sum_{n = 1}^{\infty} \frac{z^n}{n} \mb{P}(\boldbeta(n) > 0)\big)\end{equation} and
\begin{equation}\label{non-strict} \tilde{f}(z) := 1 + \sum_{n = 1}^{\infty} \tilde{p}_n z^n = \exp \big( \sum_{n = 1}^{\infty} \frac{z^n}{n} \mb{P}(\boldbeta(n) \ge 0)\big).\end{equation}
A good source for the proof is \cite[XII. 7]{feller2}. Since the $\boldxi_i$ have mean zero it makes sense to work with the following consequences of \cref{strict-generating,non-strict}: 
\[ f(z) (1- z)^{1/2}  = \exp \big( \sum_{n = 1}^{\infty} \alpha_n z^n  \big) \quad \mbox{and} \quad \tilde{f}(z)(1 - z)^{1/2} = \exp \big( \sum_{n = 1}^{\infty} \tilde{\alpha}_n z^n\big),\]
where here
\[ \alpha_n := \frac{1}{n} \big( \mb{P} (\boldbeta(n) > 0) - \frac{1}{2}\big)\quad \mbox{and} \quad \tilde\alpha_n := \frac{1}{n} \big( \mb{P} (\boldbeta(n) \ge 0) - \frac{1}{2}\big) \]

Applying a Tauberian theorem \cite[XII. 8, Theorem 1a]{feller2} we obtain
\begin{equation}\label{asymptotics} p_n \sim e^{\alpha} (\pi n)^{-1/2} \quad \mbox{and} \quad  \tilde{p}_n \sim e^{\tilde \alpha} (\pi n)^{-1/2}.\end{equation} where $\sim$ means `asymptotic to' and
\begin{equation}\label{alpha-defs} \alpha := \sum_{n = 1}^{\infty} \frac{1}{n} \big( \mb{P}(\boldbeta(n) > 0) - \frac{1}{2} \big) \quad \mbox{and} \quad  \tilde{\alpha} := \sum_{n = 1}^{\infty}\frac{1}{n} \big( \mb{P}(\boldbeta(n) \ge  0) - \frac{1}{2} \big)\end{equation} 

In particular, from \cref{asymptotics} we see that $h(-1)$ and $h(0)$ exist and that 
\begin{equation} \label{h0-form} h(-1) = e^{\alpha} \pi^{-1/2}, \quad h(0) = e^{\tilde\alpha} \pi^{-1/2}.\end{equation}
Similar statements hold for $\boldbeta'$ and the associated probabilities $p'_n, \tilde p'_n$. In particular
\begin{equation} \label{h0dash-form} h'(-1) = e^{\alpha'} \pi^{-1/2}, \quad h'(0) = e^{\tilde\alpha'} \pi^{-1/2},\end{equation} where $\alpha', \tilde\alpha'$ are defined analogously with reference to $\boldbeta'$. Observe that 
\begin{equation} \label{alpha-alpha-prime} \alpha = -\tilde\alpha', \qquad \tilde\alpha = -\alpha'.\end{equation}

We now recall a result of Spitzer \cite[Theorem 3.4]{spitzer-tauberian} (see also \cite[XVIII.5, Theorem 1]{feller2}). This states the following. Let $Z_+$ be the first strict ascending ladder height of $\boldbeta$, that is to say $\boldbeta(\tau_+)$ where $\tau_+ := \inf\{ i \ge 1 : \boldbeta(i) > 0\}$ is the first strict ascending ladder epoch of $\boldbeta$. (The use of the letter $\tau$ here is unrelated to that in \cref{tau-sec}.) Define $Z'_+$ analogously with reference to $\boldbeta'$. Define also $\tilde\tau_+ := \inf\{ i \ge 1: \boldbeta(i) \ge 0\}$ and $\tilde Z_+ := \boldbeta(\tilde\tau_+)$, and similarly for $\tilde Z'_+$.

Then Spitzer's result is that \begin{equation}\label{spitzer-beta} \mb{E} Z_+ = \frac{1}{\sqrt{2}} e^{-\alpha}, \quad \mb{E} \tilde Z'_+ = \frac{1}{\sqrt{2}} e^{\alpha},\quad \mb{E} Z'_+ = \frac{1}{\sqrt{2}} e^{-\alpha'}, \quad \mb{E} \tilde Z_+ = \frac{1}{\sqrt{2}} e^{\alpha'}.\end{equation} 
Since the steps of $\boldbeta'$ are $1 - \Pois(1)$ variables bounded above by $1$, we have $Z'_+ = 1$ deterministically and therefore 
\begin{equation}\label{alpha-dash} \alpha' = -\tfrac{1}{2}\log 2.\end{equation}
One may also observe the identity
\begin{equation}\label{borel-eval} \tilde\alpha - \alpha = \sum_{n = 1}^{\infty} \frac{1}{n} \mb{P}(\boldbeta(n) = 0) = \sum_{n = 1}^{\infty} \frac{n^{n - 1} e^{-n}}{n!} = 1,\end{equation} which goes back to Borel \cite{borel} and is equivalent to the fact that the Borel distribution with parameter $1$ is well-defined. 

Relations \cref{h0-form,h0dash-form,alpha-alpha-prime,spitzer-beta,alpha-dash,borel-eval} allow for a complete evaluation of all the constants above, and one may easily check that 
\begin{equation*} \alpha = \tfrac{1}{2} \log 2 - 1, \quad \tilde \alpha = \tfrac{1}{2} \log 2, \quad \alpha' = -\tfrac{1}{2} \log 2, \quad \tilde \alpha' = 1 - \tfrac{1}{2} \log 2  \end{equation*} and thus that $\mb{E}Z_+ = e/2$ and
\begin{equation}\label{h0h0} h(-1) = (2/\pi)^{1/2}e^{-1}, \quad h(0) = (2/\pi)^{1/2}, \quad h'(-1) = 1/\sqrt{2\pi}, \quad h'(0) = e/\sqrt{2\pi} .\end{equation}

\begin{remarks} Most of the above discussion holds with minor modification for arbitrary random walks with i.i.d. steps with mean zero and variance 1. The identities \cref{strict-generating,non-strict} make critical use of the i.i.d. assumption. From \cref{alpha-dash} onwards we made critical use of the fact that $\boldbeta'$ is `skip-free', that is to say has increments bounded above by $1$. In the general skip-free case the analogue of \cref{borel-eval} is that the RHS is $-\log \mb{P}(\boldxi' = 1)$ (we omit the proof of this fact, which is not needed in the paper).
\end{remarks}

\subsection{Renewal theorem} In the proof of \cref{first-rand-walk-pos} we will also require some basic renewal theory. In particular we will need \cref{prop12} below, which is a quantitative version of the renewal theorem in a particular setting. The result is essentially due to Gel'fond \cite{gelfond} but we give a self-contained proof (closely related to that of Gel'fond) in our setting.

Consider a random variable $X$ with the following properties:

\begin{enumerate}
\item $X$ takes values in $\{1,2,\dots\}$;
\item $\mb{P}(X = 1) > 0$;
\item We have $\mb{E} (1 + \delta_0)^X < \infty$ for some $\delta_0 > 0$.
\end{enumerate}
The key example for us will be when $X$ is the strict ascending ladder height $Z_+$ of the walk $\boldbeta$, as defined above. With this example, (1) above is obvious. Item (2) is also obvious; $\mb{P}(X = 1) \ge \mb{P}(\boldxi_1 = 1)$. Finally, $\mb{P}(X = j) \le  \mb{P}(\boldxi \ge j)$ which decays faster than exponential, so in this case we have (3) for all $\delta_0$.

Write $p_j = \mb{P}(X = j)$ for brevity in what follows. Consider the moment generating function $f(z) := \mb{E} z^X$; this will be holomorphic in $|z| < 1 + \delta_0$.

\begin{lemma}\label{lem11}
There is some $0 \le \delta_1 < \delta_0$ such that the only zero of $f(z) - 1$ in $|z| \le 1 + \delta_1$ is a simple zero at $z = 1$.
\end{lemma}
\begin{proof}
If $|z| < 1$ we have $|f(z)| < 1$ and so there are no zeros of $f(z) - 1$ with $|z| < 1$. By a compactness argument it suffices to check that the only zero of $f(z) - 1$ with $|z| = 1$ is at $z = 1$, and that it is simple. Suppose that $z = e^{i \theta}$. Then 
\[ \Re f(z) =  \sum_j p_j \cos (j\theta) \le \sum_{j \neq 1} p_j + p_1 \cos \theta =1 + p_1 (\cos \theta - 1).\]
   Since $p_1 > 0$, this is $< 1$ unless $\cos \theta =1$, that is to say $\theta$ is a multiple of $2 \pi$. Thus the only zero of $f(z) - 1$ on $|z| = 1$ is at $z = 1$. We have $f'(1) = \mb{E} X > 0$ and so the zero at $z = 1$ is simple.
\end{proof}

\begin{remark}
The use of a compactness argument means that $\delta_1$ is ineffective. In the particular case corresponding to the walk $\boldbeta$, an effective $\delta_1$ could be obtained without undue difficulty if required.    
\end{remark}

\begin{proposition}\label{prop12}
Suppose that $X$ satisfies \textup{(1)}, \textup{(2)} and \textup{(3)} above for some $\delta_0$. Let $\delta_1$ be as in the conclusion of \cref{lem11}, and set $\lambda := \mb{E} X$. Then uniformly for integers $m > 0$ we have
\[ \sum_{r = 1}^{\infty} \mb{P}(X_1 + \cdots + X_r =  m) = \frac{1}{\lambda} + O\big((1 + \delta_1)^{-m}\big).\] 
\end{proposition}
\begin{proof}
We have 
\[ \mb{P} (X_1 + \cdots + X_r = m) = \frac{1}{2\pi i} \int_{|z| = 1/2} f(z)^r z^{- m - 1} dz.\] 
Note that for uniformly for $|z| = \frac{1}{2}$ we have $|f(z)| \le \mb{P}(X = 0) + \frac{1}{2}\mb{P}(X > 0) < 1$. Summing over $r$ gives
\[ \sum_{r = 1}^{\infty} \mb{P}(X_1 + \cdots + X_r = m) = \frac{1}{2\pi i} \int_{|z| = 1/2} \frac{f(z)}{1 - f(z)} z^{-m - 1} dz. \]
Now move the contour to $|z| = 1 + \delta_1$. In doing this we pick up a contribution of $1/\lambda$ from the pole of the integrand at $z = 1$. There are no other poles by \cref{lem11}. It follows that 
\[ \sum_{r = 1}^{\infty} \mb{P}(X_1 + \cdots + X_r = m) = \frac{1}{\lambda} + \frac{1}{2\pi i} \int_{|z| = 1 + \delta_1} \frac{f(z)}{1 - f(z)} z^{-m - 1} dz . \] The result follows since $\frac{f(z)}{1 - f(z)}$ is uniformly bounded on $|z| = 1 + \delta_1$, and $|z|^{-m} = (1 + \delta_1)^{-m}$.
\end{proof}

\subsection{Proof of \cref{first-rand-walk-pos} and \cref{rand-walk-aux}} \label{a3-sec} We are now ready to prove the main statements needed in the paper concerning the existence of $h(m), h'(m)$ and the asymptotic behaviour of these quantities. We mostly consider the (harder) case of $h'$, indicating how the analysis simplifices in the much simpler case of $h$ as we go along.

We first give a reference for the proof of \cref{first-rand-walk-pos} (3), or more precisely of \cref{prop81-3-i}, since \cref{prop81-3-ii} follows immediately from it. This is the statement that \begin{equation}\label{prop81-2} N^{1/2} \mb{P}\big(\min_{1 \le i \le N} \boldbeta(i) \ge -m\big), \; N^{1/2} \mb{P}\big(\min_{1 \le i  \le N} \boldbeta'(i) \ge -m\big) \asymp m + 1,\end{equation} uniformly for $0 \le m \le \sqrt{N}$.  For $m = 0$ this is immediate from the existence of $h(0)$ and $h'(0)$. For $1 \le m \le \sqrt{N}$, items (i) (upper bound) and (iii) (lower bound) of \cite[Lemma~3.3]{PP95} give the result. 

\begin{remark} As noted in \cite{PP95}, the upper bound is due to Kozlov \cite[Theorem A, equation (13)]{Koz76}; however the proof in \cite{PP95} is different and easier. The lower bound is due to Zhang \cite[Lemma 1]{zhang}. Finally we observe that the upper bound in \cref{prop81-2} (that is, the LHS is $\ll m + 1$) holds without any restriction on $m$ for trivial reasons.
\end{remark}

We now sketch the proof of \cref{rand-walk-aux} following \cite{addario-berry-reed-2}.

\begin{proof}[Proof of \cref{rand-walk-aux} (sketch)]  Recall that this is the statement that $\mb{P}(\boldbeta(N) = x, \; \min_{1 \le i \le N} \boldbeta(i) \ge 0) \ll \min \big(N^{-1}, (1 + x) N^{-3/2}\big)$. The argument in \cite{addario-berry-reed-2} goes essentially as follows: consider the first, last and middle thirds of the walk up to $N$, with the last third of the walk reversed. The first third must be nonnegative to time $\lfloor N/3\rfloor$, an event with probability $\ll N^{-1/2}$ by the existence of $h(0)$. The final third (reversed) must be bounded below by $-x$, an event with probability $\ll \min \big(1, (1 + x) N^{-1/2}\big)$ by \cref{prop81-2}. Conditioned on the first and last thirds, the middle third of the walk sums to some fixed value, and the probability of this is $\ll N^{-1/2}$ by a result of Kesten (which in our particular case is an easy computation, see \cref{p-pointwise} below). 
\end{proof}

Before turning to the proof of items (1) and (2) of \cref{first-rand-walk-pos}, we prove the following (standard) auxiliary estimate on strict ascending ladder times.

\begin{lemma}
\label{first-ascending-ladder} Let $\tau_+ := \inf\{ i \ge 1 : \boldbeta(i) > 0\}$ be the first strict ascending ladder epoch of $\boldbeta$. We have $\mb{P}(\tau_+ = t) \ll t^{-3/2}$ uniformly in $t$. An analgous estimate holds for $\boldbeta'$.
\end{lemma}
\begin{proof} We may assume $t \ge 2$. Couple $\boldbeta$ to a copy of $\boldbeta'$ via $\boldbeta(i) = -\boldbeta'(i)$. Then $\tau_+ = t$ is precisely the event that $\min_{1 \le i \le t - 1} \boldbeta'(i) \ge 0$ and $\boldbeta(t) > 0$, the probability of which is
\[ \sum_{x \ge 0} \mb{P}\big( \boldbeta'(t - 1) = x, \min_{1 \le i \le t - 1} \boldbeta'(i) \ge 0\big) \mb{P}(\boldxi_t > x).\] By \cref{rand-walk-aux}, the first probability is $\ll (1 + x) t^{-3/2}$. Since $\mb{P}(\boldxi_t > x) \ll (1 + x)^{-3}$, the result follows.
\end{proof}

\begin{proof}[Proof of \cref{first-rand-walk-pos} (1) and (2)] We have already established the cases $m = -1,0$ in \cref{a1-sec}, so we assume $m \ge 1$ in what follows. We first give the proof for $\boldbeta'$ and $h'$, which is harder. Couple $\boldbeta'$ to a copy of $\boldbeta$ via $\boldbeta = - \boldbeta'$. Let $N$ be a large positive integer ($N > m^{20}$ will do). 

Associated to $\boldbeta$ are the strict ascending ladder epochs $T_1,T_2,\dots$, defined by $T_j = \inf\{ i : \boldbeta(i) > \boldbeta(T_{j-1})\}$, and the strict ascending ladder heights $H_j := \boldbeta(T_j)$. The variable $T_j$ is distributed as $\tau_1 + \cdots + \tau_j$ where the $\tau_i$ are i.i.d.\ copies of the first strict ascending ladder epoch $\tau_+$ and $H_j$ is distributed as $X_1 + \cdots + X_j$ where the $X_i$ are i.i.d.\ copies of the first strict ascending ladder height $Z_+ = \boldbeta(\tau_+)$. Then, foliating according to $r$, the number of ladder epochs up to time $N$, we have
\begin{equation}\label{min-minus-m} \mb{P}\big(\min_{1 \le i \le N} \boldbeta'(i) \ge -m\big) = \mb{P}\big(\max_{1 \le i \le N} \boldbeta(i) \le m\big)  = \sum_{r = 0}^m \sum_{t = 0}^N \mb{P}(T_r = t, H_r \le m) \mb{P}( \tau_+  > N - t).\end{equation} Here we adopt the convention that $T_0 = H_0 = 0$ deterministically. Note also that the number of ladder epochs is bounded above by $m$.

We assemble various facts. First, by the existence of $h'(0)$ we have for any $\ell \in \N$ that  
\begin{equation}\label{pos-prob} \mb{P}(\tau_+ > \ell) = \mb{P}\big(\max_{1 \le i \le \ell} \boldbeta(i) \le 0\big) = \mb{P}\big(\min_{1 \le i \le \ell} \boldbeta'(i) \ge 0\big) = (h'(0) + o(1)) \ell^{-1/2}.\end{equation}
Next note that uniformly for $1 \le r \le m$ and $t \ge 1$ we have
\begin{equation}\label{tk-pointwise} \mb{P}(T_r = t) \le r \sup_{x \ge t/r} \mb{P}(\tau_+ = x) \ll r^{5/2} t^{-3/2} \ll m^{5/2} t^{-3/2},\end{equation} where here we used \cref{first-ascending-ladder}. 

We first estimate the contribution to \cref{min-minus-m} from $t \ge \sqrt{N}$. For this we use $\mb{P}(\tau_+ > N - t) \ll (N+ 1- t)^{-1/2}$ (a consequence of \cref{pos-prob}) and \cref{tk-pointwise}, and we ignore the $H_r \le m$ condition. This gives a bound of
\[ \ll m^{7/2} \sum_{\sqrt{N} < t \le N} t^{-3/2} (N +1 - t)^{-1/2} \ll m^{7/2} N^{-3/4} = o(N^{1/2}).\] 
In the remaining range $t \le \sqrt{N}$ in \cref{min-minus-m}, we have $\mb{P}(\tau_+ > N-t) = (h'(0) + o(1)) N^{-1/2}$ by \cref{pos-prob}. Putting these facts together we obtain 
\[ \mb{P}\big(\min_{1 \le i \le N} \boldbeta'(i) \ge -m\big) = (h'(0) + o(1)) N^{-1/2} \sum_{r = 0}^m \sum_{t \le \sqrt{N}} \mb{P}(T_r = t, H_r \le m) + o(N^{-1/2}).\]
We can add back in the sum over $t > \sqrt{N}$, at a cost of a further error of $o(N^{-1/2})$, by further use of \cref{tk-pointwise}. The sum over $t$ then becomes redundant and we have
\begin{align*} \mb{P}\big(\min_{1 \le i \le N} \boldbeta'(i) \ge -m\big) & = (h'(0) + o(1)) N^{-1/2} \sum_{r = 0}^m \mb{P}(H_r \le m) + o(N^{-1/2}) \\ & = (h'(0) + o(1)) N^{-1/2}\big( 1 +  \sum_{r = 1}^m \mb{P}(X_1 + \cdots + X_r \le m)\big) + o(N^{-1/2}) .\end{align*}
By \cref{prop12} it now follows that $h'(m)$ exists and that 
\begin{equation}\label{hm-first} h'(m) =  h'(0)\big(\frac{m}{\mb{E} Z_+} + O(1)\big).\end{equation}

By \cref{h0dash-form,alpha-alpha-prime,spitzer-beta} we have $h'(0)/\mb{E} Z_+ = (2/\pi)^{1/2}$. Together with \cref{hm-first}, this completes the proof of \cref{first-rand-walk-pos} (1) and (2) in the case of $h'$.

The analysis of $h(m)$ is much easier. By exactly the same argument with the roles of $\boldbeta,\boldbeta'$ switched, we obtain

\[ \mb{P}\big(\min_{1 \le i \le N} \boldbeta(i) \ge -m\big)  = (h(0) + o(1)) N^{-1/2} \big( 1 + \sum_{r = 1}^m \mb{P}(X'_1 + \cdots + X'_r \le m)\big) + o(N^{-1/2}) .\]
Here, $X'_i$ are i.i.d.\ copies of $Z'_+$, the first strict ascending ladder height of $\boldbeta'$. However $Z'_+ = 1$ deterministically, and so we obtain $h(m) = h(0) (m + 1)$. By \cref{h0h0} we have $h(0) = (2/\pi)^{1/2}$ and this completes the proof of \cref{first-rand-walk-pos} (1) and (2) in the case of $h$.\end{proof}

\subsection{Local limit theorems} \label{loc-lim-thm-app} In this subsection we give the proof of \cref{lem13.5}. The main ingredient of the proof will be (a weakening of) a result of Denisov, Tarasov and Wachtel \cite{DTW24b}, which we quote as \cref{dtw-thm} below; we will give a self-contained proof of that in \cref{dtw-appendix}. The proofs of \cref{lem13.5-expr,lem13.5-exprb} are identical (we will not use the fact that $h(m)$ is exactly $(2/\pi)^{1/2}(m + 1)$ for $m \ge 0$, which does not hold for $h'(m)$), so we establish only \cref{lem13.5-expr}. 

Recall that 
\[ h(-1) := \lim_{N \rightarrow \infty} N^{1/2} \mb{P} \big(\min_{1 \le i \le N} \boldbeta(i) \ge 1\big) = \lim_{N \rightarrow \infty} N^{1/2} \mb{P} \big(\min_{1 \le i \le N} \boldbeta(i) > 0\big).\] We showed in \cref{h0h0} that $\tilde h(-1)$ exists and 
\begin{equation}\label{h0-tilde-val} \tilde h(-1) = e^{\alpha} \pi^{-1/2} = (2/\pi)^{1/2} e^{-1},\end{equation} with $\alpha$ defined as in \cref{alpha-defs}. We need the following auxiliary lemma.

\begin{lemma}\label{h-htilde} Let $m \ge 0$ be an integer. We have \[  h(m) - h(m-1) = h(-1) \sum_{q = 1}^{\infty} \mb{P}\big(\min_{1 \le i \le q} \boldbeta(i) = -m,\; \boldbeta(q) = -m\big).\] 
\end{lemma}
\begin{proof} Let $N$ be some very large positive integer. By summing over the position $q$ of the last visit to $-m$ up to time $N$ we have 
\begin{equation}\label{m-bet-gam} \mb{P}\big( \min_{1 \le i \le N} \boldbeta(i) = -m\big) = \sum_{q = 1}^N \mb{P}\big(\min_{1 \le i \le q} \boldbeta(i) = -m,\; \boldbeta(q) = -m\big) \cdot \mb{P}\big(\min_{1 \le i \le N - q} \boldgamma(i) \ge 1\big)\end{equation} where $\boldgamma$ is the shifted walk $\boldgamma(i) := \boldbeta(q + i) + m$, and the second probability term on the right is absent when $q = N$. The contribution to the sum from $q \ge N^{1/2}$ is, using \cref{argim-each-x} and the existence of $h(-1)$,
\[ \ll (1 + m) \sum_{q > N^{1/2}} q^{-3/2} (N +1 - q)^{-1/2} \ll (1 + m) N^{-3/4}.\] By the existence of $h(-1)$ again, uniformly for $q \le N^{1/2}$ we have $\mb{P}\big(\min_{1 \le i \le N - q} \boldgamma(i) \ge 1\big) =  (h(-1) + o(1)) N^{-1/2}$.  Thus \cref{m-bet-gam} gives
\begin{align*} \mb{P}\big( \min_{1 \le i \le N} \boldbeta(i) = -m\big) = (h(-1) + o(1)) N^{-1/2}\sum_{q = 1}^{N^{1/2}} \mb{P}\big(\min_{1 \le i \le q} \boldbeta(i) & = -m,\; \boldbeta(q) = -m\big) \\ & + O \big( (1 + m) N^{-3/4}\big).\end{align*}
Since 
\[ \mb{P}\big( \min_{1 \le i \le N} \boldbeta(i) = -m\big) = \mb{P}\big( \min_{1 \le i \le N} \boldbeta(i) \ge -m\big) - \mb{P}\big( \min_{1 \le i \le N} \boldbeta(i) \ge  -(m-1)\big),\] multiplying by $N^{1/2}$ and sending $N$ to infinity gives the result.
\end{proof}

Now we quote the main external input to the proof of \cref{lem13.5} (which is essentially the case $m = -1$ of the result). As previously mentioned, it is a weak version of a result of Denisov, Tarasov and Wachtel \cite[Theorem 1]{DTW24b}.

\begin{theorem}\label{dtw-thm}
There is a constant $\eps > 0$ such that the following holds. Uniformly for positive integers $N,x$ we have
\begin{equation}\label{loc-limit-crude} \mb{P}\big(\boldbeta(N) = x, \; \min_{1 \le i \le N} \boldbeta(i) > 0\big) = N^{-1}h(-1) W\big(\frac{x}{\sqrt{N}}\big) + O(N^{-1 - \eps}).
\end{equation}
\end{theorem}
\begin{proof} See \cref{dtw-appendix}. 
\end{proof}

\begin{proof}[Proof of \cref{lem13.5}] We prove \cref{lem13.5-expr}; the proof of \cref{lem13.5-exprb} is identical. \cref{dtw-thm} is the case $m = -1$, so we assume that $m \ge 0$ in what follows. We may assume $N$ sufficiently large throughout. Set $\eps_0 := \min(\frac{\eps}{2}, \frac{1}{10})$, where $\eps$ is as in \cref{dtw-thm}.  We remind the reader of the desired statement, which, replacing the dummy variable $m$ by $M$, is that uniformly for $M, x\in \Z_{\ge 0}$ we have
\begin{equation*} \mb{P} \big(\boldbeta(N) = x,\; \min_{1 \le i \le N} \boldbeta(i) \ge -M\big) = N^{-1} h(M) W\big(\frac{x}{\sqrt{N}}\big) + O\big((1 + M)^{O(1)} N^{-1 - \eps_0}\big).\end{equation*}
 We may assume throughout that $M \le N^{\eps_0}$, since otherwise the result has no content by choosing a suitably large $O(1)$ exponent.
By a telescoping sum argument, it suffices to show 
\begin{equation}\label{lem13.5-expr-to-tele} \mb{P} \big(\boldbeta(N) = x,\; \min_{1 \le i \le N} \boldbeta(i) = -m\big) = N^{-1} (h(m) - h(m-1)) W\big(\frac{x}{\sqrt{N}}\big) + O \big( (1 + m) N^{-1 - \eps_0}\big)\end{equation} uniformly for $0 \le m \le N^{\eps_0}$; to see this, sum \cref{lem13.5-expr-to-tele} over $m = 0,1,2\dots, M$ together with \cref{loc-limit-crude}. 

Henceforth we focus on \cref{lem13.5-expr-to-tele}. If $x > N^{1/2 + \eps_0}$ then both the LHS and the first term on the RHS of \cref{lem13.5-expr-to-tele} are $\ll N^{-10}$; for the LHS this follows by dropping the $\min_{1 \le i \le N} \boldbeta(i) = -m$ condition and using the large deviation bound \cref{lemmaF.1} to bound $\mb{P}(\boldbeta(N) = x)$, while for the first term on the RHS we use the bound $h(m) - h(m-1) = O(1)$, which follows from \cref{first-rand-walk-pos} (2). 

Suppose from now on that $x \le N^{1/2 + \eps_0}$. We foliate the event on the LHS of \cref{lem13.5-expr-to-tele} according to the largest value of $q$, $1 \le q  \le N$, such that $\boldbeta(q) = -m$. Set $Q := \lfloor N^{2\eps_0}\rfloor$. The contribution from $Q < q \le N$ is, using \cref{argim-each-x}, bounded above by
\begin{align*}  \sum_{Q < q \le N}\mb{P} \big(\boldbeta(N) = x,\; \min_{1 \le i \le N} \boldbeta(i) = -m, \; \boldbeta(q) = -m\big) & \ll (1 + m) \sum_{Q < q \le N} q^{-3/2} (N + 1 - q)^{-1}  \\ & \ll  (1 + m)  \big(\frac{1}{Q^{1/2} N}  + \frac{\log N}{N^{3/2}}\big) . \end{align*} This may be absorbed into the error term in \cref{lem13.5-expr-to-tele} due to the choice of $Q$. 
The remaining contribution to the LHS of \cref{lem13.5-expr-to-tele} may be written as $\sum_{q = 1}^Q p_1(q) p_2(q)$ where
\[ p_1(q) :=  \mb{P}\big( \boldgamma(N - q) = x + m, \min_{1 \le i \le N - q} \boldgamma(i) > 0\big), \;\quad p_2(q) :=  \mb{P} \big(  \min_{1 \le i \le q} \boldbeta(i) = -m, \; \boldbeta(q) = -m\big),\] and where $\boldgamma(i) = \boldbeta(q + i)$ is the shifted walk, which is a $\Pois(1) - 1$ walk independent of $(\boldbeta(i))_{1 \le i \le q}$. Thus to complete the proof of \cref{lem13.5-expr-to-tele} our task is to show, assuming $x \le N^{1/2 + \eps_0}$, that 
\begin{equation}\label{lema5-rem} \sum_{q = 1}^Q p_1(q) p_2(q) = N^{-1} (h(m) - h(m-1)) W\big(\frac{x}{\sqrt{N}}\big) + O \big( (1 + m) N^{-1 - \eps_0}\big).\end{equation}

We first dispense with the case $x < N^{1/2 - \eps_0}$. By \cref{rand-walk-aux}, $p_1(q) \ll (1 + x + m) N^{-3/2} \ll N^{-1 - \eps_0}$, and by \cref{est-72-geq}, independent of any assumption on $x$ we have\begin{equation}\label{p2q-bd} p_2(q) \ll (1 + m) q^{-3/2}.\end{equation} Thus $\sum_{q = 1}^Q p_1(q) p_2(q) \ll (1 + m) N^{-1-\eps_0}$, which is the size of the error term in \cref{lema5-rem}. Using \cref{first-rand-walk-pos} (2), the term $N^{-1} (h(m) - h(m - 1)) W\big(\frac{x}{\sqrt{N}})$ is bounded by $\ll N^{-1 - \eps_0}$, which may also be absorbed into the error term in \cref{lema5-rem}. Thus \cref{lema5-rem} holds in this regime.

For the remainder of the proof we assume $N^{1/2 - \eps_0} \le x \le N^{1/2 + \eps_0}$. We apply \cref{dtw-thm} with $\tilde x := x + m$ and $\tilde N := N - q$. This application of \cref{dtw-thm} implies that if $1 \le q \le Q$, then
\begin{equation*} p_1(q) =  (N - q)^{-1} h(-1)W\big( \frac{x + m}{\sqrt{N - q}}\big) + O(N^{-1 - \eps}).\end{equation*} We have $\frac{x + m}{\sqrt{N - q}} = \frac{x}{\sqrt{N}} + O(N^{\eps_0 - 1/2} )$ and so, since $W'$ is uniformly bounded, $W \big(\frac{x + m}{\sqrt{N - q}}\big) = W\big( \frac{x}{\sqrt{N}}\big) + O(N^{\eps_0 - 1/2})$. Also $(N - q)^{-1} = N^{-1} + O( QN^{-2}) = N^{-1} + O(N^{2\eps_0-2})$. Since $\eps_0 \le \frac{1}{10}$, we see that 
\begin{equation}\label{p1-bd} p_1(q) = N^{-1}h(-1) W\big(\frac{x}{\sqrt{N}}\big) + O(N^{-1  - \eps}).\end{equation}

Using \cref{h-htilde}, it follows that 
\[ \sum_{q = 1}^{\infty} p_1(q) p_2(q) = N^{-1} (h(m) - h(m-1)) W\big(\frac{x}{\sqrt{N}}\big) + O(N^{-1 - \eps}) \sum_{q = 1}^{\infty} p_2(q),\] and hence that 
\[ \sum_{q = 1}^{Q} p_1(q) p_2(q) = N^{-1} (h(m) - h(m-1)) W\big(\frac{x}{\sqrt{N}}\big) + O(N^{-1 - \eps}) \sum_{q = 1}^{\infty} p_2(q) + O\big( \sum_{q > Q} p_1(q) p_2(q)\big).\] 
By \cref{p2q-bd} the second term here is $\ll (1 + m) N^{-1 - \eps}$, which may be absorbed into the error term in \cref{lema5-rem}. Finally, by \cref{rand-walk-aux} (or \cref{p1-bd}) we have $p_1(q) \ll N^{-1}$, and from this and \cref{p2q-bd} we see that $\sum_{q > Q} p_1(q) p_2(q) \ll (1 + m) N^{-1} Q^{-1/2}$, which may again be absorbed into the error term in \cref{lema5-rem} since $Q = \lfloor N^{2\eps_0}\rfloor$.
\end{proof}

\section{On results of Denisov, Tarasov and Wachtel}\label{dtw-appendix}

In this appendix we give a more-or-less self-contained proof of \cref{dtw-thm}. This result (and in fact a stronger statement) follows essentially immediately from \cite[Theorem 1]{DTW24b} in the case $r = 1$, specialised to our situation. However, that reference is rather lengthy and technical (as befits the authors' aims of great precision and generality) and moreover depends on the earlier paper \cite{DTW24}, which has similar properties.

The argument we give here is much shorter and is in part very closely based on \cite{DTW-berry-esseen}. Our need for less precision, as well as the fact that our random walks have specific rather nicely-behaved increments, is what allows for a much more compact treatment.

Throughout this appendix, we work with the walk $\boldbeta$ with $\Pois(1) - 1$ steps, but we write the argument such that it would also work for $\boldbeta'$ with only trivial changes. In particular we never rely on the fact that the steps of $\boldbeta$ are bounded below by $-1$.

\subsection{Berry-Ess\'een Theorem for positive random walks} The main result of this section, \cref{berry-esseen-prop}, as well as the proof, are exactly that of \cite{DTW-berry-esseen} except we can remove a technical smoothing device (as we have a discrete random variable) and we do not need to carefully track explicit constants. Additionally, we remark on how one could, if desired, replace appeals to some of the literature on general random walks with computations using Stirling's formula.

Denote 
\begin{equation}\label{tau-ladder-def} \tau := \inf \{i \ge 1 : \boldbeta(i) \le 0\}.\end{equation}
Observe that, with the notation of the previous appendix, $\tau = \tilde \tau'_+$, and so in particular
\begin{equation} \label{boldbeta-tau} |\boldbeta(\tau)| = \tilde Z'_+.\end{equation}
Here is the main result of this subsection.
\begin{proposition}\label{berry-esseen-prop} 
Uniformly for integers $N \ge 1$ and $x > 0$ we have
\[ \mb{P}\big(\boldbeta(N) \ge x, \;\tau > N\big) = \tilde{h}(0) N^{-1/2} e^{-x^2/2N} + O(N^{-1}).\]
\end{proposition}
Before giving the proof, we begin by assembling some preliminaries. Define
\begin{equation}\label{pnx-def} p(N, x) := \mb{P}(\boldbeta(N) = x) = \mb{P}(\Pois(N) = N + x).\end{equation}
\begin{lemma}
Uniformly for integers $N \ge 1$ and $x$, we have
\begin{equation}\label{loc-limit} p(N,x) =  \frac{1}{\sqrt{2\pi N}} e^{-x^2/2N} + O(N^{-1}).\end{equation} In particular
\begin{equation}\label{p-pointwise} p(N,x) \ll N^{-1/2}\end{equation} uniformly.
We have
\begin{equation}\label{poiss-der-1} \sum_x \big|p(N , x) - p(N , x + \ell)\big| \ll |\ell| N^{-1/2},\end{equation} uniformly in $\ell \in \N$.
Uniformly for $0 \le k \le N/2$ and all $x$ we have
\begin{equation}\label{pnk-ineq} \big| p(N - k, x) - p(N, x)\big| \ll k N^{-3/2} + N^{-1}.\end{equation}
Finally, for $-N \le x \le N$ we have
\begin{equation}\label{poisson-large-der} p(N, x) \le e^{-x^2/4N}.\end{equation}
\end{lemma}
\begin{proof}
\cref{loc-limit} is an instance of the local limit theorem for which the standard general reference is the book of Petrov \cite{petrov}. In our specific case we can compute explicitly. In the range $|x| \le N^{2/3}$ this can be handled by Stirling's formula in the form $m! = (m/e)^m(2\pi m)^{1/2}(1 + O(1/m))$ and a computation. The tails $|x| \ge N^{2/3}$ follow from the large deviation estimate \cref{lemmaF.1}.  \cref{p-pointwise} follows immediately from \cref{loc-limit}.

For \cref{poiss-der-1}, by telescoping it suffices to handle the case $\ell = 1$. We compute that $p(N,x + 1) = p(N, x) \frac{N}{N + x + 1}$, so in particular $p(N,x)$ is monotonic decreasing for $x = 0,1,2,\dots$ and also for $x = -1,-2,\dots$, with $p(N,0) = p(N, -1)$. By telescoping, the LHS in \cref{poiss-der-1} (with $\ell = 1$) is simply $2 p(N,0) - p(N, -N)$, and the stated bound follows from \cref{loc-limit}.

The inequality \cref{pnk-ineq} follows from \cref{loc-limit}, the inequality $\sup_x \big| e^{-x^2/2n} - e^{-x^2/2(n - 1)}\big| \ll n^{-1}$ (for $n \ge 2$) and a telescoping sum; for details see \cite[Equation (32)]{DTW-berry-esseen}.

Finally, the large deviation estimate \cref{poisson-large-der} follows from \cref{lemmaF.1}, where a reference is given.
\end{proof}
\begin{corollary}\label{probs-cor}
Let $N \ge 1$ and let $x, y$ be integers with $y \ge 0$. Let $k$, $0 \le k \le N/2$, be an integer. Uniformly in these parameters we have
    \[ \mb{P}\big(\boldbeta(N - k) \in (x, x + y] \cup (-x, -x + y]\big) = \big(\frac{2}{\pi}\big)^{1/2} \frac{y}{\sqrt{N}} e^{-x^2/2N} + O\big( \frac{y^2}{N} + \frac{ky}{N^{3/2}}\big).\]
\end{corollary}
\begin{proof} The result is trivial when $y = 0$, so assume $y \ge 1$. Using \cref{pnk-ineq} one may reduce to handling the case $k = 0$. This in turn follows from \cref{loc-limit} using the mean value estimate $\big|e^{-x^2/2N} - e^{(x+i)^2/2N}\big| \ll i N^{-1/2}$.
\end{proof}
Recall the definition \cref{tau-ladder-def} of $\tau$.
\begin{lemma}
Uniformly for $N \ge 1$ and $y \ge 0$ we have the bound
\begin{equation}\label{last-step} \mb{P}\big(\boldbeta(N) = - y, \; \tau = N\big) \ll N^{-3/2} (1 + y)^{-10}.\end{equation}   
    In particular
    \begin{equation}\label{tau-pointwise} \mb{P}(\tau = N) \ll N^{-3/2}.\end{equation}
\end{lemma}
\begin{proof}
\cref{tau-pointwise} is an immediate consequence of \cref{last-step}. The result is clear when $N = 1$, so suppose $N \ge 2$. To prove \cref{last-step} we use \cref{rand-walk-aux} (and the fact that $\{ \tau = N\}$ is contained in the event $\{\min_{1 \le i \le N-1} \boldbeta(i) > 0\}$), which gives $\mb{P}(\boldbeta(N - 1) = z, \;\tau = N) \ll z N^{-3/2}$, uniformly for $z > 0$. This gives
\[ \mb{P}\big(\boldbeta(N) = -y,\; \tau = N\big) \ll N^{-3/2} \sum_{z > 0} z \mb{P}(\boldxi' = y + z) \ll N^{-3/2} (1 + y)^{-10}. \qedhere \]
\end{proof}

We record one more preliminary lemma. Let $x \ge 0$ and $N \ge 1$ be integers. Denote
\[ F_N(x) := \mb{P}(\boldbeta(N) \ge x) - \mb{P}(\boldbeta(N) \le - x).\]
\begin{lemma}\label{techlem1}
We have the following bounds:
\begin{enumerate}
    \item $F_N(x) \ll N^{-1/2}$ uniformly;
    \item $|F_N(x) - F_{N - k}(x)| \ll k N^{-3/2}$ uniformly for integer $k$ with $0 \le k < N/2$.
\end{enumerate}
\end{lemma}
\begin{proof}
(1) follows from the Berry-Ess\'een theorem. (2) By telescoping it suffices to check this for $k = 1$. This is done using characteristic functions in \cite[Lemma 10]{DTW-berry-esseen}. Both parts could also be done directly using Stirling's formula. 
\end{proof}

We are now ready for the proof of the main result.
\begin{proof}[Proof of \cref{berry-esseen-prop}] Let $x > 0$. By considering the time $k \ge 1$ at which the walk is first non-positive, we have
\[ \mb{P}\big(\boldbeta(N) \ge x,\, \tau > N\big) = \mb{P}(\boldbeta(N) \ge x) - \sum_{k = 1}^{N-1} \sum_{y = 0}^{\infty} \mb{P}\big(\boldbeta(k) = -y,\, \tau = k\big) \mb{P}(\boldbeta(N - k) \ge x + y).\] (Note that since $x > 0$, we cannot have $\tau = N$.) 
Similarly we have
\begin{align*} 0 & = \mb{P}\big(\boldbeta(N) \le -x, \,\tau > N)  = \mb{P}(\boldbeta(N) \le -x) \\ & - \sum_{k = 1}^{N-1} \sum_{y = 0}^{\infty} \mb{P}\big(\boldbeta(k) = -y,\, \tau = k\big) \mb{P}(\boldbeta(N - k) \le -x + y) - \mb{P}\big(\boldbeta(N) \le - x, \, \tau = N\big).\end{align*}
Subtracting the above two expressions and rearranging gives the identity
\begin{equation}\label{main-decomposition} \mb{P}\big(\boldbeta(N) \ge x, \, \tau > N\big) = E_1 - E_2 + E_3 - E_4,\end{equation} where
\[ E_1 = 
F_N(x)\mb{P}(\tau \ge N), \qquad E_2 =  \sum_{k = 1}^{N-1} \mb{P}(\tau = k) \big( F_{N - k}(x) - F_N(x) \big),\]  \[ E_3 = \sum_{k = 1}^{N-1} \sum_{y = 0}^{\infty} \mb{P}(\boldbeta(k) = -y, \tau = k) \mb{P} \big( \boldbeta(N- k) \in [x, x + y) \cup (-x, -x + y] \big) ,\] and $E_4 =  \mb{P}\big(\boldbeta(N) \le -x,\, \tau = N\big)$.
We have $E_1 \ll N^{-1}$ by \cref{techlem1} (1) and \cref{tau-pointwise}. To estimate $E_2$, we split into $1 \le k < N/2$ and $N/2 \le k \le N-1$. In the first range we use \cref{techlem1} (2) and \cref{tau-pointwise}, obtaining a bound $\ll N^{-3/2} \sum_{1 \le k < N/2} k^{-1/2} \ll N^{-1}$. In the second range we use \cref{techlem1} (1), \cref{tau-pointwise} and the triangle inequality, obtaining a bound $\ll \sum_{N/2 \le k \le N - 1} k^{-3/2} (N - k)^{-1/2} \ll N^{-1}$. Immediately from \cref{tau-pointwise} we have the bound $E_4 \ll N^{-3/2}$.

From this discussion and \cref{main-decomposition} we have
\begin{equation*} \mb{P}\big(\boldbeta(N) \ge x,\, \tau > N\big) = E_3 + O(N^{-1}).\end{equation*} 
The analysis of $E_3$, which provides the main term in the asymptotic, is the remaining task.
We first bound the contribution from $N/2 < k \le N - 1$. Using \cref{p-pointwise,last-step}, this contribution is 
\[\ll \sum_{N/2 < k \le N - 1} (N - k)^{-1/2} \sum_{y = 0}^{\infty} y \mb{P}\big(\boldbeta(k) = -y,\, \tau = k\big) \ll \sum_{N/2 < k \le N - 1} (N - k)^{-1/2} k^{-3/2} \ll N^{-1}.\]

For the remaining terms (with $k \le N/2$) in $E_3$ we use \cref{probs-cor}. We first bound the contribution of the error terms there. Using \cref{last-step}, this is 
\[ \ll \sum_{k \le N/2} \sum_{y = 0}^{\infty} \mb{P}(\boldbeta(k) = -y, \tau = k) \big( \frac{y^2}{N} + \frac{ky}{N^{3/2}}\big)\ll \sum_{k \le N/2}  k^{-3/2}  \sum_{y = 0}^{\infty} (1 + y)^{-10} \big( \frac{y^2}{N} + \frac{ky}{N^{3/2}}\big) \ll N^{-1} .\] 

Finally we come to the contribution from the main terms in \cref{probs-cor}, which is 
\[ \big(\frac{2}{\pi}\big)^{1/2}\frac{1}{\sqrt{N}} e^{-x^2/2N} \sum_{k \le N/2} \sum_{y = 0}^{\infty} y\mb{P}\big(\boldbeta(k) = -y,\, \tau = k\big) . \]
By \cref{last-step} we can complete the sum over $k$ back out to $k = \infty$ at a cost of $O(N^{-1})$. The remaining expression is then 
\[ \big(\frac{2}{\pi}\big)^{1/2} \frac{1}{\sqrt{N}} e^{-x^2/2N} \mb{E} |\boldbeta(\tau)|.\]
To conclude the proof of \cref{berry-esseen-prop} we need the fact that $(\frac{2}{\pi})^{1/2} \mb{E} |\boldbeta(\tau)| = h(-1)$. This follows from \cref{spitzer-beta,h0-tilde-val,boldbeta-tau}.\end{proof}

\subsection{Local limit from global}

We will now show that \cref{berry-esseen-prop} implies \cref{dtw-thm}, the crucial local limit theorem for random walks conditioned to stay positive. Note that \cref{berry-esseen-prop} and the fact that the derivative of $f(x) = e^{-x^2/2N}$ is $\ll N^{-1/2}$ pointwise already gives the crude bound
\begin{equation}\label{crude-local}  \mb{P}\big(\boldbeta(N) = x,\, \tau > N) = \mb{P}\big(\boldbeta(N) \ge x,\, \tau > N) - \mb{P}\big(\boldbeta(N) \ge x+1,\, \tau > N) \ll N^{-1}\end{equation}
uniformly in $x$; we will use this in the arguments.

\begin{proof}[Proof of \cref{dtw-thm}] Set $\eps := \frac{1}{40}$. 
We may assume $N$ sufficiently large throughout. We first restrict to the case
\begin{equation}\label{x-lower-upper} N^{1/2 - \eps} \le x \le N^{1/2 + \eps}.\end{equation}
If $x > N^{1/2 + \eps}$ then \cref{dtw-thm} is essentially immediate from the large deviation estimate \cref{lemmaF.1}. If $x < N^{1/2 - \eps}$ then \cref{dtw-thm} follows from \cref{rand-walk-aux}.

Henceforth, we assume that \cref{x-lower-upper} holds. Set $L' := \lfloor N^{1/4}\rfloor$ and $L := \lfloor N^{3/4}\rfloor$ (there is scope for flexibility in these parameters).
First observe that 
\[ \mb{P}\big(\boldbeta(N) = x,\, \tau > N\big) = \mb{P}\big(\boldbeta(N) = x, \, \tau > N - L) + O(N^{-10}).\]
Indeed if $\tau \in (N - L, N]$ then $\inf_{i \in (N - L, N]} \boldbeta(i) \le 0$ and so if $\boldbeta(N) = x$ then $|\sum_{i = N - L + 1}^N \boldxi_i| \ge |x| \ge N^{1/2 - \eps} > L^{3/5}$, and the probability of this can be seen to be tiny using \cref{lemmaF.1}.

We have (with $p(\cdot, \cdot)$ defined as in \cref{pnx-def})
\[ \mb{P}\big(\boldbeta(N) = x, \, \tau > N - L) = \sum_h \mb{P}\big(\boldbeta(N - L) = x - h,\, \tau > N - L\big) p(L,h) . \] 
We first replace the RHS here with 
\begin{equation}\label{aver} \sum_h \mb{P}\big(\boldbeta(N - L) = x - h,\, \tau > N - L\big) \frac{1}{L'} \sum_{h' \in [L']} p(L,h + h') .\end{equation} The error involved in doing this is, using \cref{poiss-der-1,crude-local},
\[ \ll \sup_{h' \in [L']}\sum_{h} | p(L, h) - p(L, h + h')| \cdot  \sup_{x'} \mb{P} \big(\boldbeta(n - L) = x',\, \tau > N - L\big) \ll \frac{L'}{NL^{1/2}} \ll N^{-9/8}.\]
This is an acceptable error and so henceforth we may work with \cref{aver}. Substituting $u = h + h'$ and retaining the variables $u,h'$ gives the equivalent expression
\[ \frac{1}{L'} \sum_u p(L,u) \mb{P}\big( \boldbeta(n) \in x - u + [L'], \,\tau > N - L\big).\]
Now we replace the right-hand probability term by the asymptotic from \cref{berry-esseen-prop}. The contribution of the $O(N^{-1})$ error is
\[ \ll N^{-1} \cdot \frac{1}{L'} \sum_u p(L,u) \ll \frac{1}{N L'} \ll N^{-5/4}, \] which is acceptable. It remains to show that the contribution from the main term is given by the claimed asymptotic in \cref{dtw-thm}, that is to say that 
\begin{equation}\label{main-term-2} \frac{h(-1)}{(n - L)^{3/2} L'} \sum_u p(L, u) \int^{x - u + L'}_{x - u} t e^{-t^2/2(n - L)} \,\mathrm{d}t = N^{-1}h(-1) W\big(\frac{x}{\sqrt{N}}\big) + O(N^{-1 - \eps}).\end{equation} 
The contribution from $|u| \le L^{3/5}$ is negligible using \cref{poisson-large-der}. For $|u| \le L^{3/5} \approx N^{9/20}$, since $x \ge N^{1/2 - \eps}$ we have $x - u, x - u + L' \asymp x$. The derivative of $t e^{-t^2/(N - L)}$ with respect to $t$ in the range $[x - u, x - u + L']$ is therefore $\ll 1 + |x|^2/N \ll N^{2\eps}$, and thus
\[  \frac{1}{L'} \int^{x - u + L'}_{x - u} t e^{-t^2/(N - L)} \,\mathrm{d}t =  (x - u) e^{-(x - u)^2/2(N - L)}  + O(N^{2\eps} L').  \]
When substituted into the LHS of \cref{main-term-2}, the contribution of the error term is $\ll N^{-3/2 + 2\eps} L' \ll N^{-9/8}$. Thus we are left with a main contribution to the LHS of \cref{main-term-2} of
\begin{equation}\label{main-term-3} \frac{h(-1)}{(N - L)^{3/2}} \sum_{|u| \le L^{3/5}} p(L, u)  (x - u) e^{-(x - u)^2/2(N - L)}.\end{equation} 
We replace the first $x - u$ with $x$. The error in doing this is
\[ \ll N^{-3/2}\sum_{|u| \le L^{3/5}} p(L,u) u \ll L^{3/5} N^{-3/2} \ll N^{-1 - 1/20}, \] which is acceptable. Thus we may replace \cref{main-term-3} by
\begin{equation}\label{main-term-4} \frac{h(-1) x}{(N - L)^{3/2}} \sum_{|u| \le L^{3/5}} p(L, u)  e^{-(x - u)^2/2(N - L)}.\end{equation} 
Now 
\begin{align*} \big( & \frac{N}{N - L}\big)^{3/2} e^{-(x - u)^2/2(N - L)}  = e^{-x^2/2N} \Big( e^{-\frac{x^2}{2} \big(\frac{1}{N - L} - \frac{1}{N}\big) + \frac{xu}{N - L} - \frac{u^2}{2(N - L)}}\Big) (1 + O(L/N)) \\ & \qquad = e^{-x^2/2N} \Big( 1 + O \big( \frac{|x|^2 L}{N^2} + \frac{|x|L^{3/5}}{N} + \frac{L^{6/5}}{N}\big)\Big) = e^{-x^2/2N} \big(1 + O(N^{-1/40})\big).\end{align*} 
Since $|x e^{-x^2/2N}| \ll N^{1/2}$, the total contribution of the $O(N^{-1/40})$ error term to \cref{main-term-4} is $\ll N^{-1 -1/40}$, which is acceptable. Thus we may replace \cref{main-term-4} in turn by 
\[ \big(\sum_{|u| \le L^{3/5}} p(L,u)\big) \cdot h(-1) N^{-1} W\big( \frac{x}{\sqrt{N}}\big).\]
To complete the proof, all that remains is to extend the sum back over all $u$, which can be done at negligible cost by \cref{poisson-large-der}.
\end{proof}

\section{Fourier estimates} \label{appB}

\begin{lemma}\label{a1-lem}
Let $R, N$ be integer parameters with $R \le N$, and $R$ sufficiently large. Let $X$ be sampled from $(N,  N+R]$, with $\mb{P}(X = j)$ being proportional to $1/j$ for all $j$. Consider the characteristic function $\phi_X(\xi) := \mb{E} e^{i \xi X}$ for $\xi \in [-\pi, \pi]$. Then we have the following estimates:
\begin{enumerate}
\item for $|\xi| \le 16/R$ we have $|\phi_X(\xi)| \le 1 - c\xi^2 R^2$ for some absolute $c > 0$;
\item for $16/R \le |\xi| \le \pi$ we have $|\phi_X(\xi)| \le \frac{1}{2}$.
\end{enumerate}    
\end{lemma}
\begin{proof} (1) The following argument, given by ChatGPT Pro 5.4 (but shortened, rewritten and checked by us) is cleaner than the one we originally wrote. For $1\le r \le R$, write $p_r := \mb{P}(X = N + r)$. Thus $p_r = \frac{1}{(N + r)L}$ where $L :=  \sum_{r = 1}^R \frac{1}{N + r}$. We have $L \le \frac{R}{N + 1}$, thus $ p_r \ge \frac{N+1}{R(N + r)} \ge \frac{1}{2R}$, using that $r \le R \le N$ in the second inequality.
We have $|\phi_X(\xi)| = \big| \sum_{r = 1}^R p_r e^{i \xi r}\big|$. Since $|\phi_X(\xi)| \le 1$ we have (after a short calculation)
\[ 1 - |\phi_X(\xi)| \ge \frac{1}{2}\big(1 - |\phi_X(\xi)|^2\big) = 2 \sum_{1 \le r < s \le R} p_r p_s \sin^2 \big(\frac{(s - r)\xi}{2}\big) = 2\sum_{h = 1}^R \sin^2 \big(\frac{h\xi}{2}\big)\sum_{1 \le r \le R - h} p_r p_{r + h} .\]
For $h \le R/2$ we have
\[ \sum_{1 \le r \le R - h} p_r p_{r + h} \ge \frac{R - h}{4R^2} \ge \frac{1}{8R}.\]
Therefore
\[ 1 - |\phi_X(\xi)| \ge \frac{1}{4R} \sum_{h = 1}^{\lfloor R/2\rfloor} \sin^2 \big(\frac{h\xi}{2}\big).\]
To bound this below, we use only the terms with $h \le \pi R/16$. For these, using the fact that $\sin x \ge \frac{2}{\pi} x$ for $0 \le x \le \frac{\pi}{2}$, we have $\sin^2(h\xi/2) \gg h^2|\xi|^2$. The result follows. 

(2) For this we use partial summation, noting that
\begin{equation*} L \phi_X(\xi) = \frac{1}{N+R} \sum_{j = N+1}^{N+R} e^{i \xi j} + \int^{N+R}_{N+1} \frac{\mathrm{d}t}{t^2} \Big( \sum_{j = N+1}^t e^{ i \xi j} \Big).\end{equation*}
From this and the geometric series bound $|\sum_{j \in I} e^{i\xi j}| \le \pi |\xi|^{-1}$ (for any interval $I$) we obtain $|\phi_X(\xi)| \le \pi/(N L|\xi|)$ and so, since $L \ge R/(N+R) \ge R/2N$, $|\phi_X(\xi)| \le 2\pi/R|\xi|$. The result follows.
    \end{proof}

\section{Various real-variable inequalities}\label{appC}

\begin{lemma}\label{lemmaB1}
Let $x_1,\dots, x_n \ge 0$ be real. Let $\ell \in \N$. Then 
\[ \big| \big(\sum_{i = 1}^n x_i \big)^{\ell} - \ell! \sum_{1 \le i_1 < \cdots < i_{\ell} \le n} x_{i_1} \cdots x_{i_{\ell}} \big| \le \ell (\ell-1) (\max_i x_i) \big( \sum_{i = 1}^n x_i \big)^{\ell - 1}. \]
\end{lemma}
\begin{proof} 
By expanding out, we have
\[ \big(\sum_{i = 1}^n x_i \big)^{\ell} - \ell! \sum_{1 \le i_1 < \cdots < i_{\ell} \le n} x_{i_1} \cdots x_{i_{\ell}} \] is the sum $\sum_{i_1,\dots, i_{\ell}} x_{i_1} \cdots x_{i_{\ell}}$ in which there is at least one repeated index. The contribution from (for example) $i_{\ell-1} = i_{\ell}$ is
\[ \le (\max x_i) \sum_{i_1,\dots, i_{\ell-1}}x_{i_1} \cdots x_{i_{\ell-1}} = (\max x_i) \big( \sum_{i = 1}^n x_i\big)^{\ell - 1}. \]
Summing over all possibilities for the position of the repeated indices gives the result.\end{proof}

\begin{lemma}\label{calc-1}
Let $c \in (0,1]$. Suppose that $x, y \ge 1$ and $M \ge 0$ are real numbers with $y \ge x + x^c - M$. Then $y \ge x + \frac{1}{2} y^c - M$.
\end{lemma}
\begin{proof}  First suppose that $x < M$; in this case $y \ge y^c \ge  x + y^c - M$ trivially. Now suppose that $x \ge M$. We have $2^{1/c} x \ge 2x > x + x^c - M$. The RHS is non-negative, so we may raise to the power $c$ giving $x^c \ge \frac{1}{2}(x + x^c - M)^c$ and hence $(x + x^c - M) - \frac{1}{2}(x + x^c - M)^c \ge x - M$. The result now follows from the assumption and the fact that $y - \frac{1}{2}y^c$ is an increasing function of $y$ for $y \ge 1$.\end{proof}

\begin{lemma}\label{calc-2}
Let $c \in (0,1]$. Suppose that $x, y \ge 1$ and $M \ge 0$ are real numbers with $y \le x - x^c + M$. Then $y \le x - \frac{1}{2} y^c + 2M$.
\end{lemma}
\begin{proof}
Suppose first that $x \le M$. Then $y + \frac{1}{2} y^c \le \tfrac{3}{2}y \le \tfrac{3}{2}(x + M) \le x + 2M$. Alternatively, suppose $x \ge M$. Then 
\begin{align*}
y + \tfrac{1}{2} y^c \le (x - x^c + M) & + \tfrac{1}{2}(x - x^c + M)^c \\ & \le (x - x^c + M) + \tfrac{1}{2}(x + M)^c \le (x - x^c + M) + x^c = x + M, 
\end{align*}
where we used $2^{1/c} \ge 2$ and $x \ge M$ in the penultimate step.
\end{proof}

\section{Bounded Lipschitz distance}\label{bl-app}
If $\psi \in C(\R/\Z)$ is a function, define \[ \Lip(\psi) := \sup_{x \ne y} \frac{|\psi(x) - \psi(y)|}{\Vert x - y \Vert_{\R/\Z}}.\] If this supremum is finite, we say that $\psi$ is Lipschitz. In this case we define the Lipschitz norm of $\psi$ by 
\[ \Vert \psi \Vert_{\Lip} := \Vert \psi \Vert_{\infty} + \Lip(\psi),\] where $\Vert t \Vert_{\R/\Z}$ denotes distance to the nearest integer. This norm is clearly translation-invariant.

Suppose that $\mu$ is a Borel measure on $\R/\Z$. In our paper $\mu$ will always be positive (meaning $\mu(E) \ge 0$ for any measurable set $E$). Given two such measures $\mu,\nu$ we define 
\[ \dist{\mu}{\nu} := \sup_{\Vert \psi \Vert_{\Lip} \le 1} \big| \int \psi \, \mathrm{d}\mu - \int \psi \, \mathrm{d}\nu\big|.\] This is the \emph{bounded Lipschitz} distance on the space $\mathcal{M}^+(\R/\Z)$ of positive Borel measures. That is satisfies symmetry and the triangle inequality is clear. That $\dist{\mu}{\nu} = 0$ only when $\mu = \nu$ follows from the fact that the characteristic function $1_A$ of any closed set $A \subset \R/\Z$ may be approximated monotonically by Lipschitz functions $f_n(x) := \max(1 - n \operatorname{dist}(x, A), 0)$, so by the monotone convergence theorem we have, if $\dist{\mu}{\nu} = 0$, that $\mu(A) = \nu(A)$ for all $A$.

\begin{lemma}\label{completeness-measures}
$\mathcal{P}(\R/\Z)$ is complete with respect to the bounded Lipschitz distance.
\end{lemma}

\begin{proof}
We use the Riesz representation theorem, namely that $\mathcal{P}(\R/\Z)$ may be identified with the positive linear functionals on $\Lambda : C(\R/\Z) \rightarrow \R$ with $\Lambda(1) = 1$.
Suppose that $(\mu_n)_{n =1}^{\infty}$ is a sequence of probability measures which is Cauchy in the bounded Lipschitz distance. By the Riesz theorem and Alaoglu's theorem that the unit ball of $C(\R/\Z)^*$ is sequentially compact, there is some probability measure $\mu$ such that $\int f \, \mathrm{d}\mu_n \rightarrow \int f \, \mathrm{d}\mu$ along a subsequence of the $\mu_n$, for all $f \in C(\R/\Z)$. By the Cauchy property it then follows that $\int f \, \mathrm{d}\mu_n \rightarrow \int f \, \mathrm{d}\mu$ for all Lipschitz $f$. We need this convergence to hold uniformly for all $f$ in the unit ball $\Vert f \Vert_{\Lip} \le 1$. This follows from the compactness of the latter in the uniform topology, which is an instance of the Arzel\`a-Ascoli theorem.\end{proof}

Now we prove that the space of positive (but not necessarily probability-) measures is complete in the bounded Lipschitz metric. Suppose that $(\mu_n)_{n = 1}^{\infty}$ is a Cauchy sequence. The sequence $\int 1 \, \mathrm{d}\mu_n = \mu_n(\R/\Z)$ is then Cauchy and so tends to some limit $\ell$. If $\ell = 0$, the result is trivial; we have $\mu_n \rightarrow 0$. Set $\ell_n := \mu_n(\R/\Z)$. Consider (after discarding any small values of $n$ for which $\ell_n = 0$) the probability measures $\mu'_n := \frac{1}{\ell_n} \mu_n$.
If $\Vert f \Vert_{\Lip} \le 1$ then 
\begin{align*}
\Big| \int f \, \mathrm{d}\mu'_n - \int f \, \mathrm{d}\mu'_m \Big| & \le \frac{1}{\ell_n}\Big| \int f \, \mathrm{d}\mu_n - \int f \, \mathrm{d}\mu_m\Big| + \Big| \frac{1}{\ell_n} - \frac{1}{\ell_m}\Big| \Big| \int f \, \mathrm{d}\mu_m\Big| \\ & \le \frac{1}{\ell_n} \dist{\mu_n}{\mu_m} + \Big| \frac{1}{\ell_n} - \frac{1}{\ell_m} \Big|\ell_m \rightarrow 0
\end{align*}
as $m,n \rightarrow \infty$. Thus the $\mu'_n$ are a Cauchy sequence and so by \cref{completeness-measures} there is a probability measure $\mu'$ such that $\dist{\mu'_n}{\mu'} \rightarrow 0$. We claim that $\dist{\mu_n}{\ell \mu'} \rightarrow 0$. This is from the fact that if $\Vert f \Vert_{\Lip} \le 1$ then 
\begin{align*}
\Big| \int f \, \mathrm{d}\mu_n - \ell \int f \, \mathrm{d}\mu'\Big| & \le \ell_n \Big| \int f \, \mathrm{d}\mu'_n - \int f \, \mathrm{d}\mu'\Big| + |\ell_n - \ell| \Big| \int f \, \mathrm{d}\mu' \Big| \\ & \le \ell_n \dist{\mu'_n}{\mu'} + |\ell_n - \ell| \rightarrow 0.
\end{align*}
This completes the proof. \vspace*{8pt}

\emph{Convolutions.} Recall that if $f \in C(\R/\Z)$ and $\mu$ is a measure then $f \ast \mu(x) := \int f(x - y) \, \mathrm{d} \mu(y)$.

\begin{lemma}\label{simple-lip-conv}
Let $f$ be a Lipschitz function on $\R/\Z$. Then 
\[ \Vert f \ast (\mu - \nu) \Vert_{\infty} \le \dist{\mu}{\nu} \Vert f \Vert_{\Lip}.\]
\end{lemma}
\begin{proof}
This is essentially immediate from the definition and the fact that $f \ast (\mu - \nu)(x) = \int f_x \, \mathrm{d}\mu - \int f_x \, \mathrm{d}\nu$ where $f_x(y) := f(x - y)$ has the same Lipschitz norm as $f$.
\end{proof}

We conclude with the following lemma. Here, $\Vert \mu \Vert_{\BL} := \dist{\mu}{0}$.

\begin{lemma}\label{smooth-conv-meas}
Suppose that $f \in C^{\infty}(\R/\Z)$ and that $\mu$ is a positive Borel measure with $\Vert \mu \Vert_{\BL} < \infty$. Then $f \ast \mu \in C^{\infty}(\R/\Z)$ and
\begin{equation}\label{f-mu-lip-conv}  \Vert f \ast \mu\Vert_{\Lip} \le \big( \Vert f \Vert_{\infty} +  \Vert f'\Vert_{\infty}\big) \Vert \mu \Vert_{\BL}.\end{equation}
\end{lemma}
\begin{proof}
We have \begin{equation}\label{total-meas-triv} \mu(\R/\Z) = \int 1 \,\mathrm{d}\mu \le \Vert \mu \Vert_{\BL} < \infty.\end{equation} One may check that $f \ast \mu$ is differentiable with derivative $(f \ast \mu)' = f' \ast \mu$ by using
\[ \int \big| f(x - y + h) - f(x - y) - h f'(x - y) \big| \, \mathrm{d}\mu(y)  \ll |h|^2 \Vert f'' \Vert_{\infty} \mu(\R/\Z).\]
The fact that $f \ast \mu \in C^{\infty}(\R/\Z)$ then follows by induction.

For \cref{f-mu-lip-conv} we first observe that using \cref{total-meas-triv} we have $\Vert f \ast \mu \Vert_{\infty} \le \Vert f \Vert_{\infty} \Vert \mu \Vert_{\BL}$. 
Furthermore, we have
\[ \Lip(f \ast \mu) \le \Vert (f \ast \mu)'\Vert_{\infty} = \Vert f' \ast \mu \Vert_{\infty} \le \Vert f' \Vert_{\infty} \Vert \mu \Vert_{\BL}.\] Putting these bounds together gives the result.
\end{proof}

\section{Large deviation bounds}\label{appF}
In this appendix we collect various large deviation bounds of a standard type.

\begin{lemma}\label{binom-ld}
Let $X \samedist \operatorname{Bi}(n,p)$ be a binomial random variable, and let $\eps \le \frac{3}{2}$. Then 
\[ \mb{P} \big( |X - \mb{E} X| \ge \eps \mb{E} X\big) \le 2 e^{-\frac{1}{3}\eps^2\mb{E} X}. \]
\end{lemma}
\begin{proof}
This is \cite[(2.9)]{JLR}.
\end{proof}

\begin{lemma}\label{hoeffding-inequality}
Let $X_1,\dots, X_n$ be independent random variables such that $|X_i| \le H$ for all $i$ and some $H \in (0,\infty)$. Set $X := X_1 + \cdots + X_n$. Then for all $\lambda \in (0,\infty)$ we have
\[ \mb{P} \big( | X - \mb{E} X| \ge \lambda \big) \le 2 e^{-\lambda^2/(2H^2n)}.\]
\end{lemma}
\begin{proof} This is Hoeffding's inequality; see for instance \cite[Theorem 2.8]{BLM}.\end{proof}

Now let $\boldbeta, \boldbeta'$ be random walks with $\Pois(1) - 1$ and $1 - \Pois(1)$ increments respectively. 

\begin{lemma}\label{lemmaF.1}
For $n$ sufficiently large and $0 \le i \le n$ we have $\mb{P}(\boldbeta(n) \ge i) \ll  e^{-i^2/4n}$ and similarly for $\boldbeta'$.
\end{lemma}
\begin{proof}
This is a very standard type of large deviation bound. The second estimate in \cite[Theorem A.1.15]{alon-spencer} gives that \begin{equation}\label{poisson-tail} \mb{P}(\boldbeta(n) \ge i) \le e^{-n \phi(i/n)}\end{equation} where $\phi(x) := (1 + x) \log (1 + x) - x$. The particular statement given here then follows using $\phi(x) \ge x^2/4$, valid for $x \in (0,1)$.

For $\boldbeta'$ we have a similar statement to \cref{poisson-tail} but with $\tilde\phi(x) := x^2/2$, so in this case the result follows straight away.
\end{proof}

\begin{corollary}\label{cor-F2}
With probability $\ge 1 - N^{-10}$, $|\boldbeta(j_2) - \boldbeta(j_1)| \le N^{1/2 - \eta}$ whenever $j_1, j_2 \in [N/8, N]$ and $|j_2 - j_1| \le N^{1 - 5\eta/2}$, and similarly for $\boldbeta'$.
\end{corollary}
\begin{proof} Without loss of generality, $j_2 > j_1$. 
By \cref{lemmaF.1} applied with $n = j_2 - j_1$, for each fixed $j_1, j_2$ we have 
\[ \mb{P} \big( |\boldbeta(j_2) - \boldbeta(j_1)| \ge N^{1/2 - \eta}\big) \ll e^{-N^{\eta/2}/4} \ll N^{-20}.\]
Summing over the $O(N^2)$ choices of $j_1, j_2$ gives the result.     
\end{proof}

The following slight variant of the large deviation bound was required in the proof of \cref{pos-paths-cor}.
\begin{lemma}\label{large-deviation-2}
We have $\mb{P}(\boldbeta(n) = i) \le e^{-i/10}$ uniformly for $i > n/2$, and similarly for $\boldbeta'$.
\end{lemma}
\begin{proof}
We can again use $\mb{P}(\boldbeta(n) \ge i) \le e^{-n\phi(i/n)}$, but now use the inequality $\phi(x) \ge x/6$, valid for $x \ge 1/2$. The inequality holds by a massive margin. A similar proof works for $\boldbeta'$.
\end{proof}

\bibliographystyle{amsplain0}
\bibliography{main.bib}

@misc{schlitt,
  title={Multiplication Tables for Integers with Restricted Prime Factors},
  author={Schlitt, Jeremy},
  note={\url{https://arxiv.org/abs/2603.19212}}
}

@misc{addario-berry-reed-2,
AUTHOR = {Addario-Berry, Louigi and Reed, Bruce},
     TITLE = {Ballot theorems for random walks with finite variance},
note={\url{https://problab.ca/louigi/papers/ballot.pdf}}
}

@misc{grama-xiao,
AUTHOR = {Grama, Ion and Xiao, Hui},
     TITLE = {Local limit theorems for conditioned random walks by the heat kernel approximation},
note={\url{https://arxiv.org/abs/2509.14009}}
}

@book {alon-spencer,
    AUTHOR = {Alon, Noga and Spencer, Joel},
     TITLE = {The probabilistic method},
    SERIES = {Wiley-Interscience Series in Discrete Mathematics and
              Optimization},
   EDITION = {Third},
      NOTE = {With an appendix on the life and work of Paul Erd\H{o}s},
 PUBLISHER = {John Wiley \& Sons, Inc., Hoboken, NJ},
      YEAR = {2008},
     PAGES = {xviii+352},
      ISBN = {978-0-470-17020-5},
   MRCLASS = {60-02 (05C80 60C05 60F99 60G42)},
  MRNUMBER = {2437651},
       DOI = {10.1002/9780470277331},
       URL = {https://doi.org/10.1002/9780470277331},
}

@article {ford-prob,
    AUTHOR = {Ford, Kevin},
     TITLE = {Sharp probability estimates for random walks with barriers},
   JOURNAL = {Probab. Theory Related Fields},
  FJOURNAL = {Probability Theory and Related Fields},
    VOLUME = {145},
      YEAR = {2009},
    NUMBER = {1-2},
     PAGES = {269--283},
      ISSN = {0178-8051,1432-2064},
   MRCLASS = {60G50 (60F05)},
  MRNUMBER = {2520128},
MRREVIEWER = {Vydas\ \v{C}ekanavi\v{c}ius},
       DOI = {10.1007/s00440-008-0168-4},
       URL = {https://doi.org/10.1007/s00440-008-0168-4},
}

@misc{DTW-berry-esseen,
title={Berry-{E}sse\'en inequality for random walks conditioned to stay positive},
author={Denisov, Denis and Tarasov, Alexander and Wachtel, Vitali},
note={arXiv:2412.08502}
}

@article {caravenna,
    AUTHOR = {Caravenna, Francesco},
     TITLE = {A local limit theorem for random walks conditioned to stay
              positive},
   JOURNAL = {Probab. Theory Related Fields},
  FJOURNAL = {Probability Theory and Related Fields},
    VOLUME = {133},
      YEAR = {2005},
    NUMBER = {4},
     PAGES = {508--530},
      ISSN = {0178-8051,1432-2064},
   MRCLASS = {60G50 (60F05 60K05)},
  MRNUMBER = {2197112},
MRREVIEWER = {No\"el\ Veraverbeke},
       DOI = {10.1007/s00440-005-0444-5},
       URL = {https://doi.org/10.1007/s00440-005-0444-5},
}

@book {feller2,
    AUTHOR = {Feller, William},
     TITLE = {An introduction to probability theory and its applications.
              {V}ol. {II}},
   EDITION = {Second},
 PUBLISHER = {John Wiley \& Sons, Inc., New York-London-Sydney},
      YEAR = {1971},
     PAGES = {xxiv+669},
   MRCLASS = {60.00},
  MRNUMBER = {270403},
}

@article{gelfond,
author = {Gel’fond, Aleksandr Osipovich},
title = {An Estimate for the Remainder Term in a Limit Theorem for Recurrent Events},
journal = {Theory of Probability \& Its Applications},
volume = {9},
number = {2},
pages = {299-303},
year = {1964},
doi = {10.1137/1109042},
}

@book {grafakos,
    AUTHOR = {Grafakos, Loukas},
     TITLE = {Classical {F}ourier analysis},
    SERIES = {Graduate Texts in Mathematics},
    VOLUME = {249},
   EDITION = {Second},
 PUBLISHER = {Springer, New York},
      YEAR = {2008},
     PAGES = {xvi+489},
      ISBN = {978-0-387-09431-1},
   MRCLASS = {42-01 (42Bxx)},
  MRNUMBER = {2445437},
MRREVIEWER = {Andreas\ Seeger},
}

@article {PP95,
    AUTHOR = {Pemantle, Robin and Peres, Yuval},
     TITLE = {Critical random walk in random environment on trees},
   JOURNAL = {Ann. Probab.},
  FJOURNAL = {The Annals of Probability},
    VOLUME = {23},
      YEAR = {1995},
    NUMBER = {1},
     PAGES = {105--140},
      ISSN = {0091-1798,2168-894X},
   MRCLASS = {60J15 (60G50)},
  MRNUMBER = {1330763},
MRREVIEWER = {Vadim\ A.\ Ka\u imanovich},
       URL =
              {http://links.jstor.org/sici?sici=0091-1798(199501)23:1<105:CRWIRE>2.0.CO;2-N&origin=MSN},
}

@article {For08,
    AUTHOR = {Ford, Kevin},
     TITLE = {The distribution of integers with a divisor in a given
              interval},
   JOURNAL = {Ann. of Math. (2)},
  FJOURNAL = {Annals of Mathematics. Second Series},
    VOLUME = {168},
      YEAR = {2008},
    NUMBER = {2},
     PAGES = {367--433},
      ISSN = {0003-486X,1939-8980},
   MRCLASS = {11N25 (11N37)},
  MRNUMBER = {2434882},
MRREVIEWER = {D.\ R.\ Heath-Brown},
       DOI = {10.4007/annals.2008.168.367},
       URL = {https://doi.org/10.4007/annals.2008.168.367},
}

@article{VW2009,
  author  = {Vatutin, Vladimir and Wachtel, Vitaly},
  title   = {Local probabilities for random walks conditioned to stay positive},
  journal = {Probability Theory and Related Fields},
  volume  = {143},
  number  = {1--2},
  pages   = {177--217},
  year    = {2009},
  doi     = {10.1007/s00440-007-0113-7}
}

@incollection {addario-berry-reed-1,
    AUTHOR = {Addario-Berry, Louigi and Reed, Bruce},
     TITLE = {Ballot theorems, old and new},
 BOOKTITLE = {Horizons of combinatorics},
    SERIES = {Bolyai Soc. Math. Stud.},
    VOLUME = {17},
     PAGES = {9--35},
 PUBLISHER = {Springer, Berlin},
      YEAR = {2008},
      ISBN = {978-3-540-77199-9; 3-540-77199-9; 978-963-9453-09-8},
   MRCLASS = {05A15},
  MRNUMBER = {2432525},
MRREVIEWER = {Ira\ Gessel},
       DOI = {10.1007/978-3-540-77200-2\_1},
       URL = {https://doi.org/10.1007/978-3-540-77200-2_1},
}

@article{borel,
  author  = {Borel, {\'E}mile},
  title   = {Sur l'emploi du th{\'e}or{\`e}me de {B}ernoulli pour faciliter le calcul d'un infinit{\'e} de coefficients. Application au probl{\`e}me de l'attente {\`a} un guichet},
  journal = {Comptes Rendus de l'Acad{\'e}mie des Sciences},
  year    = {1942},
  volume  = {214},
  pages   = {452--456},
  address = {Paris}
}

@article {bppw,
    AUTHOR = {Brent, Richard and Pomerance, Carl and Purdum, David and
              Webster, Jonathan},
     TITLE = {Algorithms for the multiplication table problem},
   JOURNAL = {Integers},
  FJOURNAL = {Integers. Electronic Journal of Combinatorial Number Theory},
    VOLUME = {21},
      YEAR = {2021},
     PAGES = {Paper No. A92, 19},
      ISSN = {1553-1732},
   MRCLASS = {11Y16 (11B75)},
  MRNUMBER = {4328198},
MRREVIEWER = {J.\ D.\ Dixon},
}

@article {bnpt,
    AUTHOR = {Balazard, Michel and Nicolas, Jean-Louis and Pomerance, Carl and
              Tenenbaum, G\'erald},
     TITLE = {Grandes d\'eviations pour certaines fonctions arithm\'etiques},
   JOURNAL = {J. Number Theory},
  FJOURNAL = {Journal of Number Theory},
    VOLUME = {40},
      YEAR = {1992},
    NUMBER = {2},
     PAGES = {146--164},
      ISSN = {0022-314X,1096-1658},
   MRCLASS = {11K65},
  MRNUMBER = {1149734},
MRREVIEWER = {Antanas\ Laurin\v cikas},
       DOI = {10.1016/0022-314X(92)90036-O},
       URL = {https://doi.org/10.1016/0022-314X(92)90036-O},
}

@article {AT,
    AUTHOR = {Arratia, Richard and Tavar\'e, Simon},
     TITLE = {The cycle structure of random permutations},
   JOURNAL = {Ann. Probab.},
  FJOURNAL = {The Annals of Probability},
    VOLUME = {20},
      YEAR = {1992},
    NUMBER = {3},
     PAGES = {1567--1591},
      ISSN = {0091-1798,2168-894X},
   MRCLASS = {60C05 (05A05 60B15 60F17)},
  MRNUMBER = {1175278},
MRREVIEWER = {Lars\ Holst},
       URL =
              {http://links.jstor.org/sici?sici=0091-1798(199207)20:3<1567:TCSORP>2.0.CO;2-T&origin=MSN},
}

@book{BLM,
  author    = {Boucheron, St{\'e}phane and Lugosi, G{\'a}bor and Massart, Pascal},
  title     = {Concentration Inequalities: A Nonasymptotic Theory of Independence},
  publisher = {Oxford University Press},
  year      = {2013},
  isbn      = {9780199535255}
}

@book {petrov,
    AUTHOR = {Petrov, Valentin Vladimirovich},
     TITLE = {Sums of independent random variables},
    SERIES = {Ergebnisse der Mathematik und ihrer Grenzgebiete [Results in
              Mathematics and Related Areas]},
    VOLUME = {Band 82},
      NOTE = {Translated from the Russian by A. A. Brown},
 PUBLISHER = {Springer-Verlag, New York-Heidelberg},
      YEAR = {1975},
     PAGES = {x+346},
   MRCLASS = {60FXX (60GXX)},
  MRNUMBER = {388499},
}

@article {Koz76,
    AUTHOR = {Kozlov, M. V.},
     TITLE = {The asymptotic behavior of the probability of non-extinction
              of critical branching processes in a random environment},
   JOURNAL = {Teor. Verojatnost. i Primenen.},
  FJOURNAL = {Akademija Nauk SSSR. Teorija Verojatnoste\u i\ i ee
              Primenenija},
    VOLUME = {21},
      YEAR = {1976},
    NUMBER = {4},
     PAGES = {813--825},
      ISSN = {0040-361x},
   MRCLASS = {60J80},
  MRNUMBER = {428492},
MRREVIEWER = {P.\ Jagers},
note={(English translation in Theory of Probability and Its Applications \textbf{21} (1976), 791--804.)}
}

@misc{DTW24,
  title={Expansions for random walks conditioned to stay positive},
  author={Denisov, Denis and Tarasov, Alexander and Wachtel, Vitali},
  note={arXiv:2401.09929}
}

@misc{DTW24b,
  title={Asymptotic expansions for normal deviations of random walks conditioned to stay positive},
  author={Denisov, Denis and Tarasov, Alexander and Wachtel, Vitali},
  note={arXiv:2412.09145}
}

@article {aidekon-shi,
    AUTHOR = {Aidekon, Elie and Shi, Zhan},
     TITLE = {The {S}eneta-{H}eyde scaling for the branching random walk},
   JOURNAL = {Ann. Probab.},
  FJOURNAL = {The Annals of Probability},
    VOLUME = {42},
      YEAR = {2014},
    NUMBER = {3},
     PAGES = {959--993},
      ISSN = {0091-1798,2168-894X},
   MRCLASS = {60J80 (60F05)},
  MRNUMBER = {3189063},
MRREVIEWER = {Yueyun\ Hu},
       DOI = {10.1214/12-AOP809},
       URL = {https://doi.org/10.1214/12-AOP809},
}

@article {spitzer-tauberian,
    AUTHOR = {Spitzer, Frank},
     TITLE = {A {T}auberian theorem and its probability interpretation},
   JOURNAL = {Trans. Amer. Math. Soc.},
  FJOURNAL = {Transactions of the American Mathematical Society},
    VOLUME = {94},
      YEAR = {1960},
     PAGES = {150--169},
      ISSN = {0002-9947,1088-6850},
   MRCLASS = {60.00},
  MRNUMBER = {111066},
MRREVIEWER = {J.\ L.\ Snell},
       DOI = {10.2307/1993283},
       URL = {https://doi.org/10.2307/1993283},
}

@article {EFG16,
    AUTHOR = {Eberhard, Sean and Ford, Kevin and Green, Ben},
     TITLE = {Permutations fixing a {$k$}-set},
   JOURNAL = {Int. Math. Res. Not. IMRN},
  FJOURNAL = {International Mathematics Research Notices. IMRN},
      YEAR = {2016},
    NUMBER = {21},
     PAGES = {6713--6731},
      ISSN = {1073-7928,1687-0247},
   MRCLASS = {05A05 (20B35 20P05)},
  MRNUMBER = {3579977},
       DOI = {10.1093/imrn/rnv371},
       URL = {https://doi.org/10.1093/imrn/rnv371},
}

@article {Rit81,
    AUTHOR = {Ritter, Grant},
     TITLE = {Growth of random walks conditioned to stay positive},
   JOURNAL = {Ann. Probab.},
  FJOURNAL = {The Annals of Probability},
    VOLUME = {9},
      YEAR = {1981},
    NUMBER = {4},
     PAGES = {699--704},
      ISSN = {0091-1798,2168-894X},
   MRCLASS = {60J15 (60F17)},
  MRNUMBER = {624698},
MRREVIEWER = {G\'erard\ Letac},
       URL =
              {http://links.jstor.org/sici?sici=0091-1798(198108)9:4<699:GORWCT>2.0.CO;2-J&origin=MSN},
}

@article {FGK23,
    AUTHOR = {Ford, Kevin and Green, Ben and Koukoulopoulos, Dimitris},
     TITLE = {Equal sums in random sets and the concentration of divisors},
   JOURNAL = {Invent. Math.},
  FJOURNAL = {Inventiones Mathematicae},
    VOLUME = {232},
      YEAR = {2023},
    NUMBER = {3},
     PAGES = {1027--1160},
      ISSN = {0020-9910,1432-1297},
   MRCLASS = {11K65 (05A05 11B13 11N25 11T55)},
  MRNUMBER = {4588563},
MRREVIEWER = {Peter\ Shiu},
       DOI = {10.1007/s00222-022-01177-y},
       URL = {https://doi.org/10.1007/s00222-022-01177-y},
}

@article {sparre-andersen,
    AUTHOR = {Sparre Andersen, Erik},
     TITLE = {On the fluctuations of sums of random variables. {II}},
   JOURNAL = {Math. Scand.},
  FJOURNAL = {Mathematica Scandinavica},
    VOLUME = {2},
      YEAR = {1954},
     PAGES = {195--223},
      ISSN = {0025-5521,1903-1807},
   MRCLASS = {60.0X},
  MRNUMBER = {68154},
MRREVIEWER = {K.\ L.\ Chung},
}

@article {spitzer,
    AUTHOR = {Spitzer, Frank},
     TITLE = {A combinatorial lemma and its application to probability
              theory},
   JOURNAL = {Trans. Amer. Math. Soc.},
  FJOURNAL = {Transactions of the American Mathematical Society},
    VOLUME = {82},
      YEAR = {1956},
     PAGES = {323--339},
      ISSN = {0002-9947,1088-6850},
   MRCLASS = {60.0X},
  MRNUMBER = {79851},
MRREVIEWER = {J.\ L.\ Snell},
       DOI = {10.2307/1993051},
       URL = {https://doi.org/10.2307/1993051},
}

@article {zhang,
    AUTHOR = {Zhang, Yu},
     TITLE = {A power law for connectedness of some random graphs at the
              critical point},
   JOURNAL = {Random Structures Algorithms},
  FJOURNAL = {Random Structures \& Algorithms},
    VOLUME = {2},
      YEAR = {1991},
    NUMBER = {1},
     PAGES = {101--119},
      ISSN = {1042-9832,1098-2418},
   MRCLASS = {05C80 (05C40 60C05)},
  MRNUMBER = {1099582},
MRREVIEWER = {Mark\ R.\ Jerrum},
       DOI = {10.1002/rsa.3240020108},
       URL = {https://doi.org/10.1002/rsa.3240020108},
}

@book {JLR,
    AUTHOR = {Janson, Svante and {\L}uczak, Tomasz and Rucinski, Andrzej},
     TITLE = {Random graphs},
    SERIES = {Wiley-Interscience Series in Discrete Mathematics and
              Optimization},
 PUBLISHER = {Wiley-Interscience, New York},
      YEAR = {2000},
     PAGES = {xii+333},
      ISBN = {0-471-17541-2},
   MRCLASS = {05C80 (60C05 82B41)},
  MRNUMBER = {1782847},
MRREVIEWER = {Mark\ R.\ Jerrum},
       DOI = {10.1002/9781118032718},
       URL = {https://doi.org/10.1002/9781118032718},
}

@article {erdos-mult-1,
    AUTHOR = {Erd\H{o}s, Paul},
     TITLE = {An asymptotic inequality in the theory of numbers ({R}ussian)},
   JOURNAL = {Vestnik Leningrad. Univ.},
  FJOURNAL = {Vestnik Leningrad. Univ.},
    VOLUME = {15},
      YEAR = {1960},
    NUMBER = {13},
     PAGES = {41--49},
   MRCLASS = {10.43},
  MRNUMBER = {126424},
MRREVIEWER = {H.-E.\ Richert},
}

@misc{erdos-mult-2,
  title={Some remarks on number theory ({H}ebrew with {E}nglish summary)},
  author={Erd\H{o}s, Paul},
  note={\url{https://users.renyi.hu/~p_erdos/1955-13.pdf}}
}

@misc{green-sawhney-forthcoming,
  title={The multiplication table problem},
  author={Green, Ben and Sawhney, Mehtaab},
  note={forthcoming}
}

@misc{sound-pk-remark,
  title={Ramanujan and the anatomy of integers},
  author={Soundararajan, K.},
  note={in Srinivasa Ramanujan:
Going Strong at 125,
Part II, Notices AMS 60 (2013), no. 1, 10--22.}
}

@book {flajolet-sedgewick,
    AUTHOR = {Flajolet, Philippe and Sedgewick, Robert},
     TITLE = {Analytic combinatorics},
 PUBLISHER = {Cambridge University Press, Cambridge},
      YEAR = {2009},
     PAGES = {xiv+810},
      ISBN = {978-0-521-89806-5},
   MRCLASS = {05-02 (05A15 05A16 60C05 60E10 82-01)},
  MRNUMBER = {2483235},
       DOI = {10.1017/CBO9780511801655},
       URL = {https://doi.org/10.1017/CBO9780511801655},
}

@article {granville-sedunova-sabuncu,
    AUTHOR = {Granville, Andrew and Sedunova, Alisa and Sabuncu, Cihan},
     TITLE = {The multiplication table constant and sums of two squares},
   JOURNAL = {Acta Arith.},
  FJOURNAL = {Acta Arithmetica},
    VOLUME = {214},
      YEAR = {2024},
     PAGES = {499--522},
      ISSN = {0065-1036,1730-6264},
   MRCLASS = {11N37 (11N36)},
  MRNUMBER = {4772301},
MRREVIEWER = {Kam\ Hung\ Yau},
       DOI = {10.4064/aa230828-19-4},
       URL = {https://doi.org/10.4064/aa230828-19-4},
}

@book {hardy-ramanujan,
    AUTHOR = {Hardy, G. H.},
     TITLE = {Ramanujan. {T}welve lectures on subjects suggested by his life
              and work},
 PUBLISHER = {Cambridge University Press, Cambridge; The Macmillan Company,
              New York},
      YEAR = {1940},
     PAGES = {vii+236},
   MRCLASS = {10.0X},
  MRNUMBER = {4860},
MRREVIEWER = {H.\ Rademacher},
}

@article {dfg,
    AUTHOR = {Diaconis, Persi and Fulman, Jason and Guralnick, Robert},
     TITLE = {On fixed points of permutations},
   JOURNAL = {J. Algebraic Combin.},
  FJOURNAL = {Journal of Algebraic Combinatorics. An International Journal},
    VOLUME = {28},
      YEAR = {2008},
    NUMBER = {1},
     PAGES = {189--218},
      ISSN = {0925-9899,1572-9192},
   MRCLASS = {20B05 (05A05 05E15 20B30 20P05 60B15)},
  MRNUMBER = {2420785},
MRREVIEWER = {Cheryl\ E.\ Praeger},
       DOI = {10.1007/s10801-008-0135-2},
       URL = {https://doi.org/10.1007/s10801-008-0135-2},
}

@article {luzcak-pyber,
    AUTHOR = {{\L}uczak, Tomasz and Pyber, L\'aszl\'o},
     TITLE = {On random generation of the symmetric group},
   JOURNAL = {Combin. Probab. Comput.},
  FJOURNAL = {Combinatorics, Probability and Computing},
    VOLUME = {2},
      YEAR = {1993},
    NUMBER = {4},
     PAGES = {505--512},
      ISSN = {0963-5483,1469-2163},
   MRCLASS = {20B30 (12Y05)},
  MRNUMBER = {1264722},
MRREVIEWER = {J.\ D.\ Dixon},
       DOI = {10.1017/S0963548300000869},
       URL = {https://doi.org/10.1017/S0963548300000869},
}

\end{document}